\documentclass[11pt,reqno]{amsart}
\usepackage{amssymb,amsmath,amsfonts,amsthm,enumerate,stmaryrd}
\usepackage{mathrsfs}
\usepackage{graphicx}

\usepackage[left=.75 in, right=.75 in,top=.75 in, bottom=.75 in]{geometry}
\usepackage{nameref,hyperref,cleveref}
\numberwithin{equation}{section}

\usepackage{color}

\providecommand{\abs}[1]{\left\vert#1\right\vert}
\providecommand{\norm}[1]{\left\Vert#1\right\Vert}
\providecommand{\snorm}[1]{\left[#1\right]}

\providecommand{\ip}[1]{\left(#1 \right)}

\def\nab{\nabla}
\def\dt{\partial_t}
\def\hal{\frac{1}{2}}
\def\ep{\varepsilon}
\def\vchi{\text{\large{$\chi$}}}

\def\p{\partial}

\def\sg{\mathbb{D}}

\def\naba{\nab_{\mathcal{A}}}
\def\diva{\diverge_{\mathcal{A}}}

\def\H1{{_0}H^1((-\ell,\ell))}

\providecommand{\abs}[1]{\left\vert#1\right\vert}
\providecommand{\norm}[1]{\left\Vert#1\right\Vert}

\providecommand{\br}[1]{\langle #1 \rangle}

\def\hal{\frac{1}{2}}
\def\ep{\varepsilon}
\def\vchi{\text{\large{$\chi$}}}

\def\nab{\nabla}
\def\dt{\partial_t}
\def\p{\partial}

\def\naba{\nab_{\mathcal{A}}}
\def\diva{\diverge_{\mathcal{A}}}

\def\sg{\mathbb{D}}
\def\sga{\mathbb{D}_{\mathcal{A}}}

\def\SH0{\mathcal{H}^0((-\ell,\ell))}

\def\sp{X}
\def\an{Y}

\def\A{\mathcal{A}}

\def\C{\mathbb{C}}

\def\H{\mathcal{H}}

\def\L{\mathcal{L}}

\def\n{\mathcal{N}}
\def\N{\mathbb{N}}

\def\R{\mathbb{R}}

\def\st{\;\vert\;}

\def\XXint#1#2#3{{\setbox0=\hbox{$#1{#2#3}{\int}$ }
\vcenter{\hbox{$#2#3$ }}\kern-.6\wd0}}

\DeclareMathOperator{\sgn}{sgn}
\DeclareMathOperator{\tr}{Tr}

\DeclareMathOperator{\diverge}{div}
\DeclareMathOperator{\supp}{supp}

\DeclareMathOperator{\IM}{Im}
\DeclareMathOperator{\RE}{Re}

\DeclareMathOperator{\arcsinh}{arcsinh}

\newtheorem{lemma}{Lemma}[section]
\newtheorem{corollary}[lemma]{Corollary}
\newtheorem{proposition}[lemma]{Proposition}
\newtheorem{theorem}[lemma]{Theorem}
\newtheorem{remark}[lemma]{Remark}

\title{Traveling wave solutions to the free boundary incompressible Navier-Stokes equations}

\author{Giovanni Leoni}
\address{
Department of Mathematical Sciences\\
Carnegie Mellon University\\
Pittsburgh, PA 15213, USA
}
\email[G. Leoni]{giovanni@andrew.cmu.edu}
\thanks{G. Leoni was supported by an NSF Grant (DMS \#1714098).}

\author{Ian Tice}
\address{
Department of Mathematical Sciences\\
Carnegie Mellon University\\
Pittsburgh, PA 15213, USA
}
\email[I. Tice]{iantice@andrew.cmu.edu}
\thanks{I. Tice was supported by an NSF CAREER Grant (DMS \#1653161). }



\begin{document}

\begin{abstract} 
In this paper we study a finite-depth layer of viscous incompressible fluid in dimension $n \ge 2$, modeled by the Navier-Stokes equations.  The fluid is assumed to be bounded below by a flat rigid surface and above by a free, moving interface.  A uniform gravitational field acts perpendicularly to the flat surface, and we consider the cases with and without surface tension acting on the free interface.  In addition to these gravity-capillary effects, we allow for a second force field in the bulk and an external stress tensor on the free interface, both of which are posited to be in traveling wave form, i.e. time-independent when viewed in a coordinate system moving at a constant velocity parallel to the rigid lower boundary.  We prove that, with surface tension in dimension $n \ge 2$ and without surface tension in dimension $n=2$, for every nontrivial traveling velocity there exists a nonempty open set of force and stress data that give rise to traveling wave solutions.  While the existence of inviscid traveling waves is well known, to the best of our knowledge this is the first construction of viscous traveling wave solutions.

Our proof involves a number of novel analytic ingredients, including: the study of an over-determined Stokes problem and its under-determined adjoint, a delicate asymptotic development of the symbol for a normal-stress to normal-Dirichlet map defined via the Stokes operator, a new scale of specialized anisotropic Sobolev spaces, and the study of a pseudodifferential operator that synthesizes the various operators acting on the free surface functions.
\end{abstract}

\maketitle


\section{Introduction}\label{sec_introduction}

\subsection{The equations of motion in Eulerian coordinates}\label{sec_eulerian_form}

In this paper we study traveling wave solutions to the free boundary Navier-Stokes equations, which describe the dynamics of an incompressible, viscous fluid.  We posit that the fluid evolves in an infinite layer-like domain in dimension $n \ge 2$.  Of course, the physically relevant dimensions are $n=2$ and $n=3$, but our analysis works equally well in all dimensions $n \ge 2$, so we present it in this form for the sake of generality.  In order to state the equations of motion and describe the physical features, we must first establish some notation needed to describe the fluid domain and its boundaries.  

We assume throughout the paper that $2 \le n \in \N$, and we make the standard convention of writing points $x \in \R^n$ as $x = (x',x_n) \in \R^{n-1} \times \R$.  The fluid domains of interest to us in this paper are layer-like, with fixed, flat, rigid lower boundaries and moving upper boundaries.  We will assume that the moving upper boundary can be described by the graph of a function.  Given a function $\zeta : \R^{n-1} \to (0,\infty)$ we define the set  
\begin{equation}\label{Omega_zeta}
\Omega_{\zeta} = \{x = (x',x_n)  \in \mathbb{R}^{n} \st  0< x_{n}<\zeta(x^{\prime}) \} \subseteq \R^n
\end{equation}
and we define the $\zeta$ graph surface
\begin{equation}\label{Sigma_eta}
\Sigma_{\zeta}   =\{x\in\mathbb{R}^{n} \st  x_{n}=\zeta(x^{\prime}) \text{ for some } x^{\prime} \in \mathbb{R}^{n-1}\}. 
\end{equation}
In particular, with this notation we have that if $\zeta$ is continuous, then the upper boundary of $\Omega_\zeta$ is $\Sigma_\zeta$, while the flat lower boundary is $\Sigma_{0}   =\{x\in\mathbb{R}^{n} \st x_{n}=0\}.$

With this notation established, we now turn to a description of the equations of motion for time $t \ge 0$.  We assume that in quiescent equilibrium with all external forces and stresses absent, the fluid occupies the flat equilibrium domain 
\begin{equation}\label{omega_eq}
 \Omega_b = \{ x \in \R^n \st 0 < x_n < b \}
\end{equation}
for some equilibrium depth parameter $b \in (0,\infty)$.  We further assume that when perturbed from its equilibrium state the fluid occupies the moving domain $\Omega_{b + \zeta(\cdot,t)}$, where $\zeta : \R^{n-1} \times [0,\infty) \to (-b,\infty)$ is the unknown free surface function.  

We describe the evolution of the fluid for $t \ge 0$ with its velocity field $w(\cdot,t) : \Omega_{b+ \zeta(\cdot,t)} \to \R^n$ and its pressure $P(\cdot,t) : \Omega_{b+ \zeta(\cdot,t)} \to \R$.   We posit that the fluid is acted upon by five distinct forces, two in the bulk (i.e. in $\Omega_{b+ \zeta(\cdot,t)}$), and three on the free surface (i.e. on $\Sigma_{b+\zeta(\cdot,t)}$).  The first bulk force is a uniform gravitational field pointing down: $-\rho \mathfrak{g} e_n \in \R^n$, where $\rho >0$ is the constant fluid density, $\mathfrak{g} >0$ is the gravitational field strength,  and $e_n = (0,\dotsc,1) \in \R^n$ is the vertical unit vector.  The second bulk force is a generic force described for each $t \ge 0$ by the vector field $\tilde{\mathfrak{f}}(\cdot,t) : \Omega_{b+\zeta(\cdot,t)} \to \R^n$.  The first surface force is a constant (in both space and time) external pressure applied by the fluid above $\Omega_{b+\zeta(\cdot,t)}$, which we write as $P_{ext} \in \R$.  The second surface force is generated by an externally applied stress tensor, which we describe for each $t \ge 0$ by a map $\tilde{\mathcal{T}}(\cdot,t) : \Sigma_{b+\zeta(\cdot,t)} \to \R^{n \times n}_{\operatorname*{sym}}$, where 
\begin{equation}
 \R^{n \times n}_{\operatorname*{sym}} = \{M \in \R^{n \times n} \st M = M^{\intercal} \}
\end{equation}
denotes the set of symmetric $n\times n$ matrices.  Note that symmetry is imposed to be consistent with the fact that stresses are typically symmetric in continuum mechanics, but it is not essential in our results and could be dropped.  The third surface force is the surface tension generated by the surface itself, which we model in the standard way as $-\sigma \mathcal{H}(\zeta)$, where $\sigma \ge 0$ is the coefficient of surface tension, and (writing $\nab'$ and $\diverge'$ for the gradient and divergence in $\R^{n-1}$)
\begin{equation}\label{MC_def}
 \mathcal{H}(\zeta) = \diverge'\left( \frac{\nab' \zeta}{\sqrt{1+\abs{\nab' \zeta}^2}} \right)
\end{equation}
is the mean-curvature operator.  

The equations of motion are then 
\begin{equation}\label{ns_euler}
\begin{cases}
 \rho(\dt w + w \cdot \nab w) - \mu \Delta w + \nab P = - \rho \mathfrak{g} e_n +  \tilde{\mathfrak{f}}  & \text{in } \Omega_{b+\zeta(\cdot,t)} \\
 \diverge{w}=0 & \text{in } \Omega_{b+\zeta(\cdot,t)} \\
 (P I- \mu \sg w) \nu = -\sigma \mathcal{H}(\zeta) \nu  + (P_{ext} I + \tilde{\mathcal{T}} ) \nu & \text{on } \Sigma_{b+\zeta(\cdot,t)} \\
 \dt \zeta  = w \cdot \nu \sqrt{1+ \abs{\nab' \zeta}^2} &\text{on } \Sigma_{b+\zeta(\cdot,t)} \\
 w =0 &\text{on } \Sigma_0,
\end{cases}
\end{equation}
where $\rho>0$ is the constant fluid density,  $\mu >0$ is the fluid viscosity,
\begin{equation}\label{sym_grad_def}
 \sg w = (\nab w) + (\nab w)^{\intercal} \in \R^{n \times n}_{\operatorname*{sym}}
\end{equation}
is the symmetrized gradient of $w$, and  
\begin{equation}
 \nu = \frac{(-\nab'\zeta,1)}{\sqrt{1+\abs{\nab'\zeta}^2}} \in \R^n
\end{equation}
denotes the outward pointing unit normal to the surface $\Sigma_{b+\zeta(\cdot,t)}$.  The first two equations in \eqref{ns_euler} are the incompressible Navier-Stokes equations: the first is the Newtonian balance of forces, and the second enforces mass conservation.  The third equation in \eqref{ns_euler} is called the dynamic boundary condition, and it asserts a balance of the forces acting on the free surface.  The fourth equation in \eqref{ns_euler} is called the kinematic boundary condition, as it dictates how the surface evolves with the fluid; note that it may be rewritten as a transport equation in the form
\begin{equation}
 \dt \zeta + \nab' \zeta \cdot w'\vert_{\Sigma_{b+\zeta(\cdot,t)}} = w_n \vert_{\Sigma_{b+\zeta(\cdot,t)}},
\end{equation}
which shows that $\zeta$ is transported by the horizontal component of velocity, $w'$, and driven by the vertical component $w_n$.  The fifth equation  in \eqref{ns_euler} is the usual no-slip condition enforced at rigid, unmoving boundaries.

It will be convenient to eliminate three of the physical parameters in \eqref{ns_euler}.  This may be accomplished in a standard way by dividing by $\rho$, rescaling in space and time, and renaming $b,$ $\sigma$, and the forcing terms.  Doing so, we may assume without loss of generality that $\rho = \mu = \mathfrak{g} =1$.  

Given an open set $\varnothing \neq U \subseteq \R^n$, a scalar $p \in L^2(U)$, and a vector field $u \in H^1(U;\R^n)$, we define the associated stress tensor 
\begin{equation}\label{stress_def}
 S(p,u) := p I - \sg u \in \R^{n \times n}_{\operatorname*{sym}},
\end{equation}
where $I$ denotes the $n \times n$ identity and $\sg u$ is defined as in \eqref{sym_grad_def}.  The stress tensor is of fundamental physical importance, but it also allows us to compactly rewrite terms in \eqref{ns_euler}.  Indeed, the left side of the third equation in \eqref{ns_euler} is $S(P,w) \nu$, and if we extend the divergence to act on tensors in the usual way, then 
\begin{equation}\label{stress_div}
 \diverge S(P,w) = \nab P - \Delta w - \nab \diverge{w},
\end{equation}
so the first equation may be rewritten as 
\begin{equation}
 \dt w + w \cdot \nab w + \diverge S(P,w) = -e_n + \tilde{\mathfrak{f}}.
\end{equation}

Our focus in this paper is the construction of traveling wave solutions to \eqref{ns_euler}, which are solutions that are stationary (i.e. time-independent) when viewed in an inertial coordinate system obtained from the Eulerian coordinates of \eqref{ns_euler} through a Galilean transformation.  Clearly, for the stationary condition to hold, the new coordinate system must be moving at a constant velocity parallel to $\Sigma_0$.  Up to a single rigid rotation fixing $e_n$, we may assume, without loss of generality, that the moving coordinate system's velocity relative to the Eulerian coordinates is $\gamma e_1$ for $e_1 = (1,0,\dotsc,0) \in \R^n$ and $\gamma \in \R \backslash \{0\}$.  Then $\abs{\gamma} >0$ is the speed of the traveling wave and $\sgn(\gamma)$ determines the direction of travel along the $e_1$ axis.  

In the new coordinates the stationary free surface is described by the unknown $\eta : \R^{n-1} \to (-b,\infty)$, which is related to $\zeta$ via $\zeta(x',t) = \eta(x'-\gamma t e_1 )$.  We then posit that 
\begin{multline}
w(x,t) = v(x-\gamma t e_1), \; P(x,t) = q(x-\gamma t e_1) + P_{ext} - (x_n - b),
\\
\tilde{\mathfrak{f}}(x,t) = \mathfrak{f}(x - \gamma t e_1), \text{ and } \tilde{\mathcal{T}}(x,t) = \mathcal{T}(x- \gamma t e_1),
\end{multline}
where $v: \Omega_{b+\eta} \to \R^n$, $q: \Omega_{b+\eta} \to \R$,  $\mathfrak{f}: \Omega_{b+\eta} \to \R^n$, and $\mathcal{T} : \Sigma_{b+\eta} \to \R^{n \times n}_{\operatorname*{sym}}$ are the stationary velocity field, (renormalized) pressure, external force, and external stress, respectively.  In the traveling coordinate system the equations for the unknowns $(v,q,\eta)$, given the data $\mathfrak{f}$ and $\mathcal{T}$, become
\begin{equation}\label{traveling_euler}
\begin{cases}
(v-\gamma e_1) \cdot \nab v  -  \Delta v + \nab q  = \mathfrak{f}  & \text{in } \Omega_{b+\eta} \\
 \diverge{v}=0 & \text{in } \Omega_{b+ \eta} \\
 (q I-  \mathbb{D} v) \n = (\eta -\sigma \mathcal{H}(\eta) )\n  + \mathcal{T} \n & \text{on } \Sigma_{b+\eta} \\
  - \gamma \p_1 \eta = v \cdot \n &\text{on } \Sigma_{b+\eta} \\
 v =0 &\text{on } \Sigma_0,
\end{cases}
\end{equation}
where here we have written 
\begin{equation}\label{normal_def}
 \n = (-\nab' \eta,1) \in \R^n
\end{equation}
for the non-unit normal to $\Sigma_{b+\eta}$.  Note in particular that the renormalization of the pressure has shifted the gravitational force from the bulk, where it manifested as the force vector $-e_n$, to the free surface, where it is manifested as the term $\eta \n$ on the right side of the third equations of \eqref{traveling_euler}.  The renormalization has also completely removed $P_{ext}$.

To provide some context for our result we now consider some of the basic features of the system \eqref{traveling_euler} under some modest assumptions on the solution.  Suppose we have a solution for which $\eta \in H^{5/2}(\R^{n-1})$, $\eta$ is bounded and Lipschitz, and $\inf_{\R^{n-1}} \eta  > -b$.  Note that when $n \in \{2,3\}$ the latter two conditions can be verified via the Sobolev embeddings and a smallness condition on $\norm{\eta}_{H^{5/2}}$, but for higher dimensions this is an auxiliary assumption that would need to be verified through a higher regularity argument, which we ignore for the purposes of the discussion here.  The latter two assumptions on $\eta$ guarantee that $\Omega_{b+\eta}$ is well-defined, open, and connected, and that the surface $\Sigma_{b+\eta}$ is Lipschitz and thus enjoys a trace theory.  We further suppose that $v \in H^2(\Omega_{b+\eta};\R^n) \cap L^\infty(\Omega_{b+\eta};\R^n)$, $q \in H^1(\Omega_{b+\eta})$, $\mathfrak{f} \in L^2(\Omega_{b+\eta};\R^n)$, and $\mathcal{T} \in H^{1/2}(\Sigma_{b+\eta};\R^{n \times n}_{\operatorname*{sym}})$; in other words, we posit that we have a strong solution and that $v$ is bounded.  Note again that the boundedness of $v$ follows from Sobolev embeddings when $n \in \{2,3\}$ but is an auxiliary assumption for $n \ge 4$.  Then an elementary computation, which we record in Proposition \ref{trav_prop} of the appendix, shows that 
\begin{equation}\label{power_balance}
\int_{\Omega_{b+\eta}}  \mathfrak{f} \cdot v - \int_{\Sigma_{b+\eta}} \mathcal{T} \nu \cdot v = \int_{\Omega_{b+\eta}} \hal \abs{\sg v}^2.
\end{equation}
This has a clear physical meaning: the right side is the viscous dissipation rate, and the left side is the power supplied by the external surface stress and bulk force.  These must be in perfect balance for a traveling wave solution to exist.  

In particular, if there are no sources of external surface stress and bulk force, $\mathcal{T} =0$ and $\mathfrak{f}=0$, then \eqref{power_balance} requires that $\sg v =0$ a.e. in $\Omega_{b+\eta}$.  In turn this implies (see, for instance Lemma A.4 of \cite{JTW_2016}) that $v(x) = z + A x$ for $z \in \R^n$ and $A \in \R^{n \times n}$ such that $A^{\intercal} = -A$, but since $v \in H^1(\Omega_{b+\eta};\R^n)$ this requires that $v =0$.  Plugging this into \eqref{traveling_euler} then shows that $\eta =0$ and $q=0$.  The upshot of this analysis is that within the functional framework described above, nontrivial stress or forcing is a necessary condition for the existence of nontrivial solutions to \eqref{traveling_euler}.  We emphasize, though, that this argument depends crucially on the assumed Sobolev inclusions and thus does not eliminate the possibility of nontrivial solutions to \eqref{traveling_euler} with $\mathcal{T} =0$ and $\mathfrak{f}=0$ in other functional frameworks (e.g. H\"{o}lder spaces).

In this paper we identify a Sobolev-based functional framework appropriate for constructing solutions to \eqref{traveling_euler}, and we prove that for every nontrivial wave speed there exists a nonempty open set of forcing and stress data that generate solutions to \eqref{traveling_euler}.  While the existence of traveling wave solutions to the free boundary incompressible Euler equations (the system \eqref{ns_euler} with $\mu=0$ and the no-slip condition replaced with no-penetration)  is well known with and without external sources of stress and forcing (see Section \ref{sec_prev_work}), to the best of our knowledge this paper is the first to construct traveling wave solutions to the free boundary incompressible Navier-Stokes equations.  It is important to account for the viscous case because, while many fluids have small viscosity (or more precisely, the fluid configuration has large Reynolds number), small does not mean zero, so all fluids experience some viscous effects.  Developing the viscous theory also opens the possibility of connecting the viscous and inviscid cases through vanishing viscosity limits, which could potentially yield insight into the zoo of known inviscid solutions.  In particular, it could lead to a selection mechanism for physically relevant inviscid solutions.

\subsection{Previous work}\label{sec_prev_work}

The problems \eqref{ns_euler} and \eqref{traveling_euler} and their variants have attracted enormous attention in the mathematical literature, making a complete review impossible.  We shall attempt here only a brief survey of those results most closely related to the present paper, which in particular means that we will focus exclusively on incompressible fluids in single layer geometries and neglect the expansive literature on other geometric configurations and on compressible fluids.  For more thorough reviews of the literature we refer to the works of Toland \cite{Toland_1996}, Groves \cite{Groves_2004}, and Strauss \cite{Strauss_2010} for the inviscid case and Zadrzy\'{n}ska \cite{Zadrynska_2004} and Shibata-Shimizu \cite{SS_2007} for the viscous case.

The oldest results in this area concern traveling wave solutions to the free boundary Euler equations, the inviscid analogs of \eqref{ns_euler} and \eqref{traveling_euler}.  In this case it is possible to posit that the flow is irrotational, a condition that propagates with the flow.  The rigorous construction of the first periodic solutions was completed in $2D$ by Nekrasov \cite{Nekrasov_1921} and Levi--Civita \cite{Levi-Civita_1924}.  Large amplitude $2D$ periodic solutions, including those with angle $2\pi/3$ satisfying the Stokes conjecture, were constructed later by Krasovski\u{\i} \cite{Krasovskii_1961}, Keady-Norbury \cite{KN_1978}, Toland \cite{Toland_1978}, Amick-Toland \cite{AT_1981}, Amick-Fraenkel-Toland \cite{AFT_1982}, Plotnikov \cite{Plotnikov_2002}, and McLeod \cite{Mcleod_1997}.  For more recent work on Stokes waves see Plotnikov-Toland \cite{PT_2004} and Gravina-Leoni \cite{GL_2018,GL_2019} and the references therein.    Solitary non-periodic solutions in $2D$ were constructed by  Beale \cite{Beale_1977}.

Progress on the $2D$ Euler problem with rotation came much more recently, starting with the construction of periodic rotational traveling waves by Constantin-Strauss \cite{CS_2004}.   Wahl\'{e}n \cite{Wahlen_2006,Wahlen_2006_2} then constructed periodic solutions with surface tension, and Walsh \cite{Walsh_2009,Walsh_2014,Walsh_2014_2} built solutions with density stratification and with surface tension.  Hur \cite{Hur_2008},  Groves-Wahl\'{e}n \cite{GW_2008}, and Wheeler \cite{Wheeler_2013} constructed solitary traveling waves, and Chen-Walsh-Wheeler \cite{CWW_2018, CWW_2019} recently constructed infinite depth solitary waves with and without stratification.  In these results the only forces are due to gravity and surface tension.  Recent work of Walsh-B\"{u}hler-Shatah \cite{WBS_2013} and B\"{u}hler-Shatah-Walsh-Zeng \cite{BSWZ_2016} included effects modeling forcing by wind above the fluid, and Wheeler \cite{Wheeler_2015} studied an applied spatially localized pressure force.

In $3D$ much less is known in the inviscid case.  Periodic irrotational solutions without surface tension were constructed by Iooss-Plotnikov \cite{IP_2009}.  Irrotational solitary waves in $3D$ with surface tension were first constructed by Groves-Sun \cite{GS_2008}, and then by Buffoni-Groves-Sun-Wahl\'{e}n \cite{BGSW_2013} and Buffoni-Groves-Wahl\'{e}n \cite{BGW_2018} with different techniques.

There has also been considerable recent progress on the fully dynamic inviscid and irrotational problem.  For the infinite depth problem Wu \cite{Wu_1997,Wu_1999} constructed local solutions in $2D$ and $3D$, showed almost global existence in $2D$ \cite{Wu_2009}, and then proved global well-posedness in $3D$ \cite{Wu_2011}.  Lannes \cite{Lannes_2005} developed a local well-posedness theory in finite depth in $2D$ and $3D$. In infinite depth Germain-Masmoudi-Shatah \cite{GMS_2012,GMS_2015} proved global well-posedness with gravity only and with surface tension only in $3D$, Deng-Ionescu-Pausader-Pusateri \cite{DIPP_2017} proved global well-posedness with gravity and surface tension in $3D$, and Ionescu-Pusateri \cite{IP_2015,IP_2018} proved global results in $2D$ with and without surface tension.   Wang \cite{Wang_2019} produced global solutions in finite depth with gravity but no surface tension.  Local existence in arbitrary dimension with surface tension was studied in a series of papers by Alazard-Burq-Zuily \cite{ABZ_2011,ABZ_2014,ABZ_2016}.  Alazard-Delort \cite{AD_2015_2,AD_2015} obtained $2D$ global solutions with scattering, while  Hunter-Ifrim-Tataru \cite{HIT_2016} and Ifrim-Tataru \cite{IT_2016} obtained $2D$ global solutions in an alternate framework.  To the best of our knowledge, the only result for layer geometries without the irrotationality assumption is by Zhang-Zhang \cite{ZZ_2008}, who obtained a local existence result in $3D$.

We now turn our attention to the literature associated to the dynamic viscous problem \eqref{ns_euler} in $3D$.  In contrast with the inviscid case, irrotationality is not preserved along viscous flow, so the challenges of vorticity are inherent to the viscous problem.  Beale \cite{Beale_1981} proved local well-posedness without surface tension and global well-posedness with surface tension \cite{Beale_1983}, and Beale-Nishida \cite{BN_1985} derived algebraic decay estimates for the latter solutions.  Solutions in other functional frameworks were produced with surface tension by Tani-Tanaka \cite{TT_1995}, Bae \cite{Bae_2011}, and Shibata-Shimizu \cite{SS_2011} and without surface tension by Abels \cite{Abels_2005_3}.  Guo-Tice \cite{GT_2013_inf,GT_2013_lwp} and Wu \cite{Wu_2014} proved global well-posedness without surface tension and derived decay estimates for solutions.  Masmoudi-Rousset \cite{MR_2017} proved a local-in-time vanishing viscosity result with infinite depth.  For related work on the linearized problem and resolvent estimates in various functional settings we refer to Abe-Shibata \cite{AS_2003,AS_2003_2}, Abels \cite{Abels_2005,Abels_2005_2,Abels_2006},  Abels-Wiegner \cite{AW_2005}, and Abe-Yamazaki \cite{AY_2010}.

Much is also known about periodic solutions to the viscous problem in $3D$.  Nishida-Teramoto-Yoshihara \cite{NTY_2004} constructed global, exponentially decaying solutions with surface tension.  Without surface tension, global solutions with a fixed algebraic decay rate were constructed by Hataya \cite{Hataya_2009} and with almost exponential decay by Guo-Tice \cite{GT_2013_per}.  Tan-Wang \cite{TW_2014} established the vanishing surface tension limit for global solutions.  Remond--Tiedrez-Tice \cite{R-TT_2019} proved global existence of exponentially decaying solutions with generalized bending energies, and Tice \cite{Tice_2018} constructed global decaying solutions with and without surface tension for flows with a gravitational field component parallel to the bottom.

Stationary solutions to $3D$ viscous problems, which correspond to traveling waves with zero velocity ($\gamma =0$ in \eqref{traveling_euler}), have been constructed in various settings. Jean \cite{Jean_1980} and Pileckas \cite{Pileckas_1983,Pileckas_1984} constructed solutions with a partially free boundary, corresponding to a reservoir lying above an infinite channel.  Gellrich \cite{Gellrich_1993} constructed a solution with a completely free boundary and with an affine external pressure.  Nazarov-Pileckas \cite{NP_1999,NP_1999_2}, Pileckas \cite{Pileckas_2002}, and Pileckas-Zaleskis \cite{PL_2003} built solutions in domains that are layer-like at infinity.   Bae-Cho \cite{BC_2000} found stationary solutions for incompressible non-Newtonian fluids.

To the best of our knowledge, there are no results in the literature establishing the existence of traveling wave solutions to the free boundary problem \eqref{ns_euler} with nonzero velocity.  In fixed domains there are a few results for viscous fluids.   In full space Chae-Dubovski\u{\i} \cite{CD_1996} constructed a family of traveling wave solutions to Navier-Stokes, and Freist\"{u}hler \cite{Freistuhler_2014} constructed solutions for a Navier-Stokes-Allen-Cahn system.  Kagei-Nishida \cite{KN_2019} studied traveling waves bifurcating from Poiseuille flow in rigid channels.  We refer also to Escher-Lienstromberg \cite{EL_2018} for traveling wave solutions to a related thin-film problem.

Our goal in the present paper is to construct traveling wave solutions to \eqref{ns_euler} by solving \eqref{traveling_euler} in the presence of bulk forces $\mathfrak{f}$ and surface stresses $\mathcal{T}$.  A simple version of the forcing occurs when we take $\mathfrak{f} =0$ and $\mathcal{T} = \varphi I$ for a scalar function $\varphi$.  In this case $\varphi I$ can be thought of as a spatially localized external pressure source translating in space with velocity $\gamma e_1$ above the fluid.  This is a configuration that has been realized in recent experiments in which a tube blowing air onto the surface of a viscous fluid is uniformly translated above the surface, resulting in the observation of traveling waves on the free surface.  For details of the experiments, some numerical simulations, and approximate models we refer to Akylas-Cho-Diorio-Duncan \cite{DCDA_2011,CDAD_2011}, Masnadi-Duncan \cite{MD_2017}, and Park-Cho \cite{PC_2016,PC_2018}.

\subsection{Reformulation}

A central difficulty in studying \eqref{traveling_euler} is that the domain $\Omega_{b+\eta}$, on which we seek to construct the unknowns $v$ and $q$, is itself unknown since $\eta$ is unknown.  To bypass this difficulty we follow the usual path of reformulating \eqref{traveling_euler} in a fixed domain, which comes at the price of worsening the nonlinearities.  To this end we reformulate the problem in the equilibrium domain \eqref{omega_eq}; in the interest of notational concision, throughout the rest of the paper we will typically drop the subscript $b$ and simply write
\begin{equation}\label{Omega}
\Omega = \Omega_b = \mathbb{R}^{n-1}\times(0,b).
\end{equation}

Given a continuous function $\eta: \R^{n-1} \to (-b,\infty)$ we define the flattening map $\mathfrak{F}: \bar{\Omega} \to \bar{\Omega}_{b+\eta}$ via 
\begin{equation}\label{flat_def}
\mathfrak{F}(x) = (x', x_n(1+\eta(x')/b)) = x + \frac{x_n \eta(x')}{b}e_n.  
\end{equation}
When we need to emphasize the dependence of this map on $\eta$ we will often write $\mathfrak{F}_\eta$ in place of $\mathfrak{F}$.  By construction we have that $\mathfrak{F}(x',0) = (x',0)$ and $\mathfrak{F}(x',b) = (x', b + \eta(x'))$, so $\left. \mathfrak{F} \right\vert_{\Sigma_0} = Id_{\Sigma_0}$ and $\mathfrak{F}(\Sigma_b) = \Sigma_{b+\eta}$.  Moreover, $\mathfrak{F}$ is a bijection with inverse given by $\mathfrak{F}^{-1}(y) =(y', y_n b/(b+ \eta(y')))$ for $y \in \bar{\Omega}_{b+\eta}.$  Thus $\mathfrak{F}$ is a homeomorphism that inherits the regularity of $\eta$ in the sense that if $\eta$ is Lipschitz then $\mathfrak{F}$ is a bi-Lipschitz homeomorphism, and if $\eta \in C^k(\R^{n-1})$ then $\mathfrak{F}$ is a $C^k$ diffeomorphism.

Provided that $\eta$ is differentiable, we may compute and define the following:
\begin{equation}
 \nab \mathfrak{F}(x) = 
\begin{pmatrix}
I_{(n-1) \times (n-1)} & 0_{(n-1) \times 1} \\
x_n \nab'\eta(x') / b & 1 + \eta(x')/b
\end{pmatrix},
\end{equation}
and so we define the Jacobian and inverse Jacobian  $J,K: \Omega \to (0,\infty)$ via
\begin{equation}\label{JK_def}
J = \det \nab \mathfrak{F} = 1 + \eta /b \text{ and } K = 1/J = b/(b+\eta),  
\end{equation}
and we define the matrix $\A: \Omega \to \R^{n \times n}$ via
\begin{equation}\label{A_def}
 \A(x) = (\nab \mathfrak{F}(x))^{-\intercal} = 
\begin{pmatrix}
 I_{(n-1) \times (n-1)} & - K x_n \nab' \eta(x') / b \\
0_{1 \times (n-1)} & K
\end{pmatrix}
=
\begin{pmatrix}
 I_{(n-1) \times (n-1)} & - x_n \nab' \eta(x') / (b + \eta(x')) \\
0_{1 \times (n-1)} & b/(b+ \eta(x'))
\end{pmatrix}.
\end{equation}

We now have all of the ingredients needed to reformulate \eqref{traveling_euler} in $\Omega$.  We assume that $\eta \in C^2(\R^{n-1})$ satisfies $\eta > -b$ and  define the functions $u : \Omega \to \R^n$, $p: \Omega \to \R$, $f: \Omega \to \R^n$, and $T : \Sigma_b \to \R^{n\times n}_{\operatorname*{sym}}$  via $u = v \circ \mathfrak{F}$, $p = q \circ \mathfrak{F}$,  $f = \mathfrak{f} \circ \mathfrak{F}$, and $T = \mathcal{T} \circ \mathfrak{F}$.  Then \eqref{traveling_euler} is equivalent to the following quasilinear system in the fixed domain $\Omega$:
\begin{equation}\label{flattened_system}
\begin{cases}
(u-\gamma e_1) \cdot \naba u  -  \Delta_{\A} u + \nab_{\A} p  = f  & \text{in } \Omega \\
 \diva{u}=0 & \text{in } \Omega \\
 (pI-  \sga u) \n = (\eta -\sigma \mathcal{H}(\eta) )\n  + T \n & \text{on } \Sigma_{b} \\
  u\cdot \n + \gamma \p_1 \eta = 0 &\text{on } \Sigma_{b} \\
 u =0 &\text{on } \Sigma_0.
\end{cases}
\end{equation}
Here we introduce the differential operators $\naba$, $\diva$, and $\Delta_{\A}$ with their actions given via 
\begin{equation}\label{A_op_def_1}
(\naba \psi)_i  = \sum_{j=1}^n \A_{ij} \p_j \psi, \;   
\diva X = \sum_{i,j=1}^n \A_{ij}\p_j X_i, 
\text{ and } (\Delta_{\A} X)_i = \sum_{j=1}^n \sum_{k=1}^n \sum_{m=1}^n \A_{jk}\p_k \left(\A_{jm} \p_m X_i \right) 
\end{equation}
for appropriate $\psi$ and $X$.  We also write 
\begin{equation}\label{A_op_def_2}
(X \cdot \naba u)_i = \sum_{j,k=1}^n X_j \A_{jk} \p_k u_i, \; (\sg_{\A} u)_{ij} =\sum_{k=1}^n \left( \A_{ik} \p_k u_j + \A_{jk} \p_k u_i \right), \text{ and } S_\A(p,u) = pI - \sga u. 
\end{equation}
Allowing $\diva$ to act on symmetric tensors in the usual way, we arrive at the identity
\begin{equation}\label{A_op_def_3}
 \diva S_{\A}(p,u) = \naba p - \Delta_{\A} u - \naba \diva u.
\end{equation}
This allows us to rewrite \eqref{flattened_system} as
\begin{equation} 
\begin{cases}
(u-\gamma e_1) \cdot \naba u  + \diva S_{\A}(p,u)  = f  & \text{in } \Omega \\
 \diva{u}=0 & \text{in } \Omega \\
 S_{\A}(p,u) \n = (\eta -\sigma \mathcal{H}(\eta) )\n  + T \n & \text{on } \Sigma_{b} \\
  u\cdot \n + \gamma \p_1 \eta = 0 &\text{on } \Sigma_{b} \\
 u =0 &\text{on } \Sigma_0.
\end{cases}
\end{equation}

\subsection{Discussion and statement of main results}

We now turn to a discussion of our strategy for producing solutions to \eqref{traveling_euler} by way of \eqref{flattened_system}.  First note that since \eqref{traveling_euler} is not irrotational, the Bernoulli-based surface reformulations often employed in studying the inviscid irrotational problem are not available, and so we are forced to analyze the problem directly in $\Omega$ after the reformulation \eqref{flattened_system}.  The domain $\Omega$ is unbounded, has infinite measure, and non-compact boundaries, which precludes the application of many standard tools in the theory of boundary value problems, including compactness and Fredholm techniques.  The problem \eqref{flattened_system} is quasilinear but has no variational structure, so we are left with the option of constructing solutions by way of some sort of fixed point argument built on the linearization of \eqref{flattened_system}.

An obvious strategy for attacking \eqref{flattened_system} is to employ a technique used in many of the references on the viscous problem from Section \ref{sec_prev_work}, which proceeds as follows.  First we would develop the well-posedness of the linear Stokes system with Navier boundary conditions: 
\begin{equation}\label{intro_stokes_navier}
\begin{cases}
\diverge S(p,u) -\gamma\partial_{1}u=f & \text{in }\Omega\\
\diverge u=g & \text{in }\Omega \\
(S(p,u)e_{n})'   =k',\quad u_{n} =h & \text{on }\Sigma_b \\
u=0 & \text{on }\Sigma_{0},
\end{cases}
\end{equation}
where here we recall that the stress tensor $S(p,u)$ is defined by \eqref{stress_def} and satisfies \eqref{stress_div}. Then we would use this to define a map $(v,q,\zeta) \mapsto (u,p)$ for $(u,p)$ solving \eqref{intro_stokes_navier} with $f,g,h,k'$ determined by $(v,q,\zeta)$, and then we would solve for $\eta$ in terms of $S(p,u) e_n \cdot e_n$ and $(v,q,\zeta)$ via the linearization of the gravity-capillary operator, $I - \sigma \Delta'$ (here $\Delta' = \diverge' \nab' =  \sum_{j=1}^{n-1} \p_j^2$ is the Laplacian on $\R^{n-1}$).  We would then seek to show that the map $(v,q,\zeta) \mapsto (u,p,\eta)$ is contractive on some space. 

Unfortunately, this strategy encounters a serious technical obstruction: while the elliptic system \eqref{intro_stokes_navier} provides control of $\nab p$, it fails to provide control of $p$ itself.  In a bounded domain this can be easily dealt with by simply forcing $p$ to have zero average, which gives control of $p$ via a Poincar\'{e} inequality, but this technique is unavailable in the unbounded domain $\Omega$.  Without control of $p$, the best we can hope for is that the pressure belongs to a homogeneous Sobolev space, in which case solving the elliptic problem 
\begin{equation}\label{intro_surface_elliptic}
 \eta - \sigma \Delta' \eta = S(p,u) e_n \cdot e_n + k_n = p- 2 \p_n u_n + k_n \text{ on } \Sigma_b
\end{equation}
presents a problem due to the appearance of the trace of $p$ onto $\Sigma_b$.  This is indeed a serious problem: in recent work \cite{GT_2019} we extended an earlier $2D$ result due to Strichartz \cite{Strichartz_2016} and proved that the trace space associated to homogeneous Sobolev spaces on $\Omega$ is not a standard Sobolev space, and so not only is the elliptic theory for \eqref{intro_surface_elliptic} unavailable in the literature, it has no hope of producing an $\eta$ amenable to the necessary nonlinear analysis. We are thus forced to abandon this strategy and try something else.  Note, though, that as a byproduct of our analysis we can actually characterize the data for which \eqref{intro_stokes_navier} admits solutions with $p$ under control.  We present this in Section \ref{sec_navier_bcs}, but the resulting spaces are ill-suited for the subsequent nonlinear analysis.

A possible variant of the above strategy, aimed at dealing with the pressure problem, would be to base the linear analysis on the Stokes system with stress boundary conditions:
\begin{equation}\label{intro_stokes_stress}
\begin{cases}
\diverge S(p,u) -\gamma\partial_{1}u=f & \text{in }\Omega\\
\diverge u=g & \text{in }\Omega\\
S(p,u)e_{n}   =k & \text{on }\Sigma_b\\
u=0 & \text{on }\Sigma_{0}.
\end{cases}
\end{equation}
As we show in Section \ref{sec_stress_bcs}, this does provide control of $p$, but the problem now is that in the map $(v,q,\zeta) \mapsto (u,p,\eta)$ the free surface function $\eta$ would have to be reconstructed via the equation
\begin{equation}
 \gamma \p_1 \eta = h - u_n \text{ on } \Sigma_b,
\end{equation}
and when $n \ge 3$ the operator $\gamma \p_1$ on $\R^{n-1}$ is not elliptic.  Thus, this alternate approach cannot work for the most physically relevant case, $n=3$.

We are thus led to seek another strategy.  This begins with the observation that for $f =0$, $T =0$, and any $\gamma \in \R$, a trivial solution to \eqref{flattened_system} is given by the equilibrium configuration $u=0$, $p=0$, $\eta =0$.  Linearizing \eqref{flattened_system} around this solution yields the Stokes system with traveling gravity-capillary boundary conditions:
\begin{equation}\label{intro_stokes_full}
\begin{cases}
\diverge S(p,u) -\gamma\partial_{1}u=f & \text{in }\Omega \\
\diverge u=g & \text{in }\Omega \\
S(p,u)e_{n}  -(\eta -\sigma \Delta' \eta) e_n  =k,\quad u_{n}+\gamma\partial_{1}\eta=h & \text{on }\Sigma_b \\
u=0 & \text{on }\Sigma_{0}.
\end{cases}
\end{equation}
With this in hand, we can state our strategy for solving \eqref{flattened_system}: prove that \eqref{intro_stokes_full} induces an isomorphism $(u,p,\eta) \mapsto (f,g,h,k)$ between appropriate spaces, and use this in conjunction with the implicit function theorem. 

The first key to this strategy is the linear problem \eqref{intro_stokes_full}, but at first glance this appears to be susceptible to the same problem that precludes the fixed-point strategies discussed above: the coupling between $\eta$ and $(u,p)$ occurs in two different boundary conditions.  As such, there is no clear mechanism for decoupling the problem into one for $(u,p)$ with either Navier or stress boundary conditions, and a second one for $\eta$ (with data possibly involving $(u,p)$).  We are thus led to seek a decoupling strategy that synthesizes both boundary conditions simultaneously, and this suggests that as a first step we should understand the over-determined problem 
\begin{equation} \label{intro_stokes_overdet}
\begin{cases}
\diverge S(p,u)-\gamma\partial_{1}u=f & \text{in }\Omega\\
\diverge u=g & \text{in }\Omega \\
S(p,u)e_{n}=k,\quad u_{n}=h & \text{on }\Sigma_b \\
u=0 & \text{on }\Sigma_{0}.
\end{cases}
\end{equation}

The problem \eqref{intro_stokes_overdet} is over-determined in the sense that we specify too many, namely $n+1$, boundary conditions on $\Sigma_b$, when only $n$ are needed to uniquely solve the problem.  Indeed, as a starting point for understanding \eqref{intro_stokes_overdet} we first analyze the applied stress problem \eqref{intro_stokes_stress} in Section \ref{sec_stress_bcs} and show that it induces an isomorphism $(u,p) \mapsto (f,g,k)$ between appropriate $L^2-$based Sobolev spaces (see Theorem \ref{iso_gamma_stokes} for the precise statement).  Consequently, when we specify the extra boundary condition $u_n =h$ on $\Sigma_b$ we should not expect solvability in general.  

In Section \ref{sec_overdetermined} we endeavor to precisely characterize for which data $(f,g,h,k)$ we can uniquely solve \eqref{intro_stokes_overdet}.  If everything were integrable, then a clear necessary compatibility condition would follow from integrating and applying the divergence theorem:
\begin{equation}\label{intro_cc_div}
 \int_{\Omega} g = \int_{\Omega} \diverge{u} = \int_{\Sigma_b} u_n = \int_{\Sigma_b} h.
\end{equation}
However, since we're working in $L^2-$based spaces in the infinite-measure set $\Omega$, we cannot guarantee integrability, and so this compatibility condition manifests in a more subtle way.   In Theorem \ref{cc_divergence} we show that the $L^2$ formulation of \eqref{intro_cc_div} is that 
\begin{equation}
 h - \int_0^b g(\cdot,x_n) dx_n \in \dot{H}^{-1}(\R^{n-1}),
\end{equation}
where $\dot{H}^{-1}(\R^{n-1})$ is the homogeneous Sobolev space of order $-1$ (see \eqref{homogeneous_def} for the definition).  In order to see the connection to \eqref{intro_cc_div} note that if we formally rewrite this as
\begin{equation}
0 =   \int_{\Sigma_b} h - \int_{\Omega} g= \int_{\R^{n-1}} \left(h(x') - \int_0^b g(x',x_n) dx_n \right) dx',
\end{equation}
then this tells us that the Fourier transform of the function $h - \int_0^b g(\cdot,x_n) dx_n$ vanishes at the origin.  The inclusion of this function in $\dot{H}^{-1}(\R^{n-1})$ does not require the Fourier transform to vanish at the origin but it does require that the Fourier transform is not too large near the origin, which is a sort of weak form of vanishing at the origin.   This behavior has been seen before in the analysis of viscous surface waves: we refer, for example, to \cite{BN_1985,GT_2013_inf,TZ_2019}.

The divergence structure $\diverge S(p,u)$ in \eqref{intro_stokes_overdet} and the appearance of $S(p,u)e_n$ on $\Sigma_b$ suggest that another compatibility condition should hold, but it is more subtle since we have no information about $S(p,u)e_n$ on $\Sigma_0$.  To get our hands on it we take a cue from the closed range theorem and identify the formal adjoint of the over-determined problem as the under-determined problem 
\begin{equation}\label{intro_stokes_underdet}
\begin{cases}
\diverge S(q,v)+\gamma\partial_{1}v= f & \text{in }\Omega \\
\diverge v= g & \text{in }\Omega \\
(S(q,v)e_{n})^{\prime}=k' & \text{on }\Sigma_b \\
v=0 & \text{on }\Sigma_{0},
\end{cases}
\end{equation}
which only imposes $n-1$ boundary conditions on $\Sigma_b$.  The compatibility condition can then be derived by integrating solutions to \eqref{intro_stokes_overdet} against functions in the kernel of \eqref{intro_stokes_underdet}.  From our theory of the Stokes problem with stress boundary conditions, developed in Section \ref{sec_stress_bcs}, we know that this kernel can be exactly parameterized by augmenting \eqref{intro_stokes_underdet}, with $f=0,$ $g=0,$ and $k'=0$, with the extra condition 
\begin{equation}\label{intro_adjoint_param}
S(q,v) e_n \cdot e_n = \varphi 
\end{equation}
for $\varphi$ belonging to an appropriate Sobolev space.  This leads us to Theorem \ref{cc_over-det}, which shows that the data $(f,g,h,k)$ must satisfy the second compatibility condition
\begin{equation}\label{intro_cc_psi} 
\int_{\Omega}(f\cdot v-gq)-\int_{\Sigma_b}(k\cdot v-h \varphi) =0
\end{equation}
for all appropriate $\varphi$, where $(v,q)$ are in the kernel of \eqref{intro_stokes_underdet} and satisfy \eqref{intro_adjoint_param}.

Remarkably, the two necessary compatibility conditions identified in Theorems \ref{cc_divergence} and \ref{cc_over-det} are sufficient as well.  We prove this in Theorem \ref{iso_overdetermined}, which establishes that \eqref{intro_stokes_overdet} induces an isomorphism into a space of data satisfying the compatibility conditions.

The formulation of the second compatibility condition \eqref{intro_cc_psi} is hard to work with directly, so the next step is to reformulate it on the Fourier side and eliminate $\varphi$.  We do this, among other things, in Section \ref{sec_fourier} by studying the horizontal Fourier transform of the problem \eqref{intro_stokes_stress}.  This leads to a second-order boundary-value ODE system on $(0,b)$  with the horizontal spatial frequency $\xi \in \R^{n-1}$ as a parameter.  The ODE is not particularly easy to work with, and an interesting feature of our work with it is that we use the solvability of the PDE \eqref{intro_stokes_stress} to deduce some key information about the ODE, which is backward from the usual approach of using the ODE to solve the PDE via Fourier synthesis.  In Proposition \ref{proposition_cc_fourier} we reformulate \eqref{intro_cc_psi} as  
\begin{equation}\label{intro_cc_fourier}
\int_{0}^{b}(\hat{f}(\xi,x_{n})\cdot\overline{V(\xi,x_{n}, -\gamma)} 
- \hat{g}(\xi,x_{n}) \overline{Q(\xi,x_{n}, -\gamma )})dx_{n}
-\hat{k}(\xi)\cdot \overline{V(\xi,b,-\gamma)} + \hat{h}(\xi) = 0 
\end{equation}
for almost every $\xi \in \R^{n-1}$, where $Q$ and $V$ are special solutions to the ODE (see \eqref{QVm_def} for the precise definition), and $\hat{\cdot}$ denotes the horizontal Fourier transform.  

With the solvability criteria of the over-determined problem and \eqref{intro_cc_fourier} in hand, we return to \eqref{intro_stokes_full}.  If a solution $(u,p,\eta)$ exists for given data $(f,g,h,k)$, then \eqref{intro_cc_fourier} requires  that 
\begin{equation}\label{intro_PsiDO}
\rho(\xi) \hat{\eta}(\xi) = \psi(\xi) \text{ for }\xi \in \R^{n-1},
\end{equation}
where $\psi,\rho  : \R^{n-1} \to \C$ are given by 
\begin{equation}
\psi(\xi)=
\int_{0}^{b}\left(\hat{f}(\xi,x_{n})\cdot\overline{V(\xi,x_{n},-\gamma)} - \hat{g}(\xi,x_{n})  \overline{Q(\xi,x_{n},-\gamma)} \right ) dx_{n} - \hat{k}(\xi) \cdot \overline{V(\xi,b,-\gamma)}+\hat{h}(\xi),
\end{equation}
and 
\begin{equation}\label{intro_rho}
\rho(\xi) = 2\pi i \gamma \xi_1 + (1+ 4\pi^2 \sigma \abs{\xi}^2 ) \overline{V_n(\xi,b,-\gamma)}.
\end{equation}
Here for any $\gamma \in \R$, the function $V_n(\cdot,b,\gamma)$ is the symbol associated to the pseudodifferential operator corresponding to the map 
\begin{equation}\label{intro_stress_to_dirichlet}
 H^{s}(\Sigma_b) \ni \varphi \mapsto u_n \vert_{\Sigma_b} \in H^{s+1}(\Sigma_b),
\end{equation}
where $(u,p) \in H^{s+3/2}(\Omega;\mathbb{R}^{n})\times H^{s+1/2}(\Omega)$ solve \eqref{intro_stokes_stress} with $f=0$, $g=0$, and $k=\varphi e_n$ (see Remark \ref{remark_symbol}).  This can be thought of as a Stokes system analog of the Neumann to Dirichlet operator associated to the scalar Laplacian (see Remark \ref{remark_symbol_asymp}), which one might call the normal-stress to normal-Dirichlet operator.  This reveals a remarkable fact: the two boundary conditions for $\eta$ combine via the compatibility condition into a single pseudodifferential equation on $\R^{n-1}$, $\rho(\nab/(2\pi i)) \eta = \check{\psi}$, where the symbol of the operator is a synthesis of the symbols for $\gamma \p_1$, $I - \sigma \Delta'$, and the symbol of the normal-stress to normal-Dirichlet operator.   

Clearly, for there to be any hope of solving the pseudodifferential equation \eqref{intro_PsiDO}, we need detailed information about $V$ and $Q$.  We obtain this in Section \ref{sec_fourier}, where in addition to deriving \eqref{intro_cc_fourier}, we show that $V_n(\xi,b,-\gamma) =0$ if and only if $\xi =0$, and we obtain asymptotic developments of $V$ and $Q$ as $\xi \to 0$ and $\xi \to \infty$.  The latter is particularly tricky as it is predicated on the daunting task of working out closed-form expressions for $V$ and $Q$.  The asymptotics of $V(\xi,b,-\gamma)$ reveal (see Lemma \ref{rho_lemma} for a precise statement) that for $\gamma \neq 0$ we have that  $\rho(\xi)=0$ if and only if $\xi=0$ and that
\begin{equation}\label{intro_rho_asymp}
 \abs{\rho(\xi)}^2 \asymp
\begin{cases}
\xi_1^2 + \abs{\xi}^4 &\text{for } \abs{\xi} \asymp 0 \\
1+ \abs{\xi}^2 &\text{for } \abs{\xi} \asymp \infty
\end{cases}
\text{if }\sigma >0, \text{ while }
 \abs{\rho(\xi)}^2 \asymp
\begin{cases}
\abs{\xi}^2 &\text{for } \abs{\xi} \asymp 0 \\
1+ \abs{\xi}^2 &\text{for } \abs{\xi} \asymp \infty
\end{cases}
\text{if } \sigma=0 \text{ and }n=2.
\end{equation}
Here the condition $\gamma \neq 0$ is essential: the asymptotics are worse near $0$ if $\gamma =0$.

Having derived detailed information about $V$ and $Q$, we can resume the study of the pseudodifferential equation \eqref{intro_PsiDO}.  The first observation is that since $\rho$ vanishes exactly at the origin, $\eta$ is entirely determined via  $\hat{\eta} = \psi / \rho$.  In particular, this means that in contrast with the previously discussed strategies of determining $(u,p)$ from the data and then determining $\eta$ from $(u,p)$ and the data, the path through \eqref{intro_stokes_full} allows for the determination of $\eta$ first in terms of the data, and then the determination of $(u,p)$ from $\eta$ and the data.  The second observation is that the asymptotics \eqref{intro_rho_asymp} dictate the form of the estimates we get for $\hat{\eta}$ when $\gamma \neq 0$: for $\sigma >0$ these read 
\begin{multline}\label{intro_eta_ests_ST}
 \int_{B(0,1)} \frac{\xi_1^2 + \abs{\xi}^4}{\abs{\xi}^2} \abs{\hat{\eta}(\xi)}^2 d\xi +   \int_{B(0,1)^c} (1+ \abs{\xi}^2)^{s+5/2} \abs{\hat{\eta}(\xi)}^2 d\xi   \\
\asymp \int_{B(0,1)} \frac{1}{\abs{\xi}^2} \abs{\psi(\xi)}^2 d\xi +  \int_{B(0,1)^c} (1+\abs{\xi}^2)^{s+3/2}  \abs{\psi(\xi)}^2 d\xi,
\end{multline}
while for $\sigma =0$ and $n =2$ these read 
\begin{multline}\label{intro_eta_ests_no_ST}
 \int_{B(0,1)}  \abs{\hat{\eta}(\xi)}^2 d\xi +   \int_{B(0,1)^c} (1+ \abs{\xi}^2)^{s+5/2} \abs{\hat{\eta}(\xi)}^2 d\xi   \\
\asymp \int_{B(0,1)} \frac{1}{\abs{\xi}^2} \abs{\psi(\xi)}^2 d\xi +  \int_{B(0,1)^c} (1+\abs{\xi}^2)^{s+3/2}  \abs{\psi(\xi)}^2 d\xi.
\end{multline}
Fortunately, the asymptotics of $V$ and $Q$, together with the low frequency bounds provided by \eqref{intro_cc_div}, allow us to control the right-hand sides of these expressions (see Lemma \ref{psi_integral_bounds}).  Unfortunately, while in the case $n =2$ the bounds \eqref{intro_eta_ests_ST} and \eqref{intro_eta_ests_no_ST} do provide standard $H^{s+5/2}(\R^{n-1})$ estimates of $\eta$, when $n \ge 3$ and $\sigma >0$ the bound \eqref{intro_eta_ests_ST} does not provide standard Sobolev control due to the poor low frequency control.  In this case it's not immediately clear that the resulting $\eta$ will be regular enough to use in the nonlinear analysis of \eqref{flattened_system} or, much less, even define a function.  We are thus forced to build specialized Sobolev spaces based on the left side of \eqref{intro_eta_ests_ST} and to study their properties.

To the best of our knowledge, the specialized Sobolev spaces defined via \eqref{intro_eta_ests_ST} have not been studied previously in the literature, so we turn our attention to their properties in Section \ref{sec_specialized_sobolev}.  In order for these spaces (and in turn the estimate \eqref{intro_eta_ests_ST}) to be useful, they must satisfy three mandates.  The first is that the objects in these spaces must be actual functions and not just tempered distributions or equivalence classes of functions modulo polynomials.  The source of this mandate is clear: the $\eta$ determined by the pseudodifferential equation \eqref{intro_PsiDO}, and thus satisfying \eqref{intro_eta_ests_ST}, is meant to serve as the free surface function whose graph determines the fluid domain.  The second is that these spaces must have useful properties such as good embedding and mapping properties.  In particular, as $s$ is made large we need to guarantee at the very least that the functions in these spaces are continuous and decay at infinity.  Third, the spaces have to be well-suited for the nonlinear analysis needed to invoke the implicit function theorem.  For this we need good product-type estimates and composition estimates.

Remarkably, these spaces, which we call $\sp^s(\R^{n-1})$ in Section \ref{sec_specialized_sobolev}, satisfy the above three mandates.  We show in Proposition \ref{specialized_inclusion} that $\sp^s(\R) = H^s(\R)$, so when $n=2$ these spaces are actually the standard $L^2-$Sobolev spaces.  However, when $d \ge 2$ we prove that $H^s(\R^d) \subset \sp^s(\R^d)$, so the new spaces are strictly bigger than the standard spaces.  The Fourier multiplier defining $\sp^s(\R^d)$ for $d \ge 2$ is anisotropic at low frequencies, with a special role played by the $e_1$ direction, which is the direction of motion of the traveling wave.  We prove that this induces a strong anisotropy in the space, which manifests itself in the space not being closed under composition with rigid rotations (see Remark \ref{specialized_aniso_remark}).  In addition to the spaces $\sp^s(\R^{n-1})$, in Section \ref{sec_specialized_sobolev} we also define and derive the basic properties of the spaces $\an^s(\Omega) = H^s(\Omega) + \sp^s(\R^{n-1})$, where here by abuse of notation we view functions in $\sp^s(\R^{n-1})$ as being defined in $\Omega$ in the obvious way.  We need these spaces due to a complication with the pressure that we will describe below.

The importance of $\gamma \neq 0$ here is worth emphasizing.  It is precisely this condition that yields the asymptotics \eqref{intro_rho_asymp} and in turn guarantees the inclusion $\eta \in \sp^s(\R^{n-1})$.  Without it we would only get inclusion in a space for which we could not guarantee the three mandates, and in particular in which we could not guarantee the objects in the space were actual functions.  This all highlights the interesting fact that our technique is capable of producing genuine traveling wave solutions with $\gamma \neq 0$ but is incapable of producing stationary solutions with $\gamma =0$.

Armed with the spaces $\sp^s(\R^{n-1})$ and $\an^s(\Omega)$ and our analysis of \eqref{intro_stokes_stress}, we characterize the solvability of \eqref{intro_stokes_full} in Section \ref{sec_gravity_capillary}.  To do so we first define two Banach spaces for $s \ge 0$.  The first, $\mathcal{X}^s$ defined in \eqref{Xs_def}, is built from the specialized spaces $\sp^s(\R^{n-1})$ and $\an^s(\Omega)$, and is the container space for the solutions: $(u,p,\eta) \in \mathcal{X}^s$.  The second,  $\mathcal{Y}^s$ defined in \eqref{Ys_def}, is the container space for the data: $(f,g,h,k) \in \mathcal{Y}^s$.  This space contains the data space used for the over-determined isomorphism (see Theorem \ref{iso_overdetermined}).  We prove that \eqref{intro_stokes_full} induces an isomorphism from $\mathcal{X}^s$ to $\mathcal{Y}^s$ for each $s \ge 0$ when $\gamma \neq 0$.  This is proved in Theorem \ref{iso_stokes_capillary} when $\sigma >0$ and in Theorem \ref{iso_stokes_capillary_zero} when $\sigma =0$ and $n=2$.  

The reason the dimension plays a role without surface tension (i.e. $\sigma=0$) can be seen by examining $\rho$, the symbol of the pseudodifferential operator given in \eqref{intro_rho}.  When $n=2$ we can take advantage of the fact that $\gamma \p_1$ is an elliptic operator with symbol $2\pi i \gamma \xi_1 = 2 \pi i \gamma \xi$ in $\R$ to get the asymptotics listed in \eqref{intro_rho_asymp} for $\abs{\xi} \asymp \infty$.  However, when $n \ge 3$ the operator $\gamma \p_1$ is not elliptic on $\R^{n-1}$, and since $\sigma=0$, the asymptotics of $V_n(\xi,b,-\gamma)$ derived in Theorems \ref{QVm_zero} and \ref{QVm_infty} only yield
\begin{equation}
 \abs{\rho(\xi)}^2 \asymp
\begin{cases}
 \xi_1^2 + \abs{\xi}^4 &\text{for } \abs{\xi} \asymp 0 \\
 1 + \xi_1^2  &\text{for } \abs{\xi} \asymp \infty.
\end{cases}
\end{equation}
This induces a second, high-frequency anisotropy in the analog of \eqref{intro_eta_ests_ST}.  Our linear techniques can readily extend to this case through the definition of another further specialized scale of spaces beyond $\sp^s(\R^{n-1})$.  Unfortunately, the spaces defined in this manner do not meet the second or third mandates described above, and we are unable to use them to solve the nonlinear problem \eqref{flattened_system}.  As such, we have declined to record this  extension of our linear analysis in the present paper.

The space $\an^s(\Omega)$ appears in these isomorphisms to handle an issue with the pressure.  Indeed, our proofs show that for $(u,p,\eta)$ solving \eqref{intro_stokes_full} for data $(f,g,h,k) \in \mathcal{Y}^s$, we have that $p \in \an^{s+1}(\Omega)$, $\eta \in \sp^{s+5/2}(\R^{n-1})$, and $p-\eta \in H^{s+1}(\Omega)$.  Thus, while the pressure is in the non-standard space $\an^{s+1}(\Omega) = H^{s+1}(\Omega) + \sp^{s+1}(\R^{n-1})$, we characterize precisely the source of this abnormality: $p = \eta + q$ for $q$ in the standard space $H^{s+1}(\Omega)$.  From this we see that the problems with the pressure described above in the discussion of the abandoned fixed-point strategy do not entirely go away.  However, the source of low-frequency bad behavior in the pressure is identified as exactly the bad behavior of $\eta$ at low frequencies, and so if it happens that $\eta$ is actually well-behaved at low frequencies, $p$ must be as well.

We now arrive at the second key to our strategy: the spaces $\mathcal{X}^s$ and $\mathcal{Y}^s$ are amenable to nonlinear analysis.  While the isomorphisms associated to the linearized system \eqref{intro_stokes_full} are interesting in their own right, they are useless in the study of \eqref{flattened_system} if we cannot prove that the nonlinear map from $\mathcal{X}^s$ (or really an open subset thereof) to $\mathcal{Y}^s$ defined by \eqref{flattened_system} is $C^1$.  The first difficulty is seen immediately upon examining the requirements of the space $\mathcal{Y}^s$, which in particular require that the linearized compatibility condition \eqref{intro_cc_div} holds.  This clearly does not hold for the $g$ and $h$ defined by \eqref{flattened_system}.  However, in Proposition \ref{nlin_diverge_ident} we identity a nonlinear variant of \eqref{intro_cc_div} that allows us to switch to an equivalent formulation of \eqref{flattened_system} for which the linear compatibility condition holds.  This allows us to show that the map defined by this slight reformulation of \eqref{flattened_system} is indeed well-defined from $\mathcal{X}^s$ to $\mathcal{Y}^s$.  Then the special nonlinear properties of the spaces  $\sp^s(\R^{n-1})$ and $\an^s(\Omega)$ allow us to prove in Theorem \ref{Xi_well_defd} that this map is indeed $C^1$.  

We thus arrive at the statement of our first main theorem, which establishes the solvability of \eqref{flattened_system} with surface tension ($\sigma >0$) in dimension $n \ge 2$ and without surface tension ($\sigma =0$) in dimension $n=2$.  Before giving the precise statement, a couple of comments on how we treat the bulk forcing and surface stress data are in order.  Our ultimate goal is to solve \eqref{traveling_euler} by way of \eqref{flattened_system}, so in the final part of our analysis we will want to have bulk forcing in \eqref{flattened_system} of the form $\mathfrak{f} \circ \mathfrak{F}_\eta$, where $\mathfrak{F}_\eta$ is the flattening map defined in terms of $\eta$ via \eqref{flat_def},  so that when we compose with $\mathfrak{F}_\eta^{-1}$ we have bulk forcing $\mathfrak{f}$ in the first equation of \eqref{traveling_euler}.  

The minimal assumption on $\mathfrak{f}$ is that it is defined in the domain $\Omega_{b+\eta}$, but this formulation is inconvenient for our analysis because it requires a priori knowledge of $\eta$, which is one of the unknowns we are solving for in terms of $\mathfrak{f}$.  We thus assume that $\mathfrak{f}$ is a priori defined in a fixed larger set that we can guarantee always contains $\Omega_{b+\eta}$, which without loss of generality (thanks to extension operators), we can assume is actually all of $\R^n$.  This is consistent with the usual physical understanding that bulk force fields are defined globally, not just within the set currently occupied by a continuum.  Since we employ the implicit function theorem in our proofs, we then need to show that the map $(\mathfrak{f},\eta) \mapsto \mathfrak{f}\circ \mathfrak{F}_\eta$ is $C^1$, and it is well known (see \cite{IKT_2013} and references therein) that in the context of standard Sobolev spaces this requires the domain for $\mathfrak{f}$ to enjoy one order of regularity more than the codomain (i.e. $H^{s+1}$ for the domain but $H^s$ for the codomain), and we prove in Section \ref{sec_special_nonlinear} that this holds in our context as well.

In some settings it may be advantageous to maintain the minimal regularity for the bulk force ($H^s$ for domain and codomain), and we have identified a special structural assumption on a bulk force field that allows for this.  Indeed, if $f \in H^s(\R^{n-1};\R^n)$ and we define the bounded linear map $L_{\Omega_\zeta} : H^s(\R^{n-1};\R^n) \to H^s(\Omega_\zeta;\R^n)$  via $L_{\Omega_\zeta} f(x) = f(x')$ (see Lemma \ref{sobolev_slice_extension}), then $L_{\Omega_b} f \circ \mathfrak{F}_\eta^{-1}(x) = f(x') = L_{\Omega_{b+\eta}}f(x)$.  In other words, bulk force fields with no $x_n$ dependence are invariant under composition with $\mathfrak{F}_\eta^{-1}$ and thus stay the same as we change from \eqref{flattened_system} to \eqref{traveling_euler}.  The map $f \mapsto L_{\Omega_b} f$ is also linear and thus smooth without any augmentation of regularity in its domain.  

In our formulation of the existence result for \eqref{flattened_system} we have thus chosen to incorporate both types of forces, taking the right side of the first equation in \eqref{flattened_system} to be of the form $\mathfrak{f}\circ \mathfrak{F}_\eta + L_{\Omega_b} f$ for $\mathfrak{f} \in H^{s+1}(\R^n;\R^n)$ and $f \in H^s(\R^{n-1};\R^n)$.  A similar analysis applies to the surface stresses, and we have chosen to consider stresses in the third equation of \eqref{flattened_system} of the form $\mathcal{T} \circ \mathfrak{F}_\eta \vert_{\Sigma_b} + S_{b} T$ for $\mathcal{T} \in H^{s+2}(\R^n; \R^{n\times n}_{\operatorname*{sym}})$, $T \in H^{s+1/2}(\R^{n-1}; \R^{n\times n}_{\operatorname*{sym}})$, and $S_b T(x',b) = T(x')$ (see Lemma \ref{sobolev_slice_extension_surface}).  Here we need to increase the regularity count to $s+2$ for $\mathcal{T}$ so that the map $(\mathcal{T},\eta) \mapsto \mathcal{T} \circ \mathfrak{F}_\eta$ is $C^1$ with values in $H^{s+1}(\Omega; \R^{n\times n}_{\operatorname*{sym}})$, which then allows us to take a trace to arrive in $H^{s+1/2}(\Sigma_b; \R^{n\times n}_{\operatorname*{sym}})$.  Optimal regularity is maintained for $T$, though.  Note also that in the following statement we will refer to the spaces $C^k_b$, $C^k_0$, and  ${_{0}}H^{s}(\Omega;\mathbb{R}^{n})$, defined later in Section \ref{sec_notation}. 

\begin{theorem}[Proved later in Section \ref{sec_main_thms_flat}]\label{main_thm_flat}
Suppose that either $\sigma >0$ and $n \ge 2$ or $\sigma =0$ and $n =2$.  Assume that $n/2 < s \in \N$, let $\mathcal{X}^s$ be as defined by \eqref{Xs_def}, and let $L_\Omega = L_{\Omega_b}$ be as in Lemma \ref{sobolev_slice_extension} and $S_b$ be as defined in Lemma \ref{sobolev_slice_extension_surface}.  Then there exist open sets 
\begin{equation}
\mathcal{U}^s   \subset (\R \backslash \{0\}) \times H^{s+2}(\R^{n} ; \R^{n\times n}_{\operatorname*{sym}})  \times  H^{s+1/2}(\R^{n-1} ; \R^{n\times n}_{\operatorname*{sym}}) \times H^{s+1}(\R^n;\R^n) \times H^s(\R^{n-1};\R^n)
\end{equation}
and $\mathcal{O}^s\subset \mathcal{X}^s$ such that the following hold.
\begin{enumerate}
 \item $(0,0,0) \in \mathcal{O}^s$, and for every $(u,p,\eta) \in \mathcal{O}^s$ we have that 
\begin{equation}
u \in C^{2 + \lfloor s-n/2 \rfloor}_b(\Omega;\R^n),\; p \in C^{1 + \lfloor s-n/2 \rfloor}_b(\Omega),\; \eta \in C^{3 + \lfloor s-n/2 \rfloor}_0(\R^{n-1}), 
\end{equation}
\begin{equation}
\begin{split}
 \lim_{\abs{x'} \to \infty} \p^\alpha u(x) &= 0 \text{ for all } \alpha \in \N^n \text{ such that} \abs{\alpha} \le 2 + \lfloor s-n/2 \rfloor, \text{ and }\\ 
  \lim_{\abs{x'} \to \infty} \p^\alpha p(x) &= 0 \text{ for all } \alpha \in \N^n \text{ such that} \abs{\alpha} \le 1 + \lfloor s-n/2 \rfloor,
\end{split}
\end{equation}
$\max_{\R^{n-1}} \abs{\eta} \le b/2,$ and if $\mathfrak{F}_\eta : \bar{\Omega} \to \bar{\Omega}_{b+\eta}$ denotes the map from \eqref{flat_def}, then $\mathfrak{F}_\eta$ is a bi-Lipschitz homeomorphism and is a $C^{3 + \lfloor s-n/2 \rfloor}$ diffeomorphism.

\item We have that  $(\R \backslash \{0\}) \times \{0\} \times  \{0\} \times \{0\} \times \{0\} \subset \mathcal{U}^s$.

\item For each $(\gamma,\mathcal{T}, T,\mathfrak{f},f) \in \mathcal{U}^s$ there exists a unique $(u,p,\eta) \in \mathcal{O}^s$ classically solving 
\begin{equation}\label{main_thm_flat_0} 
\begin{cases}
(u-\gamma e_1) \cdot \naba u  -  \Delta_{\A} u + \nab_{\A} p  = \mathfrak{f} \circ \mathfrak{F}_\eta + L_\Omega f  & \text{in } \Omega \\
 \diva{u}=0 & \text{in } \Omega \\
 (pI-  \sga u) \n = (\eta -\sigma \mathcal{H}(\eta) )\n  + (\mathcal{T} \circ \mathfrak{F}_\eta \vert_{\Sigma_b} + S_b T) \n & \text{on } \Sigma_{b} \\
  u\cdot \n + \gamma \p_1 \eta = 0 &\text{on } \Sigma_{b} \\
 u =0 &\text{on } \Sigma_0.
\end{cases}
\end{equation}

\item The map $\mathcal{U}^s \ni (\gamma,\mathcal{T},T,\mathfrak{f},f) \mapsto (u,p,\eta) \in \mathcal{O}^s$ is $C^1$ and locally Lipschitz.
\end{enumerate}
\end{theorem}

Note that if $n=2$ in Theorem \ref{main_thm_flat}, then in fact 
\begin{equation}\label{intro_Xs_n=2}
 \mathcal{O}^s  \subseteq \mathcal{X}^s = {_{0}}H^{s+2}(\Omega;\mathbb{R}^{2}) \times H^{s+1}(\Omega) \times H^{s+5/2}(\R),
\end{equation}
and so the solutions belong to standard Sobolev spaces.  It is only in dimension $n \ge 3$ that we need the specialized spaces $\sp^{s+5/2}(\R^{n-1})$ and $\an^{s+1}(\Omega)$, as defined in \eqref{sp_space_def} and \eqref{an_space_def}, respectively.

With Theorem \ref{main_thm_flat}  in hand, we turn our attention back to the original Eulerian problem \eqref{traveling_euler}.  Recall from the discussion at the end of Section \ref{sec_eulerian_form} that Proposition \ref{trav_prop} implies that under some mild Sobolev regularity assumptions on solutions, there cannot exist nontrivial solutions without a nontrivial stress and forcing.  When $n=2$, \eqref{intro_Xs_n=2} shows that our functional framework enforces these mild conditions, and we conclude that there cannot exist nontrivial solutions 
\begin{equation}
 \eta \in H^{s+5/2}(\R) \text{ with } \inf_{\R^{n-1}} \eta > -b, \; v \in {_{0}}H^{s+2}(\Omega_{b+\eta};\mathbb{R}^{2}), \; q \in H^{s+1}(\Omega_{b+\eta})
\end{equation}
to \eqref{traveling_euler} with $1 = n/2 < s\in \N$,  $\mathfrak{f} =0$, and $\mathcal{T}=0$.  However, when $n\ge 3$ the space $\mathcal{X}^s$ (defined in \eqref{Xs_def}) is built from our specialized Sobolev spaces, and so Proposition \ref{trav_prop} is inapplicable.  Our first result on \eqref{traveling_euler} thus addresses the question of whether traveling wave solutions exist within our functional framework without stress and forcing when $n \ge 3$.  In the statement we recall that the spaces $Y^{s}(\Omega_\zeta)$ are defined in \eqref{an_space_def}.

\begin{theorem}[Proved later in Section \ref{sec_main_thms_euler}]\label{main_thm_no_forcing}
Suppose that $\gamma \in \R \backslash \{0\}$, $\sigma >0$, and $n \ge 3$.  Let $s = \lfloor n/2 \rfloor + 1 \in \N$.  There exists $r >0$ such that if $\eta \in \sp^{s+5/2}(\R^{n-1})$, $v \in {_{0}}H^{s+2}(\Omega_{b+\eta};\mathbb{R}^{n})$, and $q \in \an^{s+1}(\Omega_{b+\eta})$  satisfy $\inf_{\R^{n-1}} \eta > -b$, $q-\eta \in H^{s+1}(\Omega_{b+\eta})$, and
\begin{equation} 
\begin{cases}
(v-\gamma e_1) \cdot \nab v  -  \Delta v + \nab q  = 0 & \text{in } \Omega_{b+\eta} \\
 \diverge{v}=0 & \text{in } \Omega_{b+ \eta} \\
 (q I-  \mathbb{D} v) \n = (\eta -\sigma \mathcal{H}(\eta) )\n   & \text{on } \Sigma_{b+\eta} \\
  - \gamma \p_1 \eta = u \cdot \n &\text{on } \Sigma_{b+\eta} \\
 u =0 &\text{on } \Sigma_0,
\end{cases}
\end{equation}
then either $v=0$, $q=0$, and $\eta =0$, or else
\begin{equation}
 \norm{v}_{{_{0}}H^{s+2}} + \norm{q}_{\an^{s+1}} + \norm{\eta}_{\sp^{s+5/2}} +  \norm{q-\eta}_{H^{s+1}} \ge r.
\end{equation}
\end{theorem}

The upshot of this theorem is that if a nontrivial traveling wave solution $(v,q,\eta)$ exists without forcing (i.e. $\mathfrak{f} =0$ and $\mathcal{T}=0$ in \eqref{traveling_euler}), then either the solution does not belong to the stated function spaces, or else it does but must exist outside a ball of known radius.  In particular, we cannot rule out the possible existence of large nontrivial unforced solutions in $\mathcal{X}^s$, though we do not expect them to exist.  We emphasize that this result implies nothing about the existence of unforced solutions in other functional frameworks, such as those built from H\"{o}lder spaces.

Finally, we turn our attention to the existence of forced solutions to \eqref{traveling_euler}.  Note that we continue to consider generalized bulk forces of the form $\mathfrak{f} + L_{\Omega_{b+\eta}} f$ where $L_{\Omega_{b+\eta}}$ is as in Lemma \ref{sobolev_slice_extension}, and we consider generalized surface stresses of the form $\mathcal{T}\vert_{\Sigma_{b+\eta}} + S_{b+\eta}T$, where we write $S_{b+\eta}T(x) = T(x')$.

\begin{theorem}[Proved later in Section  \ref{sec_main_thms_euler}]\label{main_thm_euler}
Suppose that either $\sigma >0$ and $n \ge 2$ or $\sigma =0$ and $n =2$.  Assume that $n/2 < s \in \N$, and let 
\begin{equation}
\mathcal{U}^s   \subset (\R \backslash \{0\}) \times H^{s+2}(\R^{n} ; \R^{n\times n}_{\operatorname*{sym}})  \times  H^{s+1/2}(\R^{n-1} ; \R^{n\times n}_{\operatorname*{sym}}) \times H^{s+1}(\R^n;\R^n) \times H^s(\R^{n-1};\R^n)
\end{equation}
and $\mathcal{O}^s\subset \mathcal{X}^s$ be the open sets from Theorem \ref{main_thm_flat}.  Then for each $(\gamma,\mathcal{T},T,\mathfrak{f},f) \in \mathcal{U}^s$ there exist:
\begin{enumerate}[(i)]
 \item a free surface function $\eta \in \sp^{s+5/2}(\R^{n-1}) \cap C^{3 + \lfloor s-n/2 \rfloor}_0(\R^{n-1})$ such that $\max_{\R^{n-1}} \abs{\eta} \le b/2$ and $\mathfrak{F}_\eta$, defined by \eqref{flat_def}, is a bi-Lipschitz homeomorphism and $C^{3 + \lfloor s-n/2 \rfloor}$ diffeomorphism,
 \item a velocity field $v \in {_{0}}H^{s+2}(\Omega_{b+\eta};\mathbb{R}^{n}) \cap  C^{2 + \lfloor s-n/2 \rfloor}_b(\Omega_{b+\eta};\R^n)$,
 \item a pressure $q \in \an^{s+1}(\Omega_{b+\eta}) \cap C^{1 + \lfloor s-n/2 \rfloor}_b(\Omega_{b+\eta})$, 
 \item constants $C, R>0$ 
\end{enumerate}
such that the following hold.
\begin{enumerate}
 \item $(v,q,\eta)$ classically solve 
\begin{equation}\label{main_thm_euler_0}
\begin{cases}
(v-\gamma e_1) \cdot \nab v  -  \Delta v + \nab q  = \mathfrak{f} + L_{\Omega_{b+\eta}} f  & \text{in } \Omega_{b+\eta} \\
 \diverge{v}=0 & \text{in } \Omega_{b+ \eta} \\
 (q I-  \mathbb{D} v) \n = (\eta -\sigma \mathcal{H}(\eta) )\n  + (\mathcal{T}\vert_{\Sigma_{b+\eta}} + S_{b+\eta}T) \n & \text{on } \Sigma_{b+\eta} \\
  - \gamma \p_1 \eta = v \cdot \n &\text{on } \Sigma_{b+\eta} \\
 v =0 &\text{on } \Sigma_0.
\end{cases}
\end{equation}

\item $(v\circ \mathfrak{F}_\eta, q \circ \mathfrak{F}_\eta, \eta) \in \mathcal{O}^s \subset \mathcal{X}^s$.

\item If $(\gamma_\ast,\mathcal{T}_\ast, T_\ast,\mathfrak{f}_\ast,f_\ast) \in \mathcal{U}^s$ and 
\begin{equation}
\abs{\gamma-\gamma_\ast} +\norm{\mathcal{T}-\mathcal{T}_\ast}_{H^{s+2}} +\norm{T-T_\ast}_{H^{s+1/2}} + \norm{\mathfrak{f}-\mathfrak{f}_\ast}_{H^{s+1}} + \norm{f-f_\ast}_{H^s} < R, 
\end{equation}
then for $(v_\ast,q_\ast,\eta_\ast)$ the corresponding solution triple we have the local Lipschitz estimate 
\begin{multline}
\norm{ (v\circ \mathfrak{E}_\eta, q\circ \mathfrak{E}_\eta, \eta) - (v_\ast\circ \mathfrak{E}_{\eta_\ast}, q_\ast \circ \mathfrak{E}_{\eta_\ast}, \eta_\ast) }_{\mathcal{X}^s} \\
\le C \left(\abs{\gamma-\gamma_\ast} + \norm{\mathcal{T}-\mathcal{T}_\ast}_{H^{s+2}} +\norm{T-T_\ast}_{H^{s+1/2}} + \norm{\mathfrak{f}-\mathfrak{f}_\ast}_{H^{s+1}} + \norm{f-f_\ast}_{H^s} \right).
\end{multline}

\end{enumerate}
\end{theorem} 

We conclude with a couple of remarks about Theorem \ref{main_thm_euler}.  First note that the functional framework requires that $\eta \to 0$, $\mathfrak{F}_\eta \to I$, $v \to 0$, and $q \to 0$ as $\abs{x'} \to \infty$.  This means that our traveling wave solutions correspond to what are called solitary waves in the inviscid traveling wave literature.  Second, note that solutions with different free surface functions, say $\eta$ and $\eta_\ast$, have velocities and pressures defined in different domains, $\Omega_{b+\eta}$ and $\Omega_{b+\eta_\ast}$ respectively, so there is no natural way to compare the velocities and pressures with Sobolev norms.  In the local Lipschitz estimate of the third item we have chosen to measure the difference in velocity and pressure by pulling back to the flattened domain $\Omega$ and using the $\mathcal{X}^s$ norm, which we believe is a reasonable metric given how our solutions are constructed.  
Third, note that while we have treated the bulk force and surface stress as distinct, in some cases it is possible shift terms from one to the other in the same way that we shifted the gravitational force from the bulk to the boundary.  Indeed, if $\mathfrak{f} = \mathfrak{f}_0 + \nab \psi$, then the potential gradient term can be shifted to the boundary by redefining the pressure via $q \mapsto q -\psi$ and the stress via $\mathcal{T} \mapsto \mathcal{T} -\psi I$, which leaves $\mathfrak{f}_0$ in place of $\mathfrak{f}$ in the bulk forcing.  The regularity requirements for $\psi$ are the same, though: we need $\psi \in H^{s+2}(\R^n)$ to guarantee that the bulk force term satisfies $\nab \psi \in H^{s+1}(\R^n;\R^n)$ and the stress term satisfies $\psi I \in H^{s+2}(\R^n;\R^{n \times n}_{\operatorname*{sym}})$.

\subsection{Notational conventions}\label{sec_notation}

Here we record some notational conventions used throughout the paper.   We always write $2 \le n \in \N$ for the dimension of the fluid domain $\Omega$.  We will also need to talk about function spaces defined on other sets, and in particular on subsets of $\p \Omega$.  To avoid confusion and tedious appearances of $n-1$, we will often describe these other sets as subsets of $\R^d$ for $1 \le d \in \N$.  In other words $d \ge 1$ is a generic dimensional parameter, and $n \ge 2$ always refers to the dimension of the fluid.

We write $\mathscr{S}(\R^d)$ for the usual Schwartz class of complex-valued functions and $\mathscr{S}'(\R^d)$ for the space of tempered distributions.  We define the Fourier transform, $\hat{\cdot}$, and inverse Fourier transform, $\check{\cdot}$, on $\R^d$ via  
\begin{equation}\label{FT_def}
 \hat{f}(\xi) = \int_{\R^d} f(x) e^{-2\pi i x\cdot \xi} dx \text{ and } \check{f}(x) = \int_{\R^d} f(\xi) e^{2\pi i x\cdot \xi} d\xi.
\end{equation}
By employing the Parseval and Tonelli-Fubini theorems, we extend \eqref{FT_def} to horizontal Fourier transforms acting on functions defined on $\Omega$ via 
\begin{equation}\label{FT_hor_def}
 \hat{f}(\xi,x_n) = \int_{\Omega}   f(x',x_n) e^{-2\pi i x' \cdot \xi} dx \text{ for } \xi' \in \R^{n-1}, \text{ and } \check{f}(x) = \check{f}(x',x_n) = \int_{\Omega} f(\xi,x_n) e^{2\pi i x' \cdot \xi} d\xi.
\end{equation}

For $k\in \N$, an open set $\varnothing \neq U \subseteq \R^d$, and a finite dimensional inner-product space $W$ we define the usual $L^2-$Sobolev space 
\begin{equation}
 H^k(U;W) = \{f: U \to W \st \p^\alpha f \in L^2(U;W) \text{ for } \abs{\alpha} \le k\}.
\end{equation}
For $0 \le s \in \R$ we then let $H^s(U;W)$ denote the fractional spaces obtained by interpolation.  In the event that $U = \R^d$ we take the norm on these spaces to be the standard one defined on the Fourier side, and we also extend to $s \in (-\infty,0) \subset \R$ in the usual way.  When the target is $W = \R$ we will usually drop this in the notation, writing simply $H^s(U)$.  For $0 < r \in \R$ we define the real-valued negative homogeneous Sobolev space to be
\begin{equation}\label{homogeneous_def}
 \dot{H}^{-r}(\R^d) = \{f \in \mathscr{S}(\R^d) \st f = \bar{f}, \hat{f} \in L^1_{loc}(\R^d), \text{ and }  \snorm{f}_{\dot{H}^{-r}} < \infty \}
\end{equation}
for 
\begin{equation}
 \snorm{f}_{\dot{H}^{-r}}^2 = \int_{\R^d} \frac{1}{\abs{\xi}^{2r}} \abs{\hat{f}(\xi)}^2 d\xi.
\end{equation}

Suppose now that $\zeta : \R^{n-1} \to \R$ is Lipschitz and satisfies $\inf \zeta >0$.  For $1/2 < s \in \R$ we can use trace theory to define 
\begin{equation}
{_{0}}H^{s}(\Omega_{\zeta};\R^n ) = \{u\in H^{s}(\Omega_{\zeta};\R^n) \st u=0\text{ on } \Sigma_{0}\},
\end{equation}
where the equality $u=0$ on $\Sigma_{0}$ is in the sense of traces.  We will mostly employ these spaces in the case $\Omega_\zeta = \Omega$ (i.e. $\zeta = b$), in which case we will need the following extra definitions.  Recall that the symmetrized gradient $\sg$ is defined by \eqref{sym_grad_def}.  We endow ${_{0}}H^{1}(\Omega;\R^n )$ with the inner-product 
\begin{equation}
(u,v)_{{_{0}}H^{1}} = \hal \int_{\Omega}\sg{u}:\sg{v},
\end{equation}
which, thanks to Korn's inequality (see Lemma \ref{korn}), is indeed an inner-product and generates the same topology as the standard $H^1$ norm.  We define the closed subspace of solenoidal vector fields to be
\begin{equation}\label{H1 div zero}
{_{0}}H_{\sigma}^{1}(\Omega;\mathbb{R}^{n}) = \{ u \in {_{0}}H^{1}(\Omega;\mathbb{R}^{n}) \st \diverge u=0\}. 
\end{equation}
Then ${_{0}}H^{1}_\sigma(\Omega;\R^n)$ is a Hilbert space with the same inner-product.  In what follows we will often use the fact that by the symmetry of $\sg{u}$,
\begin{equation}\label{symmetrized product}
\int_{\Omega }\mathbb{D}u:\nabla v =\frac{1}{2}\int_{\Omega} \sg{u}:\sg{v}  
\end{equation}
for all $u,v\in H^{1}(\Omega;\mathbb{R}^{n})$.

Given $k \in \N$, a real Banach space $V$, and an open set $\varnothing \neq U \subseteq \R^d$, we define the Banach space 
\begin{equation}
C^k_b(U;V) = \{f: U \to V \st f \text{ is k-times continuously differentiable, and} \norm{f}_{C^k_b} < \infty\},  
\end{equation}
where 
\begin{equation}
 \norm{f}_{C^k_b} = \sum_{\abs{\alpha} \le k} \sup_{x \in U} \norm{\p^\alpha f(x)}_{V}.
\end{equation}
When $V = \R$ we will typically write $C^k_b(U) = C^k_b(U;\R)$.  We also define $C^k_0(\R^d;V) \subset C^k_b(\R^d;V)$ to be the closed subspace
\begin{equation}
 C^k_0(\R^d;V) = \{ f\in C^k_b(\R^d;V) \st \lim_{\abs{x} \to \infty} \p^\alpha f(x) =0 \text{ for all }\abs{\alpha}\le k\},
\end{equation}
which we endow with the norm from $C^k_b(\R^d;V)$.  Again we typically write $C^k_0(\R^d) = C^k_0(\R^d;\R)$.

Finally, we introduce a convenient abuse of notation that we will use throughout the paper.  The hyperplane $\Sigma_b = \{x \in \R^n \st x_n=b \}$ is canonically diffeomorphic to $\R^{n-1}$ via the map $\Sigma_b \ni (x',b) \mapsto x' \in \R^{n-1}$.  Using this, we can identify $H^s(\Sigma_b;W)$ with $H^s(\R^{n-1};W)$ for any finite dimensional inner-product space $W$.  This abuse of notation is justified by a gain in brevity, as it allows us to write $f(x')$ in place of $f(x',b)$ for $x' \in \R^{n-1}$, etc.  It also allows us to use the Fourier transform on $\Sigma_b$ in a natural way.

\subsection{Plan of paper}

In Section \ref{sec_stress_bcs} we study the Stokes problem with stress boundary conditions \eqref{intro_stokes_stress} and characterize its solvability in standard $L^2-$based Sobolev spaces.  In Section \ref{sec_overdetermined} we study the over-determined problem \eqref{intro_stokes_overdet}, derive its compatibility conditions, and characterize its solvability in Sobolev spaces.  In Section \ref{sec_fourier} we turn our attention to an ODE associated to the horizontal Fourier transform of the problem \eqref{intro_stokes_stress}.  We study some special solutions to this ODE and derive their asymptotic developments.  In Section \ref{sec_specialized_sobolev} we study some specialized  Sobolev spaces.  Section \ref{sec_gravity_capillary} concerns the analysis of the linearized problem \eqref{intro_stokes_full}.  We characterize its solvability in terms of the specialized spaces from Section \ref{sec_specialized_sobolev}.  Section \ref{sec_navier_bcs} contains a brief digression on the solvability of the Stokes problem with Navier boundary conditions \eqref{intro_stokes_navier}.  In Section \ref{sec_nonlinear_analysis} we employ nonlinear analysis to prove all of the main theorems.  Appendix \ref{sec_analysis_tools} contains some analysis tools used throughout the paper.

\section{The $\gamma-$Stokes equations with stress boundary conditions}\label{sec_stress_bcs}

In this section we study the linear problem
\begin{equation}\label{problem_gamma_stokes_stress} 
\begin{cases}
\diverge S(p,u) - \gamma \p_1 u =f & \text{in }\Omega\\
\diverge u=g & \text{in }\Omega\\
S(p,u)e_{n}=k, & \text{on }\Sigma_b  \\
u=0 & \text{on }\Sigma_{0},
\end{cases}
\end{equation}
where $f\in ({_{0}}H^{1}(\Omega;\mathbb{R}^{n}))^{\ast}$, $g\in L^{2}(\Omega)$, $k\in H^{-1/2}(\Sigma_b;\mathbb{R}^{n})$ are given data.  A related problem with $\gamma =0$ was studied in \cite{Abels_2006} in $L^p-$Sobolev spaces.  Here we work only in $L^2-$based spaces but also go to higher regularity than \cite{Abels_2006}.  Of course, the regularity gain is not surprising and can be derived from the general theory of \cite{ADN_1964}.  Here we present a self-contained and elementary treatment for the reader's convenience.

\subsection{The specified divergence problem and the pressure as Lagrange multiplier}

Before addressing \eqref{problem_gamma_stokes_stress} we need to develop a couple auxiliary tools related to the divergence operator.  We develop these now.  The first allows us to solve the specified divergence problem, which is useful in reducing to the case $g =0$ in \eqref{problem_gamma_stokes_stress} and is essential in dealing with the pressure in the weak formulation.  The following proof is adapted from Theorem 2 in \cite{BB_2007}.

\begin{proposition}\label{specified_divergence}
Let $g\in L^{2}(\Omega)$. Then there exists $v\in {_{0}}H^{1}(\Omega;\mathbb{R}^{n})$ such that $\diverge v=g$ in $\Omega$ and
\begin{equation}\label{specified_divergence_0}
\Vert v\Vert_{ {_{0}}H^{1} }\leq c\Vert g\Vert_{L^{2}}
\end{equation}
for some constant $c=c(b,n)>0$.
\end{proposition}

\begin{proof}
Let $U=\mathbb{R}^{n-1}\times(-3b,b)$ and define $g_1 \in L^2(U)$ via 
\begin{equation}
g_{1}(x)=
\begin{cases}
g(x) & \text{in }\Omega\\
0 & \text{in }U\setminus\Omega. 
\end{cases}
\end{equation}
Consider the Dirichlet problem
\begin{equation}
\begin{cases}
\Delta\varphi=g_{1} & \text{in }U\\
\varphi=0 & \text{on }\partial U.
\end{cases}
\end{equation}
The unique weak solution  $\varphi \in H^1_0(U)$ to this problem is given by the minimizer of the functional
\begin{equation}
 H^1_0(U)\ni v \mapsto\int_{U}  \frac{1}{2}|\nabla v|^{2}+g_{1}v. 
\end{equation}
This functional is coercive thanks to the Poincar\'{e} inequality, Lemma \ref{poincare}  (which continues to hold in $H_{0}^{1}(U)$ via a translation and scaling argument),  and the Cauchy-Schwarz inequality.  Moreover, using $v=0$ as a comparison, we find that
\begin{equation}
\int_{U}   \frac{1}{2}|\nabla\varphi|^{2}+g_{1}\varphi  \leq  \int_{U} \hal \abs{\nab 0}^2 + g_1 0 =   0
\end{equation}
and so again by Poincar\'{e}'s inequality,
\begin{equation}
\Vert\nabla\varphi\Vert_{L^{2}(U)}^{2}\leq 2 \Vert\varphi\Vert_{L^2(U)} \Vert g_{1}\Vert_{L^{2}(U)} 
\leq c(b) \Vert\nabla\varphi\Vert_{L^{2}(U)}\Vert g\Vert_{L^{2}(\Omega)},
\end{equation}
which yields the estimate $\Vert\nabla\varphi\Vert_{L^{2}(U)} \leq c(b) \Vert
g\Vert_{L^{2}(\Omega)}$. Using standard regularity results (see Theorem \ref{theorem_regularity_linear} below for a sketch) we deduce that $\varphi\in H^{2}(U)$ and
\begin{equation}\label{specified_divergence_1}
\Vert\varphi\Vert_{H^{2}(U)}\leq c\Vert g\Vert_{L^{2}(\Omega)}
\end{equation}
for a constant $c = c(n,b) >0$.

We now define $v : \Omega \to \R^n$ via 
\begin{equation}
\begin{split}
v^{\prime}(x)  & = \nabla^{\prime}\varphi(x^{\prime},x_{n}) + 3\nabla^{\prime} \varphi(x^{\prime},-x_{n}) - 4\nabla^{\prime}\varphi(x^{\prime},-2x_{n}),\\
v_{n}(x)  & = \partial_{n}\varphi(x^{\prime},x_{n}) - 3\partial_{n}  \varphi(x^{\prime},-x_{n}) + 2\partial_{n}\varphi(x^{\prime},-2x_{n}).
\end{split}
\end{equation}
Then, using the fact that $g_1 = 0$ in $\R^{n-1} \times (-3b,0)$, we find that 
\begin{equation}
\diverge v(x)=\Delta\varphi(x)+3\Delta\varphi(x^{\prime} ,-x_{n})-4\Delta\varphi(x^{\prime},-2x_{n})=g(x) \text{ for }x \in \Omega.
\end{equation}
Moreover, $v=0$ on $\Sigma_{0}$ by construction, so $v\in{_{0}}H^{1}(\Omega;\mathbb{R}^{n})$.  The estimate \eqref{specified_divergence_0} then follows directly from \eqref{specified_divergence_1} and the definition of $v$.
\end{proof}

Next we aim to use Proposition \ref{specified_divergence} to perform the usual trick of introducing the pressure as a Lagrange multiplier associated to the divergence free condition.   Given $p\in L^{2}(\Omega)$, consider the linear functional $L_{p}:{_{0}}H^{1}(\Omega ; \R^{n})\rightarrow\mathbb{R}$ defined by
\begin{equation}
L_{p} v = \int_{\Omega}p\diverge v  \text{ for } v\in{_{0}}H^{1}(\Omega;\mathbb{R}^{n}).
\end{equation}
Then $\Vert L_{p}\Vert_{({_{0}}H^{1})^{\ast}} \leq c(n,b) \Vert p\Vert_{L^{2}}$, and so the Riesz representation theorem shows that there exists a unique $w_{p}\in{_{0}}H^{1}(\Omega;\mathbb{R}^{n})$ such that $\Vert w_{p}\Vert_{{_{0}}H^{1}}=\Vert L_{p}\Vert_{({_{0}}H^{1})^{\ast}}$ and 
\begin{equation}\label{operator Q}
\int_{\Omega}p\diverge v=(w_{p},v)_{{_{0}}H^{1}(\Omega)}  \text{ for all }v\in{_{0}}H^{1}(\Omega;\mathbb{R}^{n}).
\end{equation}
We then use this to define the bounded linear operator $Q:L^{2}(\Omega)  \to {_{0}}H^{1}(\Omega;\mathbb{R}^{n})$ via $Qp = w_p$.  The next result records some essential properties of $Q$.

\begin{proposition}\label{orth_decomp}
Let $Q:L^{2}(\Omega)\rightarrow{_{0}}H^{1}(\Omega;\mathbb{R}^{n})$ be the linear operator defined above. Then $Q$
has closed range, and $(\operatorname*{Ran}Q)^\bot =  {_{0}}H_{\sigma}^{1}(\Omega;\mathbb{R}^{n})$,  where ${_{0}}H_{\sigma}^{1}(\Omega;\mathbb{R}^{n})$ is defined in \eqref{H1 div zero}.  Consequently, we have the orthogonal decomposition
\begin{equation}
{_{0}}H^{1}(\Omega;\mathbb{R}^{n})={_{0}}H_{\sigma}^{1}(\Omega;\mathbb{R}^{n}) \oplus \operatorname*{Ran}Q.
\end{equation}

\end{proposition}

\begin{proof}
We divide the proof into two steps.

\textbf{Step 1 -- Closed range:} For every $p\in L^{2}(\Omega)$ we have
\begin{equation}
\Vert Q p \Vert_{{_{0}}H^{1}}= \Vert w_{p}\Vert_{{_{0}}H^{1}} \leq c(n,b) \Vert p\Vert_{L^{2}}.
\end{equation}
On the other hand, by Proposition \ref{specified_divergence} there exists $v_{0}\in{_{0}}H^{1}(\Omega;\mathbb{R}^{n})$ such that $\diverge  v_{0}=p$ and $\Vert v_{0}\Vert_{{_{0}}H^{1} }\leq c\Vert p \Vert_{L^{2}}$. Hence, by \eqref{operator Q},
\begin{equation}
\Vert p\Vert_{L^{2}}^{2}  =\int_{\Omega}p\diverge   v_{0}= (w_{p},v)_{{_{0}}H^{1}}  
  \leq \Vert w_{p}\Vert_{{_{0}}H^{1}}  \Vert v_{0} \Vert_{{_{0}}H^{1}} 
  = \Vert Qp \Vert_{{_{0}}H^{1} } \Vert v_{0}\Vert_{{_{0}}H^{1} }
  \leq c\Vert Qp \Vert_{{_{0}}H^{1} }\Vert p\Vert_{L^{2}},
\end{equation}
and so
\begin{equation}
\Vert p\Vert_{L^{2}}\leq c\Vert Q(p)\Vert_{{_{0}}H^{1}}.
\end{equation}
Hence, we have shown that
\begin{equation} \label{Q injective}
c^{-1}\Vert p\Vert_{L^{2} }\leq\Vert Q(p)\Vert_{{_{0}}H^{1}}\leq\sqrt{n}\Vert p\Vert_{L^{2}}
\end{equation}
for all $p\in L^{2}(\Omega)$, which implies that $Q$ has closed range.

\textbf{Step 2  -- Orthogonal decomposition:} From the first step we know that $\operatorname*{Ran}Q$ is closed, and so we have the orthogonal decomposition ${_{0}}H^{1}(\Omega;\mathbb{R}^{n}) = \operatorname*{Ran}Q \oplus (\operatorname*{Ran}Q)^\perp$.  We now endeavor to identify the subspace $(\operatorname*{Ran}Q)^\perp$.

Let $v\in(\operatorname*{Ran}Q)^{\perp}$, that is,
\begin{equation}
(Qp,v)_{{_{0}}H^{1}(\Omega)}=0
\end{equation}
for all $p\in L^{2}(\Omega)$. Then by \eqref{operator Q},
\begin{equation}
\int_{\Omega}p\diverge v=0
\end{equation}
for all $p\in L^{2}(\Omega)$, which implies that $\diverge v=0$, and so $v\in{_{0}}H_{\sigma}^{1}(\Omega;\mathbb{R}^{n})$. Conversely, if $v\in{_{0}}H_{\sigma}^{1}(\Omega;\mathbb{R}^{n})$, then $\diverge  v=0$ and so by \eqref{operator Q},
\begin{equation}
(Q(p),v)_{{_{0}}H^{1}(\Omega)}=0
\end{equation}
for all $p\in L^{2}(\Omega)$, which implies that $v\in(\operatorname*{Ran} Q)^{\perp}$. This shows that $(\operatorname*{Ran}Q)^{\perp}={_{0}}H_{\sigma}^{1}(\Omega;\mathbb{R}^{n})$, which completes the proof.
\end{proof}

The following corollary is essential in introducing the pressure in the weak formulation of \eqref{problem_gamma_stokes_stress}.

\begin{corollary} \label{pressure_introduction}
Let $\Lambda\in({_{0}}H^{1}(\Omega;\mathbb{R}^{n}))^{\ast}$ be such that
$\left\langle \Lambda,v\right\rangle=0$ for all $v\in{_{0}}H_{\sigma}^{1}(\Omega;\mathbb{R}^{n})$. Then
there exists a unique function $p\in L^{2}(\Omega)$ such that
\begin{equation}
\left\langle \Lambda,v\right\rangle =\int_{\Omega}p\diverge v\quad\text{for all }v\in{_{0}}H^{1}(\Omega;\mathbb{R}^{n}).
\end{equation}
Moreover, there is a constant $c = c(n,b)>0$ such that
\begin{equation}
\Vert p\Vert_{L^{2}}\leq c\Vert\Lambda\Vert_{({_{0}}H^{1})^{\ast}}.
\end{equation}
\end{corollary}

\begin{proof}
In view of the Riesz representation theorem, there exists $w\in{_{0}}H^{1}(\Omega;\mathbb{R}^{n})$ such that
\begin{equation}
\left\langle \Lambda,v\right\rangle=(w,v)_{{_{0}}H^{1}}\quad\text{for all }v\in{_{0}}%
H^{1}(\Omega;\mathbb{R}^{n}),
\end{equation}
and $\Vert w\Vert_{{_{0}}H^{1}}=\Vert\Lambda\Vert_{({_{0}}H^{1})^{\ast}}$. Then by hypothesis $w$, is orthogonal to ${_{0}}H_{\sigma}^{1}(\Omega;\mathbb{R}^{n})$, and so Proposition \ref{orth_decomp} implies that $w \in \operatorname*{Ran}Q$, which provides us with $p\in L^{2}(\Omega)$ such that $Q p=w$. It follows from \eqref{Q injective} that
\begin{equation}
\Vert p\Vert_{L^{2}}\leq c\Vert Q(p)\Vert_{{_{0}}H^{1} } = c\Vert w\Vert_{{_{0}}H^{1}}
=c\Vert\Lambda\Vert_{({_{0}}H^{1})^{\ast}}.
\end{equation}
Moreover, $p$ is unique since $Q$ is injective by \eqref{Q injective}. The conclusion now follows from \eqref{operator Q}.
\end{proof}

\subsection{Solving \eqref{problem_gamma_stokes_stress}  }

We are now ready to prove the existence of solutions to \eqref{problem_gamma_stokes_stress}.  We begin with weak solutions.  Employing the identity \eqref{symmetrized product}, a simple computation reveals that the weak formulation of \eqref{problem_gamma_stokes_stress} is to find a velocity field $u\in{_{0}}H^{1}(\Omega;\mathbb{R}^{n})$ and a pressure $p\in L^{2}(\Omega)$ satisfying $\diverge{u}=g$ in $\Omega$ as well as 
\begin{equation}\label{weak_solution} 
\int_{\Omega}\frac{1}{2}\mathbb{D}u:\mathbb{D}v-p\diverge  v - \gamma\partial_{1}u\cdot v 
=\left\langle f,v\right\rangle -\left\langle k,v \right\rangle_{\Sigma_b} 
\end{equation}
for all $v\in{_{0}}H^{1}(\Omega;\mathbb{R}^{n})$, where here $\br{f,v}$ denotes the dual pairing of $f \in ({_{0}}H^{1}(\Omega;\mathbb{R}^{n})^\ast$ and $v \in {_{0}}H^{1}(\Omega;\mathbb{R}^{n})$, and $\br{k,v}_{\Sigma_b}$ denotes the dual pairing of $k\in H^{-1/2}(\Sigma_b;\mathbb{R}^{n}) = (H^{1/2}(\Sigma_b;\mathbb{R}^{n}))^\ast$ and $v \vert_{\Sigma_b} \in H^{1/2}(\Sigma_b;\mathbb{R}^{n})$.

\begin{theorem}[Existence of weak solutions]\label{theorem existence linear}
Let $f\in({_{0}}H^{1}(\Omega;\mathbb{R}^{n}))^{\ast
}$, $g\in L^{2}(\Omega)$, and $k\in H^{-1/2}(\Sigma_b;\mathbb{R}^{n})$. Then there exist unique $u\in{_{0}}H^{1}(\Omega;\mathbb{R}^{n})$ and $p\in L^2(\Omega)$ satisfying $\diverge u=g$ in $\Omega$ and
\eqref{weak_solution}. Moreover,
\begin{equation}\label{bounds u p}
\Vert u\Vert_{{_{0}}H^{1} }  + \Vert p\Vert_{L^{2}}\leq 
c\Vert f \Vert_{({_{0}}H^{1})^{\ast}} + c\Vert g\Vert_{L^{2}}
+c\Vert k \Vert_{H^{-1/2}} 
\end{equation}
for some constant $c=c(b,n)>0$.
\end{theorem}

\begin{proof}
We divide the proof into three steps.

\textbf{Step 1  -- Setup:} Consider the bilinear map $B: {_{0}}H^{1}(\Omega;\mathbb{R}^{n})\times{_{0}}H^{1}(\Omega;\mathbb{R}^{n})\rightarrow\mathbb{R}$ given by
\begin{equation}
B(u,v) =\int_{\Omega}\frac{1}{2}\mathbb{D}u:\mathbb{D}v-\gamma\partial
_{1}u\cdot v.
\end{equation}
In light of Korn's inequality, Lemma \ref{korn}, $B$ is well-defined and continuous.   Note that
\begin{equation}\label{p1_annihilate}
 \int_{\Omega} \p_1 u \cdot u = \int_{\Omega} \p_1 \frac{\abs{u}^2}{2} =0,
\end{equation}
and hence
\begin{equation}
B(u,u) = \frac{1}{2}\int_{\Omega}|\mathbb{D}u|^{2}  =    \norm{u}_{{_{0}}H^{1}}^2,
\end{equation}
which shows that $B$ is coercive.  The Hilbert space ${_{0}}H_{\sigma}^{1}(\Omega;\mathbb{R}^{n})$, defined in \eqref{H1 div zero}, is a closed subspace of  ${_{0}}H^{1}(\Omega;\mathbb{R}^{n})$, so this analysis also shows that $B$ is well-defined, continuous, and coercive on ${_{0}}H_{\sigma}^{1}(\Omega;\mathbb{R}^{n})$.

\textbf{Step 2 -- A special case:}  Assume now that $g=0$.  Thanks to the first step, we are in a position to apply Lax--Milgram to find a unique $u\in{_{0}}H_{\sigma}^{1}(\Omega ;\mathbb{R}^{n})$ such that
\begin{equation}
B(u,v)-\left\langle f,v\right\rangle +\left\langle k,v\right\rangle _{\Sigma_b}=0
\end{equation}
for all $v\in{_{0}}H_{\sigma}^{1}(\Omega;\mathbb{R}^{n})$. Moreover,
\begin{equation}\label{lm1}
\Vert u\Vert_{{_{0}}H^{1} }\leq c\Vert f\Vert_{({_{0}}H^{1})^{\ast}}
+c\Vert k\Vert_{H^{-1/2}}
\end{equation}
for some constant $c= c(n,b)>0$.

The functional $\Lambda\in({_{0}}H^{1}(\Omega;\mathbb{R}^{n}))^{\ast}$ defined by
\begin{equation}
\left\langle \Lambda,v\right\rangle:=B(u,v)-\left\langle f,v\right\rangle +\left\langle k,v\right\rangle_{\Sigma_b} \text{ for } v\in{_{0}}H^{1}(\Omega;\mathbb{R}^{n})
\end{equation}
vanishes on ${_{0}}H_{\sigma}^{1}(\Omega;\mathbb{R}^{n})$.  Then according to Corollary
\ref{pressure_introduction} there exists a unique function $p\in L^{2}(\Omega)$
such that
\begin{equation}
B(u,v)-\left\langle f,v\right\rangle +\left\langle k,v\right\rangle _{\Sigma_b}=\int_{\Omega}p\diverge v
\end{equation}
for all $v\in{_{0}}H^{1}(\Omega;\mathbb{R}^{n})$, and we have the estimate
\begin{equation}
\Vert p\Vert_{L^{2}}    \leq c\Vert\Lambda\Vert_{({_{0}}H^{1})^{\ast}}\leq c\Vert u\Vert_{{_{0}}H^{1}}
+c\Vert f \Vert_{({_{0}}H^{1})^{\ast}} 
+ c\Vert k\Vert_{H^{-1/2} }
  \leq c\Vert f\Vert_{({_{0}}H^{1} )^{\ast}}+c\Vert k\Vert_{H^{-1/2}},
\end{equation}
where in the last inequality we used \eqref{lm1}.

\textbf{Step 3 -- The general case:} Finally, given $g\in L^{2}(\Omega)$ we use Proposition \ref{specified_divergence} to find $w\in{_{0}}H^{1}(\Omega;\mathbb{R}^{n})$ such that $\diverge w=g$ and $\Vert w\Vert_{{_{0}}H^{1}}\leq c\Vert g \Vert_{L^{2}}$.  We define $f_{1}\in({_{0}}H^{1}(\Omega;\mathbb{R}^{n}))^{\ast}$ via $\left\langle f_{1},v\right\rangle :=\left\langle f,v\right\rangle - B(w,v)$
and apply Step 2 with $f$ replaced by $f_{1}$ to find $u_{0}\in{_{0}}H_{\sigma}^{1}(\Omega;\mathbb{R}^{n})$ and $p\in L^{2}(\Omega)$ such that
\begin{equation}
\int_{\Omega} \left(\frac{1}{2}\mathbb{D}u_{0}:\mathbb{D}v    -\gamma\partial
_{1}u_{0}\cdot v \right) -\left\langle f,v\right\rangle 
  +\int_{\Omega}\left( \frac{1}{2}\mathbb{D}w:\mathbb{D}v - \gamma\partial_{1}w\cdot
v\right) + \left\langle k,v\right\rangle _{\Sigma_b}=\int_{\Omega}
p\diverge v
\end{equation}
for all $v\in{_{0}}H^{1}(\Omega;\mathbb{R}^{n})$, and
\begin{equation}\label{bounds u0} 
\Vert u_{0}\Vert_{{_{0}}H^{1}}+\Vert p\Vert_{L^{2}}     \leq
c\Vert f_{1}\Vert_{({_{0}}H^{1} )^{\ast}}
+c\Vert k\Vert_{H^{-1/2}} 
\leq c\Vert f\Vert_{({_{0}}H^{1})^{\ast}}
+c\Vert g \Vert_{L^{2}}
+c\Vert k\Vert_{H^{-1/2}},
\end{equation}
where in the last inequality we used the fact that $\Vert w\Vert_{{_{0}}H^{1}}\leq c\Vert g\Vert_{L^{2}}$. Then the function $u:=u_{0}+w \in{_{0}}H^{1}(\Omega;\mathbb{R}^{n})$
satisfies $\diverge u=g$ in $\Omega$ and
\begin{equation}
\int_{\Omega} \left( \frac{1}{2}\mathbb{D}u:\mathbb{D}v-\gamma\partial_{1}u\cdot
v \right) - \left\langle f,v\right\rangle +\left\langle k,v\right\rangle _{\Sigma_b}=\int_{\Omega}p\diverge v
\end{equation}
for all $v\in{_{0}}H^{1}(\Omega;\mathbb{R}^{n})$, which gives \eqref{weak_solution}. In view of \eqref{bounds u0} and again the fact $\Vert w \Vert_{{_{0}}H^{1}}\leq c\Vert g\Vert_{L^{2}}$ we have that the function $u$ satisfies \eqref{bounds u p}.  The uniqueness of the pair $(u,p)$ then follows the uniqueness component of Step 2.
\end{proof}

Next we prove some regularity results.  These may be derived from the well known regularity results for elliptic systems proved in \cite{ADN_1964}.  We include an elementary proof here for the convenience of the reader.

\begin{theorem}[Regularity of weak solutions]\label{theorem_regularity_linear}
Let $s \ge 0$,  $f\in H^{s}(\Omega;\mathbb{R}^{n})$, $g\in H^{s+1}(\Omega)$, and $k\in H^{s+1/2}(\Sigma_b;\mathbb{R}^{n})$.  If  $u \in {_{0}}H^{1}(\Omega;\mathbb{R}^{n})$ and $p\in L(\Omega)$ satisfy
$\diverge u=g$ in $\Omega$ and \eqref{weak_solution}, then $u\in
{_{0}}H^{s+2}(\Omega;\mathbb{R}^{n})$ and $p\in H^{s+1}(\Omega)$. Moreover, we have the estimate
\begin{equation}\label{bounds s}
\Vert u\Vert_{H^{s+2}}+\Vert p\Vert_{H^{s+1}}\leq c\Vert f \Vert_{H^{s}}+c\Vert g\Vert_{H^{s+1}}+c\Vert k\Vert_{H^{s+1/2}} 
\end{equation}
for a constant $c=c(b,n,s)>0$.
\end{theorem}

\begin{proof}
We divide the proof into two steps.

\textbf{Step 1 -- The base case:} Assume that $s=0$. Given $h\in\mathbb{R} \setminus \{0\}$, $i=1,\ldots,n-1$, and $w:\Omega\rightarrow \R^m$, we write $\delta_{h}^i w(x):=\frac{w(x+he_{i})-w(x)}{h}$ for the horizontal difference quotient in the direction $e_i$. Given $w \in{_{0}}H^{1} (\Omega;\mathbb{R}^{n})$,\ take $v:=\delta_{-h}^i w\in{_{0}}H^{1}(\Omega;\mathbb{R}^{n})$ in \eqref{weak_solution}. Then the change of variables $y=x - he_{i}$ shows that we have the identity
\begin{equation}\label{pde difference quotient}
\int_{\Omega} \frac{1}{2}\mathbb{D}\delta_{h}^i u    :\mathbb{D}w - \delta_{h}^i p \diverge w - \gamma\partial_{1}\delta_{h}^i u \cdot w 
  =\int_{\Omega}\delta_{h}^i f \cdot w - \int_{\Sigma_b} \delta_{h}^i k  \cdot w ,
\end{equation}
which shows that $\delta_{h}^i u$ and $\delta_{h}^i p$ satisfy \eqref{weak_solution} with $f$, $g$, and $k$ replaced by $\delta_{h}^i f$, $\delta_{h}^i g$, and $\delta_{h}^i k$, respectively.  Hence, by Theorem \ref{theorem existence linear},
\begin{equation}
\Vert \delta_{h}^i u\Vert_{{_{0}}H^{1}} + 
\Vert \delta_{h}^i p \Vert_{L^{2}} \leq c\Vert\delta_{h}^i f \Vert_{({_{0}}H^{1})^{\ast}} + c\Vert \delta_{h}^i g \Vert_{L^{2}} + \Vert\delta_{h}^i k\Vert_{H^{-1/2}}.
\end{equation}

Employing the change of variables $y=x+he_{i}$, the Cauchy-Schwarz inequality, and Corollary \ref{diff_quote_omega}, we may bound
\begin{equation}
\left\vert \int_{\Omega}\delta_{h}^i f \cdot v\right\vert =\left\vert \int_{\Omega} f\cdot \delta_{-h}^i v \right\vert \leq\Vert f\Vert_{L^{2}}\Vert \delta_{-h}^i v \Vert_{L^{2}} \leq \Vert f\Vert_{L^{2}}\Vert \partial_{i} v \Vert_{L^{2}}.
\end{equation}
Hence, from Korn's inequality, Lemma \ref{korn}, we have the bound $\Vert\delta_{h}^i f\Vert_{({_{0}}H^{1})^{\ast}} \leq c(n,b) \Vert f \Vert_{L^{2}}$.  Similarly, Corollary \ref{diff_quote_omega} tells us that $\Vert \delta_{h}^i g\Vert_{L^{2}} \leq \Vert\partial_{i} g \Vert_{L^{2}}$, while Proposition \ref{diff_quote_fullspace} implies that $\Vert\delta_{h}^i k \Vert_{H^{-1/2}} \leq c\norm{k}_{H^{1/2}}$.  We deduce from these that
\begin{equation}
\Vert \delta_{h}^i u\Vert_{{_{0}}H^{1}} + \Vert \delta_{h}^i p \Vert_{L^{2}} 
\leq c\Vert f\Vert_{L^{2}}+c\Vert\partial_{i} g \Vert_{L^{2}}+c\Vert k\Vert_{H^{1/2}}
\end{equation}
for all $h\neq0$ and $1 \le i \le n-1$.  In turn, these bounds imply (see, for instance, Section 11.5 of \cite{Leoni_2017}) that $\partial_{i} u \in{_{0}}H^{1}(\Omega;\mathbb{R}^{n})$ and that $\partial_{i}p\in L^{2}(\Omega)$, with
\begin{equation} \label{bounds 1} 
\Vert\partial_{i}u\Vert_{{_{0}}H^{1}} + \Vert\partial_{i}p\Vert_{L^{2}}\leq c\Vert f\Vert_{L^{2}} + c\Vert\partial_{i} 
g \Vert_{L^{2}} + c\Vert k\Vert_{H^{1/2}}
\end{equation}
for all $i=1,\ldots,n-1$.

Differentiating the equation $\diverge u=g$ with respect to $x_{n}$, we find that
\begin{equation}
\partial_{n}^{2}u_{n}=-\diverge ^{\prime}\partial_{n}u^{\prime}+\partial_{n}g
\end{equation}
and so by \eqref{bounds 1}, $\p_n^2 u_n \in L^2(\Omega)$ and
\begin{equation}
\Vert\partial_{n}^{2}u_{n}\Vert_{L^{2}}    \leq \Vert \diverge ^{\prime}\partial_{n}u^{\prime}\Vert_{L^{2}}
+\Vert\partial_{n}g\Vert_{L^{2} }  \leq c\Vert f\Vert_{L^{2}} + c\Vert\nabla g\Vert_{L^{2} }.
\end{equation}
For $i=1,\ldots,n-1$, taking $v=\varphi e_{i}$ with $\varphi\in C_{c}^{\infty}(\Omega)$, we have that 
\begin{equation}
 \hal \sg u : \sg v = \nab u_i \cdot \nab \varphi + \p_i u \cdot \nab \varphi,
\end{equation}
and so upon using $v$ in \eqref{weak_solution} we find that
\begin{multline}\label{equation 1}
0    =\int_{\Omega} \nabla u_{i}\cdot\nabla\varphi+\partial_{i}u\cdot
\nabla\varphi-p\diverge (\varphi e_i)-\gamma\partial_{1}u\cdot (\varphi e_i) -f \cdot (\varphi e_i)  \\
  =\int_{\Omega} \nabla u_{i}\cdot\nabla\varphi + (\diverge u) \p_i \varphi-p \p_i \varphi - \gamma \partial_{1}u_i \varphi 
-f_i \varphi ,
\end{multline}
where we integrated the second term by parts. Hence,
$u_{i}$ is a distributional solution to the equation
\begin{equation} \label{equation i} 
\Delta u_{i}=\partial_{i}p-\partial_{i}g-\gamma\partial_{1}u_{i}-f_{i}\in
L^{2}(\Omega).
\end{equation}
From the standard weak existence and local regularity theory for Poisson's equation, together with Weyl's lemma, we deduce that $u_{i}\in H_{\operatorname*{loc}}^{2}(\Omega)$ and that the previous equation holds almost everywhere in $\Omega$. In particular, $\p_n^2 u_i \in L^2(\Omega)$, and we have the estimate
\begin{equation}
\Vert\partial_{n}^{2}u_{i}\Vert_{L^{2}}     \leq\Vert\Delta^{\prime} u_{i}\Vert_{L^{2}} + \Vert\partial_{i} p -\partial_{i} g -\gamma \partial_{1}u_{i}-f_{i}\Vert_{L^{2}} 
 \leq c\Vert f\Vert_{L^{2} }+c\Vert\partial_{i}g\Vert_{L^{2}} + c\Vert k\Vert_{H^{1/2}}.
\end{equation}
It remains to show that $\partial_{n}p$ exists and is in $L^{2}(\Omega)$.  Taking $v=\varphi e_{n}$ in \eqref{weak_solution} for some $\varphi\in C_{c}^{\infty}(\Omega)$ and integrating by parts, we see that
\begin{equation}
\int_{\Omega}p\partial_{n}\varphi     =\int_{\Omega} \nabla u_{n} \cdot \nabla \varphi + \partial_{n} u \cdot \nabla \varphi-\gamma\partial_{1} u_{n} \varphi-f_{n}\varphi 
  =\int_{\Omega}(-\Delta u_{n} - \partial_{n}g-\gamma\partial_{1}u_{n} -f_{n})\varphi,
\end{equation}
which implies that the weak derivative $\p_n p$ exists and satisfies  $\partial_{n}p=-\Delta u_{n}+\partial
_{n}g-\gamma\partial_{1}u_{n}-f_{n} \in L^2(\Omega)$.  In turn, we may combine with the above estimates to arrive at the bound
\begin{equation}
 \norm{\p_n p}_{L^2} \le c\Vert f\Vert_{L^{2} }+c\Vert\partial_{i}g\Vert_{L^{2}} + c\Vert k\Vert_{H^{1/2}}.
\end{equation}
This completes the proof in the case $s =0$.

\textbf{Step 2 -- Induction and interpolation:} The case $s\in\mathbb{N}$ can be obtained through
induction by reasoning as in Step 1. Indeed, the base case $s=0$ has been
established in Step 1. Assume that the result is true for $s\in \N$.  More precisely, assume that if $f\in H^{s}(\Omega;\mathbb{R}^{n})$, $g\in H^{s+1}(\Omega)$, and $k\in H^{s+1/2}(\Sigma_b;\mathbb{R}^{n})$, then $u\in{_{0}}H^{s+2}(\Omega;\mathbb{R}^{n})$, $p\in H^{s+1}(\Omega)$, and the bound \eqref{bounds s}
holds. Let $f\in H^{s+1}(\Omega;\mathbb{R}^{n})$, $g\in H^{s+2}(\Omega)$, and $k\in H^{s+3/2}(\Sigma_b;\mathbb{R}^{n})$. Then by \eqref{pde difference quotient} we have that $\delta_{h}^i u$ and $\delta_{h}^i p$
satisfy \eqref{weak_solution} with $f$, $g$, and $k$ replaced by $\delta_{h}^i f$, $\delta_{h}^i g$, and $\delta_{h}^i k$, respectively. Hence, by the induction hypothesis
\begin{equation}
\Vert\delta_{h}^i u \Vert_{H^{s+2}} + \Vert\delta_{h}^i p \Vert_{H^{s+1}}    \leq c\Vert\delta_{h}^i f \Vert_{H^{s}} 
+c\Vert\delta_{h}^i g \Vert_{H^{s+1}}+c\Vert\delta_{h}^i k \Vert_{H^{s+1/2}}.
\end{equation}
Reasoning as in Step 1 and again employing Proposition \ref{diff_quote_fullspace} and Corollary \ref{diff_quote_omega},  we can bound the right-hand side from above by $c\Vert f\Vert_{H^{s+1} }+c\Vert g\Vert_{H^{s+2}}+c\Vert k\Vert_{H^{s+3/2}}$ and then in turn conclude that
$\partial_{i}u\in H^{s+2}(\Omega;\mathbb{R}^{n})$ and $\partial_{i}p\in
H^{s+2}(\Omega)$ for all $i=1,\ldots,n-1$. As in Step 1, we then use the identity $\diverge u=g$ to show that $\partial_{n}^{s+3}u_{n}$ exists in $L^{2}(\Omega)$ with the appropriate bounds.  We then use \eqref{equation i} to show that $\partial_{n}^{s+3}u_{i}$ exists in $L^{2}(\Omega)$ for $1 \le i \le n-1$, and then use
the equation $-\partial_{n}p=-\Delta u_{n}- \partial_{n}g-\gamma\partial_{1}u_{n}-f_{n}$ to prove that $\partial_{n}^{s+2}p$ exists in $L^{2}(\Omega)$ and obeys the appropriate bounds.

The non integer case $s \in (0,\infty) \backslash \N$ can then be obtained by interpolation. Indeed, we have now shown that the linear operator
\begin{equation}
T:H^{s}(\Omega;\mathbb{R}^{n})\times H^{s+1}(\Omega)\times H^{s+1/2}(\Sigma_b;\mathbb{R}^{n})    \rightarrow{_{0}}H^{s+2}(\Omega;\mathbb{R}^{n}) \times H^{s+1}(\Omega)
\end{equation}
defined by $T(f,g,k) = (u,p)$  is continuous for all $s\in\mathbb{N}$. We can now use classical
interpolation theory (see, for instance \cite{BL_1976,Leoni_2017,Triebel_1995}) to prove that $T$ is continuous for all $s>0$. 
\end{proof}

We are now ready to prove the main theorem of this section.

\begin{theorem}\label{iso_gamma_stokes}
For every $\gamma\in\mathbb{R}$ and every $s \ge 0$, the bounded linear operator
\begin{equation}
 \Phi_\gamma :  {_{0}}H^{s+2}(\Omega;\mathbb{R}^{n})\times H^{s+1}(\Omega)  
\rightarrow H^{s}(\Omega;\mathbb{R}^{n})\times H^{s+1}(\Omega) \times
H^{s+1/2}(\Sigma_b;\mathbb{R}^{n})
\end{equation}
given by 
\begin{equation}
 \Phi_\gamma(u,p) = (\diverge{S(p,u)} - \gamma \p_1 u, \diverge{u}, \left.  S(p,u)e_{n}\right\vert _{\Sigma_b}   )
\end{equation}
is an isomorphism.
\end{theorem}

\begin{proof}
Theorems \ref{theorem existence linear} and \ref{theorem_regularity_linear}  show that the bounded linear operator $\Phi_\gamma$ is surjective.  Theorem \ref{theorem existence linear} shows that it is injective.
\end{proof}

\section{The over-determined $\gamma-$Stokes equations}\label{sec_overdetermined}

In this section we study the over-determined problem
\begin{equation} \label{problem_gamma_stokes_overdet}
\begin{cases}
\diverge S(p,u)-\gamma\partial_{1}u=f & \text{in }\Omega \\
\diverge u=g & \text{in }\Omega \\
S(p,u)e_{n}=k,\quad u_{n}=h & \text{on }\Sigma_b \\
u=0 & \text{on }\Sigma_{0},
\end{cases}
\end{equation}
where, for $s \ge 0$, $f\in H^{s}(\Omega;\mathbb{R}^{n})$, $g\in H^{s+1}(\Omega)$, $k \in H^{s+1/2}(\Sigma_b;\mathbb{R}^{n})$, and  $h\in H^{s+3/2}(\Sigma_b)$.  In view of Theorem \ref{iso_gamma_stokes}, the value of $u_n$ on $\Sigma_b$ is completely determined by $f$, $g$, and $k$.  Hence, in general the problem \eqref{problem_gamma_stokes_overdet} is over-determined and admits no solution.  In this section we identify compatibility conditions on the data $(f,g,h,k)$ that are necessary and sufficient for solutions to \eqref{problem_gamma_stokes_overdet} to exist, and we prove a corresponding isomorphism theorem.

\subsection{Divergence compatibility }

In the over-determined problem \eqref{problem_gamma_stokes_overdet} we seek to specify both $\diverge{u} = g$ in $\Omega$ and the boundary conditions $u_n =0$ on $\Sigma_0$ and $u_n = h$ on $\Sigma_b$.  If we were to posit integrability of $g$ and $h$, then the divergence theorem would require the compatibility condition 
\begin{equation}
 \int_{\Omega} g = \int_{\Sigma_b} h.
\end{equation}
The functional framework we employ in this paper is built on subspaces of $L^2(\Omega)$, and $\Omega$ has infinite measure, so in general we cannot verify these integrability conditions.  As such, the form of compatibility between $g$ and $h$ is somewhat more subtle than the condition stated above.  We record this condition now.

\begin{theorem}[Divergence-trace compatibility condition]\label{cc_divergence}
Suppose that $u\in {_{0}}H^{1}(\Omega;\mathbb{R}^{n})$ and define  $g = \diverge{u} \in L^2(\Omega)$ and $h = u_n\vert_{\Sigma_b} \in H^{1/2}(\Sigma_b;\R)$.  Then 
\begin{equation}\label{H minus 1}
h-\int_{0}^{b}g(\cdot,x_{n}) dx_{n} \in \dot{H}^{-1}(\mathbb{R}^{n-1})
\end{equation}
and 
\begin{equation}\label{H minus 1 alt}
\snorm{h-\int_{0}^{b}g(\cdot,x_{n}) dx_{n}}_{\dot{H}^{-1}} \le 2\pi \sqrt{b} \norm{u}_{L^2}. 
\end{equation}

\end{theorem}
\begin{proof}
Since $u_n \in H^1(\Omega)$ we have that $u_n(x',\cdot)$ is absolutely continuous for almost every $x' \in \R^{n-1}$ (see, for instance, Theorem 11.45 in \cite{Leoni_2017}).  Since $u=0$ on $\Sigma_{0}$ and $\diverge u=g$ in $\Omega$, we may then compute  
\begin{equation}
u_{n}(x^{\prime},b)=\int_{0}^{b}\partial_{n}u_{n}(x^{\prime},x_{n} ) dx_{n}
=\int_{0}^{b}(g(x^{\prime},x_{n})-\diverge ^{\prime} u^{\prime}(x^{\prime},x_{n})) dx_{n}
\end{equation}
for almost every $x' \in \R^{n-1}$.  Hence,
\begin{equation}
u_{n}(x^{\prime},b)-\int_{0}^{b}g(x^{\prime},x_{n})dx_{n} 
=-\diverge ^{\prime}\int_{0}^{b}u^{\prime}(x^{\prime},x_{n}) dx_{n}.
\end{equation}
Write $R\in H^1(\R^{n-1};\R^{n-1})$ for $R(x') = \int_0^b u'(x',x_n) dx_n$.  Then  we may use the Cauchy-Schwarz inequality, Parseval's identity, and Tonelli's theorem to bound
\begin{multline}
 \snorm{ \diverge{R}}_{\dot{H}^{-1}}^2 = \int_{\R^{n-1}} \frac{1}{\abs{\xi}^2} \abs{2\pi i \xi \cdot \hat{R}(\xi)}^2 d\xi \le 4\pi^2 \int_{\R^{n-1}} \abs{\hat{R}(\xi)}^2 d\xi = 4\pi^2 \int_{\R^{n-1}} \abs{R(x')}^2 dx' \\
 \le 4\pi^2 b \int_{\Omega} \abs{u'(x)}^2 dx = 4\pi^2 b \norm{u'}_{L^2}^2,
\end{multline}
which proves \eqref{H minus 1} and \eqref{H minus 1 alt}.

\end{proof}

\subsection{Adjoint problem and compatibility }

In the spirit of the closed range theorem, we seek to understand when the over-determined problem \eqref{problem_gamma_stokes_overdet} admits a solution in terms of a corresponding adjoint problem.  To motivate the form of the adjoint problem we first present the following calculation.

\begin{lemma}\label{adjoint_calc}
Suppose that $u,v \in {_{0}}H^{2}(\Omega;\R^n)$ and $p,q \in H^1(\Omega)$.  Then 
\begin{multline}
\int_{\Omega} (\diverge S(p,u) - \gamma \p_1 u) \cdot v -  (\diverge{u})q  - \int_{\Omega} u\cdot (\diverge S(q,v) + \gamma \p_1 v)  - p \diverge{v} \\
= \int_{\Sigma_b} S(p,u) e_n \cdot v - u \cdot S(q,v) e_n.
\end{multline}
\end{lemma}
\begin{proof}
We simply integrate by parts to see that 
\begin{multline}
\int_{\Omega} (\diverge S(p,u) - \gamma \p_1 u) \cdot v -  (\diverge{u})q =  \int_{\Omega} - S(p,u) : \nab v + \gamma  u \cdot \p_1 v -  (\diverge{u})q  + \int_{\Sigma_b} S(p,u) e_n \cdot v \\
= \int_{\Omega} \hal \sg u : \sg v -p \diverge{v} + \gamma  u \cdot \p_1 v -  (\diverge{u})q  + \int_{\Sigma_b} S(p,u) e_n \cdot v,
\end{multline}
and similarly, 
\begin{equation}
 \int_{\Omega} u\cdot (\diverge S(q,v) + \gamma \p_1 v)  - p \diverge{v} = \int_{\Omega} \hal \sg u : \sg v - (\diverge{u}) q + \gamma u \cdot \p_1 v - p \diverge{v} + \int_{\Sigma_b} u \cdot S(q,v) e_n.
\end{equation}
The result follows by subtracting these expressions.

\end{proof}

This lemma shows that the formal adjoint of the over-determined problem \eqref{problem_gamma_stokes_overdet} is the under-determined problem  
\begin{equation}\label{underdetermined}
\begin{cases}
\diverge S(q,v)+\gamma\partial_{1}v= f & \text{in }\Omega \\
\diverge v= g & \text{in }\Omega \\
(S(q,v)e_{n})^{\prime}=k' & \text{on }\Sigma_b \\
v=0 & \text{on }\Sigma_{0}.
\end{cases}
\end{equation}
Note that this is under-determined in the sense that on $\Sigma_b$ we only specify $n-1$ boundary conditions instead of the standard $n$.  Taking a cue from the closed range theorem, we then examine the space of solutions to the homogeneous under-determined problem, i.e. \eqref{underdetermined} with $f=0$, $g=0$, and $k'=0$.   In light of Theorem \ref{iso_gamma_stokes} (with $\gamma$ replaced by $-\gamma$) the solution to this problem is completely determined by the boundary condition $S(p,u)e_{n}=\psi e_{n}$ on $\Sigma_b$.  In other words, we may parameterize the space of homogeneous solutions to the under-determined problem \eqref{underdetermined} with $\psi$ by way of the $(-\gamma)-$Stokes problem 
\begin{equation}\label{problem_adjoint}
\begin{cases}
\diverge S(q,v)+\gamma\partial_{1}v=0 & \text{in }\Omega \\
\diverge v=0 & \text{in }\Omega \\
S(q,v)e_{n}=\psi e_n & \text{on }\Sigma_b \\
v=0 & \text{on }\Sigma_{0}.
\end{cases}
\end{equation}

Using this parameterization, we arrive at a convenient formulation of the second compatibility condition associated to the over-determined problem.

\begin{theorem}[Over-determined compatibility condition]\label{cc_over-det}
Let $s \ge 0$ and suppose that $f\in H^{s}(\Omega;\mathbb{R}^{n})$, $g \in H^{s+1}(\Omega)$, $h\in H^{s+3/2}(\Sigma_b)$, and $k\in H^{s+1/2}(\Sigma_b;\mathbb{R}^{n})$.  Assume that the problem
\eqref{problem_gamma_stokes_overdet} admits a solution $u\in{_{0}}%
H^{s+2}(\Omega;\mathbb{R}^{n})$ and $p\in H^{s+1}(\Omega)$. For every $\psi\in
H^{s+1/2}(\Sigma_b)$ let $v\in{_{0}}H^{s+2}(\Omega;\mathbb{R}^{n})$ and $q\in
H^{s+1}(\Omega)$ be the unique solution (given by Theorem \ref{iso_gamma_stokes}) to the adjoint problem \eqref{problem_adjoint}.  Then the following compatibility condition holds:
\begin{equation}\label{compatibility_condition}
\int_{\Omega}(f\cdot v-gq)-\int_{\Sigma_b}(k\cdot v-h\psi) =0.
\end{equation}

\end{theorem}

\begin{proof}
In light of Lemma \ref{adjoint_calc},  \eqref{problem_gamma_stokes_overdet}, and \eqref{problem_adjoint} we have that
\begin{equation}
\int_{\Omega} f \cdot v - gq = \int_{\Sigma_b} k \cdot v - u\cdot  \psi e_n  = \int_{\Sigma_b} k\cdot v - h \psi.
\end{equation}
Then \eqref{compatibility_condition} follows by rearranging.
\end{proof}

\subsection{Some function spaces and the over-determined isomorphism }

With the compatibility conditions of Theorems \ref{cc_divergence} and \ref{cc_over-det} in hand, we may now completely characterize the solvability of the over-determined problem \eqref{problem_gamma_stokes_overdet}.  To do so, we first need to introduce a pair of function spaces for the data.

For $s \ge 0$ we define the space 
\begin{equation}\label{Ys_def}
 \mathcal{Y}^s = \{(f,g,h,k) \in H^s(\Omega; \R^n) \times H^{s+1}(\Omega) \times H^{s+3/2}(\Sigma_b) \times H^{s+1/2}(\Sigma_b;\R^n) \st 
 h\text{ and }g \text{ satisfy }  \eqref{H minus 1}  \}.
\end{equation}
We endow $\mathcal{Y}^s$ with the norm defined by
\begin{equation}
 \norm{(f,g,h,k)}_{\mathcal{Y}^s}^2 = \norm{f}_{H^s}^2 + \norm{g}_{H^{s+1}}^2 + \norm{h}_{H^{s+3/2}}^2 + \norm{k}_{H^{s+1/2}}^2 + \snorm{h - \int_0^b g(\cdot,x_n) dx_n}_{\dot{H}^{-1}}^2,
\end{equation}
which clearly makes $\mathcal{Y}^s$ into a Hilbert space (with the obvious inner-product associated to the norm).  Similarly, for $s \ge 0$ we define the subspace 
\begin{equation}\label{Zs_def}
 \mathcal{Z}^s = \{(f,g,h,k) \in \mathcal{Y}^s \st  \eqref{compatibility_condition} \text{ holds for every } \psi \in H^{s+1/2}(\Sigma_b) \}.
\end{equation}
The topology of $\mathcal{Y}^s$ guarantees that $\mathcal{Z}^s$ is a closed subspace, and so $\mathcal{Z}^s$ is a Hilbert space when endowed with the inner-product from $\mathcal{Y}^s$.

Next we establish the main result of this section, which shows that a necessary and sufficient condition for the existence of a solution to \eqref{problem_gamma_stokes_overdet} is that the  $f$, $g$, $k$, $h$ satisfy the
compatibility conditions \eqref{H minus 1} and \eqref{compatibility_condition} for every $\psi\in H^{s+1/2}(\Sigma)$.

\begin{theorem}\label{iso_overdetermined}
Let  $\gamma \in \R$, $s \ge 0$, and $\mathcal{Z}^s$ be the Hilbert space defined in \eqref{Zs_def}.  Then the bounded linear operator $\Psi_\gamma : {_{0}}H^{s+2}(\Omega;\mathbb{R}^{n})\times H^{s+1}(\Omega)  \rightarrow \mathcal{Z}^{s}$ given by 
\begin{equation}
 \Psi_\gamma(u,p) = 
 (\diverge{S(p,u)} - \gamma \p_1 u, 
 \diverge{u}, 
 \left. u_{n}\right\vert_{\Sigma_b},
\left.  S(p,u)e_{n}\right\vert_{\Sigma_b})
\end{equation}
is an isomorphism.
\end{theorem}

\begin{proof}
First note that in light of Theorems \ref{cc_divergence} and \ref{cc_over-det}, the map $\Psi_\gamma$ takes values in $\mathcal{Z}^s$ and is thus well-defined.  It is clearly a bounded linear operator.  The injectivity of $\Psi_\gamma$ follows from Theorem \ref{iso_gamma_stokes}. To
prove that $\Psi_\gamma$ is surjective let $(f,g,h,k)\in \mathcal{Z}^{s}$.  Using $f$, $g$, and $k$ in Theorem \ref{iso_gamma_stokes}, we find the  unique solution $u\in{_{0}}H^{s+2}(\Omega;\mathbb{R}^{n})$ and $p\in H^{s+1}(\Omega)$ to \eqref{problem_gamma_stokes_stress}. Given $\psi\in H^{s+1/2}(\Sigma)$, let $v\in{_{0}}H^{s+2}(\Omega;\mathbb{R}^{n})$ and $q\in H^{s+1}(\Omega)$ be the unique solution to \eqref{problem_adjoint} (the existence of which is again guaranteed by Theorem \ref{iso_gamma_stokes}).  Applying Theorem \ref{cc_over-det} and using the fact that $(f,g,h,k)$ satisfy the compatibility condition  \eqref{compatibility_condition}, we then find that 
\begin{equation}
\int_{\Sigma_b}u_{n}\psi=-\int_{\Omega}(f\cdot v-gq) + \int_{\Sigma_b}k\cdot v=\int_{\Sigma_b}h\psi.
\end{equation}
Then $\int_{\Sigma_b}(u_{n}-h)\psi=0$ for all $\psi\in H^{s+1/2}(\Sigma_b)$, which implies that $u_{n}=h$ on $\Sigma_b$.  Hence $\Psi_\gamma$ is surjective.
\end{proof}

\section{Fourier analysis}\label{sec_fourier}

In this section we consider the horizontal Fourier transform (as defined in Section \ref{sec_notation}) of the  linear problem \eqref{problem_gamma_stokes_stress}, where $f\in H^{s}(\Omega;\mathbb{R}^{n})$, $g\in H^{s+1}(\Omega)$, $k\in
H^{s+1/2}(\Sigma_b;\mathbb{R}^{n})$.   Note that the boundary condition
$S(p,u)e_{n}=k$ on $\Sigma_b$ may be decomposed into horizontal and vertical components: $-\p_n u' - \nabla' u_n = k'$ and $p - 2 \p_n u_n = k_n.$  Applying the horizontal Fourier transform to \eqref{problem_gamma_stokes_stress} then yields the following ODE boundary value problem for $\hat{u}(\xi,\cdot) \in H^2((0,b);\C^n)$ and $\hat{p}(\xi,\cdot) \in H^1((0,b);\C)$:
\begin{equation}\label{fourier system}
\begin{cases}
\left(  -\partial_{n}^{2}+4\pi^{2} \abs{\xi}^{2} \right)  \hat{u}' + 2\pi i\xi\hat
{p}- 2\pi i\xi_1 \gamma\hat{u}' =\hat{f}'+2\pi i \xi \hat{g} & \text{in } (0,b) \\
\left(  -\partial_{n}^{2}+4\pi^{2} \abs{\xi}^{2} \right)  \hat{u}_{n}+\partial_{n}
\hat{p} - 2 \pi i \xi_1 \gamma \hat{u}_{n} = \hat{f}_{n}+\partial_{n}\hat{g} & \text{in } (0,b)\\
2\pi i\xi \cdot \hat{u}'+\partial_{n}\hat{u}_{n}=\hat{g} & \text{in } (0,b)\\
-\partial_{n}\hat{u}' -2\pi i\xi \hat{u}_{n} = \hat{k}',\quad \hat{p} - 2\partial_{n}\hat{u}_{n}=\hat{k}_{n} & \text{for }x_{n}=b\\
\hat{u}=0 & \text{for }x_{n}=0.
\end{cases}
\end{equation}

\subsection{Generalities about the ODE system \eqref{fourier system}}

We begin our discussion of the ODE system \eqref{fourier system} by deriving an ODE variant of \eqref{weak_solution} and proving uniqueness of solutions.

\begin{proposition}\label{ODE_int_unique}
Suppose that $F \in L^2((0,b);\C^2),$ $G \in H^1((0,b);\C)$, and $K \in \C^2$.   Then the following hold.
\begin{enumerate}
 \item If $w \in H^2((0,b);\C^n)$ and $q \in H^1((0,b);\C)$ satisfy 
\begin{equation}\label{ODE_int_unique_01}
\begin{cases}
\left(  -\partial_{n}^{2}+4\pi^{2} \abs{\xi}^{2} \right)  w' + 2\pi i\xi q- 2\pi i\xi_1 \gamma w' = F'+2\pi i \xi G & \text{in } (0,b)\\
\left(  -\partial_{n}^{2}+4\pi^{2} \abs{\xi}^{2} \right)  w_{n} + \partial_{n}q - 2 \pi i \xi_1 \gamma w_{n} = F_{n}+\partial_{n} G & \text{in } (0,b)\\
2\pi i\xi \cdot w'+\partial_{n} w_{n}= G & \text{in } (0,b) \\
-\partial_{n} w' -2\pi i\xi w_{n} = K',\quad q-2\partial_{n} w_{n}=K_{n}, & \text{for }x_{n}=b \\
w=0 & \text{for }x_{n}=0,
\end{cases}
\end{equation}
then for $v \in H^1((0,b);\C^n)$ satisfying $v(0)=0$ we have that 
\begin{multline}\label{ODE_int_unique_02}
- K \cdot \overline{v(b)}   + \int_0^b  F \cdot \overline{v}  + q \overline{\left(2\pi i \xi \cdot v' + \p_n v_n \right)}    =
\int_0^b -\gamma 2\pi i \xi_1 w \cdot \overline{v} + 2 \p_n w_n \overline{\p_n v_n} +(\p_n w' + 2\pi i \xi w_n) \cdot \overline{(\p_n v' + 2\pi i \xi v_n) }
\\
+ \hal \int_0^b (2\pi i \xi \otimes w' + w' \otimes 2\pi i \xi) : \overline{(2\pi i \xi \otimes v' + v' \otimes 2\pi i \xi)}.
\end{multline}

 \item There exists at most one pair $(w,q) \in H^2((0,b);\C^n) \times H^1((0,b);\C)$ solving \eqref{ODE_int_unique_01}.
\end{enumerate}

\end{proposition}
\begin{proof}
Using the third equation in \eqref{ODE_int_unique_01}, we compute 
\begin{multline}
 (-\p_n^2 + 4 \pi^2 \abs{\xi}^2) w' + 2\pi i \xi q - 2\pi i \xi G =  (-\p_n^2 + 4 \pi^2 \abs{\xi}^2) w' + 2\pi i \xi q - 2\pi i \xi (2\pi i \xi \cdot w' + \p_n w_n) \\
= 2\pi i \xi q  - (2\pi i \xi \otimes w' + w' \otimes 2\pi i \xi) 2\pi i \xi - \p_n (\p_n w' + 2\pi i \xi w_n)
\end{multline}
and 
\begin{multline}
 (-\p_n^2 + 4 \pi^2 \abs{\xi}^2) w_n + \p_n q - \p_n G  =  (-\p_n^2 + 4 \pi^2 \abs{\xi}^2) w_n + \p_n q - \p_n (2\pi i \xi\cdot w' + \p_n w_n) \\
 = -2\pi i \xi \cdot (\p_n w' + 2\pi i \xi w_n) + \p_n(q-2\p_n w_n).
\end{multline}
Using these and the first two equations of \eqref{ODE_int_unique_01}, we then find that
\begin{equation}
 \int_0^b F' \cdot \overline{v'} + \gamma 2\pi i \xi_1 w'\cdot \overline{v'} = \int_0^b -q \overline{2\pi i \xi \cdot v'} + (2\pi i \xi \otimes w' + w' \otimes 2\pi i \xi) : \overline{ v' \otimes 2 \pi i \xi  } -   \p_n(\p_n w' + 2\pi i \xi w_n) \cdot \overline{v'}
\end{equation}
and 
\begin{equation}
 \int_0^b F_n \overline{v_n} + \gamma 2\pi i \xi_1 w_n \overline{v_n} = \int_0^b (\p_n w' + 2\pi i \xi w_n) \cdot \overline{2 \pi i \xi v_n} + \p_n (q-2\p_n w) \overline{v_n}.
\end{equation}
We then integrate by parts and use the boundary conditions in \eqref{ODE_int_unique_01} to see that 
\begin{equation}
 - \int_0^b \p_n(\p_n w' + 2\pi i \xi w_n) \cdot \overline{v'} = K' \cdot \overline{v'(b)} + \int_0^b (\p_n w' + 2\pi i \xi w_n) \cdot \overline{\p_n v'}
\end{equation}
and 
\begin{equation}
 \int_0^b  \p_n (q-2\p_n w) \overline{v_n}= K_n \overline{v_n}(b) - \int_0^b   (q-2\p_n w) \overline{\p_n v_n}.
\end{equation}
Combining these then shows that 
\begin{multline}
- K \cdot \overline{v(b)}   + \int_0^b  F \cdot \overline{v}  + q \overline{\left(2\pi i \xi \cdot v' + \p_n v_n \right)}    \\ =
\int_0^b -\gamma 2\pi i \xi_1 w \cdot \overline{v} + 2 \p_n w_n \overline{\p_n v_n} +(\p_n w' + 2\pi i \xi w_n) \cdot \overline{(\p_n v' + 2\pi i \xi v_n) }
\\
+  \int_0^b (2\pi i \xi \otimes w' + w' \otimes 2\pi i \xi) : \overline{v' \otimes 2\pi i \xi},
\end{multline}
and we conclude the proof of the first item by using the symmetry of $(2\pi i \xi \otimes w' + w' \otimes 2\pi i \xi)$ to rewrite 
\begin{equation}
 (2\pi i \xi \otimes w' + w' \otimes 2\pi i \xi) : \overline{v' \otimes 2\pi i \xi} = \hal (2\pi i \xi \otimes w' + w' \otimes 2\pi i \xi) : \overline{(2\pi i \xi \otimes v' + v' \otimes 2\pi i \xi)}.
\end{equation}

We now prove the second item.  If $w^j \in H^2((0,b);\C^n)$ and $q^j \in H^1((0,b);\C)$ for $j=1,2$ solve \eqref{ODE_int_unique_01}, then $w = w^1-w^2 \in H^2((0,b);\C^n)$ and $q = q^1 -q^2 \in H^1((0,b);\C)$ solve \eqref{ODE_int_unique_01} with $F=0$, $G=0$, $K=0$.  The first item with $v=w$ then implies that 
\begin{equation}
 \int_0^b -\gamma 2\pi i \xi_1  \abs{w}^2  + 2 \abs{\p_n w_n}^2 + \abs{\p_n w' + 2\pi i \xi w_n}^2 
+ \hal  \abs{2\pi i \xi \otimes w' + w' \otimes 2\pi i \xi}^2 =0.
\end{equation}
Taking the real part of this identity then shows that $\p_n w_n =0$ and $\p_n w' + 2\pi i \xi w_n =0$ in $(0,b)$.  Due to the boundary condition $w_n(0)=0$, we then have that $w_n=0$, which then implies that $\p_n w' =0$ and hence that $w'=0$ since $w'(0)=0$.  The second and fifth equations in \eqref{ODE_int_unique_01}  then require that $\p_n q =0$ and $q(b)=0$, which imply that $q=0$.  Hence $w^1=w^2$ and $q^1=q^2$, which proves the second item.

\end{proof}

In order to analyze the system \eqref{fourier system} it is convenient to decompose it into a pair of decoupled sub-systems.  We present this decoupling now. In the following result we suppress the functional dependence on $\xi$ for the sake of brevity, i.e. we write simply $\hat{u}(x_n)$ in place of $\hat{u}(\xi,x_n)$, etc.

\begin{proposition}\label{ODE_equivalence_full}
Suppose that $\hat{f} \in L^2((0,b);\C^n)$, $\hat{g} \in H^1((0,b);\C)$ and $\hat{k} \in \C^n$.  Further suppose that  $\hat{u} \in H^2((0,b);\C^n)$, $\hat{p} \in H^1((0,b);\C)$, $\varphi,\psi \in H^2((0,b);\C)$, $q \in H^1((0,b);\C)$, and $\vartheta \in H^2((0,b);\C^{n-1})$.  Then the following are equivalent for every $\xi \in \mathbb{R}^{n-1} \backslash \{0\}$.
\begin{enumerate}
 \item $\hat{p},\hat{u}$ solve \eqref{fourier system}.
 \item We have that 
\begin{equation}\label{ODE_equivalence_0}
\hat{p} = q, \; \hat{u}' = -i \varphi \frac{\xi}{\abs{\xi}} + \vartheta, \text{ and } \hat{u}_n = \psi, 
\end{equation}
$\varphi,\psi,q$ solve 
\begin{equation}\label{phi_psi_system}
\begin{cases}
\left(  -\partial_{n}^{2}+4\pi^{2} \abs{\xi}^{2} \right)  \varphi - 2\pi \abs{\xi} q- 2\pi i\xi_1 \gamma \varphi =i \hat{f}'\cdot \xi/\abs{\xi} - 2\pi \abs{\xi} \hat{g} & \text{in }  (0,b) \\
\left(  -\partial_{n}^{2}+4\pi^{2} \abs{\xi}^{2} \right)  \psi + \partial_n q - 2 \pi i \xi_1 \gamma \psi = \hat{f}_{n} + \partial_{n} \hat{g} & \text{in }  (0,b) \\
2\pi \abs{\xi}   \varphi +\partial_{n} \psi =\hat{g} & \text{in  } (0,b) \\
-\partial_{n} \varphi  + 2\pi \abs{\xi} \psi = i\hat{k}'\cdot \xi/\abs{\xi} ,\quad q-2\partial_{n} \psi =\hat{k}_{n} & \text{for }x_{n}=b \\
\varphi = \psi =0 & \text{for }x_{n}=0,
\end{cases}
\end{equation}
and $\vartheta$  solves
\begin{equation}\label{theta_system}
\begin{cases}
\left(  -\partial_{n}^{2}+4\pi^{2} \abs{\xi}^{2} \right) \vartheta - 2\pi i\xi_1 \gamma \vartheta = (1-\xi \otimes \xi/ \abs{\xi}^2)  \hat{f}' & \text{in } (0,b) \\
-\partial_{n} \vartheta  = (1-\xi \otimes \xi/ \abs{\xi}^2) \hat{k}'   & \text{for }x_{n}=b \\
\vartheta = 0 & \text{for }x_{n}=0,
\end{cases}
\end{equation}
which in particular requires that $\vartheta \cdot \xi =0$ on $(0,b)$.
\end{enumerate}
In either case (and hence both), the solutions are unique.
\end{proposition}
\begin{proof}
First note that if $\vartheta$ solves \eqref{theta_system}, then taking the dot product with $\xi$ reveals that $\chi := \xi \cdot \vartheta \in H^2((0,b);\C)$ solves 
\begin{equation}
\begin{cases}
\left(  -\partial_{n}^{2}+4\pi^{2} \abs{\xi}^{2} \right) \chi - 2\pi i\xi_1 \gamma \chi = 0 & \text{in
} (0,b) \\
-\partial_{n} \chi  = 0  & \text{for }x_{n}=b \\
\chi = 0 & \text{for }x_{n}=0.
\end{cases}
\end{equation}
We then multiply the first equation by $\bar{\chi}$ and integrate by parts over $(0,b)$ to conclude that 
\begin{equation}
 \int_0^b \abs{\p_n \chi}^2 + (4 \pi^2 \abs{\xi}^2 -  2\pi i \xi_1 \gamma)  \abs{\chi}^2    =0.
\end{equation}
Taking the real part of this equation then shows that $\chi =0$ on $(0,b)$, and hence $\vartheta \cdot \xi =0$.  

Now suppose $\hat{p},\hat{u}$ solve \eqref{fourier system}.  Then we define $q = \hat{p}$, $\varphi = i \hat{u}' \cdot \xi/\abs{\xi}$, $\psi = \hat{u}_n$, and $\vartheta = (1-\xi\otimes \xi/\abs{\xi}^2) \hat{u}'$, which implies \eqref{ODE_equivalence_0}.  Then \eqref{phi_psi_system} follows from \eqref{fourier system} by taking the dot product with $i \xi /\abs{\xi}$, and \eqref{theta_system} follows by multiplying by the projector matrix $(1-\xi \otimes \xi/ \abs{\xi}^2)$.  

On the other hand, if $\varphi,\psi,q$ solve \eqref{phi_psi_system} and $\vartheta$ solves \eqref{theta_system}, then we define $\hat{u}$ and $\hat{p}$ via \eqref{ODE_equivalence_0}.  We then multiply the first and fourth equations in \eqref{phi_psi_system} by $i \xi/\abs{\xi}$ and combine with \eqref{theta_system} and the remaining equations in \eqref{phi_psi_system} to obtain \eqref{fourier system}.

The uniqueness claim follows from the uniqueness result of Proposition \ref{ODE_int_unique}.
\end{proof}

It is also convenient to reformulate the coupled system \eqref{phi_psi_system} as a first-order equation.  We present this equivalent formulation now.  Note that in this result we present the system with slightly more general data and we allow for $\xi =0$ as well.

\begin{proposition}\label{ODE_equivalence_reduced}
Suppose that $F \in L^2((0,b);\C^2),$ $G \in H^1((0,b);\C)$, and $K \in \C^2$.  Further suppose that $y \in H^1((0,b);\C^4)$,  $\varphi,\psi \in H^2((0,b);\C)$, $q \in H^1((0,b);\C)$.  Then the following are equivalent for every $\xi \in \R^{n-1}$.
\begin{enumerate}
 \item $\varphi,\psi,q$ solve the second-order boundary value problem
\begin{equation}\label{general_phi_psi} 
\begin{cases}
\left(  -\partial_{n}^{2}+4\pi^{2} \abs{\xi}^{2} \right)  \varphi - 2\pi \abs{\xi} q- 2\pi i\xi_1 \gamma \varphi =F_1 - 2\pi \abs{\xi} G & \text{in } (0,b) \\
\left(  -\partial_{n}^{2}+4\pi^{2} \abs{\xi}^{2} \right)  \psi + \partial_n q - 2 \pi i \xi_1 \gamma \psi = F_2 + \partial_{n} G & \text{in  } (0,b) \\
2\pi \abs{\xi}   \varphi +\partial_{n} \psi =G & \text{in }(0,b) \\
-\partial_{n} \varphi  + 2\pi \abs{\xi} \psi = K_1 ,\quad q-2\partial_{n} \psi =K_2 & \text{for }x_{n}=b \\
\varphi = \psi =0 & \text{for }x_{n}=0.
\end{cases}
\end{equation}

 \item $y=(\varphi,\psi,q,\p_n \varphi)$ and $y$ solves the first-order two-point boundary value problem 
\begin{equation}\label{problem two point} 
\begin{cases}
\partial_n y =A y +z \text{ in } (0,b) \\
M y(0)+N y(b)=d,
\end{cases}
\end{equation}
\end{enumerate}
where $A \in \C^{4 \times 4}$ is given by
\begin{equation}\label{matrix A}
A=
\begin{pmatrix}
0 & 0 & 0 & 1\\
-2\pi \abs{\xi} & 0 & 0 & 0\\
0 & -(4\pi^2 \abs{\xi}^2- i 2\pi \xi_1 \gamma ) & 0 & -2\pi \abs{\xi} \\
4\pi^2 \abs{\xi}^2- i 2\pi \xi_1 \gamma  & 0 & -2\pi \abs{\xi} & 0,
\end{pmatrix},
\end{equation}
$z \in L^2((0,b);\C^4)$ and $d \in \C^4$ are given by 
\begin{equation}
 z(x_n)=
\begin{pmatrix}
0\\
G(x_n)\\
F_2(x_n)+2\partial_{n}G(x_n)\\
-F_1(x_n)  + 2 \pi \abs{\xi} G(x_n),
\end{pmatrix}
\text{ and } 
d=
\begin{pmatrix}
0\\
0\\
K_1 \\
K_2 + 2G(b)
\end{pmatrix},
\end{equation}
and $M,N \in \C^{4 \times 4}$ are given by
\begin{equation}\label{matrices M and N} 
M =
\begin{pmatrix}
1 & 0 & 0 & 0\\
0 & 1 & 0 & 0\\
0 & 0 & 0 & 0\\
0 & 0 & 0 & 0
\end{pmatrix}
\text{ and } N=
\begin{pmatrix}
0 & 0 & 0 & 0\\
0 & 0 & 0 & 0\\
0 & 2\pi \abs{\xi} & 0 & -1\\
4\pi \abs{\xi} & 0 & 1 & 0
\end{pmatrix}.
\end{equation}
\end{proposition}
\begin{proof}
Suppose that $\varphi,\psi$ and $q$ solve \eqref{general_phi_psi} and let $y = (\varphi,\psi,q,\p_n \varphi)$.  Note that $y_1,y_2 \in H^2((0,b);\C)$.  We differentiate the third equation to obtain the equation 
\begin{equation}\label{ODE_equivalence_reduced_2}
\p_n^2 y_2 = \p_n^2 \psi = \p_n G - 2\pi \abs{\xi} \p_n \varphi = \p_n G - 2\pi \abs{\xi} y_4.
\end{equation}
From this we readily deduce that $y$ solves the system 
\begin{equation}\label{ODE_equivalence_reduced_1}
\begin{cases}
\partial_n y_{1}=y_{4} & \text{in } (0,b) \\
\partial_n y_{2}=-2\pi \abs{\xi} y_{1}+G & \text{in } (0,b) \\
\partial_n y_{3}=-(4\pi^2 \abs{\xi}^2-  2\pi i \xi_1  \gamma )y_{2}- 2\pi \abs{\xi} y_{4} + F_2 + 2\partial_{n} G  & \text{in } (0,b) \\
\partial_n y_{4}=(4\pi^2 \abs{\xi}^2-  2\pi i \xi_1 \gamma )y_{1}- 2\pi \abs{\xi} y_{3}- F_1 +2\pi \abs{\xi} G & \text{in } (0,b) \\
-y_{4} + 2\pi \abs{\xi} y_{2} = K_1 ,\quad y_{3}+4\pi \abs{\xi} y_{1}=K_2 + 2 G & \text{for }x_{n}=b \\
y_{1}=0,\quad y_{2}=0 & \text{for }x_{n}=0,
\end{cases}
\end{equation}
which may be compactly rewritten as \eqref{problem two point}.

Now suppose that $y$ solves \eqref{problem two point}, which is equivalent to \eqref{ODE_equivalence_reduced_1}.  Define $\varphi = y_1$, $\psi = y_2$, and $q = y_3$, all of which then belong to $H^1((0,b);\C)$.  However, $\p_n \varphi = \p_n y_1 = y_4 \in H^1((0,b);\C)$ and $\p_n \psi = \p_n y_2 = G -2 \pi \abs{\xi} \varphi \in H^1((0,b);\C)$, so $\varphi,\psi \in H^2((0,b);\C)$.  In turn this implies that we may differentiate the second equation in \eqref{ODE_equivalence_reduced_1} to see that \eqref{ODE_equivalence_reduced_2} holds.  Then the second equation in \eqref{ODE_equivalence_reduced_1} corresponds to the third in \eqref{general_phi_psi}, the fourth in \eqref{ODE_equivalence_reduced_1} corresponds to the first in \eqref{general_phi_psi}, and the third in \eqref{ODE_equivalence_reduced_1} corresponds to the second in \eqref{general_phi_psi} in light of the identity \eqref{ODE_equivalence_reduced_2}.  The equivalence of the boundary conditions follows similarly.
\end{proof}

Consider the matrix $A\in \C^{4 \times 4}$ given by \eqref{matrix A}.  Given $z \in L^2((0,b);\C^4)$, the unique solution $y\in H^1((0,b);\C^4)$ to the ODE 
\begin{equation}
\begin{cases}
\p_n y = A y + z &\text{in }(0,b) \\
y(0) = y_0
\end{cases}
\end{equation}
is given by 
\begin{equation}\label{ODE_general_soln}
 y(x_n) = \exp(x_n A) y_0 + \int_0^{x_n} \exp((x_n-t)A) z(t) dt.
\end{equation}
Let $M,N \in \C^{4 \times 4}$ be given by \eqref{matrices M and N} and define the boundary matrix
\begin{equation}\label{boundary matrix} 
B:=M+N\exp(bA) \in \C^{4 \times 4}. 
\end{equation}
Thus the solvability of the two-point problem \eqref{problem two point} reduces to solving for $y_0 \in \C^4$ such that $d = M y_0 + N y(b)$, which in light of \eqref{ODE_general_soln} is equivalent to 
\begin{equation}\label{B y0} 
By_{0} = M y_{0} + N \exp(b A)y_{0}  = d -N\int_{0}^{b}\exp((b-t)A)z(t)dt .
\end{equation}

Our next result establishes that $B$ is invertible for every $\xi \in \R^{n-1}$, which then allows us to make various conclusions about \eqref{problem two point}.  An interesting feature of our approach is that we establish the invertibility of $B$ by using the isomorphism from Theorem \ref{iso_gamma_stokes} rather than through direct computation.  We do this because although $\det{B}$ can be computed by hand (and we will do so later in Section \ref{sec_infty_asympt}), the resulting expression is quite cumbersome, and it is rather tricky to prove directly that it never vanishes.

\begin{theorem}\label{ODE_B_inversion}
Let $\xi \in \R^{n-1}$ and $A,M,N,B \in \C^{4 \times 4}$ be given by \eqref{matrix A},  \eqref{matrices M and N}, and \eqref{boundary matrix}, respectively.  Then the following hold.
\begin{enumerate}
 \item The boundary matrix $B$ has the block structure 
\begin{equation}
B=
\begin{pmatrix}
I_{2\times 2} & 0_{2\times 2}\\
B_{3} & B_{4}
\end{pmatrix}
\end{equation}
where $B_3,B_4 \in \C^{2 \times 2}$ are given by
\begin{equation}
B_3 =
\begin{pmatrix}
 2\pi \abs{\xi} \exp(bA)_{21} - \exp(bA)_{41} & 2\pi \abs{\xi} \exp(bA)_{22} - \exp(bA)_{42} \\
 4 \pi \abs{\xi} \exp(bA)_{11} + \exp(bA)_{31} & 4 \pi \abs{\xi} \exp(bA)_{12} + \exp(bA)_{32}
\end{pmatrix}
\end{equation}
and 
\begin{equation}
B_4 = 
\begin{pmatrix}
2\pi \abs{\xi} \exp(bA)_{23} - \exp(bA)_{43} & 2\pi \abs{\xi} \exp(bA)_{24} - \exp(bA)_{44} \\
4\pi \abs{\xi} \exp(bA)_{13} + \exp(bA)_{33} & 4\pi \abs{\xi} \exp(bA)_{14} + \exp(bA)_{34}
\end{pmatrix}.
\end{equation}

 \item $B_4 \in \C^{2 \times 2}$ is invertible.

 \item $B$ is invertible, and we have the identities $\det{B} = \det{B_4}$ and
\begin{equation}
 B^{-1} = 
\begin{pmatrix}
I_{2 \times 2} & 0_{2 \times 2} \\
- B_4^{-1} B_3 & B_4^{-1}
\end{pmatrix}.
\end{equation}

 \item For every $z \in L^2((0,b);\C^4)$ and $d \in \C^4$ there exists a unique solution $y \in H^1((0,b);\C^4)$ to the problem 
\begin{equation}
\begin{cases}
\p_n y = Ay + z &\text{in }(0,b) \\
My(0) + N y(b) = d,
\end{cases}
\end{equation}
which is given by
\begin{equation}\label{general solution}
y(x_n) = \exp(x_n A)B^{-1} \left(d - N \int_{0}^{b}\exp((b-t)A)z(t)dt \right) + \int_0^{x_n} \exp((x_n-t)A) z(t) dt.
\end{equation}

\end{enumerate}
\end{theorem}

\begin{proof}

The first item follows from a direct calculation, using the block structure of $M,N$:
\begin{equation}\label{MN_block_form}
M=
\begin{pmatrix}
I_{2\times 2} & 0_{2\times 2}\\
0_{2\times 2} & 0_{2\times 2}%
\end{pmatrix},
\text{ and }
N=
\begin{pmatrix}
0_{2\times 2} & 0_{2\times 2}\\
N_{3} & N_{4}
\end{pmatrix}
\end{equation}
for $N_3,N_4 \in \C^{2 \times 2}$.  The third item follows from the second and a simple calculation.  The fourth item then follows from the third item, combined with \eqref{ODE_general_soln} and \eqref{B y0}.  It remains only to prove the second item.  

Suppose initially that $\xi =0$.  In this case we may readily compute 
\begin{equation}
 B_4 = N_4 =
\begin{pmatrix}
0 & -1 \\
1 & 0
\end{pmatrix}
\end{equation}
to deduce that $B_4$ is invertible.  In the case $\xi \in \R^{n-1} \backslash \{0\}$ the value of $\det{B_4}$ can be computed explicitly from the first item, but the resulting expression is rather complicated.  To avoid working directly with $\det{B_4}$ we will instead employ Theorem \ref{iso_gamma_stokes} to show that $B_4$ is invertible.  Let $m \in \N$ and pick a radial function $\zeta \in C^\infty_c(\R^{n-1})$ such that $\zeta =1$ on $B(0,2^m) \backslash B[0,2^{-m}]$.  For $j =1,2$ let $k^1,k^2 \in \mathscr{S}(\R^{n-1};\C^n)$ be given via 
\begin{equation}\label{ODE_B_inversion_1}
\hat{k}^1(\xi) = (-i \zeta(\xi) \xi / \abs{\xi},0) \text{ and } \hat{k}^2(\xi) = \zeta(\xi) e_n. 
\end{equation}
Then by construction $\overline{\hat{k}^j(\xi)} = \hat{k}^j(-\xi)$, and so Lemma \ref{tempered_real_lemma} shows that $k^j$ actually takes values in $\R^n$.  

We may then use $f=0$, $g =0$, and $k = k^j \in \bigcap\limits_{s>0} H^s(\Sigma_b;\R^n)$ for $j=1,2$ in Theorem \ref{iso_gamma_stokes} to produce $(u^j,p^j)  \in \bigcap\limits_{s >0}  {_{0}}H^{s+2}(\Omega;\mathbb{R}^{n})\times H^{s+1}(\Omega)$ solving \eqref{problem_gamma_stokes_stress}.  For $\xi \in \R^{n-1}\backslash \{0\}$ define $y^j(\xi,\cdot) \in C^\infty([0,b];\C^4)$   via 
\begin{equation}
y^j(\xi,x_n)  = (i \hat{u}^j(\xi,x_n) \cdot \xi /\abs{\xi}, \hat{u}^j_n(\xi,x_n), \hat{p}^j(\xi,x_n), i \p_n \hat{u}^j(\xi,x_n) ). 
\end{equation}
Since $(\hat{u}^j,\hat{p}^j)$ satisfy \eqref{fourier system}, Propositions \ref{ODE_equivalence_full} and \ref{ODE_equivalence_reduced}, together with \eqref{B y0} and \eqref{ODE_B_inversion_1} and the fact that $z=0$, imply that if $2^{-m} < \abs{\xi} < 2^m$ then $B y^j(\xi,0) = e_{2+j}$.  Since $y^j(\xi,0) \cdot e_1 =  y^j(\xi,0) \cdot e_2 =0$ for all $\xi \neq 0$, we may write $y^j(\xi,0) = (0,0,\nu^j(\xi))$ for $\nu^j(\xi) \in \C^2$.  Then the identity $B y^j(\xi,0) = e_{2+j}$ is equivalent to $B_4 \nu^j(\xi) = e_j$ for $j =1,2$, and we deduce that for $2^{-m} < \abs{\xi} < 2^m$ the matrix $B_4 \in \C^{2\times 2}$ has rank two and is thus invertible.  Since $m \in \N$ was arbitrary we then conclude that $B_4$ is invertible for all $\xi \in \R^{n-1} \backslash \{0\}$, which concludes the proof of the second item.

\end{proof}

\subsection{Some special functions}

With Theorem \ref{ODE_B_inversion} in hand we are now in a position to introduce some functions that will play a fundamental role in our subsequent analysis.  For $\xi \in \R^{n-1}$ and $\gamma \in \R$ write $A(\xi,\gamma), B(\xi,\gamma) \in \C^{4 \times 4}$ for the matrices defined by \eqref{matrix A} and \eqref{boundary matrix}, respectively.  In light of Theorem \ref{ODE_B_inversion} we may then define $Q: \R^{n-1} \times [0,b] \times \R \to \C$, $V : \R^{n-1} \times [0,b] \times \R \to \C^n$, and $m: \R^{n-1}\times \R \to \C$ via 
\begin{equation}\label{QVm_def}
\begin{split}
Q(\xi,x_n,\gamma)  &= \exp(x_n A(\xi,\gamma)) B^{-1}(\xi,\gamma) e_4 \cdot e_3  \in \C \\
V'(\xi,x_n,\gamma) &= -i  \left( \exp(x_n A(\xi,\gamma)) B^{-1}(\xi,\gamma) e_4 \cdot e_1 \right) \frac{\xi}{\abs{\xi}}  \in \C^{n-1} \text{ for } \xi \neq 0 \text{ and } V'(0,x_n,\gamma) = 0 \in \C^{n-1}  \\
V_n(\xi,x_n,\gamma) &= \exp(x_n A(\xi,\gamma)) B^{-1}(\xi,\gamma) e_4 \cdot e_2  \in \C \\
m(\xi,\gamma) &= V_n(\xi,b,\gamma) =  \exp(b A(\xi,\gamma)) B^{-1}(\xi,\gamma) e_4 \cdot e_2 \in \C.
\end{split}
\end{equation}

The following result records some essential properties of these functions.

\begin{theorem}\label{QVm_properties}
Let  $Q: \R^{n-1} \times [0,b] \times \R \to \C$, $V : \R^{n-1} \times [0,b] \times \R \to \C^n$, and $m: \R^{n-1}\times \R \to \C$ be as defined in \eqref{QVm_def}.  Then the following hold.
\begin{enumerate}
  \item $Q$, $V$, and $m$ are continuous, $Q$ and $V$ are smooth on $(\R^{n-1} \backslash \{0\}) \times [0,b] \times \R$, and $m$ is smooth on  $(\R^{n-1} \backslash \{0\} )\times \R$.  Also, for each $\xi \in \R^{n-1}$ we have that $Q(\xi,\cdot)$ and $V(\xi,\cdot)$ are smooth on $[0,b]$.  
  \item $V(0,x_n,\gamma) =0$, $Q(0,x_n,\gamma) =1$, and $m(0,\gamma) =0$.
  \item For each $\xi \in \R^{n-1}$, $x_n \in [0,b]$, and $\gamma \in \R$ we have that $\overline{V(\xi,x_n,\gamma)} = V(-\xi,x_n,\gamma)$, $\overline{Q(\xi,x_n,\gamma)} = Q(-\xi,x_n,\gamma)$, and $\overline{m(\xi,\gamma)} = m(-\xi,\gamma)$.
  \item For each $\xi \in \R^{n-1}$ we have that $Q(\xi,\cdot,\gamma)$, $V(\xi,\cdot,\gamma)$ solve 
\begin{equation}\label{QVm_properties_0}
\begin{cases}
\left(  -\partial_{n}^{2}+4\pi^{2} \abs{\xi}^{2} \right)  V' + 2\pi i \xi Q - 2\pi i\xi_1 \gamma V' =0 & \text{in } (0,b) \\
\left(  -\partial_{n}^{2}+4\pi^{2} \abs{\xi}^{2} \right)  V_{n} + \partial_{n} Q - 2 \pi i \xi_1 \gamma V_{n} =0 & \text{in } (0,b) \\
2\pi i\xi \cdot V' + \partial_{n}V_{n}=0 & \text{in } (0,b) \\
-\partial_{n} V' -2\pi i\xi V_{n} = 0,\quad Q - 2\partial_{n} V_{n}=1 & \text{for }x_{n}=b \\
V=0 & \text{for }x_{n}=0.  
\end{cases}
\end{equation}

\item If $(u,p) \in {_{0}}H^{2}(\Omega;\mathbb{R}^{n})\times H^{1}(\Omega)$ solve \eqref{problem_gamma_stokes_stress} with $f=0$, $g=0$, and $k = \zeta e_n$ for $\zeta \in H^{1/2}(\R^{n-1})$, then $\hat{u} = \hat{\zeta} V(\cdot,\cdot,\gamma)$ and $\hat{p} = \hat{\zeta} Q(\cdot,\cdot,\gamma)$.

\item $\RE{m(\xi,\gamma)} \le 0$ for all $\xi \in \R^{n-1}$ and $\gamma \in \R$, and $\RE{m(\xi,\gamma)} =0$ if and only if $\xi =0$.
  
\end{enumerate}
\end{theorem}
\begin{proof}
Define $y : \R^{n-1} \times [0,b] \times \R \to \C^4$ via $y(\xi,x_n,\gamma) = \exp(x_n A(\xi,\gamma)) B^{-1}(\xi,\gamma) e_4$.  Theorem \ref{ODE_B_inversion} shows that $y$ is continuous, smooth on $(\R^{n-1} \backslash \{0\}) \times [0,b] \times \R$, and that for $\xi$ fixed $y(\xi,\cdot,\cdot)$ is smooth on $[0,b] \times \R$.  We have that $Q= y_3$, $V_n = y_2$, $m = y_2(\cdot,b,\cdot)$, and for $\xi \neq 0$, $V'(\xi,x_n,\gamma) = -i y_1(\xi,x_n,\gamma) \xi / \abs{\xi}$.  Thus, to complete the proof of the first two items it suffices to notice that 
\begin{equation}
 \lim_{(\xi,t,\gamma) \to (0,x_n,\gamma_0)} y(\xi,t,\gamma) = \exp(x_n A(0,\gamma_0)) B^{-1}(0,\gamma_0) e_4 = 
\begin{pmatrix}
1 & 0 & -x_n & 0 \\
0 & 1 & 0 & 0 \\
0 & 0 & 0 & 1 \\
0 & 0 &-1 & 0
\end{pmatrix}
\begin{pmatrix}
 0 \\ 0 \\ 0 \\ 1
\end{pmatrix}
= 
\begin{pmatrix}
 0 \\ 0 \\ 1 \\ 0
\end{pmatrix},
\end{equation}
and hence $y(\xi,t,\gamma) \to e_3 = y(0,x_n,\gamma_0)$ as $(\xi,t,\gamma) \to (0,x_n,\gamma_0)$.

To prove the third item we note that $A(-\xi,\gamma) = \overline{A(\xi,\gamma)}$, and if we write $N(\xi) \in \C^{4 \times 4}$ to emphasize the $\xi$ dependence of the matrix defined in \eqref{matrices M and N}, then $N(-\xi) = N(\xi)$.  From this we have that $B(-\xi,\gamma) = M + N(-\xi) \exp(x_n A(-\xi,\gamma)) = M + N(\xi) \exp(x_n \overline{A(\xi,\gamma)}) = \overline{B(\xi,\gamma)}$, and hence that $B^{-1}(-\xi,\gamma) = \overline{B^{-1}(\xi,\gamma)}$.  Hence $\overline{y(\xi,x_n,\gamma)} = y(-\xi,x_n,\gamma)$ for all $\xi \in \R^{n-1}$, $x_n \in [0,b]$, and $\gamma \in \R$.  The third item then follows directly from this and the definitions of $V,Q,$ and $m$ in terms of $y$.  

The fourth item follows immediately from Propositions \ref{ODE_equivalence_full} and \ref{ODE_equivalence_reduced} when $\xi \neq 0$ and from the second item and a trivial calculation when $\xi =0$.  The fifth item follows from the fourth and Proposition \ref{ODE_equivalence_full}.

We now turn to the proof of the sixth item.  In light of the fourth item and  Proposition \ref{ODE_int_unique} we have the identity 
\begin{multline}
 \int_0^b \left( -\gamma 2\pi i \xi_1  \abs{V(\xi,x_n,\gamma)}^2  + 2 \abs{\p_n V_n(\xi,x_n,\gamma)}^2 + \abs{\p_n V'(\xi,x_n,\gamma) + 2\pi i \xi V_n(\xi,x_n,\gamma)}^2 \right) dx_n \\
+ \hal\int_0^b   \abs{2\pi i \xi \otimes V'(\xi,x_n,\gamma) + V'(\xi,x_n,\gamma) \otimes 2\pi i \xi}^2   dx_n=  -m(\xi,\gamma). 
\end{multline}
Taking the real part of this identity yields
\begin{multline}
- \RE{m(\xi,\gamma)} = \int_0^b \left(  2 \abs{\p_n V_n(\xi,x_n,\gamma)}^2 + \abs{\p_n V'(\xi,x_n,\gamma) + 2\pi i \xi V_n(\xi,x_n,\gamma)}^2 \right) dx_n \\
+ \hal\int_0^b   \abs{2\pi i \xi \otimes V'(\xi,x_n,\gamma) + V'(\xi,x_n,\gamma) \otimes 2\pi i \xi}^2   dx_n,
\end{multline}
which immediately implies that $\RE{m(\xi,\gamma)} \le 0$ for all $\xi \in \R^{n-1}$ and $\gamma \in \R$.  Moreover, if $\RE{m(\xi,\gamma)}=0$ for $\xi \neq 0$, then this identity and the sixth equation in \eqref{QVm_properties_0} show that $V(\xi,\cdot,\gamma) =0$, and so the first and fifth equations in \eqref{QVm_properties_0} show that $Q(\xi,\cdot,\gamma)=0$ but $Q(\xi,b,\gamma) =1$, a contradiction.  Hence $\RE{m(\xi,\gamma)} < 0$ for $\xi \neq 0$.
 
\end{proof}

\begin{remark}\label{remark_symbol}
 The fifth item of Theorem \ref{QVm_properties} shows that $m(\cdot,\gamma)$ is the symbol of pseudodifferential operator corresponding to the normal-stress to normal-Dirichlet map given by \eqref{intro_stress_to_dirichlet}.
\end{remark}

We know from Theorem \ref{QVm_properties} that $V(0,x_n,\gamma) =0$, $Q(0,x_n,\gamma) =1$, and $m(0,\gamma) =0$.  Later in the paper we will crucially require a finer asymptotic development as $\abs{\xi} \to 0$.  We derive this now.

\begin{theorem}\label{QVm_zero}
Let  $Q: \R^{n-1} \times [0,b] \times \R \to \C$, $V : \R^{n-1} \times [0,b] \times \R\to \C^n$, and $m: \R^{n-1} \times \R\to \C$ be as defined in \eqref{QVm_def}.   Then for $\abs{\xi} \ll 1$ we have the asymptotic developments 
\begin{align}\label{QVm_zero_0}
 V'(\xi,x_n,\gamma)  &= -i \pi \xi(2x_n b - x_n^2) + O(\abs{\xi}^2),   & V_n(\xi,x_n,\gamma) &= 2 \pi^2 \abs{\xi}^2 x_n^2 \left(\frac{x_n}{3} - b \right) +O(\abs{\xi}^3), \\
 m(\xi,\gamma)   &= -\frac{4\pi^2 \abs{\xi}^2 b^3}{3} + O(\abs{\xi}^3),  &
 Q(\xi,x_n,\gamma) &=   1 + O(\abs{\xi}^2), \nonumber
\end{align}
where here we write $F(\xi,x_n) = O(\abs{\xi}^k)$ to mean that
\begin{equation}
\limsup_{\abs{\xi} \to 0}  \sup_{0 \le x_n \le b}  \frac{\abs{F(\xi,x_n)}}{\abs{\xi}^k} < \infty.
\end{equation}
\end{theorem}
\begin{proof}
Fix $\gamma \in \R$.  Throughout the proof we will suppress the functional dependence on $\gamma$ in $A$ and $B$, writing $A(\xi)$, $B(\xi)$ in place of $A(\xi,\gamma)$ and $B(\xi,\gamma)$.

Write $P(\xi,x_n) = \exp(x_n A(\xi))$, and introduce the block form 
\begin{equation}
 P(\xi,x_n) = 
\begin{pmatrix}
 P_1(\xi,x_n) & P_2(\xi,x_n) \\
 P_3(\xi,x_n) & P_4(\xi,x_n)
\end{pmatrix}
\end{equation}
for $P_j(\xi,x_n) \in \C^{2 \times 2}$ for $1\le j \le 4$.  Using the block form of $B^{-1}(\xi)$ from Theorem \ref{ODE_B_inversion}, we may compute 
\begin{equation}\label{QVm_zero_1}
 \exp(x_n A(\xi)) B^{-1}(\xi) =
\begin{pmatrix}
P_1(\xi,x_n) - P_2(\xi,x_n) B_4^{-1}(\xi) B_3(\xi) & P_2(\xi,x_n) B_4^{-1}(\xi) \\
P_3(\xi,x_n) - P_4(\xi,x_n) B_4^{-1}(\xi) B_3(\xi) & P_4(\xi,x_n) B_4^{-1}(\xi)
\end{pmatrix}.
\end{equation}

A simple computation, which may be done by hand or rapidly verified with a computer algebra system, shows that $\abs{A(\xi)} = O(1)$ and $\abs{A(\xi)^6} = O(\abs{\xi}^3)$, from which we deduce that $\abs{A(\xi)^k} = O(\abs{\xi}^3)$ for $k \ge 6$.  Then $\exp(x_n A(\xi)) = P(\xi,x_n) = \sum_{j=0}^6 A(\xi)^j / j! + O(\abs{\xi}^3)$, and we may compute $\sum_{j=0}^6 A(\xi)^j / j!$ by hand (or with a computer algebra system) and truncate to second order to write 
\begin{equation}
P_2(\xi,x_n) = 
\begin{pmatrix}
 -\pi \abs{\xi} x_n^2 + x_n^4 \frac{\pi^2 i \gamma \abs{\xi} \xi_1 }{6} 
    & x_n + x_n^3 \frac{4\pi^2 \abs{\xi}^2 - \pi i \gamma \xi_1}{3} - x_n^5 \frac{\pi^2 \gamma^2 \xi_1^2}{30} \\
x_n^3 \frac{2\pi^2 \abs{\xi}^2}{3} 
    & -\pi \abs{\xi} x_n^2 + x_n^4 \frac{\pi^2 i \gamma \abs{\xi}\xi_1}{6} 
\end{pmatrix}
+ O(\abs{\xi}^3)
=: Q_2(\xi,x_n)+ O(\abs{\xi}^3)
\end{equation}
and 
\begin{multline}
P_4(\xi,x_n) = 
\begin{pmatrix}
1 + 2\pi^2 \abs{\xi}^2 x_n^2  
    & -2\pi \abs{\xi} x_n \\
-2\pi \abs{\xi} x_n + x_n^3 \frac{2 \pi^2 i \gamma \abs{\xi} \xi_1}{3} 
    & 1 + x_n^2 (4\pi^2 \abs{\xi}^2 - \pi i \gamma \xi_1) - x_n^4 \frac{\pi^2 \gamma^2 \xi_1^2}{6}
\end{pmatrix}
+ O(\abs{\xi}^3)\\
=: Q_4(\xi,x_n) +  O(\abs{\xi}^3).
\end{multline}
Using this, the expression for $B_4$ from Theorem \ref{ODE_B_inversion}, and the block form of \eqref{MN_block_form}, we may then compute $B_4(\xi) = N_4 + R(\xi) + O(\abs{\xi}^3)$, for 
\begin{equation}
 R(\xi) =  
\begin{pmatrix}
 2 \pi b \abs{\xi} - \frac{2\pi^2 b^3 i \gamma \abs{\xi} \xi_1}{3} & \frac{\pi^2(-36 b^2 \abs{\xi}^2 + b^4 \gamma^2 \xi_1^2)}{6} + \pi b^2 i \gamma \xi_1 \\
 -2\pi^2 b^2 \abs{\xi}^2 
& \frac{\pi b \abs{\xi}(6 -4\pi b^2 i \gamma \xi_1)  }{3}
\end{pmatrix}.
\end{equation}
Then for $\abs{\xi} \ll 1$, we have the expansion
\begin{multline}
B_4^{-1}(\xi) = (I + N_4^{-1} R(\xi)   + N_4^{-1} O(\abs{\xi}^3) )^{-1} N_4^{-1} = (I - N_4^{-1} R(\xi) + (N_4^{-1} R(\xi))^2 ) N_4^{-1}  + O(\abs{\xi}^3)     \\
= 
\begin{pmatrix}
 2\pi \abs{\xi}\left( 1 + \frac{ \pi b^2 i \gamma \xi_1}{3} \right) & 1 - 2 \pi^2 b^2 \abs{\xi}^2   \\
 -1 + b^2(10 \pi^2 \abs{\xi}^2 - \pi i \gamma \xi_1) + \frac{5 \pi^2 b^4\gamma^2 \xi_1^2}{6} & 2\pi \abs{\xi} b \left(1 + \frac{4 \pi b^2 i \gamma \xi_1 }{3} \right) 
\end{pmatrix}
+O(\abs{\xi}^3) =: W(\xi) + O(\abs{\xi}^3).
\end{multline}
Returning to \eqref{QVm_zero_1}, we compute 
\begin{equation}
y(\xi,x_n) :=  \exp(x_n A(\xi)) B^{-1}(\xi)  e_4 = 
\begin{pmatrix}
P_2(\xi,x_n)  B_4^{-1}(\xi) e_2 \\
P_4(\xi,x_n)  B_4^{-1}(\xi) e_2
\end{pmatrix}
= 
\begin{pmatrix}
Q_2(\xi,x_n) W(\xi) e_2 \\
Q_4(\xi,x_n) W(\xi) e_2
\end{pmatrix}
+O(\abs{\xi}^3).
\end{equation}
From these we then  compute 
\begin{equation}
\begin{split}
 y_1(\xi,x_n) &= P_2(\xi,x_n) B_4^{-1}(\xi) e_2 \cdot e_1 = \pi \abs{\xi}(2x_n b - x_n^2) + O(\abs{\xi}^2) \\
 y_2(\xi,x_n) &= P_2(\xi,x_n) B_4^{-1}(\xi) e_2 \cdot e_2 = 2 \pi^2 \abs{\xi}^2 x_n^2 \left(\frac{x_n}{3} - b \right) +O(\abs{\xi}^3) \\
 y_3(\xi,x_n) &= P_4(\xi,x_n) B_4^{-1}(\xi) e_2 \cdot e_1 = 1 + O(\abs{\xi}^2).
\end{split}
\end{equation}
Then \eqref{QVm_zero_0} follows from this and the definitions \eqref{QVm_def}.

\end{proof}

\begin{remark}\label{remark_symbol_asymp}
Naively, one might expect that $m(\cdot,\gamma)$, the symbol for the normal-stress to normal-Dirichlet operator, should have the same essential behavior as the Neumann to Dirichlet operator defined via the scalar Laplacian, i.e. the map 
$H^{s}(\Sigma_b) \ni \psi \mapsto u \vert_{\Sigma_b} \in H^{s+1}(\Sigma_b)$, where
\begin{equation}
\begin{cases}
-\Delta u =0  &\text{in }\Omega \\
\p_n u = \psi &\text{on }\Sigma_b \\
u =0 &\text{on }\Sigma_0.
\end{cases}
\end{equation}
However, the symbol for this operator is 
\begin{equation}
 \frac{\sinh(2\pi \abs{\xi} b)}{2\pi \abs{\xi} \cosh(2\pi \abs{\xi} b)}, \text{ which behaves like } b \left(1 - \frac{4 \pi^2 \abs{\xi}^2 b^2}{3}\right)
\end{equation}
for $\abs{\xi} \ll 1$, and the lack of vanishing at the origin makes this operator significantly easier to work with.  Note, though, that the asymptotics of $m(\cdot,\gamma)$ exactly match the second term in the above development.

\end{remark}

\subsection{Asymptotics of the special functions \eqref{QVm_def} as $\abs{\xi} \to \infty$} \label{sec_infty_asympt}

We now turn our attention to the question of the asymptotics of the functions defined in \eqref{QVm_def} as $\abs{\xi} \to \infty$.  Unfortunately, due to the essential singularity of the exponential map at infinity, we cannot employ a simple Taylor expansion at infinity, as we did at zero in Theorem \ref{QVm_zero}.  Instead we must employ a more delicate strategy in which we actually compute $\exp(x_n A) B^{-1} e_4$.  Doing this directly is prohibitively difficult, so we first introduce a reparameterization that makes the algebraic manipulations more tractable.  We begin our pursuit of this strategy with the following lemma, which introduces the reparameterization.

\begin{lemma}\label{s_lemma}
Let  $s: [0,\infty) \times \R \to \C$ via
\begin{equation}\label{s_def}
 s(r,\kappa) = \sqrt{\frac{r^2 + \sqrt{r^4 + r^2 \kappa^2}}{2}} - i \frac{r \kappa}{\sqrt{2} \sqrt{r^2 + \sqrt{r^4 + r^2 \kappa^2}}}
\end{equation}
for $r>0$ and $s(0,\kappa) =0$.  Then the following hold.
\begin{enumerate}
 \item $(s(r,\kappa))^2 = r^2 - i r \kappa$ for every $(r,\kappa) \in [0,\infty) \times \R$.
 \item $s$ is continuous on $[0,\infty) \times \R$ and smooth on $(0,\infty) \times \R$.
 \item For all  $(r,\kappa) \in [0,\infty) \times \R$ we have the bounds 
\begin{equation}\label{s_lemma_01}
 0 \le \RE(s(r,\kappa)) -r \le \frac{\kappa^2}{8r} \text{ and } 0 \le -\sgn(\kappa) \IM(s(r,\kappa))  \le \frac{\abs{\kappa}}{2}.
\end{equation}
 \item There exists $c >0$ such that if $\abs{\kappa} \le r$, then 
\begin{equation}\label{s_lemma_04}
 \abs{\RE(s(r,\kappa)) -r - \frac{\kappa^2}{8r}  } \le   c \frac{\kappa^4}{r^3} \text{ and } \abs{\IM(s(r,\kappa)) + \frac{\kappa}{2} - \frac{\kappa^3}{16 r^2}} \le c  \frac{\kappa^5}{r^4}.
\end{equation}
\end{enumerate}

\end{lemma}

\begin{proof}
The first two items are trivial.  To prove the third item we first note that since $r^2 \kappa^2 \ge 0$, 
\begin{equation}
 \RE(s(r,\kappa)) \ge \sqrt{\frac{r^2 + \sqrt{r^4}}{2}  } = r.
\end{equation}
On the other hand, we have that
\begin{multline}
 \RE(s(r,\kappa)) \le r + \frac{\kappa^2}{8r} \Leftrightarrow \sqrt{r^4 + r^2 \kappa^2} \le r^2 + \frac{\kappa^2}{2} + \frac{\kappa^4}{32 r^2} \\
 \Leftrightarrow r^4 + r^2 \kappa^2 \le r^4 + 2r^2\left( \frac{\kappa^2}{2} + \frac{\kappa^4}{32 r^2}\right) + \left(\frac{\kappa^2}{2} + \frac{\kappa^4}{32 r^2} \right)^2,
\end{multline}
and the final inequality is trivially true, which means that the first estimate of \eqref{s_lemma_01} holds.  In turn this implies that  
\begin{equation}\label{s_lemma_1}
\begin{split}
\kappa &\ge 0 \Rightarrow  -\frac{\kappa}{2} \le \IM(s(r,\kappa)) \le -\frac{\kappa}{2} \frac{8r^2}{8r^2+\kappa^2} \le 0  \text{ and } \\
\kappa &< 0 \Rightarrow  0 \le  -\frac{\kappa}{2} \frac{8r^2}{8r^2+\kappa^2}      \le \IM(s(r,\kappa)) \le   -\frac{\kappa}{2},  
\end{split}
\end{equation}
which implies the second estimate of \eqref{s_lemma_01}.

Finally, for the fourth item we note that if $r = 1/\rho$ for $\rho >0$, then Taylor expanding around $\rho =0$ shows that
\begin{equation}
 \RE(s(r,\kappa)) = \frac{1}{\sqrt{2} \rho} \sqrt{1 + \sqrt{1+ (\kappa \rho)^2}} = \frac{1}{\sqrt{2} \rho} \left(\sqrt{2} + \frac{\sqrt{2} (\kappa \rho)^2}{8} + O((\kappa \rho)^4)   \right)
\end{equation}
and 
\begin{equation}
 \IM(s(r,\kappa)) = -\frac{\kappa}{\sqrt{2}} \left( 1 + \sqrt{1+ (\kappa \rho)^2}\right)^{-1/2} = -\kappa \left(\frac{1}{2} - \frac{(\kappa\rho)^2}{16 } + O((\kappa \rho)^4)   \right).
\end{equation}
Hence there exists $c >0$, depending only on the smooth maps $\R \ni z\mapsto \sqrt{1+ \sqrt{1+z^2}} \in \R$ and $\R \ni z\mapsto 1/\sqrt{1+ \sqrt{1+z^2}} \in \R$, such that if $\abs{\kappa} < r$, then
\begin{equation}
 \abs{\RE(s(r,\kappa)) -r - \frac{\kappa^2}{8r}  } \le cr  (\kappa \rho)^4  =  c \frac{\kappa^4}{r^3} \text{ and } \abs{\IM(s(r,\kappa)) + \frac{i\kappa}{2} - \frac{\kappa^3}{16 r^2}} \le c  \frac{\kappa^5}{r^4}.
\end{equation}

\end{proof}

We now aim to reparameterize the matrices $A,N,B \in \C^{4 \times 4}$ as defined in \eqref{matrix A}, \eqref{matrices M and N}, and \eqref{boundary matrix}, respectively.  To this end we first note that for given $\xi \in \R^{n-1}$ and $\gamma \in \R$, each of these matrices only depends on $\abs{\xi}$ and $\gamma \xi_1$.  This suggests that we introduce the reparameterization $\R^{n-1}\backslash \{0\} \ni \xi \mapsto (r, \kappa) \in [0,\infty) \times [-\abs{\gamma},\abs{\gamma}]$ given by $r = 2 \pi \abs{\xi}$ and $\kappa =   \gamma \xi_1 /\abs{\xi}$.  We then introduce the function $s : [0,\infty) \times [-\abs{\gamma}, \abs{\gamma}] \to \C$ defined as the restriction to $[0,\infty) \times [-\abs{\gamma}, \abs{\gamma}]$ of the function defined in \eqref{s_def}; by construction $(s(r,\kappa))^2 = r^2 - i r \kappa = 4\pi^2 \abs{\xi}^2 - 2\pi i \gamma \xi_1$, $s \in \{z \in \C \st r \le \RE(z), \abs{\IM(z)} \le \abs{\kappa}/2 \le \abs{\gamma}/2   \}$, and $s = r$ if and only if $\kappa =0$.  We then reparameterize $A,N,B$ in terms of $r$ and $s$ via
\begin{equation}\label{ANB_reparam}
A(r,s) = 
\begin{pmatrix}
0 & 0 & 0 & 1 \\
-r & 0 & 0 & 0 \\
0 & -s^2 & 0 & -r \\
s^2 & 0 & -r & 0
\end{pmatrix},
N(r) = 
\begin{pmatrix}
0 & 0 & 0 & 0 \\
0 & 0 & 0 & 0 \\
0 & r & 0 & -1 \\
2r & 0 & 1 & 0
\end{pmatrix},
\text{ and }
B(r,s) = M + N(r) \exp(b A(r,s)).
\end{equation}

Written in this form, for $s \neq r$ (i.e. when $\kappa \neq 0$) we have that $A(r,s)$ is diagonalizable with spectrum $\{s,-s,r,-r\}$.  Exploiting this, we may readily compute the columns $\exp(x_n A(r,s))^{(j)} \in \C^4$ for $j=1,\dotsc,4$:
\begin{equation}
\exp(x_n A(r,s))^{(1)} = 
\begin{pmatrix}
  \cosh(x_n s)  \\
  -\frac{ r \sinh(x_n s)}{s} \\
  0 \\
  s \sinh(x_n s) 
\end{pmatrix},
\exp(x_n A(r,s))^{(2)} =
\begin{pmatrix}
  \frac{s}{r^2-s^2} (-r \sinh(x_n s) + s \sinh(x_n r)) \\
  \frac{1}{r^2-s^2}(r^2 \cosh(x_n s) - s^2 \cosh(x_n r)) \\
  \frac{-s^2 \sinh(x_n r)}{r} \\
  \frac{s^2 r}{r^2-s^2}(-\cosh(x_n s) + \cosh(x_n r))
\end{pmatrix},
\end{equation}
\begin{equation}
 \exp(x_n A(r,s))^{(3)} = 
\begin{pmatrix}
 \frac{r}{r^2-s^2}(\cosh(x_n s)-\cosh(x_n r)) \\
 \frac{r}{s(r^2-s^2)}(-r \sinh(x_n s) + s \sinh(x_n r)) \\
 \cosh(x_n r) \\
 \frac{r}{r^2-s^2}(s \sinh(x_n s) - r \sinh(x_n r))
\end{pmatrix},
\end{equation}
and 
\begin{equation}
 \exp(x_n A(r,s))^{(4)}= 
\begin{pmatrix}
  \frac{1}{r^2-s^2}(-s \sinh(x_n s) + r \sinh(x_n r))  \\
  \frac{r}{r^2-s^2}(\cosh(x_n s) - \cosh(x_n r)) \\
  -\sinh(x_n r) \\
  \frac{1}{r^2-s^2}(-s^2 \cosh(x_n s) +r^2 \cosh(x_n r) )
\end{pmatrix}.
\end{equation}
Note that when $s=r$ (i.e. $\kappa =0$) $A(r,r)$ fails be diagonalizable (though it still has a nice Jordan form), but we may recover the value of $\exp(x_n A(r,r))$ by sending $s \to r$ in these expressions.

We then define the reparameterized form of $y(\xi,x_n,\gamma) = \exp(x_n A(\xi,\gamma)) B^{-1}(\xi,\gamma) e_4$, as used in the definition of the special functions in \eqref{QVm_def}, to be 
\begin{equation}\label{Y_reparam_def}
Y(r,\kappa,x_n) := \exp(x_n A(r,s(r,\kappa))) B^{-1}(r,s(r,\kappa)) e_4 \in \C^4.
\end{equation}
Employing Theorem \ref{ODE_B_inversion}, for $1 \le j \le 3$ and $r \neq s$ we may explicitly compute:
\begin{multline}
 Y_j(r,\kappa,x_n) = \frac{1}{\det B} \left[ \exp(x_n A)_{j3}(\exp(bA)_{44} + r \exp(bA)_{24}  )  + \exp(x_n A)_{j4} (r \exp(bA)_{23} - \exp(bA)_{43} )  \right] \\
 =: \frac{n_j}{\det B}.
\end{multline}
Using the identity $s^2 = r^2 - i \kappa r$ simplifies the resulting expressions for $n_j$ and $\det{B}$, and after some elementary, if tedious calculations,  we arrive at:
\begin{multline}\label{n1_reparam_def}
n_1 = -\frac{1}{2 \kappa^2 s}\left[ (r+s)(2r-i\kappa)\cosh(bs-x_n r) + 2s(r+s) \cosh(br-x_n s) -2s(2r-i\kappa)\cosh(s(b-x_n)) \right. \\
\left. - 4 rs \cosh(r(b-x_n)) -(r-s)(2r-i\kappa) \cosh(bs+x_n r) + 2s(r-s) \cosh(br+x_n s) \right],
\end{multline}
\begin{multline}\label{n2_reparam_def}
n_2 = \frac{1}{2\kappa^2 s} \left[-(r+s)(2r-i\kappa) \sinh(bs-x_n r) -2r(r+s)\sinh(br-x_n s) +2r(2r-i\kappa) \sinh(s(b-x_n))    \right. \\
\left. +4rs \sinh(r(b-x_n)) - (r-s)(2r-i\kappa) \sinh(bs+x_n r) +2r(r-s)\sinh(br+x_n s) \right],
\end{multline}
\begin{equation}\label{n2b_reparam_def}
n_2\vert_{x_n=b}  
=\frac{i}{\kappa s} \left[r \sinh(bs) \cosh(br)  - s \cosh(bs) \sinh(br)  \right],  
\end{equation}
\begin{equation}\label{n3_reparam_def}
n_3 = \frac{-i}{2 \kappa s} \left[-(r+s)(2r-i\kappa)\cosh(bs-x_n r) + 4rs\cosh(r(b-x_n)) + (r-s)(2r-i\kappa)\cosh(bs+x_n r) \right],
\end{equation}
and
\begin{multline}\label{d_reparam_def}
\det B = \frac{-1}{\kappa^2 s} \left[s(8r^2 - \kappa^2 - i 4\kappa r) \cosh(br)\cosh(bs) \right. \\
\left.- r(8r^2-\kappa^2 - i 8\kappa r)\sinh(br) \sinh(bs) -4rs(2r-i\kappa) \right].
\end{multline}
The value of $Y_j(r,0,x_n)$ may then be obtained by sending $s \to r$ in these expressions:
\begin{equation}\label{Y1_rr_def}
Y_1(r,0,x_n) = \frac{(b-x_n)(\sinh(r(b+x_n)) - \sinh(r(b-x_n)) ) + 2brx_n \cosh(r(b-x_n)) }{2(\cosh(2rb) + 1 + 2 b^2 r^2)},
\end{equation}
\begin{multline}\label{Y2_rr_def}
  Y_2(r,0,x_n) = \frac{-\sinh(r(b+x_n)) -r(b-x_n) \cosh(r(b+x_n))    }{2r (\cosh(2rb) + 1 + 2 b^2 r^2  )} \\
+  \frac{ (1+2br^2 x_n))\sinh(r(b-x_n)) + r(b+x_n) \cosh(r(b-x_n))   }{2r (\cosh(2rb) + 1 + 2 b^2 r^2  )},
\end{multline}
\begin{equation}\label{Y3_rr_def}
 Y_3(r,0,x_n) = \frac{\cosh(r(b+x_n)) + \cosh(r(b-x_n)) + 2rb \sinh(r(b-x_n))}{\cosh(2rb) + 1 + 2 b^2 r^2 }.
\end{equation}

With all of these computations in hand, we are now ready to derive the asymptotics as $\abs{\xi} \to \infty$.

\begin{theorem}\label{QVm_infty}
Let  $Q: \R^{n-1} \times [0,b] \times \R \to \C$, $V : \R^{n-1} \times [0,b] \times \R \to \C^n$, and $m: \R^{n-1} \times \R \to \C$ be as defined in \eqref{QVm_def}.  Then for each $\gamma \in \R$ there exist constants $c=c(\gamma,b)>0$ and $R = R(\gamma,b) >0$ such that if $x_n \in [0,b]$ and $\abs{\xi} > R$, then 
\begin{equation}\label{QVm_infty_0}
\begin{split}
\abs{V'(\xi,x_n,\gamma)}  &\le c \left( \frac{\abs{\gamma}}{\abs{\xi}^2} + (b-x_n)  \right) e^{-2\pi \abs{\xi}(b-x_n)} + c e^{-2\pi \abs{\xi} b},  \\  
 \abs{V_n(\xi,x_n,\gamma)} &\le   c\left( \frac{1}{\abs{\xi}} + (b-x_n)  \right) e^{-2\pi \abs{\xi} (b-x_n)} + c e^{-2\pi \abs{\xi} b}  , \\
 \abs{m(\xi,\gamma) + \frac{1}{4 \pi \abs{\xi}}}   &\le  c \frac{1}{\abs{\xi}^2},  \\
 \abs{Q(\xi,x_n,\gamma)} &\le c e^{-2\pi \abs{\xi} (b-x_n)}  + c e^{-2\pi \abs{\xi} b}.
\end{split}
\end{equation}
\end{theorem}
\begin{proof}
We will present the proof under the assumption that $\gamma \neq 0$.  The proof when $\gamma =0$ is simpler and can be readily extracted from the first two steps of the following argument, so we omit the details.  We divide the proof into steps.

\textbf{Step 1 - A claim and its consequences:}  We claim that there exists constants $c>0$ and $R= R(\gamma,b)>0$ such that 
\begin{equation}
 \abs{Y_1(r,\kappa,x_n)} \le c\left( \frac{\abs{\kappa}}{r^2} + (b-x_n)  \right) e^{-r(b-x_n)} + c e^{-br},
 \quad 
  \abs{Y_3(r,\kappa,x_n)} \le c  e^{-r(b-x_n)} + c e^{-br}
\end{equation}
and
\begin{equation}
 \abs{Y_2(r,\kappa,x_n)} \le c\left( \frac{1}{r} + (b-x_n)  \right) e^{-r(b-x_n)} + c e^{-br},
\quad  \abs{Y_2(r,\kappa,b) + \frac{1}{2r}  } \le   c \frac{1}{r^2},
\end{equation}
for all $\abs{\kappa} \le \gamma$, $x_n \in [0,b]$, and $r \ge R$.  Once the claim is established we consider 
\begin{equation}
y(\xi,x_n,\gamma) = \exp(x_n A(\xi,\gamma)) B^{-1}(\xi,\gamma) e_4 = Y(2\pi \abs{\xi}, \gamma \xi_i/\abs{\xi},x_n) 
\end{equation}
and simply plug into the definitions in \eqref{QVm_def} to deduce \eqref{QVm_infty_0}.  It remains only to prove the claim.  We break to two cases: $\kappa =0$ and $\kappa \neq 0$.

\textbf{Step 2 - Asymptotic development of $Y(r,0,x_n)$:}  Clearly, for $r$ large the dominant terms in the denominators of  \eqref{Y1_rr_def}--\eqref{Y3_rr_def} are the $\cosh(2rb)$ terms.  Similarly, since $0 \le x_n \le b$, the dominant terms in the numerators of \eqref{Y1_rr_def}--\eqref{Y3_rr_def} are the hyperbolic functions with arguments $r(b+x_n)$.  From these observations we then deduce that 
\begin{equation}
\sup_{0\le x_n \le b} \abs{ Y_1(r,0,x_n) - \frac{(b-x_n) e^{-r(b-x_n)} }{2} } = O(e^{-rb}),
\end{equation}
\begin{equation}
\sup_{0\le x_n \le b} \abs{ Y_2(r,0,x_n) - \frac{[-1-r( b-x_n) ] e^{-r(b-x_n)} }{2 r} } = O(e^{-rb}),
\quad
\abs{ Y_2(r,r,b) + \frac{1  }{2 r} } = O(e^{-rb}),
\end{equation}
and
\begin{equation}
\sup_{0\le x_n \le b} \abs{ Y_3(r,0,x_n) - e^{-r(b-x_n)}   } = O(e^{-rb}).
\end{equation}

\textbf{Step 3 - Asymptotic development of $Y(r,\kappa,x_n)$ for $\kappa \neq 0$:}  First recall that Lemma \ref{s_lemma} tells us that $s$ has the asymptotic development
\begin{equation}
 s = r - i \frac{\kappa}{2} + \frac{\kappa^2}{8r} + i \frac{\kappa^3}{16 r^2} + O((\kappa^4/r^3)).
\end{equation}
We begin by using the expression for $\det B$ in \eqref{d_reparam_def} together with the asymptotic development of $s$ to write 
\begin{multline}
 \det B(r,\kappa) = \frac{-1}{4 \kappa^2 s} \left[ \left( s(8r^2 - \kappa^2 - i \kappa r) - r(8r^2-\kappa^2 -i \kappa r) \right)e^{b(r+s)} + O(e^{-3br/2})  \right] \\
 = \frac{-1}{4 \kappa^2 s} \left[ \left(- \kappa^2 r + i \frac{\kappa^3}{2}  +\kappa^2 O(\kappa^2/r)  \right)e^{b(r+s)} + O(e^{-3br/2})  \right].
\end{multline}
This allows us to use \eqref{n1_reparam_def} to write 
\begin{multline}
Y_1(r,\kappa, x_n) = \frac{e^{-b(r+s)}}{- \kappa^2 r + i \frac{\kappa^3}{2}  +\kappa^2 O(\kappa^2/r) } \left[ -(r-s)(2r-i\kappa) e^{bs+x_nr} + 2s(r-s) e^{br+x_ns}  \right] +O(e^{-br}) \\
= \frac{r-s}{- \kappa^2 r  + i \frac{\kappa^3}{2} + \kappa^2 O(\kappa^2/r) } e^{-r(b-x_n)} \left[2s-2r+i\kappa +2s\left(e^{(r-s)(b-x_n)}-1 \right) \right] +O(e^{-br}),
\end{multline}
and so when we plug in the $s$ asymptotics we find that there exists a $c>0$ and $R = R(\gamma,b) >0$ such that 
\begin{equation}
 \abs{Y_1(r,\kappa,x_n)} \le c\left( \frac{\abs{\kappa}}{r^2} + (b-x_n)  \right) e^{-r(b-x_n)} + c e^{-br}
\end{equation}
for all $0 < \abs{\kappa} \le \gamma$, $x_n \in [0,b]$, and $r \ge R$.  Arguing similarly with \eqref{n2_reparam_def}--\eqref{n3_reparam_def} and enlarging $R$ if necessary, we find that 
\begin{equation}
 \abs{Y_2(r,\kappa,x_n)} \le c\left( \frac{1}{r} + (b-x_n)  \right) e^{-r(b-x_n)} + c e^{-br},
\quad  \abs{Y_2(r,\kappa,b) + \frac{1  }{2 r}   } \le c \frac{1}{r^2} +  c e^{-br},
\end{equation}
and 
\begin{equation}
 \abs{Y_3(r,\kappa,x_n)} \le c  e^{-r(b-x_n)} + c e^{-br}
\end{equation}
for all $0 < \abs{\kappa} \le \gamma$, $x_n \in [0,b]$, and $r \ge R$.

\textbf{Step 4 - Proof of the claim:} The claim now follows by combining the results of Steps 2 and 3.

\end{proof}

The asymptotic developments of Theorem \ref{QVm_infty} may be combined with the results of Theorem  \ref{QVm_properties} to deduce some integral bounds.  We record these now.

\begin{corollary}\label{QVm_integrals}
Let  $Q: \R^{n-1} \times [0,b] \times \R \to \C$ and $V : \R^{n-1} \times [0,b] \times \R \to \C^n$ be as defined in \eqref{QVm_def}.  Then for each $\gamma \in \R$ there exists a constant $c = c(n,\gamma,b) >0$ such that
\begin{equation}
(1+ \abs{\xi}^3) \int_0^b \abs{V(\xi,x_n,\gamma)}^2 dx_n + (1+ \abs{\xi}) \int_0^b \abs{Q(\xi,x_n,\gamma)}^2 dx_n \le c \text{ for all } \xi \in \R^{n-1}.
\end{equation}
\end{corollary}
\begin{proof}
From Theorems \ref{QVm_properties} and \ref{QVm_infty} we can choose $c = c(n,\gamma,b) >0$ such that 
\begin{equation}
 \abs{V(\xi,x_n,\gamma)}^2 \le c \left(\frac{1}{1+ \abs{\xi}^2} + (b-x_n)^2      \right)e^{-4\pi \abs{\xi}(b-x_n)}
\text{ and } 
\abs{Q(\xi,x_n,\gamma)}^2 \le c\left(e^{-4 \pi \abs{\xi}b} + e^{-4\pi \abs{\xi}(b-x_n)}   \right)
\end{equation}
for all $x_n \in [0,b]$ and $\xi \in \R^{n-1}$.  The result then follows directly from this and the fact that 
\begin{equation}
 \int_0^\infty z^t e^{-r z} dz = \frac{\Gamma(t+1)}{r^{t+1}} \text{ for every } r,t \in (0,\infty). 
\end{equation}

\end{proof}

\subsection{The over-determined problem }

We now write the compatibility conditions \eqref{compatibility_condition} using the Fourier transform. 

\begin{proposition}\label{proposition_cc_fourier}
Let $\gamma \in \R$, $s \ge 0$, and suppose that $f\in H^{s}(\Omega;\mathbb{R}^{n})$, $g\in H^{s+1}(\Omega)$, $h\in H^{s+3/2}(\Sigma_b)$, and $k\in H^{s+1/2}(\Sigma_b;\mathbb{R}^{n})$.  Then \eqref{compatibility_condition} holds if
and only if
\begin{equation}\label{compatibility condition fourier}
\int_{0}^{b}(\hat{f}(\xi,x_{n})\cdot\overline{V(\xi,x_{n}, -\gamma)} 
- \hat{g}(\xi,x_{n}) \overline{Q(\xi,x_{n}, -\gamma )})dx_{n}
-\hat{k}(\xi)\cdot \overline{V(\xi,b,-\gamma)} + \hat{h}(\xi) = 0 
\end{equation}
for almost every $\xi\in\mathbb{R}^{n-1}$, where $Q$ and $V$ are as defined in \eqref{QVm_def}.

\end{proposition}

\begin{proof}

For  $\psi \in H^{s+1/2}(\Sigma_b)$ and $v$, $q$ as in Theorem \ref{cc_over-det}  we apply Parseval's theorem, the fifth item of Theorem \ref{QVm_properties}, and Fubini's theorem to see that
\begin{multline}\label{proposition_cc_fourier_1}
\int_{\Omega}(f\cdot v-gq)-\int_{\Sigma}(k\cdot v-h\psi)\\
  =\int_{\mathbb{R}^{n-1}}\int_{0}^{b} \left(\hat{f}(\xi,x_{n})\cdot\overline{\hat{v} (\xi,x_{n})}-\hat{g}(\xi,x_{n}) \overline{\hat{q}(\xi,x_{n})} \right) dx_{n}d\xi
-\int_{\mathbb{R}^{n-1}} \left(\hat{k}(\xi) \cdot\overline{\hat{v}(\xi,b)}-\hat{h}
(\xi)\overline{\hat{\psi}(\xi)} \right) d\xi\\
  =\int_{\mathbb{R}^{n-1}}\left[  \int_{0}^{b}(\hat{f}(\xi,x_{n})\cdot \overline{V(\xi,x_{n},-\gamma)} - \hat{g}(\xi,x_{n}) \overline{Q(\xi,x_{n},-\gamma )})dx_{n} \right] \overline{\hat{\psi}(\xi)} \,d\xi \\
 + \int_{\R^{n-1}} \left[-\hat{k}(\xi)\cdot\overline{V(\xi,b,-\gamma)}+\hat{h}(\xi)\right]
\overline{\hat{\psi}(\xi)} \,d\xi.
\end{multline}
If \eqref{compatibility condition fourier} holds, then this implies that \eqref{compatibility_condition} holds.

Conversely, suppose that \eqref{compatibility_condition} holds.  Let $\hat{\psi} \in C^\infty_c(\R^{n-1};\C)$ be such that $\overline{\hat{\psi}(\xi)} = \hat{\psi}(-\xi)$.  From Lemma \ref{tempered_real_lemma} we then know that $\psi = (\hat{\psi})^\vee \in \mathscr{S}(\R^{n-1})$ is real-valued.  We then use this $\psi$ in Theorem \ref{cc_over-det} to see that the left term in \eqref{proposition_cc_fourier_1} vanishes, which yields an identity of the form 
\begin{equation}
 0 = \int_{\R^{n-1}} \phi(\xi) \hat{\psi}(-\xi) d\xi \text{ for all } \hat{\psi} \in C^\infty_c(\R^{n-1};\C) \text{ such that } \overline{\hat{\psi}(\xi)} = \hat{\psi}(-\xi),
\end{equation}
where $\phi(\xi)$ is the left side of \eqref{compatibility condition fourier}.  According to Lemma \ref{tempered_real_lemma} and the third item of Theorem \ref{QVm_properties} we have that $\overline{\phi(\xi)} = \phi(-\xi)$.  Since we then know that $\RE{\phi}$ and $\RE{\hat{\psi}}$ are even and $\IM{\phi}$ and $\IM{\hat{\psi}}$ are odd, the previous identity reduces to 
\begin{equation}
 0 = \int_{\R^{n-1}} \left( \RE{\phi(\xi)} \RE{\hat{\psi}(\xi)}  + \IM{\phi(\xi)} \IM{\hat{\psi}(\xi)} \right) d\xi
\end{equation}
for all such $\hat{\psi}$.  Let $\chi, \zeta \in C^\infty_c(\R^{n-1})$ be such that $\supp(\chi),\supp(\zeta) \subset \R^{n-1}_+ = \{ x_{n-1} >0 \}$ and set 
\begin{equation}
 \hat{\psi}(\xi) = \left(\frac{\chi(\xi) + \chi(-\xi)}{2} \right) + i \left(\frac{\zeta(\xi) - \zeta(-\xi)}{2} \right),
\end{equation}
which satisfies $\overline{\hat{\psi}(\xi)} = \hat{\psi}(-\xi)$.  Then from the previous identity we deduce that 
\begin{equation}
 0 = \int_{\R^{n-1}_+} \left( \RE{\phi(\xi)} \chi(\xi)  + \IM{\phi(\xi)} \zeta(\xi) \right) d\xi,
\end{equation}
and from the arbitrariness of $\chi,\zeta$ we then deduce that $\RE{\phi} = \IM{\phi}=0$ almost everywhere in $\R^{n-1}_+$ and hence almost everywhere in $\R^{n-1}$ as well.  Thus  \eqref{compatibility condition fourier} holds for almost every $\xi\in\mathbb{R}^{n-1}$.  
\end{proof}

\section{Some specialized Sobolev spaces}\label{sec_specialized_sobolev}

In this section we introduce a pair of specialized Sobolev spaces that play an essential role in constructing solutions to \eqref{traveling_euler}. The first space, $\sp^s(\R^d)$, is the space to which the free surface function will belong.  It is defined through an anisotropic Fourier multiplier and is, at least when $d \ge 2$, strictly larger than the standard fractional $L^2-$based Sobolev space $H^s(\R^d)$.  The second space, $\an^s(\Omega)$, is the space to which the pressure will belong.  It is defined in terms of $\sp^s(\R^{n-1})$ and is again strictly larger than $H^s(\Omega)$ when $n \ge 3$.  Note that throughout this section we continue the practice described in Section \ref{sec_notation} of using $1 \le d \in \N$ for a generic dimension and $2 \le n \in \N$ for the dimension of $\Omega$.  

To the best of our knowledge, neither of these spaces has been previously studied in the literature.  As such, we develop their basic properties here.  We will need to work with these spaces in a nonlinear context, so we also develop a number of nonlinear tools.

\subsection{Preliminary estimates }

We record here two preliminary results that will play an essential role in defining the specialized Sobolev spaces.  The first is a simple integral computation.

\begin{lemma}\label{integral_computation_lemma}
For $a >0$ we have that 
\begin{equation}
 \int_0^{2\pi} \frac{d\theta}{a + \cos^2(\theta)} = \frac{2\pi}{\sqrt{a(1+a)}} .
\end{equation}
\end{lemma}
\begin{proof}
We begin by noting that for $\theta \in (0,\pi)$ if we set $z = \tan(\theta/2)$, then 
\begin{equation}
 \cos(\theta) = \frac{1-z^2}{1+z^2} \text{ and } d \theta = 2 \frac{dz}{1+z^2}.
\end{equation}
Using periodicity and this change of variables, we may then compute 
\begin{equation}
 \int_0^{2\pi} \frac{d\theta}{a + \cos^2(\theta)} = 2  \int_0^{\pi} \frac{d\theta}{a + \cos^2(\theta)} = 4 \int_0^\infty \frac{1+z^2}{a(1+z^2)^2 + (1-z^2)^2}dz 
 = \frac{4}{1+a} \int_0^\infty \frac{1+z^2}{z^4 + 2 \beta z^2 +1} dz
\end{equation}
for 
\begin{equation}
 \beta := \frac{a-1}{a+1} \in (-1,1).
\end{equation}
Next write $\alpha := \sqrt{2(1-\beta)} \in (0,2)$ and note that we have the partial fraction decomposition
\begin{equation}
\frac{1+z^2}{z^4 + 2 \beta z^2 +1} dz  = \hal \left( \frac{1}{z^2 + \alpha z + 1} +  \frac{1}{z^2 - \alpha z +1}\right)
\end{equation}
as well as the identities 
\begin{equation}\label{integral_computation_lemma_1}
1- \frac{\alpha^2}{4} = \frac{1+\beta}{2} = \frac{a}{1+a} \text{ and } \frac{\alpha}{2} \sqrt{\frac{2}{1+\beta}} = \sqrt{\frac{1-\beta}{1+\beta}} = \frac{1}{\sqrt{a}}.
\end{equation}
Hence 
\begin{multline}
  \int_0^{2\pi} \frac{d\theta}{a + \cos^2(\theta)} = \frac{2}{1+a} \int_0^\infty \left( \frac{1}{z^2 + \alpha z + 1} +  \frac{1}{z^2 - \alpha z +1} \right)dz \\
  = \frac{2}{1+a} \int_0^\infty \left( \frac{1}{(z + \alpha/2)^2  + 1-\alpha^2/4} +  \frac{1}{(z - \alpha/2)^2  + 1-\alpha^2/4} \right)dz \\
  = \frac{2}{\sqrt{a(1+a)}} \left( \int_{1/\sqrt{a}}^\infty \frac{dw}{1+w^2} + \int_{-1/\sqrt{a}}^\infty \frac{dw}{1+w^2}   \right)   = \frac{4}{\sqrt{a(1+a)}} \int_0^\infty \frac{dw}{1+w^2}  = \frac{2\pi}{\sqrt{a(1+a)}},
\end{multline}
where in the third inequality we have used the change of variables 
\begin{equation}
 z \pm \frac{\alpha}{2} = w \sqrt{1-\frac{\alpha^2}{4}}
\end{equation}
together with the identities \eqref{integral_computation_lemma_1}.

\end{proof}

We now parlay the computation of Lemma \ref{integral_computation_lemma} into an estimate for a certain integral.  Note that the lemma is only used here in the case $d=2$, as the other cases are easier.

\begin{proposition}\label{integral_finite}
Let $R >0$ and consider the ball $ B(0,R) \subset \R^d$ for $d \ge 1$.  Then 
\begin{equation}
 \int_{B(0,R)} \frac{\abs{x}^2}{x_1^2 + \abs{x}^4} dx < \infty.
\end{equation}
\end{proposition}
\begin{proof}

If $d \ge 3$ then we simply bound 
\begin{equation}
 \int_{B(0,R)} \frac{\abs{x}^2}{x_1^2 + \abs{x}^4}dx  \le \int_{B(0,R)} \frac{dx}{\abs{x}^2} = \mathcal{H}^{d-1}(\p B(0,1)) \int_0^R r^{d-3} dr = \mathcal{H}^{d-1}(\p B(0,1)) \frac{ R^{d-2}}{d-2}  < \infty.
\end{equation}
On the other hand, in the case $d =2$ we may use polar coordinates and Lemma \ref{integral_computation_lemma} to compute
\begin{multline}
 \int_{B(0,R)}  \frac{\abs{x}^2}{x_1^2 + \abs{x}^4} dx = \int_0^R \int_{0}^{2\pi} \frac{r^2}{r^2 \cos^2(\theta) + r^4} r d\theta dr = \int_0^R r \int_{0}^{2\pi} \frac{d\theta}{r^2 + \cos^2(\theta)} dr \\
 = \int_0^R r \frac{2\pi}{r \sqrt{r^2 + 1}}dr = 2\pi \int_0^R \frac{dr}{\sqrt{1+r^2}} = 2\pi \arcsinh(R) < \infty.
\end{multline}
Finally, if $d=1$ then 
\begin{equation}
 \int_{B(0,R)} \frac{\abs{x}^2}{x_1^2 + \abs{x}^4}dx  = \int_{-R}^R \frac{dr}{1+r^2} = 2 \arctan(R) < \infty. 
\end{equation}
\end{proof}

\subsection{A class of specialized Sobolev spaces on $\R^d$}

For $0 \le s \in \R$ and $1 \le d \in \N$ we define the measurable function $\omega_s : \R^d \to [0,\infty)$ via 
\begin{equation}\label{omega_s_def}
 \omega_s(\xi) = \frac{\xi_1^2 + \abs{\xi}^4}{\abs{\xi}^2}\vchi_{B(0,1)}(\xi) +    (1+\abs{\xi}^2)^{s} \vchi_{B(0,1)^c}(\xi).
\end{equation}
Then for $s \ge 0$ we define the (real) specialized Sobolev space  
\begin{equation}\label{sp_space_def}
 \sp^s(\R^d) = \{f \in \mathscr{S}'(\R^d) \st f = \bar{f}, \hat{f} \in L^1_{loc}(\R^d), \text{ and } \norm{f}_{\sp^s} < \infty \},
\end{equation}
where 
\begin{equation}
 \norm{f}_{\sp^s}^2 := \int_{\R^d} \omega_s(\xi) \abs{\hat{f}(\xi)}^2 d\xi.
\end{equation}
We endow the vector space $\sp^s(\R^d)$ with the associated inner-product 
\begin{equation}
 \ip{f,g}_{\sp^s} := \int_{\R^d} \omega_s(\xi) \hat{f}(\xi) \overline{\hat{g}(\xi)}d\xi,
\end{equation}
which takes values in $\R$ due to Lemma \ref{tempered_real_lemma}.  Note that we can similarly define complex-valued analogs of $\sp^s(\R^d)$ by dropping the condition that $f = \bar{f}$.  We will not need these spaces, so we focus on the real case here.

We begin our study of these spaces by showing that they contain the usual Sobolev spaces $H^s(\R^d)$ and that the containment is strict for $d \ge 2$.

\begin{proposition}\label{specialized_inclusion}
For $s \ge 0$ the following hold.
\begin{enumerate}
 \item We have that $\sp^s(\R) = H^s(\R)$, and $\norm{\cdot}_{\sp^s}$ and  $\norm{\cdot}_{H^s}$ are equivalent norms.
 \item If $d \ge 2$, then we have the strict inclusion $H^s(\R^d) \subset \sp^s(\R^d)$ and  $\norm{f}_{\sp^s} \le 2\norm{f}_{H^s}$ for all $f \in H^s(\R^d)$.
 \item If $d \ge 2$, then $\sp^s(\R^d)$ is not closed under rotation in the sense that for every $Q \in O(d)$ such that $\abs{Q e_1 \cdot e_1} < 1$ there exists $f \in \sp^s(\R^d) \cap C^\infty_0(\R^d)$ such that $f(Q\cdot) \notin \sp^s(\R^d)$.
\end{enumerate}
\end{proposition}
\begin{proof}
Clearly $\omega_s(\xi) \le 2 (1+\abs{\xi}^2)^s$ for all $\xi \in \R^d$, and hence 
\begin{equation}
 \norm{f}_{\sp^s}^2 = \int_{\R^d} \omega_s(\xi) \abs{\hat{f}(\xi)}^2 d\xi \le 2\int_{\R^d}(1+\abs{\xi}^2) \abs{\hat{f}(\xi)}^2 d\xi = 2\norm{f}_{H^s}^2
\end{equation}
for all $f \in H^s(\R^d)$.  Thus $H^s(\R^d) \subseteq \sp^s(\R^d)$.  On other hand, if $d =1$, then $(\xi_1^2 + \abs{\xi}^4)/\abs{\xi}^2 = 1 + \abs{\xi}^2$, and for $\abs{\xi} \le 1$ we have that
\begin{equation}
\frac{(1+\abs{\xi}^2)^s}{1+ \abs{\xi}^2} \in 
\begin{cases}
[2^{s-1},1] &\text{if } 0 \le s \le 1 \\
[1,2^{s-1}] &\text{if } 1 < s.
\end{cases}
\end{equation}
Hence, we can choose a constant $c = c(s) >0$ such that 
\begin{equation}
\frac{1}{c} \int_{\R} (1+\abs{\xi}^2)^{s} \abs{\hat{f}(\xi)}^2 d\xi \le  \int_{\R} \omega_s(\xi) \abs{\hat{f}(\xi)}^2 d\xi \le c  \int_{\R} (1+\abs{\xi}^2)^{s} \abs{\hat{f}(\xi)}^2 d\xi  
\end{equation}
to deduce that $\norm{\cdot}_{\sp^s}$ is a norm equivalent to $\norm{\cdot}_{H^s}$.   

Now assume that $d \ge 2$ and let $Q \in O(d)$ be such that $\abs{Qe_1 \cdot e_1} < 1$, which is equivalent to the existence of $2 \le j \le d$ such that $\abs{Q e_1\cdot e_j} >0$.  We will construct $f \in \sp^s(\R^d) \cap C^\infty_0(\R^d)$ such that $f(Q\cdot) \notin \sp^s(\R^d)$ and $f \notin L^2(\R^d)$, which will complete the proof since the latter also shows that $f \notin H^s(\R^d)$.  For $1 \le i \le d$ write 
\begin{equation}
 \sigma_i = 
\begin{cases}
1 & \text{if } Qe_1 \cdot e_i \ge 0 \\
-1 &\text{if } Qe_1 \cdot e_i < 0.
\end{cases}
\end{equation}
For $0 <\ep < \frac{2}{3\sqrt{d}}$ we then define $R_{\ep} = \sigma_1(\ep^2/2,3\ep^2/2) \times \prod_{j=2}^{d} \sigma_j (\ep/2,3\ep/2) \subset B(0,1).$  By construction, for $\xi \in R_{\ep} \cup (-R_\ep) \subset B(0,1)$ we have that
\begin{equation}
 \abs{\xi \cdot Q e_1} = \abs{\sum_{i=1}^d \xi_i (Qe_1 \cdot e_i) }  = \abs{\sum_{i=1}^d \sigma_i \xi_i \abs{Qe_1 \cdot e_i} } = \sum_{i=1}^d \abs{\xi_i} \abs{Q e_1 \cdot e_i}
\end{equation}
and since $\abs{Qe_1 \cdot e_j} >0$, we readily deduce the equivalences
\begin{equation}
\omega_s(\xi) = \frac{\abs{\xi_1}^2}{\abs{\xi}^2} + \abs{\xi}^2 \asymp \frac{\ep^4}{\ep^4 + \ep^2} + (\ep^4 + \ep^2) \asymp \ep^2, 
\end{equation}
and 
\begin{equation}
\omega_s(Q^T \xi)  = \frac{\abs{Q^T \xi \cdot e_1}^2}{\abs{Q^T \xi}^2} + \abs{Q^T \xi}^2    =\frac{\abs{ \xi \cdot Qe_1}^2}{\abs{ \xi}^2} + \abs{ \xi}^2 \asymp \frac{\ep^2}{\ep^4 + \ep^2} + (\ep^4 + \ep^2) \asymp 1 + \ep^2 \asymp 1.
\end{equation}
Define $F_\ep = \vchi_{R_{\ep}} + \vchi_{-R_{\ep}}$ and note that $F_\ep(-\xi) = F_\ep(\xi) = \overline{F_\ep(\xi)}$.  The above calculations then show that we have the equivalences
\begin{equation}\label{specialized_inclusion_7}
 \int_{\R^d} \omega_s(\xi) \abs{F_\ep(\xi)}^2 d\xi \asymp \ep^2   \cdot (\ep^2 \cdot \ep^{d-1})  = \ep^{d+3},
\end{equation}
\begin{equation}
 \int_{\R^d} \abs{F_\ep(\xi)} d\xi = \int_{\R^d}  \abs{F_\ep(\xi)}^2 d\xi \asymp   1 \cdot (\ep^2 \cdot \ep^{d-1})  = \ep^{d+1},
\end{equation}
and 
\begin{equation}\label{specialized_inclusion_8}
 \int_{\R^d} \omega_s(\xi) \abs{F_\ep(Q\xi)}^2 d\xi =  \int_{\R^d} \omega_s(Q^T\xi) \abs{F_\ep(\xi)}^2 d\xi\asymp 1   \cdot (\ep^2 \cdot \ep^{d-1})  = \ep^{d+1}.  
\end{equation}

Now fix $r>3$ and $K \in \N$ such that $K \log{r} > \log(3\sqrt{d}/2)$.  Define $F = \sum_{k=K}^\infty r^{k(d+1)/2} F_{r^{-k}}$, which converges pointwise since the supports of the $F_{r^{-k}}$ are pairwise disjoint thanks to the bound $r>3$.  Then \eqref{specialized_inclusion_7}--\eqref{specialized_inclusion_8}  imply that 
\begin{equation}
 \int_{\R^d} \omega_s(\xi) \abs{F(\xi)}^2 d\xi \asymp \sum_{k=K}^\infty r^{k(d+1)} r^{-k(d+3)} = \sum_{k=K}^\infty r^{-2k} < \infty
\end{equation}
and
\begin{equation}
 \int_{\R^d}  \abs{F(\xi)} d\xi \asymp \sum_{k=K}^\infty r^{k(d+1)/2} r^{-k(d+1)} = \sum_{k=K}^\infty r^{-k(d+1)/2} < \infty,
\end{equation}
while 
\begin{equation}
\int_{\R^d} \omega_s(\xi) \abs{F(Q\xi)}^2 d\xi \asymp \int_{\R^d}  \abs{F(\xi)}^2 d\xi \asymp \sum_{k=K}^\infty r^{k(d+1)} r^{-k(d+1)} =\infty.  
\end{equation}
Hence, $f := \check{F} \in \sp^s(\R^d)$, but $f(Q\cdot) \notin \sp^s(\R^d)$ and $f \notin L^2(\R^d)$.  The inclusion $f \in C^\infty_0(\R^d)$ follows from the fact that $f$ is band-limited and $\hat{f} \in L^1(\R^d)$.
\end{proof}

\begin{remark}\label{specialized_aniso_remark}
The third item of Proposition \ref{specialized_inclusion} shows that $\sp^s(\R^d)$ is not closed under composition with rotations when $d \ge 2$, which is a strong form of anisotropy.  
\end{remark}

Next we prove a technical lemma that, in particular, will allow us to show that the elements of $\sp^s(\R^d)$ are actually functions and not just tempered distributions.

\begin{lemma}\label{localization_lemma}
 Let $s \ge 0$ and $R>0$.  Then there exists $c = c(d,R,s) >0$ such that if $f \in \sp^s(\R^d)$, then
\begin{equation}
 \int_{B(0,R)} \abs{\hat{f}(\xi)} d\xi + \left(\int_{B(0,R)^c} (1+ \abs{\xi}^2)^s \abs{\hat{f}(\xi)}^2 d\xi\right)^{1/2} \le c \norm{f}_{\sp^s}.
\end{equation}
In particular, if $s >d/2$, then there exists a constant $c =c(d,s)>0$ such that
\begin{equation}
 \norm{\hat{f}}_{L^1} \le c \norm{f}_{\sp^s}.
\end{equation}
\end{lemma}
\begin{proof}
First note that we have the trivial norm equivalence
\begin{equation}
   \norm{f}_{\sp^s}^2  \asymp \int_{B(0,R)} \frac{\xi_1^2 + \abs{\xi}^4}{\abs{\xi}^2} \abs{\hat{f}(\xi)}^2 d\xi    +\int_{B(0,R)^c} (1+\abs{\xi}^2)^{s} \abs{\hat{f}(\xi)}^2 d\xi 
\end{equation}
where the constants in the equivalence depend on $d,R,s$.  To complete the proof of the first estimate we use the Cauchy-Schwarz inequality and Proposition \ref{integral_finite}  to bound
\begin{multline}
 \int_{B(0,R)} \abs{\hat{f}(\xi)} d\xi  = \int_{B(0,R)}  \frac{\abs{\xi}}{\sqrt{\xi_1^2 + \abs{\xi}^4}} \frac{\sqrt{\xi_1^2 + \abs{\xi}^4}}{\abs{\xi}} \abs{\hat{f}(\xi)} d\xi  \\
\le  \left(\int_{B(0,R)} \frac{\abs{\xi}^2}{\xi_1^2 + \abs{\xi}^4} d\xi \right)^{1/2}  \left(\int_{B(0,R)} \frac{\xi_1^2 + \abs{\xi}^4}{\abs{\xi}^2} \abs{\hat{f}(\xi)}^2 d\xi   \right)^{1/2} \\
= c(d,R)   \left(\int_{B(0,R)} \frac{\xi_1^2 + \abs{\xi}^4}{\abs{\xi}^2} \abs{\hat{f}(\xi)}^2 d\xi   \right)^{1/2}.
\end{multline}
In the supercritical case $s > d/2$ we may then further bound 
\begin{multline}
 \int_{B(0,R)^c} \abs{\hat{f}(\xi)} d\xi \le \left(\int_{B(0,R)^c} (1+ \abs{\xi}^2)^s \abs{\hat{f}(\xi)}^2 d\xi\right)^{1/2} \left(\int_{B(0,R)^c} \frac{1}{(1+ \abs{\xi}^2)^s}  d\xi\right)^{1/2} \\
 \le c(d,R,s) \left(\int_{B(0,R)^c} (1+ \abs{\xi}^2)^s \abs{\hat{f}(\xi)}^2 d\xi\right)^{1/2}
\end{multline}
to arrive at the estimate $\norm{\hat{f}}_{L^1} \le c \norm{f}_{\sp^s}$.
\end{proof}

Next we show that all elements of $\sp^s(\R^d)$ can be decomposed into a sum of low and high frequency localizations with certain nice properties.  In particular, the decomposition shows that $\sp^s(\R^d) \subseteq C^0_0(\R^d) + H^s(\R^d)$ and hence that the elements of this space are actually functions.  Here and in the following statement we recall that the spaces $C^k_b$ and $C^k_0$ are defined in Section \ref{sec_notation}.

\begin{theorem}\label{localization_properties}
Let $s \ge 0$ and $R>0$.  For each $f \in \sp^s(\R^d)$ define the low-frequency localization  $f_{l,R} = (\hat{f} \vchi_{B(0,R)})^{\vee}$ and the high-frequency localization $f_{h,R} = (\hat{f} \vchi_{B(0,R)^c})^\vee$, both of which are well-defined as elements of $\mathscr{S}'(\R^d)$ by virtue of Lemma \ref{localization_lemma}.   Then the following hold.
\begin{enumerate}
 \item $f_{l,R}, f_{h,R} \in \sp^s(\R^d)$ and $f = f_{l,R} + f_{h,R}$.  Moreover, we have the estimates
\begin{equation}
 \norm{f_{l,R}}_{\sp^s} \le \norm{f}_{\sp^s} \text{ and }  \norm{f_{h,R}}_{\sp^s} \le \norm{f}_{\sp^s}. 
\end{equation}

 \item For each $k \in \mathbb{N}$ we have that $f_{l,R} \in C^k_0(\R^d) \subset C^k_b(\R^d)$ and there exists a constant $c = c(c,R,s,k) >0$ such that   
\begin{equation}
 \norm{f_{l,R}}_{C^k_b} = \sum_{\abs{\alpha} \le k} \norm{ \p^\alpha f_{l,R}}_{L^\infty} \le c \norm{f_{l,R}}_{\sp^s}.
\end{equation}
In particular, $f_{l,R} \in C^\infty_0(\R^d) = \bigcap_{k \in \N} C^k_0(\R^d)$.

 \item $f_{h,R} \in H^s(\R^d)$ and there exists a constant $c = c(d,R,s) >0$ such that
\begin{equation}
 \norm{f_{h,R}}_{H^s} \le c \norm{f_{h,R}}_{\sp^s}.
\end{equation}

\end{enumerate}
\end{theorem}

\begin{proof}
Lemma \ref{tempered_real_lemma} and the fact that balls are reflection-invariant imply that $f_{l,R}, f_{h,R} \in \mathscr{S}'(\R^d)$ are real-valued.  The first item then follows directly from this.  To prove the second item we first note that $f_{l,R}$ is band-limited and hence smooth.  The stated estimate then follows from the bound 
\begin{equation}
  \sum_{\abs{\alpha} \le k} \norm{ \p^\alpha f_{l,R}}_{L^\infty}  \le   \sum_{\abs{\alpha} \le k} \norm{ \widehat{ \p^\alpha f_{l,R}}}_{L^1}  \le c \int_{B(0,R)} (1+ \abs{\xi}^2)^k \abs{\hat{f}(\xi)} d\xi \le c \int_{B(0,R)} \abs{\hat{f}(\xi)} d\xi
\end{equation}
and the estimate of Lemma \ref{localization_lemma}. The fact that $\p^\alpha  f_{l,R} \to 0$ as $\abs{x} \to \infty$ for any multi-index $\alpha \in \N^d$ follows from the Riemann-Lebesgue lemma.  The third item follows directly from Lemma \ref{localization_lemma}.
\end{proof}

Our next result establishes some fundamental completeness, inclusion, and mapping properties of the space $\sp^s(\R^d)$. 

\begin{theorem}\label{specialized_properties}
Let $s \ge 0$.  Then the following hold.
\begin{enumerate}
 \item $\sp^s(\R^d)$ is a Hilbert space.
  \item The subspace $\{f \in \sp^s(\R^d) \st  \hat{f} \in C_c^\infty(\R^d) \text{ and } 0 \notin \supp(\hat{f})  \} \subset \sp^s(\R^d)$ is dense.  In particular, the set of real-valued Schwartz functions is dense in $\sp^s(\R^d)$.

 \item If $t \in \R$ and $s < t$, then we have the continuous inclusion $\sp^t(\R^d) \subset \sp^s(\R^d)$.
 
 \item For each $k \in \N$ we have the continuous inclusion $\sp^s(\R^d) \subseteq C^k_0(\R^d) + H^s(\R^d)$.

 \item If $k \in \N$ and $s >k+ d/2$, then we have the continuous inclusion  $\sp^s(\R^d) \subseteq C^k_0(\R^d) $ and there exists a constant $c = c(d,k,s) >0$ such that 
\begin{equation}
 \norm{f}_{C^k_b} \le c \norm{f}_{\sp^s} \text{ for all }f \in \sp^s(\R^d).
\end{equation}

 \item If $s \ge 1$, then there exists a constant $c = c(d,s)>0$ such that 
\begin{equation}
 \norm{\sqrt{-\Delta} f}_{H^{s-1}} \le c \norm{f}_{\sp^s} \text{ for each } f \in \sp^s(\R^d).
\end{equation}
In particular, we have that $\sqrt{-\Delta} : \sp^s(\R^d) \to H^{s-1}(\R^d)$ is a bounded linear map.

 \item If $s \ge 1$, then there exists a constant $c = c(d,s)>0$ such that 
\begin{equation}\label{specialized_properties_00}
 \norm{\nab f}_{H^{s-1}} \le c \norm{f}_{\sp^s} \text{ for each } f \in \sp^s(\R^d).
\end{equation}
In particular, we have that $\nab : \sp^s(\R^d) \to H^{s-1}(\R^d;\R^d)$ is a bounded linear map.  This map is injective.

 \item If $s \ge 1$, then there exists a constant $c = c(d,s)>0$ such that 
\begin{equation}
 \snorm{\p_1 f}_{\dot{H}^{-1}} \le c \norm{f}_{\sp^s} \text{ for each } f \in \sp^s(\R^d).
\end{equation}
In particular, we have that $\p_1 : \sp^s(\R^d) \to H^{s-1}(\R^d)\cap \dot{H}^{-1}(\R^d)$ is a bounded linear map.  This map is injective.
\end{enumerate}

\end{theorem}

\begin{proof}
Suppose that $\{f_m\}_{m \in \N} \subset \sp^s(\R^d)$ is Cauchy.  Then $\{\hat{f}_m\}_{m \in \N} \subset L^2(\R^d; \omega_s d\xi)$ is Cauchy, and hence there exists $F \in L^2(\R^d; \omega_s d\xi)$ such that $\hat{f}_m \to F$ in $L^2(\R^d; \omega_s d\xi)$ as $m \to \infty$.   The same argument used to prove Lemma \ref{localization_lemma} shows that $F \in L^1(\R^d) + L^2(\R^d) \subset \mathscr{S}'(\R^d)$ and that $\hat{f}_m \to F$ in $\mathscr{S}'(\R^d)$ as $m \to \infty$.  As such we may define $f = \check{F} \in \mathscr{S}'(\R^d)$.  By Lemma \ref{tempered_real_lemma} we know that $\overline{\hat{f}_m} = R \hat{f}_m$, but since $\hat{f}_m \to F$ in $\mathscr{S}'(\R^d)$ we deduce that $\overline{F} = R F$ and hence that $\overline{\hat{f}} = R \hat{f}$, which again by the lemma tells us that $f$ is real-valued, i.e. $\bar{f} = f$.  Then $f \in \sp^s(\R^d)$, 
\begin{equation}
 \norm{f}_{\sp^s}^2 = \int_{\R^d} \omega_s(\xi) \abs{F(\xi)}^2 d\xi, \text{ and }  \norm{f-f_m}_{\sp^s}^2 = \int_{\R^d} \omega_s(\xi) \abs{F(\xi) - \hat{f}_m(\xi)}^2 d\xi,
\end{equation}
and we conclude that $\sp^s(\R^d)$ is complete.  This proves the first item.

We now prove the second item.  Let $f \in \sp^s(\R^d)$ and $\ep >0$.  By the monotone convergence theorem we may choose $0 < R_1 < R_2 < \infty$ such that if we define the annulus $\mathfrak{A}(R_1,R_2) = B(0,R_2) \backslash B[0,R_1]$, then
\begin{equation}\label{specialized_properties_1}
 \int_{\mathfrak{A}(R_1,R_2)^c  } \omega_s(\xi) \abs{\hat{f}(\xi)  }^2 d\xi < \frac{\ep^2}{4}.
\end{equation}
We then select a non-negative and radial function $\varphi \in C^\infty_c(\R^d)$ with $\supp(\varphi) \subseteq B(0,1)$ and $\int_{\R^d} \varphi =1$.  Then for $0 < \delta < R_1/4$ we define the function $F_\delta \in C_c^\infty(\R^d)$ via
\begin{equation}
 F_\delta(\xi) = \int_{\mathfrak{A}(R_1,R_2)} \frac{1}{\delta^d} \varphi\left(\frac{\xi-z}{\delta} \right) \hat{f}(z) dz
\end{equation}
and note that $\supp(F_\delta) \subset \mathfrak{A}(R_1 /2, R_2 + R_1)$ and 
\begin{multline}
 \overline{F_\delta(\xi)} = \int_{\mathfrak{A}(R_1,R_2)} \frac{1}{\delta^d} \varphi\left(\frac{\xi-z}{\delta} \right)  \overline{\hat{f}(z)} dz = \int_{\mathfrak{A}(R_1,R_2)} \frac{1}{\delta^d} \varphi\left(\frac{\xi-z}{\delta} \right)  \hat{f}(-z) dz \\
 = \int_{\mathfrak{A}(R_1,R_2)} \frac{1}{\delta^d} \varphi\left(\frac{\xi+z}{\delta} \right)  \hat{f}(z) dz
= \int_{\mathfrak{A}(R_1,R_2)} \frac{1}{\delta^d} \varphi\left(\frac{-\xi-z}{\delta} \right)  \hat{f}(z) dz
= F_\delta(-\xi),
\end{multline}
which implies, by virtue of Lemma \ref{tempered_real_lemma}, that $\check{F}_\delta \in \mathscr{S}(\R^d)$ is real-valued.  On the annulus $\mathfrak{A}(R_1/2, R_1+R_2)$ we have the equivalence $\omega_s(\xi) \asymp 1$ (with equivalence constants depending on $d,s,R_1,R_2$), and so the usual theory of mollification (see, for instance, Appendix C of \cite{Leoni_2017}) provides us with $0 < \delta_0 < R_1/2$ such that 
\begin{equation}\label{specialized_properties_2}
 \int_{\mathfrak{A}(R_1,R_2)} \omega_s(\xi) \abs{\hat{f}(\xi)  - F_{\delta_0}(\xi)    }^2 d\xi  + \int_{\mathfrak{A}(R_1 /2, R_1 + R_2) \backslash \mathfrak{A}(R_1,R_2)} \omega_s(\xi) \abs{F_{\delta_0}(\xi)}^2 d\xi  < \frac{\ep^2}{8}.
\end{equation}
Thus, if we define $f_{\delta_0} = \check{F}_{\delta_0}$,  then  $f_{\delta_0} \in \sp^s(\R^d) \cap  \mathscr{S}(\R^d)$, $\supp(\hat{f}_{\delta_0}) \subset \mathfrak{A}(R_1/2,R_1+R_2)$, and the estimates \eqref{specialized_properties_1} and \eqref{specialized_properties_2}, together with the inclusion $\mathfrak{A}(R_1,R_2) \subseteq \mathfrak{A}(R_1/2,R_1+R_2)$, imply that
\begin{multline}
 \norm{f - f_{\delta_0}}_{\sp^s}^2 = \int_{\mathfrak{A}(R_1/2,R_1 + R_2)^c} \omega_s(\xi) \abs{\hat{f}(\xi)}^2 d\xi + \int_{\mathfrak{A}(R_1/2,R_1 + R_2)}  \omega_s(\xi) \abs{\hat{f}(\xi)  - F_{\delta_0}(\xi)    }^2 d\xi \\
< \frac{\ep^2}{4} + \int_{\mathfrak{A}(R_1, R_2)}  \omega_s(\xi) \abs{\hat{f}(\xi)  - F_{\delta_0}(\xi)    }^2 d\xi + \int_{\mathfrak{A}(R_1/2,R_1 + R_2) \backslash \mathfrak{A}(R_1,R_2)}  \omega_s(\xi) \abs{\hat{f}(\xi)  - F_{\delta_0}(\xi)    }^2 d\xi   \\
< \frac{\ep^2}{4} + \int_{\mathfrak{A}(R_1, R_2)}  \omega_s(\xi) \abs{\hat{f}(\xi)  - F_{\delta_0}(\xi)    }^2 d\xi
+ 2 \int_{\mathfrak{A}(R_1/2,R_1 + R_2) \backslash \mathfrak{A}(R_1,R_2)}  \omega_s(\xi) \abs{F_{\delta_0}(\xi)}^2 d\xi \\
+ 2 \int_{  \mathfrak{A}(R_1,R_2)^c}  \omega_s(\xi) \abs{\hat{f}(\xi)}^2 d\xi
< \frac{\ep^2}{4} + 2 \frac{\ep^2}{8} + 2\frac{\ep^2}{4} = \ep^2,
\end{multline}
which completes the proof of the second item.

The third item follows trivially from the pointwise estimate $\omega_s \le \omega_t$, and the fourth follows immediately from Theorem \ref{localization_properties}.  The fifth item follows from the fourth and the standard Sobolev embedding $H^s(\R^d) \hookrightarrow C^k_0(\R^d)$ for $s > k + d/2$.

We now turn to the proof of the sixth item. Assume $s \ge 1$.  First note that there is a constant $c = c(s) >0$ such that $\abs{\xi}^2(1+\abs{\xi}^2)^{s-1} \le c \omega_s(\xi)$  for all $\xi \in \R^d.$  Then for a real-valued  $f \in \mathscr{S}(\R^d)$ we may bound 
\begin{equation}
 \norm{\sqrt{-\Delta} f}_{H^{s-1}}^2 = 4\pi^2 \int_{\R^d} \abs{\xi}^2(1+\abs{\xi}^2)^{s-1} \abs{\hat{f}(\xi)}^2 d\xi \le 4\pi^2 c\int_{\R^d} \omega_s(\xi) \abs{\hat{f}(\xi)}^2 d\xi  = 4\pi^2 c \norm{f}_{\sp^s}^2.
\end{equation}
The sixth item then follows from this and the density result of the second item.  The seventh item then follows from the second and sixth items, together with the identity 
\begin{equation}
  \norm{\nab f}_{H^{s-1}} = \norm{\sqrt{-\Delta} f}_{H^{s-1}} \text{ for all }f \in \mathscr{S}(\R^d),
\end{equation}
and the observation that $\nab f =0$ if and only if $\abs{\xi} \abs{\hat{f}(\xi)} =0$, which requires that $\hat{f} =0$ almost everywhere.

To prove the eighth item we first note that $\frac{\xi_1^2}{\abs{\xi}^2} \le \omega_s(\xi)$ for all $\xi \in \R^{n}$.  Then for $f \in \mathscr{S}(\R^d)$ we bound 
\begin{equation}
 \snorm{\p_1 f}_{\dot{H}^{-1}}^2 \le c \int_{\R^d} \frac{\xi_1^2 }{\abs{\xi}^2} \abs{\hat{f}(\xi)}^2d\xi \le c\int_{\R^d} \omega_s(\xi) \abs{\hat{f}(\xi)}^2 d\xi  = c \norm{f}_{\sp^s}^2,
\end{equation}
and we again use the second item to conclude the estimates holds for general $f \in \sp^s$.  Injectivity follows since $\p_1 f =0$ if and only if $\abs{\xi_1} \abs{\hat{f}(\xi)} =0$, which requires that $\hat{f}=0$ almost everywhere.
\end{proof}

\subsection{A class of specialized Sobolev spaces on $\Omega$ built from $\sp^s(\R^{n-1})$}

For $0 \le s \in \R$, $n \ge 2$, and $\zeta \in C^{0,1}_b(\R^{n-1})$ such that $\inf \zeta > 0$ we define the space 
\begin{multline}\label{an_space_def}
 \an^s(\Omega_\zeta)  = H^s(\Omega_\zeta) + \sp^s(\R^{n-1}) = \{f \in L^1_{loc}(\Omega_\zeta) \st \text{there exist }g \in H^s(\Omega_\zeta) \text{ and } h \in \sp^s(\R^{n-1})  \\
 \text{ such that }f(x) = g(x) + h(x') \text{ for almost every }x \in \Omega_\zeta \},
\end{multline}
and we endow this space with the norm 
\begin{equation}
 \norm{f}_{\an^s} = \inf\{ \norm{g}_{H^s} + \norm{h}_{\sp^s} \st f = g +h    \}.
\end{equation}
Note that $\Omega = \Omega_b$, so in particular this defines a scale of spaces with functions defined on $\Omega$.

Our first result shows that this is a Banach space.

\begin{theorem}\label{omega_aniso_complete}
Let $s \ge 0$, $n \ge 2$,  and $\zeta \in C^{0,1}_b(\R^{n-1})$  such that $\inf \zeta > 0$.  Then $\an^s(\Omega_\zeta)$ is a Banach space.
\end{theorem}
\begin{proof}
Let $\{f_m\}_{m \in \N} \subset \an^s(\Omega_\zeta)$ be such that $\sum_{m \in \N} \norm{f_m}_{\an^s} <\infty.$  We may then select $\{g_m\}_{m \in \N} \subset H^s(\Omega_\zeta)$ and $\{h_m\}_{m \in \N} \subset \sp^s(\R^{n-1})$ such that  $\norm{g_m}_{H^s} + \norm{h_m}_{\sp^s} \le 2 \norm{f_m}_{\an^s},$ which in particular means that 
$\sum_{m \in \N} \norm{g_m}_{H^s} < \infty$ and $\sum_{m \in \N} \norm{h_m}_{\sp^s} <\infty.$   Since $H^s(\Omega_\zeta)$ and $\sp^s(\R^{n-1})$ (see Theorem \ref{specialized_properties}) are Banach spaces, there exist $g \in H^s(\Omega_\zeta)$ and $h \in \sp^s(\R^{n-1})$ such that $g = \sum_{m \in \N} g_m$ and  $h = \sum_{m \in \N} h_m,$ with the convergence of the sums occurring in $H^s(\Omega_\zeta)$ and $\sp^s(\R^{n-1})$, respectively.  From this we deduce that $f := g + h \in \an^s(\Omega_\zeta)$ is such that $f = \sum_{m \in \N} f_m$, with the sum converging in $\an^s(\Omega_\zeta)$.  Thus every absolutely summable sequence is summable, and so $\an^s(\Omega_\zeta)$ is a Banach space.
\end{proof}

\begin{remark}\label{omega_aniso_remark}
When $n=2$ Proposition \ref{specialized_inclusion} implies that $H^s(\R^{n-1}) = \sp^s(\R^{n-1})$ algebraically and topologically, so in this case $\an^s(\Omega_\zeta) = H^s(\Omega_\zeta)+H^s(\R^{n-1}) = H^s(\Omega_\zeta)$.  When $n \ge 3$, it's clear that we have the continuous inclusion $H^s(\Omega_\zeta) \subseteq \an^s(\Omega_\zeta)$, but due to the strict inclusion $H^s(\R^{n-1}) \subset \sp^s(\R^{n-1})$ from Proposition \ref{specialized_inclusion}, the previous inclusion is strict as well.
\end{remark}

Our next result shows that the trace operator may be extended to act on $\an^s(\Omega)$ when $s >1/2$.  Recall that we employ the abuse of notation for functions on $\Sigma_b$ described at the end of Section \ref{sec_notation}.

\begin{theorem}\label{omega_aniso_trace}
Let $s > 1/2$ and $n \ge 2$.  Then the trace map $\tr: H^s(\Omega) \to H^{s-1/2}(\Sigma_b)$ extends to a bounded linear map $\tr: \an^s(\Omega) \to \sp^{s-1/2}(\R^{n-1})$.  More precisely, the following hold.
\begin{enumerate}
 \item If $f \in C^0(\bar{\Omega}) \cap \an^s(\Omega)$, then $\tr f = f \vert_{\Sigma_b}$.
 \item If $\varphi \in C_c^1(\R^{n-1} \times (0,b])$, then 
\begin{equation}
 \int_{\Sigma_b} \tr f \varphi = \int_{\Omega} \p_n f \varphi + f \p_n \varphi \text{ for all } f \in \an^s(\Omega).
\end{equation}
 \item There exists a constant $c = c(n,s,b) >0$ such that 
\begin{equation}
 \norm{\tr f}_{\sp^{s-1/2}} \le c \norm{f}_{\an^s} \text{ for all } f\in \an^s(\Omega).
\end{equation} 
\end{enumerate}
\end{theorem}
\begin{proof}
Let $f \in \an^s(\Omega)$ and suppose that $f = g_1 + h_1 = g_2 + h_2$ for $g_1,g_2 \in H^s(\Omega)$ and $h_1,h_2 \in \sp^s(\R^{n-1})$, which in particular requires that $\p_n g_1 = \p_n g_2$ in $\Omega$.  Let $\varphi \in C_c^1(\R^{n-1} \times (0,b])$.  From the usual trace theory in $H^s(\Omega)$ and the fact that $h_1,h_2$ do not depend on $x_n$ we may compute 
\begin{equation}
 \int_{\Sigma_b} (\tr g_1 + h_1) \varphi = \int_{\Omega} (g_1 + h_1) \p_n \varphi +  \p_n g_1 \varphi = \int_{\Omega} (g_2 + h_2) \p_n \varphi +  \p_n g_2 \varphi =  \int_{\Sigma_b} (\tr g_2 + h_2) \varphi.
\end{equation}
Hence $\tr g_1 + h_1 = \tr g_2 + h_2$, and so we unambiguously define $\tr f = \tr g + h \in H^{s-1/2}(\Sigma_b) + \sp^s(\R^{n-1}) \subset \sp^{s-1/2}(\R^{n-1})$.  The stated properties of $\tr : H^s(\Omega) \to \sp^{s-1/2}(\R^{n-1})$ then follow from the standard trace theory and Theorem \ref{specialized_properties}.

\end{proof}

The next result shows that functions in $\an^s(\Omega)$ interact nicely with the horizontal Fourier transform.

\begin{proposition}\label{omega_aniso_fourier}
Let $s \ge 0$, $n \ge 2$, and $f \in \an^s(\Omega)$.  Then the following hold. 

\begin{enumerate}

\item For almost every $x_n \in (0,b)$ we have that $f(\cdot,x_n) \in \sp^s(\R^{n-1})$ and if we write $\hat{\cdot}$ for the Fourier transform with respect to $x' \in \R^{n-1}$, then $\hat{f}(\cdot,x_n) \in L^1(\R^{n-1}) + L^2(\R^{n-1})$.

\item If $s \in \N$, then for almost every $\xi \in \R^{n-1}$ we have that $\hat{f}(\xi,\cdot) \in H^s((0,b);\C)$.  

\end{enumerate}
\end{proposition}
\begin{proof}
Since $f \in \an^s(\Omega)$ we can write $f(x) = g(x) + h(x')$ for $g \in H^s(\Omega)$ and $h \in \sp^s(\R^{n-1})$.  The Parseval and Tonelli theorems imply that $\hat{g}(\cdot, x_n) \in L^2(\R^{n-1};\C)$ for almost every $x_n \in (0,b)$, and  Lemma \ref{localization_lemma} implies that $\hat{h} \in L^1(\R^{n-1};\C) + L^2(\R^{n-1};\C)$.  This completes the proof of the first item.  For the second item we again use the Tonelli and Parseval theorems to see that if $0 \le j \le s$, then $\p_n^j \hat{g}(\xi,\cdot) \in L^2((0,b);\C)$ for almost every $\xi \in \R^{n-1}$.  On the other hand, $\hat{h}(\xi)$ does not depend on $x_n$ and $(0,b)$ has finite measure, so we conclude that for $0 \le j \le s$ we have the inclusion $\p_n^j \hat{f}(\xi,\cdot) \in L^2((0,b);\C)$ for almost every $\xi \in \R^{n-1}$. 
\end{proof}

Now we record some essential inclusion and mapping properties of $\an^s(\Omega)$.

\begin{theorem}\label{omega_aniso_properties}
Let $s \ge 0$, $n \ge 2$,  and $\zeta \in C^{0,1}_b(\R^{n-1})$ such that $\inf \zeta > 0$.  Then the following hold. 
\begin{enumerate}
\item If $t \in \R$ and $s < t$, then we have the continuous inclusion $\an^t(\Omega_\zeta) \subset \an^s(\Omega_\zeta)$.

\item For each $f \in \sp^s(\R^{n-1})$ we have that $\norm{f}_{\an^s} \le \norm{f}_{\sp^s}$, and hence we have the continuous inclusion $\sp^s(\R^{n-1}) \subset \an^s(\Omega_\zeta)$.

\item If $k \in \N$ and $s >k+ n/2$, then there exists a constant $c = c(n,k,s,\zeta) >0$ such that 
\begin{equation}
 \norm{f}_{C^k_b} \le c \norm{f}_{\an^s} \text{ for all }f \in \an^s(\Omega_\zeta).
\end{equation}
Moreover, for $\zeta = b$ (in which case $\Omega_\zeta = \Omega$) we have the continuous inclusion  
\begin{equation}
\an^s(\Omega) \subseteq \{f \in C^k_b(\Omega) \st \lim_{\abs{x'} \to \infty} \p^\alpha f(x) = 0 \text{ for } \abs{\alpha} \le k\} \subset     C^k_b(\Omega).
\end{equation}

 \item If $s \ge 1$, then there exists a constant $c = c(n,s,\zeta)>0$ such that 
\begin{equation}
 \norm{\nab f}_{H^{s-1}} \le c \norm{f}_{\an^s} \text{ for each } f \in \an^s(\Omega_\zeta).
\end{equation}
In particular, we have that $\nab : \an^s(\Omega_\zeta) \to H^{s-1}(\Omega_\zeta;\R^n)$ is a bounded linear map.
\end{enumerate}
\end{theorem}
\begin{proof}
 These follow immediately from Theorem \ref{specialized_properties} and the usual properties of the Sobolev space $H^s(\Omega_\zeta)$.
\end{proof}

\subsection{Nonlinear analysis tools in the specialized spaces}\label{sec_special_nonlinear}

Later in the paper we will employ our specialized Sobolev spaces to produce solutions to \eqref{traveling_euler}.  In doing so, we will need a number of nonlinear tools in these spaces, and our goal now is to develop these.  We begin with four important results about products involving the specialized spaces. 

We first investigate how products $fg$ of functions $f\in \sp^s(\R^d)$ and $g \in H^s(\R^d)$ behave in the supercritical case $s > d/2$.  

\begin{theorem}\label{specialized_product_supercrit}
Suppose that $s > d/2$.  There exists a constant $c = c(d,s)>0$ such that 
\begin{equation}\label{specialized_product_supercrit_00}
 \norm{fg}_{H^s} \le c \norm{f}_{\sp^s} \norm{g}_{H^s} \text{ for all }f\in \sp^s(\R^d) \text{ and }g \in H^s(\R^d).
\end{equation}
Consequently, for $1 \le k \in \N$ the mapping 
\begin{equation}\label{specialized_product_supercrit_01}
H^s(\R^d) \times \prod_{j=1}^k \sp^s(\R^d) \ni (g,f_1,\dotsc,f_k) \mapsto g \prod_{j=1}^k f_j \in H^s(\R^d) 
\end{equation}
is a bounded $(k+1)-$linear map.
\end{theorem}
\begin{proof}
First recall that since $s > d/2$,  Lemma \ref{localization_lemma} provides a constant $c>0$ such that $\norm{\hat{f}}_{L^1} \le c \norm{f}_{\sp^s}$ for all $f \in \sp^s(\R^d)$.  Similarly, $\norm{\hat{f}}_{L^1} \le c \norm{f}_{H^s}$ for all $f \in H^s(\R^d)$.

Now let $f,g \in \mathscr{S}(\R^d)$ be real-valued.  Then 
\begin{equation}
 \abs{\widehat{fg}(\xi)} = \abs{\hat{f} \ast \hat{g}(\xi) }\le \int_{\R^d} \abs{\hat{f}(z)} \abs{\hat{g}(\xi-z)} dz,
\end{equation}
which may be combined with the elementary estimate 
\begin{equation}
 (1+\abs{\xi}^2)^{s/2} \le c\left( (1+\abs{\xi-z}^2)^{s/2} +  \abs{z}^s \right)
\le 
 c\left( (1+\abs{\xi-z}^2)^{s/2} +  \abs{z} (1+ \abs{z}^2)^{(s-1)/2} \right)
\end{equation}
to arrive at the bound 
\begin{equation}
 (1+\abs{\xi}^2)^{s/2}  \abs{\widehat{fg}(\xi)} \le c \int_{\R^d}  \abs{z} (1+ \abs{z}^2)^{(s-1)/2}  \abs{\hat{f}(z)} \abs{\hat{g}(\xi-z)} dz + c \int_{\R^d}  \abs{\hat{f}(z)} (1+\abs{\xi-z}^2)^{s/2}\abs{\hat{g}(\xi-z)} dz.
\end{equation}
From this, Young's inequality, the above $L^1$ estimates, and Theorem \ref{specialized_properties} we deduce that 
\begin{equation}
\norm{fg}_{H^s} = \norm{ (1+\abs{\cdot}^2)^{s/2}  \widehat{fg} }_{L^2} \le c \norm{\sqrt{-\Delta} f}_{H^{s-1}} \norm{\hat{g}}_{L^1} + c \norm{\hat{f}}_{L^1} \norm{g}_{H^s}  \le c \norm{f}_{\sp^s} \norm{g}_{H^s}.
\end{equation}
The estimate \eqref{specialized_product_supercrit_00} then holds for all $f\in \sp^s(\R^d)$ and $g \in H^s(\R^d)$ due to the density of real-valued Schwartz functions in both spaces.  The boundedness of the mapping \eqref{specialized_product_supercrit_01} then follows from \eqref{specialized_product_supercrit_00}, the fact that $H^s(\R^d)$ is an algebra for $s >d/2$, and an induction argument.
\end{proof}

Our next product result is a variant of Theorem \ref{specialized_product_supercrit} that assumes one of the factors also has a special product form.

\begin{theorem}\label{Rn_specialized_product}
Let $\varphi \in C^\infty_c(\R)$, $n/2 < s \in \R$, and $V$ be a real finite dimensional inner-product space.  Then for $0 \le r \le s$ there exists a constant $c = c(n,V,s,r,\varphi) >0$ such that if $f \in H^r(\R^n;V)$, $\eta \in \sp^s(\R^{n-1})$, and $\varphi \eta f : \R^n \to V$ is defined via $(\varphi \eta f)(x) = \varphi(x_n) \eta(x') f(x)$, then $\varphi \eta f \in H^r(\R^n;V)$ and 
\begin{equation}
 \norm{\varphi \eta f}_{H^r} \le c \norm{\eta}_{\sp^s} \norm{f}_{H^r}.
\end{equation}
\end{theorem}
\begin{proof}
We first use  Theorem \ref{localization_properties} with $R=1$ to write $\eta = \eta_0 + \eta_1$ with $\eta_0 = \eta_{l,1}$ and $\eta_1 = \eta_{h,1}$ in the notation of the theorem.  Then  
\begin{equation}
\norm{\varphi \eta f}_{H^r} \le  \norm{\varphi \eta_0 f}_{H^r} +\norm{\varphi \eta_1 f}_{H^r}. 
\end{equation}
By the second item of Theorem \ref{localization_properties} we can bound, for any $s \le k \in \N$,
\begin{equation}
\norm{\varphi \eta_0 f}_{H^r} \le c \norm{\varphi \eta_0}_{C^k_b(\R^n)} \norm{f}_{H^r} \le c \norm{ \eta_0}_{C^k_b(\R^{n-1})} \norm{f}_{H^r} \le c \norm{ \eta}_{\sp^s} \norm{f}_{H^r}.
\end{equation}
On the other hand, from  Lemma \ref{prod_full_space}, the Fourier characterization of $H^s(\R^n)$, and the third item of Theorem \ref{localization_properties} we can bound 
\begin{equation}
\norm{\varphi \eta_1 f}_{H^r} \le c \norm{\varphi \eta_1}_{H^s(\R^n)} \norm{f}_{H^r} \le c \norm{\eta_1}_{H^s(\R^{n-1})} \norm{f}_{H^r} \le c \norm{\eta}_{\sp^s} \norm{f}_{H^r}.
\end{equation}
Combining these three estimates then yields the stated inclusion and estimate.
\end{proof}

Next, we turn our attention to establishing an analog of Theorem \ref{specialized_product_supercrit} for the spaces $\an^s(\Omega_\zeta)$.

\begin{theorem}\label{omega_product_supercrit}
Let $n \ge 2$, $s > n/2$, and $\zeta \in C^{0,1}_b(\R^{n-1})$ such that $\inf \zeta > 0$.  There exists a constant $c = c(n,s,\zeta)>0$ such that 
\begin{equation}\label{omega_product_supercrit_00}
 \norm{fg}_{H^s} \le c \norm{f}_{\an^s} \norm{g}_{H^s} \text{ for all }f\in \an^s(\Omega_\zeta) \text{ and }g \in H^s(\Omega_\zeta).
\end{equation}
In particular, for $1 \le k \in \N$ the mapping 
\begin{equation}\label{omega_product_supercrit_01}
H^s(\Omega_\zeta) \times \prod_{j=1}^k \an^s(\Omega_\zeta) \ni (g,f_1,\dotsc,f_k) \mapsto g \prod_{j=1}^k f_j \in H^s(\Omega_\zeta) 
\end{equation}
is a bounded $(k+1)-$linear map.    
\end{theorem}
\begin{proof}
The boundedness of the $(k+1)-$linear map \eqref{omega_product_supercrit_01} follows from \eqref{omega_product_supercrit_00} and an induction argument, so we will only prove \eqref{omega_product_supercrit_00}.  Let $f \in \an^s(\Omega_\zeta)$ and $g \in H^s(\Omega_\zeta)$.  Write $f(x) = h(x) + \varphi(x')$ for $h \in H^s(\Omega_\zeta)$ and $\varphi \in \sp^s(\R^{n-1})$.  Then $fg = hg + \varphi g$, but from the standard theory of Sobolev spaces we have that $hg \in H^s(\Omega_\zeta)$ with $\norm{hg}_{H^s} \le c \norm{h}_{H^s} \norm{g}_{H^s}$ for a constant $c = c(n,s,\zeta)>0$.  Thus it suffices to show that $\varphi g \in H^s(\Omega_\zeta)$ and $\norm{\varphi g}_{H^s} \le c \norm{\varphi}_{\sp^s} \norm{g}_{H^s}$ for a constant $c = c(n,s,\zeta)>0$.  

To prove this we first use the Stein extension theorem (see the proof of Lemma \ref{extension_char}) to pick $G = Eg \in H^s(\R^n)$ such that $G = g$ almost everywhere in $\Omega_\zeta$ and $\norm{G}_{H^s(\R^n)} \le c \norm{g}_{H^s(\Omega_\zeta)}$ for a constant $c = c(n,s,\zeta)>0$.  Then we use Lemma \ref{slicing_lemma} to bound
\begin{equation}
\frac{1}{c} \norm{\varphi G}_{H^s(\R^n)}^2  \le \int_{\R} \norm{\varphi(\cdot) G(\cdot,x_n)   }_{H^s(\R^{n-1})}^2 dx_n + \int_{\R} (1+\tau^2)^s \norm{\mathcal{F}_n (\varphi G)(\cdot,\tau)}_{L^2(\R^{n-1})}^2 d\tau,
\end{equation}
where $\mathcal{F}_n$ denotes the Fourier transform with respect to the $n^{th}$ variable.  For the latter term, since $\varphi$ does not depend on $x_n$ we can use Theorem \ref{specialized_properties} to bound
\begin{multline}
  \int_{\R} (1+\tau^2)^s \norm{\mathcal{F}_n (\varphi G)(\cdot,\tau)}_{L^2(\R^{n-1})}^2 d\tau =  \int_{\R} (1+\tau^2)^s \norm{\varphi \mathcal{F}_n  G(\cdot,\tau)}_{L^2(\R^{n-1})}^2 d\tau \\
 \le \int_{\R} (1+\tau^2)^s \norm{\varphi}_{L^\infty(\R^{n-1})}^2 \norm{ \mathcal{F}_n  G(\cdot,\tau)}_{L^2(\R^{n-1})}^2 d\tau \\
 \le c\norm{\varphi}_{\sp^s}^2  \int_{\R} (1+\tau^2)^s  \norm{ \mathcal{F}_n  G(\cdot,\tau)}_{L^2(\R^{n-1})}^2 d\tau.
\end{multline}
For the former term we use Theorem \ref{specialized_product_supercrit} for almost every $x_n \in \R$ to bound 
\begin{equation}
\int_{\R} \norm{\varphi(\cdot) G(\cdot,x_n)   }_{H^s(\R^{n-1})}^2 dx_n \le c  \norm{\varphi}_{\sp^s(\R^{n-1})}^2   \int_{\R}   \norm{ G(\cdot,x_n)   }_{H^s(\R^{n-1})}^2 dx_n.
\end{equation}

Hence, upon combining these estimates and again employing Lemma \ref{slicing_lemma}, we deduce that 
\begin{equation}
 \norm{\varphi G}_{H^s(\R^n)} \le c \norm{\varphi}_{\sp^s} \norm{G}_{H^s(\R^n)}
\end{equation}
for a constant $c = c(n,s,\zeta)>0$.  Since $\varphi G = \varphi g$ almost everywhere in $\Omega_\zeta$ we may then use Lemma \ref{extension_char} to conclude that 
\begin{equation}
\norm{\varphi g}_{H^s(\Omega_\zeta)} \le c\norm{\varphi}_{\sp^s} \norm{G}_{H^s(\R^n)}. 
\end{equation}
This proves the desired inclusion and estimate.

\end{proof}

The following is a variant of the results in Theorems \ref{specialized_product_supercrit} and \ref{omega_product_supercrit}  that works in more general domains but in a slightly lower regularity class with integer regularity bounds.

\begin{theorem}\label{omega_zeta_products}
Let $\zeta \in C^{0,1}_b(\R^{n-1})$ be such that $\inf \zeta >0$.  Suppose that $\eta \in \sp^{k+2}(\R^{n-1})$ for $n/2 < k \in \N$ and $\norm{\eta}_{L^\infty} \le b/2$.  For $0 \le j \le n-1$ define $\mu_j : \Omega_\zeta \to \R$ via $\mu_0(x) = \eta(x')$ and $\mu_j(x) = x_n \partial_j \eta(x')$.  For $1 \le \ell \in \N$ define $M_\ell : \Omega_\zeta \to \R$ via $M_\ell(x) = (b+ \eta(x'))^{-\ell}$.  Then the following hold.
\begin{enumerate}
 \item For every $0 \le s \le k+2$ there exists a constant $c = c(n,\norm{\zeta}_{C^{0,1}_b},k,s) >0$ such that 
\begin{equation}
 \norm{\mu_0  f }_{H^s} \le c \norm{\eta}_{\sp^{k+2}} \norm{f}_{H^s} \text{ for all }f\in H^s(\Omega_\zeta).
\end{equation}

 \item For every $0 \le s \le k+1$ there exists a constant $c = c(n,\norm{\zeta}_{C^{0,1}_b},k,s) >0$ such that if $1 \le j \le n-1$, then
\begin{equation}
 \norm{\mu_j  f }_{H^s} \le c \norm{\eta}_{\sp^{k+2}} \norm{f}_{H^s} \text{ for all }f\in H^s(\Omega_\zeta).
\end{equation}

\item For every $0 \le s \le k+2$ and $1 \le \ell \in \N$ there exists a constant $c = c(n,\norm{\zeta}_{C^{0,1}_b},s,k,\ell)>0$ such that 
\begin{equation}
 \norm{M_\ell f}_{H^s} \le c(1+ \norm{\eta}_{\sp^{k+2}}) \norm{f}_{H^s} \text{ for all }f \in H^s(\Omega_\zeta).
\end{equation}

\end{enumerate}
\end{theorem}
\begin{proof}
We begin with the proof of the first item.  Clearly 
\begin{equation}
 \norm{\mu_0 f}_{L^2} \le \frac{b}{2} \norm{f}_{L^2} \text{ for every } f \in L^2(\Omega_\zeta).
\end{equation}
Suppose now, that for $0 \le m \le k+1$ there exists a constant $c =c(n,\zeta,k,m) >0$ such that 
\begin{equation}
 \norm{\mu_0 f}_{H^m} \le c \norm{\eta}_{\sp^{k+2}} \norm{f}_{H^m} \text{ for all }f \in H^m(\Omega_\zeta).
\end{equation}
For $f \in H^{m+1}(\Omega_\zeta)$ and $1 \le j \le n$ we have that 
\begin{equation}
 \p_j (\mu_0 f) = \mu_0 \p_j f +   \p_j \eta f.
\end{equation}
According to Theorem \ref{specialized_properties} we have that $\p_j \eta \in H^{k+1}(\R^{n-1})$.  We can then use these, the induction hypothesis,  Lemma \ref{omega_zeta_subcrit}, and the seventh item of Theorem \ref{specialized_properties} to bound 
\begin{equation}
\norm{\p_j(\mu_0 f)}_{H^m} \le c \norm{\eta}_{\sp^{k+2}} \norm{\p_j f}_{H^m} +c \norm{\p_j \eta}_{H^{k}} \norm{f}_{H^m} \le c \norm{\eta}_{\sp^{k+2}} \norm{f}_{H^{m+1}}.
\end{equation}
Combining this with the induction hypothesis then shows that there is a constant $c = c(n,\norm{\zeta}_{C^{0,1}_b},k,m)>0$ such that
\begin{equation}
 \norm{\mu_0 f}_{H^{m+1}} \le c \norm{\eta}_{\sp^{k+2}} \norm{f}_{H^{m+1}} \text{ for all }f \in H^{m+1}(\Omega_\zeta).
\end{equation}
Proceeding with a finite induction then proves that there is a constant $c = c(n,\norm{\zeta}_{C^{0,1}_b},k)>0$ such that if $0 \le m \le k+2$ is an integer, then
\begin{equation}
 \norm{\mu_0 f}_{H^{m}} \le c \norm{\eta}_{\sp^{k+2}} \norm{f}_{H^{m}} \text{ for all }f \in H^{m}(\Omega_\zeta).
\end{equation}
This shows that the linear map $H^m(\Omega_\zeta) \ni f \mapsto \mu_0 f \in H^m(\Omega_\zeta)$ is bounded for each $0 \le m \le k+2$ with operator norm bounded by $c\norm{\eta}_{\sp^{k+2}}$.  Standard interpolation theory (see, for instance, \cite{BL_1976,Leoni_2017,Triebel_1995}) then shows that this map is bounded on $H^s(\Omega_\zeta)$ for $0 \le s \le k+2$ with operator norm bounded by $c\norm{\eta}_{\sp^{k+2}}$.  This proves the first item.

We now turn to the proof of the second item.  In light of Theorem \ref{specialized_properties} we have the inclusion $\mu_j \in H^{k+1}(\Omega_\zeta)$ and the  bound $\norm{\mu_j}_{H^{k+1}} \le c(n,\norm{\zeta}_{C^{0,1}_b},k) \norm{\eta}_{\sp^{k+2}}$.  With these in hand, the second item then follows directly from Theorem \ref{omega_product_supercrit}.

For the third item we note that for $1 \le \ell \in \N$ and $f \in H^1(\Omega_\zeta)$ we have the identity
\begin{equation}
 \p_j (M_\ell f) = M_\ell \p_j f - \ell  M_{\ell+1} \p_j \eta f.
\end{equation}
With this and the trivial bound 
\begin{equation}
\norm{M_\ell f}_{L^2} \le \norm{\frac{1}{(b+\eta)^\ell}}_{L^\infty} \norm{f}_{L^2} \le \left(\frac{2}{b}\right)^{\ell} \norm{f}_{L^2} \le c(1+ \norm{\eta}_{\sp^{k+2}}) \norm{f}_{L^2} \text{ for all }f \in L^2(\Omega_\zeta)
\end{equation}
in hand, we may then argue as in the proof of the first item to conclude that the third item holds.
 
\end{proof}

From Theorem \ref{omega_zeta_products} we know that under some assumptions on $\eta$, we have the inclusion $M_1 f \in H^s$ whenever $f \in H^s$.  We now aim to investigate the smoothness of a generalization of the map $(\eta,f) \mapsto M_1 f$.  This will be essential later in our nonlinear analysis.

\begin{theorem}\label{omega_power_series}
Let $n \ge 2$, $s > n/2$,  and $\zeta \in C^{0,1}_b(\R^{n-1})$ such that $\inf \zeta > 0$. Let $c = c(n,s,\zeta)>0$ denote the larger of the constant from the third item of Theorem \ref{omega_aniso_properties} with $k=0$ and the constant from Theorem \ref{omega_product_supercrit}.   Define the ball $B_{\an^s}(0, b/(2c)) = \{f \in \an^{s}(\Omega_\zeta) \st \norm{f}_{\an^{s}} < \frac{b}{2c} \}$.  Then the maps $\Gamma_1,\Gamma_2 :   B_{\an^s}(0, b/(2c)) \times H^s(\Omega_\zeta)  \to H^s(\Omega_\zeta)$ given by 
\begin{equation}
 \Gamma_1(f,g) = \frac{g}{b+f} \text{ and } \Gamma_2(f,g) = \frac{gf}{b+f}
\end{equation}
are well-defined and smooth.
\end{theorem}
\begin{proof}
We have that $\Gamma_2(f,g) = f \Gamma_1(f,g)$, so if $\Gamma_1$ is well-defined and smooth, then Theorem \ref{omega_product_supercrit} guarantees that $\Gamma_2$ is also well-defined and smooth.  It thus suffices to prove that $\Gamma_1$ is well-defined and smooth.

We begin by showing that $\Gamma_1$ is well-defined, i.e. it actually takes values in $H^s(\Omega)$.  To this end we note that Theorem \ref{omega_aniso_properties} provides the estimate $\norm{f}_{C^0_b} \le c \norm{f}_{\an^s},$
while Theorem \ref{omega_product_supercrit} provides the estimate 
\begin{equation}
 \norm{g f^k}_{\an^s} \le c^k \norm{g}_{H^s} \norm{f}_{\an^s}^k \text{ for every } k \ge 1, f \in H^s(\Omega_\zeta), g \in \an^s(\Omega_\zeta).
\end{equation}
Using this, we see that 
\begin{equation}
  \sum_{k=0}^\infty \frac{1}{b^k} \norm{f}_{C^0_b}^k \le   \sum_{k=0}^\infty \frac{c^k}{b^k} \norm{f}_{\an^s}^k < \sum_{k=0}^\infty \frac{1}{2^k} = 2
\end{equation}
and
\begin{equation}
 \sum_{k=1}^\infty \frac{1}{b^k}  \norm{g f^k}_{H^s} \le  \sum_{k=1}^\infty \left(\frac{c}{b}\norm{f}_{\an^s}\right)^k  \norm{g}_{H^s} < \norm{g}_{H^s} \sum_{k=1}^\infty \frac{1}{2^k} = \norm{g}_{H^s},
\end{equation}
and hence the series
\begin{equation}
 \sum_{k=0}^\infty \frac{(-1)^k}{b^k} f^k = \frac{b}{b + f} 
\end{equation}
converges uniformly in $\Omega_\zeta$, and the series $ \sum_{k=1}^\infty \frac{(-1)^{k}}{b^{k}} g f^{k}$ converges in $H^s(\Omega_\zeta)$.  However, 
\begin{equation}
\Gamma_1(f,g) = \frac{g}{b+f} = \frac{g}{b} \frac{b}{b+f} = \frac{g}{b} + \frac{1}{b} \sum_{k=0}^\infty \frac{(-1)^k}{b^k} g f^k \in H^s(\Omega_\zeta),
\end{equation}
so $\Gamma_1(f,g)$ is well-defined.

We now turn to the proof of smoothness.  Define the linear map $T : \an^s(\Omega_\zeta) \to \L(H^s(\Omega_\zeta))$ via $T(f) g = gf$.  By virtue of Theorem \ref{omega_product_supercrit}, $T$ is bounded and $\norm{T}_{\L(\an^s;\L(H^s))} \le c$.  In the unital Banach algebra $\L(H^s(\Omega_\zeta))$ we have that the power series $F(L) = \sum_{k=0}^\infty \frac{1}{b^k} L^k$ converges and defines a smooth function in the open ball $\{L \in \L(H^s(\Omega_\zeta)) \st \norm{L}_{\L(H^s)} < b\}$.  Thus, $F \circ T : \an^s(\Omega_\zeta) \to \L(H^s(\Omega_\zeta))$ defines a smooth function.  Since $\Gamma_1(f,g) = F(Tf) g$ we immediately deduce that $\Gamma_1$ is smooth on $B_{Y^s}(0, b/(2c)) \times H^s(\Omega_\zeta)$.

\end{proof}

In shifting from the Eulerian problem \eqref{traveling_euler} to the flattened problem \eqref{flattened_system} we employ the flattening map $\mathfrak{F}$ defined in terms of the free surface function via \eqref{flat_def}.  We thus require some information about the operators defined by composition with $\mathfrak{F}_\eta$ and its inverse.  We record this now.

\begin{theorem}\label{G_composition}
Let $n \ge 2$, $n/2 <k \in \N$, and $\eta \in \sp^{k+5/2}(\R^{n-1})$ be such that $\norm{\eta}_{C^0_b} \le b/2$.  Define $\mathfrak{G} : \Omega_{b+\eta} \to \Omega$ via $\mathfrak{G}(x) =(x', x_n b/(b+ \eta(x')))$.  Suppose that $V$ is a real finite dimensional inner-product space.  Then the following hold.
\begin{enumerate}
 \item $\mathfrak{G} \in C^r(\Omega_{b+\eta};\Omega)$ is a diffeomorphism for $r = 3+ \lfloor k-n/2 \rfloor \in \N$,  with inverse $\mathfrak{F}\in C^r(\Omega;\Omega_{b+\eta})$ defined by \eqref{flat_def}.  Moreover, $\nab \mathfrak{G} \in C^{r-1}_b(\Omega_{b+\eta};\R^{n\times n})$ and $\nab \mathfrak{F} \in C^{r-1}_b(\Omega;\R^{n\times n})$.
 
 \item If $0 \le s \le k+2$ and $f \in H^s(\Omega;V)$, then $f \circ \mathfrak{G} \in H^s(\Omega_{b+\eta};V)$.  Moreover, there exists a constant $c = c(n,s,k,\norm{\eta}_{\sp^{k+2}}) >0$ such that 
\begin{equation}
 \norm{f\circ \mathfrak{G}}_{H^s(\Omega_{b+\eta};V)} \le c  \norm{f }_{H^s(\Omega;V)},
\end{equation}
and the map $[0,\infty) \ni r \mapsto c(n,s,k,r) \in (0,\infty)$ is non-decreasing.

 \item If $0 \le s \le k+2$ and $f \in H^s(\Omega_{b+\eta};V)$, then $f \circ \mathfrak{F} \in H^s(\Omega;V)$.  Moreover, there exists a constant $c = c(n,s,k,\norm{\eta}_{\sp^{k+2}}) >0$ such that 
\begin{equation}
 \norm{f\circ \mathfrak{F}}_{H^s(\Omega;V)} \le c  \norm{f }_{H^s(\Omega_{b+\eta};V)},
\end{equation}
and the map $[0,\infty) \ni r \mapsto c(n,s,k,r) \in (0,\infty)$ is non-decreasing.

\end{enumerate}
\end{theorem}
\begin{proof}
We will prove only the results for $\mathfrak{G}$.  The corresponding results for $\mathfrak{F}$ follow from similar arguments.

Note that $k + 5/2 > 3 + \lfloor k - n/2 \rfloor + (n-1)/2$.  Then according to Theorem \ref{specialized_properties} we have that $\eta \in C^r_0(\R^{n-1})$, and from this we readily deduce that $\mathfrak{G} \in C^r(\Omega_{b+\eta};\Omega)$.  We compute 
\begin{equation}\label{G_composition_1}
 \nab \mathfrak{G}(x) - I = 
\begin{pmatrix}
0_{(n-1) \times (n-1)} & 0_{(n-1) \times 1} \\
-\frac{x_n b \nab' \eta(x')}{(b+\eta(x'))^2} & -\frac{\eta(x')}{b+\eta(x')}
\end{pmatrix},
\end{equation}
which shows that $\nab \mathfrak{G} \in C^{r-1}_b(\Omega_{b+\eta};\R^{n \times n})$ .  Clearly, $\mathfrak{G} = \mathfrak{F}^{-1}$ and $\mathfrak{F} \in C^r(\Omega;\Omega_{b+\eta})$ from a similar argument.  This proves the first item.

We now turn to the proof of the second item.  Suppose $f \in L^2(\Omega;V)$.  Then we use the change variables $x = \mathfrak{F}(y)$ and the identity $J = \det \nab \mathfrak{F}$, for $J$ defined in \eqref{JK_def}, to estimate
\begin{equation}
\norm{f\circ \mathfrak{G}}_{L^2}^2 = \int_{\Omega_{b+\eta}} \norm{f\circ \mathfrak{G}(x) }_V^2  dx = \int_{\Omega} \norm{f(y)}_V^2 \abs{\det \nab \mathfrak{F}(y)} dy \le \norm{J}_{L^\infty} \norm{f}_{L^2}^2.
\end{equation}
By hypothesis we have that
\begin{equation}
\norm{J}_{L^\infty} = \norm{1 + \frac{\eta}{b}}_{L^\infty} \le \frac{3}{2},
\end{equation}
and we deduce from this that 
\begin{equation}
\norm{f \circ \mathfrak{G}}_{L^2} \le \sqrt{\frac{3}{2}} \norm{f}_{L^2} \text{ for all }f \in L^2(\Omega;V). 
\end{equation}

Suppose now that for $0 \le m \le k+1$ there exists a constant $c = c(n,m,k,\norm{\eta}_{\sp^{k+2}}) >0$, which is non-decreasing in the last argument, such that 
\begin{equation}\label{G_composition_2}
\norm{f \circ \mathfrak{G}}_{H^m} \le c \norm{f}_{H^m} \text{ for all }f \in H^m(\Omega;V). 
\end{equation}
Consider $f \in H^{m+1}(\Omega;V)$.  For $1 \le j \le n$ we compute  
\begin{equation}
 \p_j( f \circ \mathfrak{G}) = \sum_{i=1}^n \p_i f \circ \mathfrak{G} \p_j \mathfrak{G}_i = \p_j f \circ \mathfrak{G} + \sum_{i=1}^n \p_i f \circ \mathfrak{G} (\p_j \mathfrak{G}_i - \delta_{ij}).
\end{equation}
From this, \eqref{G_composition_1}, \eqref{G_composition_2},   and Theorem \ref{omega_zeta_products} we may then bound 
\begin{equation}
 \norm{\p_j (f \circ \mathfrak{G})}_{H^m} \le c \norm{\p_j f}_{H^m} + c(1+ \norm{\eta}_{\sp^{k+2}}) \norm{\eta}_{\sp^{k+2}} \sum_{i=1}^n \norm{\p_i f \circ \mathfrak{G}}_{H^m} \le c \norm{f}_{H^{m+1}}
\end{equation}
for a new constant $c = c(n,m,k,\norm{\eta}_{\sp^{k+2}}) >0$ that is non-decreasing in the last argument.  Summing over $1 \le j \le n$ and again using the induction hypothesis \eqref{G_composition_2}, we conclude that there exists a constant $c = c(n,m+1,k,\norm{\eta}_{\sp^{k+2}}) >0$ that is again non-decreasing in the last argument such that 
\begin{equation}
\norm{f \circ \mathfrak{G}}_{H^{m+1}} \le c \norm{f}_{H^{m+1}} \text{ for all }f \in H^{m+1}(\Omega;V). 
\end{equation}
Proceeding with a finite induction then shows that for each $0 \le m \le k+2$ there exists a constant $c = c(n,m,k,\norm{\eta}_{\sp^{k+2}}) >0$ such that 
\begin{equation}
\norm{f \circ \mathfrak{G}}_{H^{m}} \le c \norm{f}_{H^{m}} \text{ for all }f \in H^{m}(\Omega;V). 
\end{equation}

We have now shown that the linear map $H^m(\Omega;V) \ni f \mapsto f\circ \mathfrak{G} \in H^m(\Omega_{b+\eta};V)$ is bounded for each $0 \le m \le k+2$ with operator norm bounded by a constant $c =  c(n,m,k,\norm{\eta}_{\sp^{k+2}})>0$.  The theory of operator interpolation then guarantees that this map is bounded from $H^s(\Omega;V)$ to $H^s(\Omega_{b+\eta};V)$ for every $0 \le s \le k+2$ and that the operator norm is a constant of the form   $c =  c(n,s,k,\norm{\eta}_{\sp^{k+2}})>0$.  This completes the proof of the second item.

\end{proof}

Theorem \ref{G_composition} tells us that under some assumptions on $\eta$, we can guarantee that $f \circ \mathfrak{F}_\eta \in H^s(\Omega;\R^n)$ whenever $f \in H^s(\Omega_{b+\eta};\R^n)$.  In our nonlinear analysis of \eqref{traveling_euler} we will need to show that a variant of this map is jointly $C^1$ in $\eta$ and $f$.  The complication with working directly with the composition in Theorem \ref{G_composition} is that in the theorem the function $f$ is defined on $\Omega_{b+\eta}$, a set that depends on $\eta$.  To avoid this technical complication, we instead investigate the continuous differentiability of the map $(f,\eta) \mapsto f\circ \mathfrak{E}_\eta$, where $f : \R^n \to \R^n$, i.e. $f$ is defined everywhere rather than just $\Omega_{b+\eta}$ and $\mathfrak{E}_\eta$ is a diffeomorphism that agrees with $\mathfrak{F}_\eta$ on $\bar{\Omega}$.  This is a variant of the ``$\omega-$lemma'' (see, for instance, Proposition 2.4.18 in \cite{AMR_1988} for a proof in $C^0$ spaces over compact topological spaces, and \cite{IKT_2013} for a proof in standard Sobolev spaces on $\R^n$ or manifolds) for the specialized spaces.  As in the standard $\omega-$lemma, we need to impose an extra order of regularity on the vector field in order to show the map is $C^1$.

Although the map $\mathfrak{F}_\eta$ from \eqref{flat_def} can be naturally extended as a map from $\R^n$ to $\R^n$, the unbounded term $x_n \eta(x')/b$ causes some technical problems in proving the $\omega-$lemma.  As such, we need to introduce a better behaved map $\mathfrak{E}_\eta : \R^n \to \R^n$ that agrees with $\mathfrak{F}_\eta$ on $\bar{\Omega}$.  We do this now.

\begin{proposition}\label{E_eta_properties}
Let $\psi \in C^\infty_c(\R)$ be such that $0 \le \psi \le 1$, $\psi =1$ on $[-2b,2b]$, and $\supp(\psi) \subset (-3b,3b)$.  Let $\varphi \in C^\infty_c(\R)$ be given by $\varphi(t) = t \psi(t)$.  Given $\eta \in \sp^{s+1/2}(\R^{n-1})$ define $\mathfrak{E}_\eta : \R^n \to \R^n$ via
\begin{equation}
 \mathfrak{E}_\eta(x) = (x', x_n + \varphi(x_n) \eta(x')/b).
\end{equation}
Then the following hold.
\begin{enumerate}
 \item  The map $\mathfrak{E}_\eta$ is Lipschitz and $C^1$, and there exists a constant $c = c(n,s,\varphi) >0$ such that 
\begin{equation}\label{E_eta_properties_00}
 \norm{\nab \mathfrak{E}_\eta - I}_{C^0_b} \le c \norm{\eta}_{\sp^{s+1/2}}.
\end{equation}

 \item If $V$ is a real finite dimensional inner-product space and  $0 \le r \le s$, then there exists a constant $c = c(n,r,s,V,\varphi) >0$ such that
\begin{equation}
\sup_{1\le j,k \le n} \norm{\p_j \mathfrak{E}_\eta \cdot e_k f  }_{H^r} \le c(1 + \norm{\eta}_{\sp^{s+1/2}}) \norm{f}_{H^r}
\end{equation}
and 
\begin{equation}
\sup_{1\le j,k \le n} \norm{(\p_j \mathfrak{E}_\eta \cdot e_k - \p_j \mathfrak{E}_\zeta \cdot e_k ) f  }_{H^r} \le c\norm{\eta-\zeta}_{\sp^{s+1/2}} \norm{f}_{H^r} 
\end{equation}
for every $\eta,\zeta \in \sp^{s+1/2}(\R^{n-1})$ and $f \in H^r(\R^n;V)$. 
 
\item There exists $0 < \delta_\ast <1$ such that if $\norm{\eta}_{\sp^{s+1/2}} < \delta_\ast$, then $\mathfrak{E}_\eta$ is a bi-Lipschitz homeomorphism and a $C^1$ diffeomorphism, and we have the estimate $\norm{\nab \mathfrak{E}_\eta - I}_{C^0_b} < 1/2$.
\end{enumerate}

\end{proposition}
\begin{proof}
Since $s > n/2 = (n-1)/2 + 1/2$, we have that $s + 1/2 > (n-1)/2 + 1$, so the first item follows from the fifth item of Theorem \ref{specialized_properties} and the formula 
\begin{equation}
 \nab \mathfrak{E}_\eta(x) = I + e_n \otimes ( \varphi(x_n) \nab'\eta(x') /b, \varphi'(x_n) \eta(x')/b).
\end{equation}
The second item follows from this formula, Lemma \ref{prod_full_space}, and Theorem \ref{Rn_specialized_product}.  We now turn to the proof of the third item.  First note that if the map $\R^n \ni x \mapsto (\varphi(x_n) \eta(x')/b) e_n$ is a contraction, then the Banach fixed point theorem readily implies that $\mathfrak{E}_\eta$ is a bi-Lipschitz homeomorphism.  On the other hand, the estimate \eqref{E_eta_properties_00} shows that if $\norm{\eta}_{\sp^{s+1/2}}$ is sufficiently small, then we can apply the inverse function theorem to see that $\mathfrak{E}_\eta$ is a local $C^1$ diffeomorphism in a neighborhood of every point.  Since the smallness of $\norm{\eta}_{\sp^{s+1/2}}$ can also be used to guarantee the smallness of $\norm{\eta}_{C^1_b}$, thanks to the fifth item of Theorem \ref{specialized_properties}, we then deduce the existence of a $0 < \delta_\ast <1$ satisfying the third item.
\end{proof}

Next we establish some essential continuity properties of the map we will use in the $\omega-$lemma.

\begin{proposition}\label{Lambda_continuity}
Let $n/2 < s \in \N$ and let $0 < \delta_\ast <1$ be as in the third item of Proposition \ref{E_eta_properties}.  Let $r \in \N$ be such that $0 \le r \le s$ and $V$ be a real finite dimensional inner-product space.  Consider the map $\Lambda : H^r(\R^n;V) \times B_{\sp^{s+1/2}(\R^{n-1})}(0,\delta_\ast) \to H^r(\R^n;V)$ given by 
\begin{equation}
 \Lambda(f,\eta) = f\circ \mathfrak{E}_\eta,
\end{equation}
where $\mathfrak{E}_\eta : \R^n \to \R^n$ is as defined in Proposition \ref{E_eta_properties}.  Then $\Lambda$ is well-defined and continuous, and there exists a constant $c = c(n,V,s,r,\varphi)>0$ (where $\varphi$ is as in the definition of $\mathfrak{E}_\eta$) such that  
\begin{equation}\label{Lambda_continuity_0}
\norm{\Lambda(f,\eta)}_{H^r} \le c (1+\norm{\eta}_{\sp^{s+1/2}} )^{1+r} \norm{f}_{H^r}.
\end{equation}
\end{proposition}
\begin{proof}
We proceed by finite induction on $0 \le r \le s$.

Suppose initially that $r=0$ and let $f,g \in H^0(\R^n;V) = L^2(\R^n;V)$ and $\eta,\zeta \in B_{\sp^{s+1/2}(\R^{n-1})}(0,\delta_\ast)$.  A change of variables shows that 
\begin{equation}
 \norm{f \circ \mathfrak{E}_\eta}_{H^0} =  \left( \int_{\R^n} \abs{ f \circ \mathfrak{E}_\eta}^2 \right)^{1/2} \le \norm{\det \nab \mathfrak{E}_\eta^{-1} }_{L^\infty}^{1/2} \norm{f}_{H^0},
\end{equation}
but the first and third items of Proposition \ref{E_eta_properties}, together with the fact that $\norm{\eta}_{\sp^{s+1/2}} < 1$, allow us to estimate 
\begin{equation}
 \norm{\det \nab \mathfrak{E}_\eta^{-1} }_{L^\infty}^{1/2} \le \norm{ (\nab \mathfrak{E}_\eta)^{-1} }_{L^\infty}^{n/2} \le c( 1+ \norm{\eta}_{\sp^{s+1/2}})^{n/2} \le  c( 1+ \norm{\eta}_{\sp^{s+1/2}}).
\end{equation}
Hence,
\begin{equation}
 \norm{f \circ \mathfrak{E}_\eta}_{H^0}  \le c(1+ \norm{\eta}_{\sp^{s+1/2}})\norm{f}_{H^0},
\end{equation}
which is \eqref{Lambda_continuity_0} with $r=0$.  Next we compute 
\begin{equation}
 \Lambda(f,\eta) - \Lambda(g,\zeta) = f\circ \mathfrak{E}_\eta - f \circ \mathfrak{E}_\zeta + (f-g) \circ \mathfrak{E}_\zeta,
\end{equation}
so that 
\begin{equation}
 \norm{ \Lambda(f,\eta) - \Lambda(g,\zeta)}_{H^0} \le \norm{f\circ \mathfrak{E}_\eta - f \circ \mathfrak{E}_\zeta}_{H^0} + \norm{(f-g) \circ \mathfrak{E}_\zeta}_{H^0}.
\end{equation}
Note that if $\zeta \to \eta$ in $\sp^{s+1/2}(\R^{n-1})$, then the fifth item of Theorem \ref{specialized_properties} implies that $\mathfrak{E}_\zeta \to \mathfrak{E}_\eta$ uniformly; this fact, together with the density of $C^\infty_c(\R^n;V)$ in $H^0(\R^n;V)$ and the dominated convergence convergence theorem, show that 
\begin{equation}
 \norm{f\circ \mathfrak{E}_\eta - f \circ \mathfrak{E}_\zeta}_{H^0} \to 0 \text{ as } \zeta \to \eta \text{ in } \sp^{s+1/2}(\R^{n-1}).
\end{equation}
On the other hand, \eqref{Lambda_continuity_0} with $r=0$ implies that 
\begin{equation}
\norm{(f-g) \circ \mathfrak{E}_\zeta}_{H^0} \to 0 \text{ as }  f \to g \text{ in } H^0(\R^n;V) \text{ and } \zeta \to \eta \text{ in } \sp^{s+1/2}(\R^{n-1}).
\end{equation}
Thus, the continuity assertion is proved, and the result is proved in the case $r=0$.

Suppose now that the result holds for $t \in \N$ such that $0 \le t \le r \le  s-1$  and consider the case $r+1\le s$.  Let $f,g \in H^{r+1}(\R^n;V)$ and $\eta,\zeta \in B_{\sp^{s+1/2}(\R^{n-1})}(0,\delta_\ast)$.  For $1 \le j \le n$ we have that $\p_j \Lambda(f,\eta) = \sum_{k=1}^n \p_k f\circ \mathfrak{E}_\eta \p_j \mathfrak{E}_\eta \cdot e_k$, and the induction hypothesis (applied to $\p_k f$ instead of $f$) implies that $\p_k f \circ \mathfrak{E}_\eta \in H^{r}(\R^n;V)$.  Thus, the induction hypothesis and Proposition \ref{E_eta_properties} show that
\begin{equation}
 \norm{\Lambda(f,\eta)}_{H^{r+1}} \le c( \norm{\Lambda(f,\eta)}_{H^0} + \sum_{j=1}^n \norm{\p_j \Lambda(f,\eta)}_{H^r}  )  \le c(1 + \norm{\eta}_{\sp^{s+1/2}})^{2+r} \norm{f}_{H^{r+1}},
\end{equation}
which is \eqref{Lambda_continuity_0} for $r+1$.   On the other hand, for $1 \le j \le n$ we also compute 
\begin{multline}
 \p_j\left[ \Lambda(f,\eta) - \Lambda(g,\zeta) \right] = \sum_{k=1}^n   \left(\p_k f\circ \mathfrak{E}_\eta  - \p_k f \circ \mathfrak{E}_\zeta \right)\p_j \mathfrak{E}_\eta \cdot e_k  \\ 
 +  \sum_{k=1}^n  \p_k f \circ \mathfrak{E}_\zeta \left( \p_j \mathfrak{E}_\eta \cdot e_k - \p_j \mathfrak{E}_\zeta \cdot e_k \right)  
+ \sum_{k=1}^n  \p_k(f-g) \circ \mathfrak{E}_\zeta  \p_j \mathfrak{E}_\zeta \cdot e_k .
\end{multline}
Again using the induction hypothesis (applied to $\p_k f$) and Proposition \ref{E_eta_properties}, we readily deduce from this that 
\begin{equation}
 \norm{\Lambda(f,\eta) - \Lambda(g,\zeta)}_{H^{r+1}} \to 0 \text{ as }  f \to g \text{ in } H^{r+1}(\R^n;V) \text{ and } \zeta \to \eta \text{ in } \sp^{s+1/2}(\R^{n-1}).
\end{equation}
This proves continuity assertion, so the result is proved for $r+1$.  Proceeding by finite induction, we see that the result holds for all $0 \le r \le s$, as desired.
\end{proof}
  
We now have the tools needed to prove a version of the $\omega-$lemma.  Note that we need to impose higher regularity on the domain of $\Lambda$ in order to prove that it is $C^1$.

\begin{theorem}\label{Lambda_diff}
Let $n/2 < s \in \N$,  $0 < \delta_\ast <1$ be as in the third item of Proposition \ref{E_eta_properties}, and $V$ be a real finite dimensional inner-product space.  Consider the map $\Lambda : H^{s+1}(\R^n;V) \times B_{\sp^{s+1/2}(\R^{n-1})}(0,\delta_\ast) \to H^s(\R^n;V)$ given by 
\begin{equation}
 \Lambda(f,\eta) = f\circ \mathfrak{E}_\eta,
\end{equation}
where $\mathfrak{E}_\eta : \R^n \to \R^n$ is as defined in Proposition \ref{E_eta_properties}.  Then $\Lambda$ is $C^1$ and $D\Lambda(f,\eta)(g,\zeta) = \frac{\varphi}{b} (\p_n f \circ \mathfrak{E}_{\eta}) \zeta + g \circ \mathfrak{E}_{\eta}$.
\end{theorem}
\begin{proof}
Let $f,g \in H^{s+1}(\R^n;V)$ and $\eta \in B_{\sp^{s+1/2}(\R^{n-1})}(0,\delta_\ast)$.  Let $R>0$ be such that $B_{\sp^{s+1/2}(\R^{n-1})}(\eta,R) \subset B_{\sp^{s+1/2}(\R^{n-1})}(0,\delta_\ast)$ and consider $\zeta \in B_{\sp^{s+1/2}(\R^{n-1})}(\eta,R)$.  Define the map 
\begin{equation}
Q_{f,\eta} : H^{s+1}(\R^n;V) \times \sp^{s+1/2}(\R^{n-1}) \to H^s(\R^n;V) 
\end{equation}
via 
\begin{equation}
 Q_{f,\eta}(h,\vartheta) =  \frac{\varphi}{b} (\p_n f \circ \mathfrak{E}_{\eta} ) \vartheta + h \circ \mathfrak{E}_{\eta}.  
\end{equation}
This is obviously linear, provided that it is well-defined.  It is indeed well-defined and bounded due to Theorem \ref{Rn_specialized_product} and Proposition \ref{Lambda_continuity}, which show that 
\begin{multline}
\norm{ Q_{f,\eta}(h,\vartheta) }_{H^s} \le c \left( \norm{\p_n f \circ \mathfrak{E}_{\eta}}_{H^s} \norm{\vartheta}_{\sp^{s+1/2}}  + \norm{h \circ \mathfrak{E}_{\eta}}_{H^s} \right) \\ 
\le c\left(1+ \norm{\eta}_{\sp^{s+1/2}} \right)^{1+s} \left(1 + \norm{f}_{H^{s+1}} \right)  \left(\norm{h }_{H^s} +\norm{\vartheta}_{\sp^{s+1/2}} \right) \\
\le c\left(1+ \norm{\eta}_{\sp^{s+1/2}} \right)^{1+s} \left(1 + \norm{f}_{H^{s+1}} \right)  \left(\norm{h }_{H^{s+1}} +\norm{\vartheta}_{\sp^{s+1/2}} \right).
\end{multline}

We next claim that $\Lambda$ is differentiable at $(f,\eta)$ and $D\Lambda(f,\eta) = Q_{f,\eta}$. Since $s+1 > 1 + n/2$ we have that $f,g \in C^1_b(\R^n;V)$.  Then 
\begin{equation}
 \Lambda(f+g,\eta + \zeta) - \Lambda(f,\eta) - Q_{f,\eta}(g,\zeta) = [f \circ \mathfrak{E}_{\eta + \zeta} - f \circ \mathfrak{E}_{\eta} - \frac{\varphi}{b} (\p_n f  \circ \mathfrak{E}_{\eta}) \zeta] 
 +[g \circ \mathfrak{E}_{\eta+ \zeta} - g \circ \mathfrak{E}_{\eta}].
\end{equation}
We will handle each term on the right in turn.  For the first we use the fundamental theorem of calculus to write 
\begin{equation}
  f \circ \mathfrak{E}_{\eta + \zeta} - f \circ \mathfrak{E}_{\eta} - \frac{\varphi}{b} (\p_n f \circ \mathfrak{E}_{\eta})\zeta =  \frac{\varphi \zeta}{b}\int_0^1 [\p_n f \circ \mathfrak{E}_{\eta + t \zeta} - \p_n f \circ \mathfrak{E}_{\eta}  ] dt.
\end{equation}
Using Theorem \ref{Rn_specialized_product}, this allows us to estimate
\begin{equation}
 \norm{f \circ \mathfrak{E}_{\eta + \zeta} - f \circ \mathfrak{E}_{\eta} - \frac{\varphi}{b} (\p_n f  \circ \mathfrak{E}_{\eta})\zeta}_{H^s} \le c \norm{\zeta}_{\sp^{s+1/2}} \int_0^1 \norm{\p_n f \circ \mathfrak{E}_{\eta + t \zeta} - \p_n f \circ \mathfrak{E}_{\eta}}_{H^s} dt,
\end{equation}
and since $\eta + t \zeta \in B_{\sp^{s+1/2}}(0,\delta_\ast)$, Proposition \ref{Lambda_continuity} guarantees that 
\begin{equation}
 \lim_{(g,\zeta) \to 0} \frac{ \norm{f \circ \mathfrak{E}_{\eta + \zeta} - f \circ \mathfrak{E}_{\eta} - \frac{\varphi}{b} (\p_n f \circ  \mathfrak{E}_{\eta})\zeta}_{H^s}}{\norm{g}_{H^{s+1}} + \norm{\zeta}_{\sp^{s+1/2}}} =0.
\end{equation}
Similarly, we can again use the fundamental theorem of calculus, Theorem \ref{Rn_specialized_product}, and Proposition \ref{Lambda_continuity} to see that 
\begin{multline}
\norm{g \circ \mathfrak{E}_{\eta+ \zeta} - g \circ \mathfrak{E}_{\eta}}_{H^s} = \norm{\frac{\varphi \zeta}{b}  \int_0^1 \p_n g \circ \mathfrak{E}_{\eta + t \zeta} dt}_{H^s} \le c \norm{\zeta}_{\sp^{s+1/2}}   \int_0^1\norm{   \p_n g \circ \mathfrak{E}_{\eta + t \zeta}}_{H^s}dt \\
\le c \norm{\zeta}_{\sp^{s+1/2}} \norm{g}_{H^{s+1}}  \int_0^1 \left(1 + \norm{\eta + t \zeta}_{\sp^{s+1/2}} \right)^{s+1} dt \le c (1+ \delta_\ast)^{s+1} \norm{\zeta}_{\sp^{s+1/2}} \norm{g}_{H^{s+1}}, 
\end{multline}
and hence 
\begin{equation}
  \lim_{(g,\zeta) \to 0}\frac{\norm{g \circ \mathfrak{E}_{\eta+ \zeta} - g \circ \mathfrak{E}_{\eta}}_{H^s}}{\norm{g}_{H^{s+1}} + \norm{\zeta}_{\sp^{s+1/2}}} =0.
\end{equation}
Combining these then completes the proof of the claim.

To show that $\Lambda$ is $C^1$ it remains only to prove that the map 
\begin{equation}
D\Lambda :  H^{s+1}(\R^n;V) \times B_{\sp^{s+1/2}(\R^{n-1})}(0,\delta_\ast) \to \L( H^{s+1}(\R^n;V) \times \sp^{s+1/2}(\R^{n-1}) ; H^s(\R^n;V))  
\end{equation}
is continuous.  We compute 
\begin{multline}
 D\Lambda(f,\eta)(h,\vartheta) - D\Lambda(g,\zeta)(h,\vartheta) = \frac{\varphi \vartheta}{b} \left(\p_n f \circ \mathfrak{E}_\eta - \p_n f \circ \mathfrak{E}_\zeta \right) \\
+  \frac{\varphi \vartheta}{b} \left(\p_n f - \p_n g \right) \circ \mathfrak{E}_\zeta + (h \circ \mathfrak{E}_\eta - h \circ \mathfrak{E}_\zeta ).
\end{multline}
Again using Theorem \ref{Rn_specialized_product} and Proposition \ref{Lambda_continuity} we may estimate 
\begin{equation}
\norm{\frac{\varphi \vartheta}{b} \left(\p_n f \circ \mathfrak{E}_\eta - \p_n f \circ \mathfrak{E}_\zeta \right)}_{H^s} \le c \norm{\vartheta}_{\sp^{s+1/2}} \norm{\p_n f \circ \mathfrak{E}_\eta - \p_n f \circ \mathfrak{E}_\zeta}_{H^s},
\end{equation}
\begin{equation}
 \norm{ \frac{\varphi \vartheta}{b} \left(\p_n f - \p_n g \right) \circ \mathfrak{E}_\zeta }_{H^s} \le c \norm{\vartheta}_{\sp^{s+1/2}} (1 + \norm{\zeta}_{\sp^{s+1/2}})^{1+s}  \norm{\p_n f - \p_n g}_{H^s} 
\end{equation}
and (also using the fundamental theorem of calculus)
\begin{multline}
 \norm{h \circ \mathfrak{E}_\eta - h \circ \mathfrak{E}_\zeta}_{H^s} = \norm{\frac{\varphi (\eta - \zeta)}{b} \int_0^1 \p_n h \circ \mathfrak{E}_{t \eta + (1-t) \zeta} dt }_{H^s} \\
 \le c \norm{\eta - \zeta}_{\sp^{s+1/2}} \norm{  \int_0^1 \p_n h \circ \mathfrak{E}_{t \eta + (1-t) \zeta}dt }_{H^s}
  \le c \norm{\eta - \zeta}_{\sp^{s+1/2}}  \int_0^1 \norm{\p_n h \circ \mathfrak{E}_{t \eta + (1-t) \zeta} }_{H^s} dt \\
 \le c \norm{\eta - \zeta}_{\sp^{s+1/2}} \norm{\p_n h }_{H^s}  \int_0^1 (1 + \norm{t\eta - (1-t)\zeta}_{\sp^{s+1/2}})^{1+s}  dt   \le c (1+ \delta_\ast)^{1+s}\norm{\eta - \zeta}_{\sp^{s+1/2}} \norm{h }_{H^{s+1}}.
\end{multline}
Hence, we may bound the operator norm via
\begin{equation}
 \norm{D \Lambda(f,\eta) - D \Lambda(g,\zeta)  }_{\L} \le c \left(\norm{\p_n f \circ \mathfrak{E}_\eta - \p_n f \circ \mathfrak{E}_\zeta}_{H^s} + \norm{ f -  g}_{H^{s+1}} + \norm{\eta - \zeta}_{\sp^{s+1/2}} \right),
\end{equation}
and the continuity of $D\Lambda$ then follows from this estimate and Proposition \ref{Lambda_continuity}.  Thus $\Lambda$ is $C^1$.
\end{proof}

Our final result gives another version of the $\omega-$lemma for the original flattening map given in \eqref{flat_def}.

\begin{corollary}\label{Lambda_Omega_diff}
Let $n/2 < s \in \N$, $0 < \delta_\ast <1$ be as in the third item of Proposition \ref{E_eta_properties}, and $V$ be a real finite dimensional inner-product space.  For $\eta \in \sp^{s+1/2}(\R^{n-1})$ define $\mathfrak{F}_\eta : \Omega \to \Omega_{b+\eta}$ as in \eqref{flat_def}.  Then the following hold.
\begin{enumerate}
 \item The map $\Lambda_\Omega : H^{s+1}(\R^n;V) \times B_{\sp^{s+1/2}(\R^{n-1})}(0,\delta_\ast) \to H^s(\Omega;V)$ given by 
\begin{equation}
 \Lambda_\Omega(f,\eta) = f\circ \mathfrak{F}_\eta
\end{equation}
is well-defined and $C^1$, with $D\Lambda_\Omega(f,\eta)(g,\zeta) = \frac{\varphi}{b} (\p_n f \circ \mathfrak{F}_{\eta}) \zeta + g \circ \mathfrak{F}_{\eta}$.

\item The map $\mathfrak{S}_b :  H^{s+2}(\R^n ; V) \times B_{\sp^{s+3/2}(\R^{n-1})}(0,\delta_\ast)  \to  H^{s+1/2}(\Sigma_b ; V)$ given by
\begin{equation}
 \mathfrak{S}_b (f,\eta) = f \circ \mathfrak{F}_\eta \vert_{\Sigma_b}
\end{equation}
is well-defined and $C^1$, with $D\mathfrak{S}_b(f,\eta)(g,\zeta) = \left(\frac{\varphi}{b} (\p_n f \circ \mathfrak{F}_{\eta}) \zeta + g \circ \mathfrak{F}_{\eta} \right)\vert_{\Sigma_b}$.
\end{enumerate}
\end{corollary}
\begin{proof}
Let $\mathfrak{E}_\eta: \R^n \to \R^n$ be as in Proposition \ref{E_eta_properties}.  By construction, we have that $\mathfrak{F}_\eta$ equals the restriction of $\mathfrak{E}_\eta$ to $\Omega$.  Then $\Lambda_\Omega = R_\Omega \circ \Lambda$, where $\Lambda$ is as in Theorem \ref{Lambda_diff} and $R_\Omega : H^{s}(\R^n;V) \to H^s(\Omega;V)$ is the bounded linear restriction map $R_\Omega g = g\vert_\Omega$. Theorem \ref{Lambda_diff} shows that $\Lambda$ is $C^1$, and $R_\Omega$ is linear, and hence smooth, so $\Lambda_\Omega$ is $C^1$ by the chain rule.  This proves the first item, and the second follows from the first (applied with $s+1$ in place of $s$) and the fact that the trace operator from $H^{s+1}(\Omega;V)$ to $H^{s+1/2}(\Sigma_b;V)$ is bounded and linear, and hence smooth.
\end{proof}

\section{The $\gamma-$Stokes equations with traveling gravity-capillary boundary conditions}\label{sec_gravity_capillary}
 
In this section we turn our attention to the $\gamma-$Stokes problem with boundary conditions that couple the stress tensor to the linearized gravity-capillary operator.  That is, we seek solution triples $(u,p,\eta)$ to the problem
\begin{equation} \label{problem_gamma_stokes_stress_capillary}
\begin{cases}
\diverge S(p,u) -\gamma\partial_{1}u=f & \text{in }\Omega \\
\diverge u=g & \text{in }\Omega \\
S(p,u)e_{n}  -(\eta -\sigma \Delta' \eta) e_n  =k,\quad u_{n}+\gamma\partial_{1}\eta=h & \text{on }\Sigma_b \\
u=0 & \text{on }\Sigma_{0},
\end{cases}
\end{equation}
for given data $(f,g,h,k) \in \mathcal{Y}^s$, as defined in \eqref{Ys_def}.  In order to solve this problem for data in $\mathcal{Y}^s$ we will employ the specialized Sobolev spaces $\sp^s(\R^{n-1})$ and $\an^s(\Omega)$ introduced in Section \ref{sec_specialized_sobolev}.

\subsection{Preliminaries}

We begin our analysis by studying the mapping defined by the problem \eqref{problem_gamma_stokes_stress_capillary}, with the aim being to find a domain space for the triple $(u,p,\eta)$ that yields a well-defined linear map into $\mathcal{Y}^s$.  We begin with two crucial lemmas that establish key properties of some auxiliary functions.

The first lemma studies a function defined in terms of the function $m$ from \eqref{QVm_def}.

\begin{lemma}\label{rho_lemma}
Let $n \ge 2$ and $m: \R^{n-1} \times \R \to \C$ be given by \eqref{QVm_def}.   Let $\gamma \in \R \backslash \{0\}$,  $\sigma \ge 0$, and define $\rho : \R^{n-1} \to \C$ via $\rho(\xi) = 2\pi i \gamma \xi_1 + (1+ 4\pi^2 \abs{\xi}^2 \sigma) \overline{m(\xi,-\gamma)}$.  Then the following hold.
\begin{enumerate}
 \item $\rho$ is continuous, and it is also smooth when restricted to $\R^{n-1}\backslash \{0\}$.
 \item $\RE{\rho} \le 0$, and  $\RE{\rho(\xi)}=0$ if and only if $\xi =0$.  In particular, $\rho(\xi)=0$ if and only if $\xi=0$.  
 \item $\overline{\rho(\xi)} = \rho(-\xi)$ for all $\xi \in \R^{n-1}$.
 \item For $\sigma >0$ there exists a constant $c = c(n,\gamma,\sigma,b) >0$ such that 
\begin{equation}\label{rho_lemma_01}
 \frac{1}{c} \frac{\abs{\rho(\xi)}^2}{\abs{\xi}^2}    \le   \frac{\xi_1^2 + \abs{\xi}^4}{\abs{\xi}^2} \le c \frac{\abs{\rho(\xi)}^2}{\abs{\xi}^2}  
\end{equation}
for $\abs{\xi} \le 1$, and 
\begin{equation}\label{rho_lemma_02}
 \frac{1}{c} \abs{\rho(\xi)}^2    \le 1+ \abs{\xi}^2  \le c  \abs{\rho(\xi)}^2    
\end{equation}
for $\abs{\xi} \ge 1$.

  \item For $\sigma =0$ and $n =2$ there exists a constant $c = c(\gamma,b) >0$ such that 
\begin{equation}\label{rho_lemma__03}
 \frac{1}{c} \frac{\abs{\rho(\xi)}^2}{\abs{\xi}^2}    \le  1+ \abs{\xi}^2 \le c \frac{\abs{\rho(\xi)}^2}{\abs{\xi}^2}  
\end{equation}
for $\abs{\xi} \le 1$, and 
\begin{equation}\label{rho_lemma_04}
 \frac{1}{c} \abs{\rho(\xi)}^2    \le 1+ \abs{\xi}^2  \le c  \abs{\rho(\xi)}^2    
\end{equation}
for $\abs{\xi} \ge 1$.

\end{enumerate}
\end{lemma}
\begin{proof}
The first item follows from the continuity of $m(\cdot,-\gamma)$ and its smoothness away from the origin, which was proved in Theorem \ref{QVm_properties}.  Clearly $\rho(0)=0$, but $\RE{\rho(\xi)} = (1+4\pi^2 \abs{\xi}^2 \sigma) \RE{m(\xi,-\gamma)} < 0$ for $\xi \neq 0$, thanks again to Theorem \ref{QVm_properties}.  This proves the second item.  The third item follows from the third item of Theorem \ref{QVm_properties}: 
\begin{equation}
\overline{\rho(\xi)} = 
-2\pi i \gamma \xi_1 + (1+ 4\pi^2 \abs{\xi}^2 \sigma) m(\xi,-\gamma)
=-2\pi i \gamma \xi_1 + (1+ 4\pi^2 \abs{\xi}^2 \sigma) \overline{m(-\xi,-\gamma)} = \rho(-\xi).  
\end{equation}

We now turn to the proof of the fourth item.   According to Theorems \ref{QVm_zero} and \ref{QVm_infty} we have the asymptotic developments  
\begin{equation}
  m(\xi,-\gamma)  = -\frac{4\pi^2 \abs{\xi}^2 b^3}{3} + O(\abs{\xi}^3) \text{ as } \xi \to 0 \text{ and }m(\xi,-\gamma) =- \frac{1}{4 \pi \abs{\xi}} +O\left(\frac{1}{\abs{\xi}^2}\right) \text{ as } \abs{\xi} \to \infty.
\end{equation}
Thus 
\begin{equation}
\begin{split}
 \rho(\xi) &= -\frac{4\pi^2 \abs{\xi}^2 b^3}{3} + 2\pi i \gamma \xi_1 + O(\abs{\xi}^3) \text{ as } \xi \to 0, \text{ and} \\ \rho(\xi) &= - \frac{1+ 4\pi^2 \abs{\xi}^2 \sigma}{4 \pi \abs{\xi}} + 2 \pi i \gamma \xi_1    +O\left(\frac{1}{\abs{\xi}^2}\right) \text{ as } \abs{\xi} \to \infty,
\end{split}
\end{equation}
and so we we can pick constants $c = c(n,\gamma,\sigma,b)>0$ and $R = R(n,\gamma,\sigma,b)>1$ such that \eqref{rho_lemma_01} holds for $\abs{\xi} \le 1/R$ and \eqref{rho_lemma_02} holds for $\abs{\xi} \ge R$.  However, on the compact set $\{x \in \R^{n-1} \st 1/R \le \abs{\xi} \le R\}$ the quantities in the middle of \eqref{rho_lemma_01} and \eqref{rho_lemma_02} cannot vanish, nor can $\abs{\rho}$  by the second item.  Hence, the middle and outer quantities are equivalent on this compact set, and so upon possibly enlarging $c$, we conclude that \eqref{rho_lemma_01} holds for $\abs{\xi} \le 1$ and  \eqref{rho_lemma_02} holds for $\abs{\xi} \ge 1$, which completes the proof of the fourth item.  The fifth item follows from a similar argument.

\end{proof}

The second lemma studies an auxiliary function defined in terms of $Q$ and $V$ from \eqref{QVm_def}.

\begin{lemma}\label{psi_integral_bounds}
Let $\gamma \in \R$, $s \ge 0$, and let  $Q(\cdot,\cdot,-\gamma): \R^{n-1} \times [0,b] \to \C$ and $V(\cdot,\cdot,-\gamma) : \R^{n-1} \times [0,b] \to \C^n$  be as defined in \eqref{QVm_def}.  There exists a constant $c = c(n,s,\gamma,b)>0$ such that if $(f,g,h,k) \in \mathcal{Y}^s$, where $\mathcal{Y}^s$ is the Hilbert space defined in \eqref{Ys_def}, and we define the measurable function $\psi: \R^{n-1} \to \C$ via 
\begin{equation}
\psi(\xi)=
\int_{0}^{b}\left(\hat{f}(\xi,x_{n})\cdot\overline{V(\xi,x_{n},-\gamma)} - \hat{g}(\xi,x_{n})  \overline{Q(\xi,x_{n},-\gamma)} \right ) dx_{n} - \hat{k}(\xi) \cdot \overline{V(\xi,b,-\gamma)}+\hat{h}(\xi),
\end{equation}
then
\begin{equation}\label{psi_integral_bounds_0}
 \int_{B'(0,1)} \frac{1}{\abs{\xi}^2} \abs{\psi(\xi)}^2 d\xi + \int_{B'(0,1)^c} (1+\abs{\xi}^2)^{s+3/2}  \abs{\psi(\xi)}^2 d\xi \le c \norm{(f,g,h,k)}_{\mathcal{Y}^s}^2.
\end{equation}
Moreover, the function $\psi$ satisfies $\overline{\psi(\xi)} = \psi(-\xi)$ for every $\xi \in \R^{n-1}$.

\end{lemma}
\begin{proof}
For $\abs{\xi} \le 1$ we regroup the sum defining $\psi$ via  
\begin{multline}
\psi(\xi)=
\int_{0}^{b}\left(\hat{f}(\xi,x_{n})\cdot\overline{V(\xi,x_{n},-\gamma)} - \hat{g}(\xi,x_{n})  \left(\overline{Q(\xi,x_{n},-\gamma)} -1 \right) \right ) dx_{n} \\
- \hat{k}(\xi) \cdot \overline{V(\xi,b,-\gamma)} 
+ \left(\hat{h}(\xi) - \int_0^b \hat{g}(\xi,x_n) dx_n \right)
\end{multline} 
and then apply the Cauchy-Schwarz inequality and square to arrive at the estimate 
\begin{multline}
\abs{\psi(\xi)}^2 \le 4 \left( \int_0^b \abs{\hat{f}(\xi,x_n)}^2 d\xi  \right) \left( \int_0^b \abs{V(\xi,x_n)}^2 d\xi \right)  + 4\abs{\hat{k}(\xi)}^2 \abs{V(\xi,b,-\gamma)}^2 \\
+ 4\left( \int_0^b \abs{\hat{g}(\xi,x_n,-\gamma)}^2 d\xi  \right) \left( \int_0^b \abs{Q(\xi,x_n,-\gamma)-1}^2 d\xi \right) 
 + 4\abs{ \hat{h}(\xi) - \int_0^b \hat{g}(\xi,x_n) dx_n}^2. 
\end{multline}
Using the continuity of $V$ and $Q$, as proved in Theorem \ref{QVm_properties}, together with the asymptotic developments as $\xi \to 0$ from Theorem \ref{QVm_zero}, we find that there exists a constant $c = c(n,b,\gamma)>0$ such that 
\begin{equation}
\sup_{\abs{\xi} \le 1}  \frac{1}{\abs{\xi}^2}  \left( \int_0^b \left( \abs{V(\xi,x_n,-\gamma)}^2 +  \abs{Q(\xi,x_n,-\gamma) -1}^2\right) dx_n + \abs{V(\xi,b,-\gamma)}^2 \right) \le c.
\end{equation}
On the other hand, from the definition of $\dot{H}^{-1}$ (see \eqref{homogeneous_def}) we have that
\begin{equation}
\int_{B'(0,1)} \frac{1}{\abs{\xi}^2} \abs{\hat{h}(\xi) - \int_0^b \hat{g}(\xi,x_n) dx_n}^2 d\xi \le \snorm{h - \int_0^b g(\cdot,x_n) dx_n }_{\dot{H}^{-1}}^2.
\end{equation}
Combining these and employing the Tonelli and Parseval theorems, we deduce that
\begin{multline}\label{psi_integral_bounds_1}
 \int_{B'(0,1)}  \frac{1}{\abs{\xi}^2} \abs{\psi(\xi)}^2 d\xi  \le c  \int_{B'(0,1)}\int_0^b  \left( \abs{\hat{f}(\xi,x_n)}^2  +  \abs{\hat{g}(\xi,x_n)}^2 \right) dx_n d\xi + c \int_{B'(0,1)} \abs{\hat{k}(\xi)}^2 d\xi  \\
 + c \snorm{h - \int_0^b g(\cdot,x_n) dx_n }_{\dot{H}^{-1}}^2 \le c \left( \norm{f}_{L^2}^2 +   \norm{g}_{L^2}^2 + \norm{k}_{L^2}^2 +  \snorm{h - \int_0^b g(\cdot,x_n) dx_n }_{\dot{H}^{-1}}^2 \right) \\
 \le c  \norm{(f,g,h,k)}_{\mathcal{Y}^s}^2
\end{multline}
for another constant $c = c(n,\gamma,b)>0$.

For $\abs{\xi} \ge 1$ we don't regroup but use Cauchy-Schwarz again to bound 
\begin{multline}
\abs{\psi(\xi)}^2 \le 4 \left( \int_0^b \abs{\hat{f}(\xi,x_n)}^2 d\xi  \right) \left( \int_0^b \abs{V(\xi,x_n,-\gamma)}^2 d\xi \right) + 4\abs{\hat{k}(\xi)}^2 \abs{V(\xi,b,-\gamma)}^2 
\\
+ 4\left( \int_0^b \abs{\hat{g}(\xi,x_n)}^2 d\xi  \right) \left( \int_0^b \abs{Q(\xi,x_n,-\gamma)}^2 d\xi \right) 
+ 4\abs{ \hat{h}(\xi)}^2. 
\end{multline}
From Theorem \ref{QVm_infty}, Corollary \ref{QVm_integrals},  and the continuity of $V$ we deduce that there is a constant  $c = c(n,b,\gamma)>0$ such that
\begin{equation}
(1+ \abs{\xi}^2)^{3/2} \int_0^b \abs{V(\xi,x_n,-\gamma)}^2 dx_n + (1+ \abs{\xi}^2)^{1/2} \int_0^b \abs{Q(\xi,x_n,-\gamma)}^2 dx_n  + (1+\abs{\xi}^2) \abs{V(\xi,b,-\gamma)}^2 \le c 
\end{equation}
for all $\xi \in \R^{n-1}$.   Combining these, and again using the Tonelli and Parseval theorems, as well as Corollary \ref{omega_slicing_bound}, we deduce that 
\begin{multline}\label{psi_integral_bounds_2}
 \int_{B'(0,1)^c} (1+\abs{\xi}^2)^{s+3/2}  \abs{\psi(\xi)}^2 d\xi \le 
 c \int_0^b \int_{\R^{n-1}} (1+\abs{\xi}^2)^s\abs{\hat{f}(\xi,x_n)}^2 d\xi dx_n \\
 +  c \int_0^b \int_{\R^{n-1}} (1+\abs{\xi}^2)^{s+1} \abs{\hat{g}(\xi,x_n)}^2 d\xi dx_n 
 + \int_{\R^{n-1}} (1+\abs{\xi}^2)^{s+1/2} \abs{\hat{k}(\xi)}^2 d\xi 
 + \int_{\R^{n-1}} (1+\abs{\xi}^2)^{s+3/2} \abs{\hat{h}(\xi)}^2 d\xi \\
\le c \int_0^b \norm{f(\cdot,x_n)}_{H^s(\R^{n-1})}^2 dx_n + c \int_0^b \norm{g(\cdot,x_n)}_{H^{s+1}(\R^{n-1})}^2 dx_n 
+ c \norm{k}_{H^{s+1/2}}^2 + c \norm{h}_{H^{s+3/2}}^2 \\
\le c \left( \norm{f}_{H^{s}}^2 + \norm{g}_{H^{s+1}}^2 + \norm{k}_{H^{s+1/2}}^2 + \norm{h}_{H^{s+3/2}}^2 \right) \le
c  \norm{(f,g,h,k)}_{\mathcal{Y}^s}^2.
\end{multline}

Then \eqref{psi_integral_bounds_0} follows by summing \eqref{psi_integral_bounds_1} and \eqref{psi_integral_bounds_2}.  To conclude we use Lemma \ref{tempered_real_lemma} and the third item of Theorem \ref{QVm_properties} to compute
\begin{multline}
\overline{\psi(\xi)} =
\int_{0}^{b}\left( \overline{\hat{f}(\xi,x_{n})} \cdot V(\xi,x_{n},-\gamma) - \overline{\hat{g}(\xi,x_{n})}  Q(\xi,x_{n},-\gamma) \right ) dx_{n} - \overline{\hat{k}(\xi)} \cdot V(\xi,b,-\gamma)+ \overline{\hat{h}(\xi)} \\
=
\int_{0}^{b}\left(\hat{f}(-\xi,x_{n})\cdot\overline{V(-\xi,x_{n},-\gamma)} - \hat{g}(-\xi,x_{n})  \overline{Q(-\xi,x_{n},-\gamma)} \right ) dx_{n} - \hat{k}(-\xi) \cdot \overline{V(-\xi,b,-\gamma)}+\hat{h}(-\xi) \\
= \psi(-\xi).
\end{multline}

\end{proof}

With these lemmas in hand, we now turn to the question of defining the domain of the map determined by \eqref{problem_gamma_stokes_stress_capillary}.  For $s \ge 0$ we define the space 
\begin{equation}\label{Xs_def}
 \mathcal{X}^s = \{(u,p,\eta) \in {_{0}}H^{s+2}(\Omega;\mathbb{R}^{n}) \times \an^{s+1}(\Omega) \times \sp^{s+5/2}(\R^{n-1})  \st p - \eta \in H^{s+1}(\Omega) \},
\end{equation}
where here the condition $p-\eta \in H^{s+1}(\Omega)$ is understood in the sense of the definition of the space $\an^{s+1}(\Omega)$ from \eqref{an_space_def} and is well-defined due to the inclusion  $\sp^{s+5/2}(\R^{n-1}) \subseteq \sp^{s+1}(\R^{n-1})$ from Theorem \ref{specialized_properties}.  We endow $\mathcal{X}^s$ with the norm 
\begin{equation}
 \norm{(u,p,\eta)}_{\mathcal{X}^s} = \norm{u}_{{_{0}}H^{s+2}} + \norm{p}_{\an^{s+1}} +  \norm{\eta}_{\sp^{s+5/2}} + \norm{p-\eta}_{H^{s+1}}.
\end{equation}
It is a simple matter to verify that $\mathcal{X}^s$ is a Banach space.  Moreover, we have the following embedding result.

\begin{proposition}\label{Xs_embed}
Let $\mathcal{X}^s$ be defined by \eqref{Xs_def} and suppose $s >n/2$.  Then we have the continuous inclusion
\begin{equation}
 \mathcal{X}^s \subseteq C^{2 + \lfloor s-n/2 \rfloor}_b(\Omega;\R^n) \times  C^{1 + \lfloor s-n/2 \rfloor}_b(\Omega) \times C^{3 + \lfloor s-n/2 \rfloor}_0(\R^{n-1}), 
\end{equation}
where $C^k_b$ and $C^k_0$ are defined in Section \ref{sec_notation}.   Moreover, if $(u,p,\eta)\in \mathcal{X}^s$, then
\begin{equation}
\begin{split}
 \lim_{\abs{x'} \to \infty} \p^\alpha u(x) &= 0 \text{ for all } \alpha \in \N^n \text{ such that} \abs{\alpha} \le 2 + \lfloor s-n/2 \rfloor, \text{ and }\\ 
  \lim_{\abs{x'} \to \infty} \p^\alpha p(x) &= 0 \text{ for all } \alpha \in \N^n \text{ such that} \abs{\alpha} \le 1 + \lfloor s-n/2 \rfloor.
\end{split}
\end{equation}

\end{proposition}
\begin{proof}
This follows from the usual Sobolev embedding, the third item of Theorem \ref{omega_aniso_properties}, and the fifth item of Theorem \ref{specialized_properties}.
\end{proof}

The next result shows that the map $(u,p,\eta) \mapsto (f,g,h,k)$ defined by \eqref{problem_gamma_stokes_stress_capillary} is well-defined from $\mathcal{X}^s$ to $\mathcal{Y}^s$.

\begin{lemma}\label{upsilon_well_defd}
Let $s \ge 0$ and suppose that $(u,p,\eta) \in \mathcal{X}^s$, where $\mathcal{X}^s$ is the Banach space defined by \eqref{Xs_def}.  Define $f: \Omega \to \R^n$, $g : \Omega \to \R$, $h: \Sigma_b \to \R$, and $k : \Sigma_b \to \R^n$ via 
$f = \diverge S(p,u) - \gamma \p_1 u$, $g = \diverge u$, $h = u_n + \gamma \p_1 \eta$, and $k = S(p,u) e_n - (\eta - \sigma \Delta'\eta)e_n$.  Then $(f,g,h,k) \in \mathcal{Y}^s$, where $\mathcal{Y}^s$ is the Hilbert space defined in \eqref{Ys_def}, and there exists a constant $c = c(\Omega,s,\gamma,\sigma) >0$ such that 
\begin{equation}\label{upsilon_well_defd_0}
\norm{(f,g,h,k)}_{\mathcal{Y}^s} \le c \norm{(u,p,\eta)}_{\mathcal{X}^s}. 
\end{equation}
\end{lemma}
\begin{proof}
Theorem \ref{omega_aniso_properties} shows that $\nab : \an^{s+1}(\Omega) \to H^s(\Omega;\R^n)$ is a bounded linear map, so according to \eqref{stress_div} we have that $f \in H^s(\Omega;\R^n)$ and 
\begin{equation}
 \norm{f}_{H^s} \le c \norm{u}_{{_{0}}H^{s+2}} + c\norm{p}_{\an^{s+1}}.
\end{equation}
Clearly, $g \in H^{s+1}(\Omega)$ and $\norm{g}_{H^{s+1}} \le c \norm{u}_{{_{0}}H^{s+2}}$.    Then Theorem \ref{cc_divergence} shows that we have the inclusion $u_n(\cdot,b) - \int_0^b g(\cdot,x_n) dx_n \in \dot{H}^{-1}$ and 
\begin{equation}
 \snorm{u_n(\cdot,b) - \int_0^b g(\cdot,x_n) dx_n) }_{\dot{H}^{-1}} \le c \norm{u}_{L^2 }.
\end{equation}
Theorem \ref{specialized_properties} shows that the linear maps $\p_1 : \sp^{s+5/2}(\Sigma_b) \to H^{s+3/2}(\Sigma_b) \cap \dot{H}^{-1}(\Sigma_b)$ and $\Delta' : \sp^{s+5/2}(\Sigma_b) \to H^{s+1/2}(\Sigma_b)$ are bounded.  These and standard trace theory then show that we have the inclusions $h \in  H^{s+3/2}(\Sigma_b)$ and $h - \int_0^b g(\cdot,x_n) dx_n \in \dot{H}^{-1}$, and that we have the bounds
\begin{equation}
 \norm{h}_{H^{s+3/2}}  \le c \norm{u}_{{_{0}}H^{s+2}}  + c \norm{\eta}_{\sp^{s+5/2}},
\end{equation}
\begin{equation}
 \snorm{h - \int_0^b g(\cdot,x_n) dx_n) }_{\dot{H}^{-1}} \le 
 \snorm{u_n(\cdot,b) - \int_0^b g(\cdot,x_n) dx_n) }_{\dot{H}^{-1}} + \snorm{\gamma \p_1 \eta }_{\dot{H}^{-1}}  
 \le   c \norm{\eta}_{\sp^{s+5/2}} + c \norm{u}_{L^2 }.
\end{equation}
Finally, we again use the above, the inclusion $p - \eta \in H^{s+1}(\Omega)$, and trace theory to see that  $k \in H^{s+1/2}(\Sigma_b;\R^n)$ 
and 
\begin{multline}
 \norm{k}_{H^{s+1/2}} \le \norm{p-\eta}_{H^{s+1/2}(\Sigma_b)} + \norm{\sigma \Delta' \eta}_{H^{s+1/2}(\Sigma_b)} + \norm{\mathbb{D} u }_{H^{s+1/2}(\Sigma_b)} \\
 \le c \norm{p-\eta}_{H^{s+1}(\Omega)} + c \norm{\eta}_{\sp^{s+5/2}} + c \norm{u}_{{_{0}}H^{s+2}}.
\end{multline}
Synthesizing these shows that $(f,g,h,k) \in \mathcal{Y}^s$ and that \eqref{upsilon_well_defd_0} holds.

\end{proof}

For $\gamma \in \R \backslash \{0\}$, $\sigma \ge 0$, $s \ge 0$, and  $\mathcal{X}^s$ and $\mathcal{Y}^s$ the Banach spaces defined by \eqref{Xs_def} and \eqref{Ys_def}, respectively, we define the operator $\Upsilon_{\gamma,\sigma} : \mathcal{X}^s \to \mathcal{Y}^s$ via    
\begin{equation}\label{Upsilon_def}
\Upsilon_{\gamma,\sigma}(u,p,\eta) =
( \diverge{S(p,u)} - \gamma \p_1 u, 
  \diverge{u}, 
  \left. u_{n}\right\vert_{\Sigma_b} + \gamma \p_1 \eta,
  \left. S(p,u)e_{n}\right\vert_{\Sigma_b} -(\eta-\sigma \Delta' \eta)e_n).
\end{equation}
This is well-defined and bounded by virtue of Lemma \ref{upsilon_well_defd}.  The map $\Upsilon_{\gamma,\sigma}$ is injective, as we now prove.

\begin{proposition}\label{Upsilon_injective}
Assume that $\gamma \in \R \backslash \{0\}$, $\sigma \ge 0$, $s \ge 0$, and let $\mathcal{X}^s$ and $\mathcal{Y}^s$ be the Banach spaces defined by \eqref{Xs_def} and \eqref{Ys_def}, respectively.  Then the bounded linear operator $\Upsilon_{\gamma,\sigma} : \mathcal{X}^s \to \mathcal{Y}^s$ defined by \eqref{Upsilon_def}   is injective. 
\end{proposition}
\begin{proof}
Suppose that $(u,p,\eta) \in \mathcal{X}^s$ and $\Upsilon_{\gamma,\sigma}(u,p,\eta)=0$, which is equivalent to the problem 
\begin{equation}\label{iso_stokes_capillary_1}
\begin{cases}
\diverge{S(p,u)} -\gamma\partial_{1}u=0 & \text{in }\Omega \\
\diverge{u}=0 & \text{in }\Omega \\
S(p,u)e_{n} = (\eta - \sigma \Delta' \eta) e_{n},\quad u_{n} + \gamma\partial_{1}\eta=0 & \text{on }\Sigma_b \\
u=0 & \text{on }\Sigma_{0}.
\end{cases}
\end{equation}
Since $s \ge 0$, when we apply the horizontal Fourier transform we find from the Tonelli and Parseval theorems that $\hat{u}(\xi,\cdot) \in H^2((0,b);\C^n)$ for almost every $\xi \in \R^{n-1}$.  We also know from Proposition \ref{omega_aniso_fourier} that $\hat{p}(\xi,x_n) \in H^1((0,b);\C)$ for almost every $\xi \in \R^{n-1}$.  Furthermore, Lemma \ref{localization_lemma} shows that $\hat{\eta} \in L^1(\R^{n-1}) + L^2(\R^{n-1};(1+\abs{\xi}^2)^{(s+5/2)/2} d\xi)$. We may thus apply the horizontal Fourier transform to \eqref{iso_stokes_capillary_1} to deduce that for almost every $\xi \in \R^{n-1}$ the pair  $w := \hat{u}(\xi,\cdot)$,  $q:=\hat{p}(\xi,\cdot)$ satisfies \eqref{ODE_int_unique_01} with $F=0,$ $G=0$, and $K = (1 + 4\pi^2 \abs{\xi}^2 \sigma   )\hat{\eta}(\xi) e_n$, and that $\hat{u}_n(\xi,b) + 2\pi i \gamma \xi_1 \hat{\eta}(\xi) =0$.  Fix one of these almost every $\xi \in \R^{n-1} \backslash \{0\}$.   Using \eqref{ODE_int_unique_02} from Proposition \ref{ODE_int_unique} with $v = w$, we find that 
\begin{multline}
 \int_0^b -\gamma 2\pi i \xi_1  \abs{w}^2  + 2 \abs{\p_n w_n}^2 + \abs{\p_n w' + 2\pi i \xi w_n}^2 
+ \hal  \abs{2\pi i \xi \otimes w' + w' \otimes 2\pi i \xi}^2 \\
= - (1 + 4\pi^2 \abs{\xi}^2 \sigma   )\hat{\eta}(\xi) \overline{\hat{u}_n(\xi,b)} 
= -2\pi i \gamma \xi_1  (1 + 4\pi^2 \abs{\xi}^2 \sigma   ) \abs{\hat{\eta}(\xi)}^2.
\end{multline}
Taking the real part of this expression then shows that $\p_n w_n =0$ and $\p_n w' + 2\pi i \xi w_n =0$ in $(0,b)$, but $w(0)=0$, so $w =0$ in $[0,b]$.  Since $w=0$ and $\xi \neq 0$, the first equation in \eqref{ODE_int_unique_01} then shows that $q=0$ in $[0,b]$, but then the fifth equation in \eqref{ODE_int_unique_01} requires that $0 = q(b) - 2\p_n w_n(b) = K_n =  (1 + 4\pi^2 \abs{\xi}^2 \sigma   )\hat{\eta}(\xi)$, and we find that $\hat{\eta}(\xi) =0$.  We have thus proved that for almost every $\xi \in \R^{n-1}$ we have that $\hat{u}(\xi,\cdot) =0$, $\hat{p}(\xi,\cdot)=0$, and $\hat{\eta}(\xi) =0$, which then implies that $u =0$, $p=0$, and $\eta =0$ and hence that $\Upsilon_{\gamma,\sigma}$ is injective.

\end{proof}

\subsection{Solvability of \eqref{problem_gamma_stokes_stress_capillary} when $\sigma >0$}

We are now ready to completely characterize the solvability of \eqref{problem_gamma_stokes_stress_capillary} for data belonging to $\mathcal{Y}^s$ in the case of positive surface tension, i.e. $\sigma >0$.

\begin{theorem}\label{iso_stokes_capillary}
Assume that $\gamma \in \R \backslash \{0\}$, $\sigma >0$, $s \ge 0$, and let $\mathcal{X}^s$ and $\mathcal{Y}^s$ be the Banach spaces defined by \eqref{Xs_def} and \eqref{Ys_def}, respectively.  Then the bounded linear operator $\Upsilon_{\gamma,\sigma} : \mathcal{X}^s \to \mathcal{Y}^s$ defined by \eqref{Upsilon_def}   is an isomorphism.
\end{theorem}
\begin{proof}
Proposition \ref{Upsilon_injective} established that the map $\Upsilon_{\gamma,\sigma}$ is injective, so we must only prove that it is surjective.  Fix $(f,g,h,k) \in \mathcal{Y}^s$, let $\psi: \R^{n-1} \to \C$ be defined in terms of $(f,g,h,k)$ as in Lemma \ref{psi_integral_bounds}, and consider $\rho : \R^{n-1} \to \C$ given by $\rho(\xi) = 2\pi i \gamma \xi_1 + (1+ 4\pi^2 \abs{\xi}^2 \sigma) \overline{m(\xi)}$ as in Lemma \ref{rho_lemma}.  Lemma \ref{rho_lemma} tells us that $\rho(\xi) =0$ if and only if $\xi =0$, so we may define $\hat{\eta} : \R^{n-1} \to \C$  via $\hat{\eta}(\xi) = \psi(\xi) / \rho(\xi)$ for $\xi \neq 0$ and $\hat{\eta}(0) =0$.   Note that $\overline{\hat{\eta}(\xi)} = \hat{\eta}(-\xi)$ due to the corresponding properties of $\rho$ and $\psi$, as established in Lemmas \ref{rho_lemma} and \ref{psi_integral_bounds}.  Then from Lemmas \ref{rho_lemma} and \ref{psi_integral_bounds} we find a constant $c = c(n,\gamma,b,s,\sigma)>0$ such that
\begin{multline}
 \int_{B'(0,1)} \frac{\xi_1^2 + \abs{\xi}^4}{\abs{\xi}^2} \abs{\hat{\eta}(\xi)}^2 d\xi +   \int_{B'(0,1)^c} (1+ \abs{\xi}^2)^{s+5/2} \abs{\hat{\eta}(\xi)}^2 d\xi
\\
\le 
c \int_{B'(0,1)} \frac{\abs{\rho(\xi)}^2}{\abs{\xi}^2} \abs{\hat{\eta}(\xi)}^2 d\xi +  c \int_{B'(0,1)^c} (1+ \abs{\xi}^2)^{s+3/2} \abs{\rho(\xi)}^2 \abs{\hat{\eta}(\xi)}^2 d\xi \\
\le 
c\int_{B'(0,1)} \frac{1}{\abs{\xi}^2} \abs{\psi(\xi)}^2 d\xi + c \int_{B'(0,1)^c} (1+\abs{\xi}^2)^{s+3/2}  \abs{\psi(\xi)}^2 d\xi \le c \norm{(f,g,h,k)}_{\mathcal{Y}^s}^2.
\end{multline}
Consequently, we may define $\eta = (\hat{\eta})^{\vee} \in \sp^{s+5/2}(\R^{n-1})$; the above estimate then says that $\norm{\eta}_{\sp^s} \le c \norm{(f,g,h,k)}_{\mathcal{Y}^s}$.  

Next we recall from \eqref{QVm_def} that $m(\xi) = V_n(\xi,b)$.  Using this and the definitions of $\rho$ and $\psi$, we  rearrange the identity $\rho \hat{\eta} = \psi$ to find that 
\begin{multline}\label{iso_stokes_capillary_2}
0=
\int_{0}^{b}\left(\hat{f}(\xi,x_{n})\cdot\overline{V(\xi,x_{n},-\gamma)} - \hat{g}(\xi,x_{n})  \overline{Q(\xi,x_{n},-\gamma)} \right ) dx_{n} 
\\
-  \left( \hat{k}(\xi)  +(1 + 4 \pi^2 \abs{\xi}^2 \sigma)\hat{\eta}(\xi) e_n   ) \right)\cdot \overline{V(\xi,b,-\gamma)}+  \hat{h}(\xi) - 2\pi i \gamma \xi_1  \hat{\eta}(\xi)  
\end{multline}
for $\xi \in \R^{n-1} \backslash \{0\}$.  From the third equation in \eqref{QVm_properties_0} from Theorem \ref{QVm_properties} we know that $2\pi i \xi \cdot V'(\xi,x_n,-\gamma) + \p_n V_n(\xi,x_n,-\gamma) = 0$.  Since $V(\xi,0,-\gamma)=0$, this allows us to compute 
\begin{equation}\label{iso_stokes_capillary_3}
\overline{V_n(\xi,b,-\gamma)} = \int_0^b \p_n \overline{V_n(\xi,x_n,-\gamma)} dx_n =  \int_0^b  2\pi i \xi \cdot \overline{V'(\xi,x_n,-\gamma)} dx_n.
\end{equation}
Combining \eqref{iso_stokes_capillary_2} and \eqref{iso_stokes_capillary_3}, we deduce that 
\begin{multline}\label{iso_stokes_capillary_4}
0=
\int_{0}^{b}\left(\left( \hat{f}(\xi,x_{n}) - (2\pi i \xi \hat{\eta}(\xi),0)  \right)\cdot\overline{V(\xi,x_{n},-\gamma)} - \hat{g}(\xi,x_{n})  \overline{Q(\xi,x_{n},-\gamma)} \right ) dx_{n} 
\\
-  \left( \hat{k}(\xi)  +  4 \pi^2 \abs{\xi}^2 \sigma \hat{\eta}(\xi)  e_n   ) \right)\cdot \overline{V(\xi,b,-\gamma)}+  \hat{h}(\xi) - 2\pi i \gamma \xi_1  \hat{\eta}(\xi)  .
\end{multline}

Since $\eta \in \sp^{s+5/2}(\R^{n-1})$, Theorem \ref{specialized_properties} guarantees that $(-\nab' \eta,0) \in H^{s+3/2}(\Omega;\R^n) \subset H^{s}(\Omega;\R^n)$, $\Delta' \eta \in H^{s+1/2}(\R^{n-1})$, and $\p_1 \eta \in H^{s+3/2}(\R^{n-1}) \cap \dot{H}^{-1}(\R^{n-1})$.  Consequently,  $f -(\nab' \eta,0) \in H^s(\Omega;\R^n)$, $g \in H^{s+1}(\Omega)$, $h - \gamma \p_1 \eta \in H^{s+3/2}(\Sigma_b)$,  $k - \sigma \Delta' \eta e_n \in H^{s+1/2}(\Sigma_b;\R^n)$, and $(h - \gamma \p_1 \eta) - \int_0^b g(\cdot,x_n) dx_n \in \dot{H}^{-1}(\R^{n-1})$.  We have that  $\widehat{\nab' \eta}(\xi) = 2 \pi i \xi \hat{\eta}(\xi)$ and $\widehat{\sigma \Delta' \eta}(\xi) = -4\pi^2 \abs{\xi}^2 \sigma \hat{\eta}(\xi)$, so the identity \eqref{iso_stokes_capillary_4} and Proposition  \ref{proposition_cc_fourier} imply that the modified data quadruple $(f-(\nab'\eta,0), g, h- \gamma \p_1 \eta, k - \sigma \Delta' \eta e_n)$ satisfy the compatibility condition \eqref{compatibility_condition} and hence belongs to the Hilbert space $\mathcal{Z}^s$, as defined in \eqref{Zs_def}.  Thus, we may apply Theorem \ref{iso_overdetermined} to find a unique pair $u \in {_{0}}H^{s+2}(\Omega;\mathbb{R}^{n})$ and $q \in H^{s+1}(\Omega)$ solving 
\begin{equation}
\begin{cases}
 \diverge S(q,u) - \gamma \p_1 u = f - (\nab' \eta,0) &\text{in }\Omega \\
 \diverge u = g &\text{in }\Omega \\
 S(q,u) e_n = k -  \sigma \Delta' \eta e_n   & \text{on }\Sigma_b   \\ 
 u_n = h - \gamma \p_1 \eta &\text{on } \Sigma_b \\
 u =0 & \text{on } \Sigma_0.
\end{cases}
\end{equation}
Since $S(q,u) = qI - \mathbb{D}u$ and $\diverge S(q,u) = -\Delta u - \nab \diverge{u} + \nab q$, we may then define $p = q + \eta \in \an^{s+1}(\Omega)$ and deduce that $(u,p,\eta) \in \mathcal{X}^s$ satisfies 
\begin{equation}
\begin{cases}
 \diverge S(p,u) - \gamma \p_1 u = f  &\text{in }\Omega \\
 \diverge u = g &\text{in }\Omega \\
 S(p,u) e_n - ( \eta -  \sigma \Delta' \eta) e_n = k     & \text{on }\Sigma_b   \\ 
 u_n + \gamma \p_1 \eta= h  &\text{on } \Sigma_b \\
 u =0 & \text{on } \Sigma_0.
\end{cases}
\end{equation}
Hence $\Upsilon_{\gamma,\sigma}(u,p,\eta) = (f,g,h,k)$, and we conclude that $\Upsilon_{\gamma,\sigma}$ is an isomorphism.

\end{proof}

\subsection{Solvability of \eqref{problem_gamma_stokes_stress_capillary} when $\sigma =0$ and $n=2$}

We now turn our attention to the solvability of \eqref{problem_gamma_stokes_stress_capillary} in the case without surface tension, i.e. $\sigma =0$.  Due to technical obstructions, we must restrict to the dimension $n=2$.  In this case, for $s \ge 0$ Proposition \ref{specialized_inclusion} and Remark \ref{omega_aniso_remark} imply that the Banach space $\mathcal{X}^s$ defined by \eqref{Xs_def} satisfies the algebraic identity
\begin{equation}\label{Xs_2equiv}
 \mathcal{X}^s =  {_{0}}H^{s+2}(\Omega;\mathbb{R}^{2}) \times H^{s+1}(\Omega) \times H^{s+5/2}(\R)
\end{equation}
and that we have the norm equivalence 
\begin{equation}
 \norm{(u,p,\eta)}_{\mathcal{X}^s} \asymp  \norm{u}_{{_{0}}H^{s+2}} + \norm{p}_{H^{s+1}} +  \norm{\eta}_{H^{s+5/2}}
\end{equation}
for $(u,p,\eta) \in \mathcal{X}^s$.  In particular, this means that in this case $\mathcal{X}^s$ possesses an equivalent Hilbert topology.

The following characterizes the solvability of \eqref{problem_gamma_stokes_stress_capillary} when $\sigma =0$ and $n=2$.

\begin{theorem}\label{iso_stokes_capillary_zero}
Assume that $n=2$, $\gamma \in \R \backslash \{0\}$, $\sigma =0$, $s \ge 0$, and let $\mathcal{X}^s$ and $\mathcal{Y}^s$ be the Banach spaces defined by \eqref{Xs_def} and \eqref{Ys_def}, respectively.  Then the bounded linear operator $\Upsilon_{\gamma,0} : \mathcal{X}^s \to \mathcal{Y}^s$ defined by \eqref{Upsilon_def} is an isomorphism.
\end{theorem}
\begin{proof}
Again, we know from Proposition \ref{Upsilon_injective} that $\Upsilon_{\gamma,0}$ is injective, so we must only establish surjectivity.

Fix $(f,g,h,k) \in \mathcal{Y}^s$, let $\psi: \R \to \C$ be defined in terms of $(f,g,h,k)$ as in Lemma \ref{psi_integral_bounds}, and consider $\rho : \R \to \C$ given by $\rho(\xi) = 2\pi i \gamma \xi+ \overline{m(\xi)}$ as in Lemma \ref{rho_lemma} with $\sigma =0$.  Arguing as in the proof of Theorem \ref{iso_stokes_capillary}, we may define $\hat{\eta} : \R \to \C$  via $\hat{\eta}(\xi) = \psi(\xi) / \rho(\xi)$ for $\xi \neq 0$ and $\hat{\eta}(0) =0$, and we have that  $\overline{\hat{\eta}(\xi)} = \hat{\eta}(-\xi)$.  Moreover, thanks to  Lemmas \ref{rho_lemma} and \ref{psi_integral_bounds} there exists a constant $c = c(\gamma,b,s)>0$ such that
\begin{multline}
 \int_{B'(0,1)} (1+\abs{\xi}^2)  \abs{\hat{\eta}(\xi)}^2 d\xi +   \int_{B'(0,1)^c} (1+ \abs{\xi}^2)^{s+5/2} \abs{\hat{\eta}(\xi)}^2 d\xi
\\
\le 
c \int_{B'(0,1)} \frac{\abs{\rho(\xi)}^2}{\abs{\xi}^2} \abs{\hat{\eta}(\xi)}^2 d\xi +  c \int_{B'(0,1)^c} (1+ \abs{\xi}^2)^{s+3/2} \abs{\rho(\xi)}^2 \abs{\hat{\eta}(\xi)}^2 d\xi \\
\le 
c\int_{B'(0,1)} \frac{1}{\abs{\xi}^2} \abs{\psi(\xi)}^2 d\xi + c \int_{B'(0,1)^c} (1+\abs{\xi}^2)^{s+3/2}  \abs{\psi(\xi)}^2 d\xi \le c \norm{(f,g,h,k)}_{\mathcal{Y}^s}^2.
\end{multline}
Consequently, we may define $\eta = (\hat{\eta})^{\vee} \in H^{s+5/2}(\R) = \sp^{s+5/2}(\R)$ (recall that the latter identity was established in Proposition \ref{specialized_inclusion}); the above estimate then says that $\norm{\eta}_{\sp^s} \le c \norm{(f,g,h,k)}_{\mathcal{Y}^s}$.  

Next we argue as in the proof of Theorem \ref{iso_stokes_capillary} to see that 
\begin{multline}
0=
\int_{0}^{b}\left(\hat{f}(\xi,x_{2})\cdot\overline{V(\xi,x_{2},-\gamma)} - \hat{g}(\xi,x_{2})  \overline{Q(\xi,x_{2},-\gamma)} \right ) dx_{2} 
\\
-  \left( \hat{k}(\xi)  + \hat{\eta}(\xi) e_2    \right)\cdot \overline{V(\xi,b,-\gamma)}+  \hat{h}(\xi) - 2\pi i \gamma \xi   \hat{\eta}(\xi) 
\end{multline}
for $\xi \in \R \backslash \{0\}$, and that we have the inclusions $h - \gamma \p_1 \eta \in H^{s+3/2}(\Sigma_b)$,  $k + \eta e_2 \in H^{s+1/2}(\Sigma_b;\R^2)$, and $(h - \gamma \p_1 \eta) - \int_0^b g(\cdot,x_2) dx_2 \in \dot{H}^{-1}(\R)$.   These and Proposition  \ref{proposition_cc_fourier} imply that the modified data quadruple $(f, g, h- \gamma \p_1 \eta, k +\eta e_2)$ satisfy the compatibility condition \eqref{compatibility_condition} and hence belongs to the Hilbert space $\mathcal{Z}^s$, as defined in \eqref{Zs_def}.  Thus, we may apply Theorem \ref{iso_overdetermined} to find a unique pair $u \in {_{0}}H^{s+2}(\Omega;\mathbb{R}^{2})$ and $p \in H^{s+1}(\Omega)$ solving 
\begin{equation}
\begin{cases}
 \diverge S(p,u) - \gamma \p_1 u = f &\text{in }\Omega \\
 \diverge u = g &\text{in }\Omega \\
 S(p,u) e_n = k + \eta e_2   & \text{on }\Sigma_b   \\ 
 u_2 = h - \gamma \p_1 \eta &\text{on } \Sigma_b \\
 u =0 & \text{on } \Sigma_0.
\end{cases}
\end{equation}
Hence $\Upsilon_{\gamma,0}(u,p,\eta) = (f,g,h,k)$, and we conclude that $\Upsilon_{\gamma,0}$ is an isomorphism.

\end{proof}

\section{The $\gamma-$Stokes problem with Navier boundary conditions}\label{sec_navier_bcs}

In this section we make a brief digression from the main line of analysis to discuss how a variant of the technique from the previous section can be used to study the $\gamma-$Stokes problem with Navier boundary conditions:
\begin{equation} \label{problem_gamma_stokes_navier}
\begin{cases}
\diverge S(p,u)-\gamma\partial_{1}u=f & \text{in }\Omega \\
\diverge u=g & \text{in }\Omega \\
(S(p,u)e_{n})' =k',\quad u_{n}=h & \text{on }\Sigma_b \\
u=0 & \text{on }\Sigma_{0},
\end{cases}
\end{equation}
for given $s \ge 0$ and $f \in H^s(\Omega;\R^n)$, $g \in H^{s+1}(\Omega)$, $h \in H^{s+3/2}(\Sigma_b)$, and $k' \in H^{s+1/2}(\Sigma_b;\R^{n-1})$.  Note that 
\begin{equation}
 (S(p,u)e_{n})' = (p e_n)' - (\sg u e_n)' = - \p_n u' - \nab' u_n,
\end{equation}
so the boundary conditions on $\Sigma_b$ are equivalent to 
\begin{equation}\label{problem_gamma_stokes_navier_bc2}
-\p_n u' - \nab' u_n = k' \text{ and } u_n = h, \text{ or equivalently }
-\p_n u'  = k' + \nab' h \text{ and } u_n = h.
\end{equation}

\subsection{A function space for the data in \eqref{problem_gamma_stokes_navier} }

The data for which we can build a good solvability theory for \eqref{problem_gamma_stokes_navier} require another function space that encodes a new compatibility condition.  In order to define this we have to introduce an auxiliary function associated to a data quadruple $(f,g,h,k')$.   Given $\gamma \in \R$,  $f \in L^2(\Omega;\R^n)$, $g \in L^2(\Omega)$, $h \in L^2(\R^{n-1})$, and $k' \in L^2(\R^{n-1};\R^{n-1})$ we define $\hat{W} :\R^{n-1} \to \C$ via 
\begin{equation}\label{W_fn_def}
 \hat{W}(f,g,h,k')(\xi) = \int_{0}^{b}(\hat{f}(\xi,x_{n})\cdot\overline{V(\xi,x_{n}, -\gamma)}-\hat{g}(\xi,x_{n}) \overline{Q(\xi,x_{n}, -\gamma )})dx_{n} - \hat{k}'(\xi)\cdot \overline{V'(\xi,b,-\gamma)} + 
\hat{h}(\xi),
\end{equation}
where $V$, $Q$, and $m$ are as defined in \eqref{QVm_def}.  By virtue of Theorems \ref{QVm_properties} and \ref{QVm_infty}, Corollary \ref{QVm_integrals}, Lemma \ref{tempered_real_lemma}, and the Cauchy-Schwarz inequality, we may argue as in the proof of Lemma \ref{psi_integral_bounds} to see that $\hat{W} \in L^2(\R^{n-1})$ and $\overline{\hat{W}(\xi)} = \hat{W}(\xi)$.  As such, $W = (\hat{W})^\vee \in L^2(\R^{n-1})$ is well-defined and real-valued.  Moreover, there exists a constant $c = c(n,\gamma,b) >0$ such that 
\begin{equation}
 \norm{W(f,g,h,k')}_{L^2} \le c \left(\norm{f}_{L^2} + \norm{g}_{L^2} + \norm{h}_{L^2} + \norm{k'}_{L^2}  \right).
\end{equation}

In order to study the problem \eqref{problem_gamma_stokes_navier} we must introduce the following variant of the space $\mathcal{Y}^s$ from \eqref{Ys_def}.  For $s \ge 0$ we define 
\begin{multline}\label{Ws_def}
 \mathcal{W}^s = \{(f,g,h,k') \in H^s(\Omega; \R^n) \times H^{s+1}(\Omega) \times H^{s+3/2}(\Sigma_b) \times H^{s+1/2}(\Sigma_b;\R^{n-1}) \st \\ 
 h\text{ and }g \text{ satisfy }  \eqref{H minus 1}   
\text{ and } W(f,g,h,k') \in \dot{H}^{-2}(\R^{n-1}) \},
\end{multline}
and we endow $\mathcal{W}^s$ with the norm
\begin{multline}
 \norm{(f,g,h,k')}_{\mathcal{W}^s}^2 = \norm{f}_{H^s}^2 + \norm{g}_{H^{s+1}}^2 + \norm{h}_{H^{s+3/2}}^2 + \norm{k'}_{H^{s+1/2}}^2 \\
 + \snorm{h - \int_0^b g(\cdot,x_n) dx_n}_{\dot{H}^{-1}}^2 + \snorm{W(f,g,h,k')}_{\dot{H}^{-2}}^2.
\end{multline}
This clearly makes $\mathcal{W}^s$ into a Hilbert space (with the obvious inner-product associated to the norm).

\subsection{The isomorphism}

We now characterize the solvability of \eqref{problem_gamma_stokes_navier}.

\begin{theorem}\label{iso_gamma_stokes_navier}
Let $\gamma\in\mathbb{R}$, $s \ge 0$, and $\mathcal{W}^s$ be the Hilbert space defined in \eqref{Ws_def}.  Then the  bounded linear operator $ \Theta_\gamma :  {_{0}}H^{s+2}(\Omega;\mathbb{R}^{n})\times H^{s+1}(\Omega)  
\rightarrow  \mathcal{W}^s$ given by 
\begin{equation}
 \Theta_\gamma(u,p) = (\diverge{S(p,u)} - \gamma \p_1 u, \diverge{u}, u_n \vert_{\Sigma_b},    \left(S(p,u)e_{n}\vert _{\Sigma_b} \right)'  )
\end{equation}
is an isomorphism.
\end{theorem}
\begin{proof}
Let $(u,p) \in {_{0}}H^{s+2}(\Omega;\mathbb{R}^{n})\times H^{s+1}(\Omega)$ and define $(f,g,h,k') = \Theta_\gamma(u,p)$.  Define also $k_n \in H^{s+1/2}(\Sigma_b)$ via $k_n = (p - 2 \p_n u_n)\vert_{\Sigma_b} = S(p,u) e_n \cdot e_n \vert_{\Sigma_b}$.  We may then define $k = (k',k_n) \in H^{s+1/2}(\Sigma_b;\R^n)$.  Then $\Psi_\gamma(u,p) = (f,g,h,k) \in \mathcal{Z}^s$ by Theorem \ref{iso_overdetermined}.  Consequently, Proposition \ref{proposition_cc_fourier} tells us that 
\begin{equation}
\hat{W}(f,g,h,k')(\xi) = \hat{k}_n(\xi)  \overline{m(\xi,-\gamma)} 
\end{equation}
for almost every $\xi\in\mathbb{R}^{n-1}$, where $W$ is defined by \eqref{W_fn_def}.  According to Theorems \ref{QVm_properties}, \ref{QVm_zero}, and \ref{QVm_infty}, we may thus bound 
\begin{equation}
 \snorm{W(f,g,h,k'}_{\dot{H}^{-2}}^2 =   \int_{\R^{n-1}} \frac{1}{\abs{\xi}^4} \abs{\hat{W}(f,g,h,k')(\xi)}^2 d\xi \le c \int_{\R^{n-1}} \abs{\hat{k}_n(\xi)}^2 d\xi \le c \norm{u}_{{_{0}}H^{2}}^2 + c \norm{p}_{H^1}^2.
\end{equation}
From this and the definition of $\mathcal{Z}^s$ we deduce that $\Theta_\gamma$ defines a bounded linear operator taking values in $\mathcal{W}^s$.

We now prove that $\Theta_\gamma$ is injective.  Suppose that $\Theta_{\gamma} (u,p) =(0,0,0,0)$ and define $k_n$ as above.  Then $\Psi_{\gamma}(u,p) = (0,0,0,k_n e_n)$, and so Theorem \ref{iso_overdetermined} and Proposition \ref{proposition_cc_fourier} tell us that 
\begin{equation}
 0 = \hat{W}(f,g,h,k')(\xi) = \hat{k}_n(\xi)  \overline{m(\xi,-\gamma)}. 
\end{equation}
Since Theorem \ref{QVm_properties} guarantees that $m(\xi,-\gamma) =0$ if and only $\xi=0$, we deduce from this that $k_n =0$, and hence $\Psi_{\gamma}(u,p) = (0,0,0,0)$.  Then Theorem \ref{iso_overdetermined} implies that $u=0$ and $p=0$.  Hence $\Theta_{\gamma}$ is injective.

We now turn to the proof that $\Theta_{\gamma}$ is surjective.  Let $(f,g,h,k') \in \mathcal{W}^s$.  Define $\hat{\psi} : \R^{n-1} \to \C$ via $\hat{\psi}(0) =0$ and 
\begin{equation}
 \hat{\psi}(\xi) = \frac{1}{\overline{m(\xi,-\gamma)}} \hat{W}(f,g,h,k')(\xi) \text{ for } \xi \neq 0.
\end{equation}
From the properties of $Q$, $V$, and $m$ from Theorem \ref{QVm_properties} we have that $\overline{\hat{\psi}(\xi)} = \hat{\psi}(-\xi)$.  Also, according to Theorems \ref{QVm_properties}, \ref{QVm_zero}, and \ref{QVm_infty} and Corollary \ref{QVm_integrals}, we have that 
\begin{multline}
\int_{B'(0,1)} \abs{\hat{\psi}(\xi)}^2 d\xi \le  \sup_{\xi \in B'(0,1)} \frac{\abs{\xi}^4}{\abs{m(\xi,-\gamma)}^2}    \int_{B'(0,1)} \frac{1}{\abs{\xi}^4} \abs{\hat{W}(f,g,h,k')(\xi)}^2 d\xi  \\
\le c \int_{\R^{n-1}} \frac{1}{\abs{\xi}^4} \abs{\hat{W}(f,g,h,k')(\xi)}^2 d\xi 
= c \snorm{W(f,g,h,k')}_{\dot{H}^{-2}}^2 \le c \norm{(f,g,h,k')}_{\mathcal{W}^s}^2
\end{multline}
and (also using the Cauchy-Schwarz inequality and the Tonelli and Parseval theorems)
\begin{multline}
\int_{B'(0,1)^c} (1+\abs{\xi}^2)^{s+1/2} \abs{\hat{\psi}(\xi)}^2 d\xi \le c\int_{\R^{n-1}} \int_0^b \left( (1+\abs{\xi}^2)^s \abs{\hat{f}(\xi,x_n)}^2 + (1+\abs{\xi}^2)^{s+1} \abs{\hat{g}(\xi,x_n)}^2  \right) dx_n d\xi  \\
+c \int_{\R^{n-1}}  \left((1+\abs{\xi}^2)^{s-1/2} \abs{\hat{k}'(\xi)}^2 + (1+\abs{\xi}^2)^{s+3/2} \abs{\hat{h}(\xi)}^2     \right) d\xi  \\
\le c \left(\norm{f}_{H^s}^2 + \norm{g}_{H^{s+1}}^2 + \norm{k'}_{H^{s-1/2}}^2 + \norm{h}_{H^{s+3/2}}^2    \right) \le c \norm{(f,g,h,k')}_{\mathcal{W}^s}^2.
\end{multline}
Hence, we may define $\psi = (\hat{\psi})^\vee \in H^{s+1/2}(\R^{n-1})$ and note that $\psi$ is real-valued and satisfies $\hat{\psi}(\xi) \overline{m(\xi,-\gamma)} = \hat{W}(f,g,h,k')(\xi)$ for almost every $\xi \in \R^{n-1}$.  Rearranging this identity and again using Proposition \ref{proposition_cc_fourier}, we find that if we define $k \in H^{s+1/2}(\Sigma_b;\R^{n})$ via $k = (k',\psi)$, then $(f,g,h,k) \in \mathcal{Z}^s$.  Hence Theorem \ref{iso_overdetermined} provides us with $(u,p) \in {_{0}}H^{s+2}(\Omega;\mathbb{R}^{n})\times H^{s+1}(\Omega)$ satisfying $\Psi_{\gamma}(u,p) = (f,g,h,k)$, which in particular implies that $\Theta_\gamma(u,p) = (f,g,h,k')$.  Thus $\Theta_\gamma$ is surjective.

\end{proof}

The $\dot{H}^{-2}$ constraints on the data are somewhat severe compared to the $\dot{H}^{-1}$ constraints built into $\mathcal{Y}^s$.  This makes Theorem \ref{iso_gamma_stokes_navier} unsuitable for our subsequent nonlinear analysis.

\section{Nonlinear analysis}\label{sec_nonlinear_analysis}

In this section we prove Theorems \ref{main_thm_flat}, \ref{main_thm_no_forcing}, and \ref{main_thm_euler}.  The latter two are consequences of the first and the mapping properties of the flattening diffeomorphism $\mathfrak{F}$ defined by \eqref{flat_def}.  As such, the crux of the matter is to prove the first theorem and study $\mathfrak{F}$.  We prove Theorem \ref{main_thm_flat}  with the help of the implicit function theorem and the isomorphisms of Theorems \ref{iso_stokes_capillary} and \ref{iso_stokes_capillary_zero}.  In order to apply the implicit function theorem, we must first establish the smoothness of various maps.

\subsection{Preliminaries}

We now turn our attention to proving some preliminary results needed to define the nonlinear maps associated to \eqref{flattened_system}.  The first such result is a simple quantitative $L^\infty$ bound for functions in $\sp^s(\R^{n-1})$.

\begin{lemma}\label{nlin_eta_lem}
Suppose that $s > (n-1)/2$.  Then there exists a constant $\delta = \delta(n,s,b)>0$ such that if $\eta \in \sp^s(\R^{n-1})$ and  $\norm{\eta}_{\sp^s} < \delta$, then $\norm{\eta}_{C^0_b} < b/2$.
\end{lemma}
\begin{proof}
This follows immediately from the fifth item of Theorem \ref{specialized_properties}. 
\end{proof}

Our next result is a nonlinear analog of Theorem \ref{cc_divergence}.  

\begin{proposition}\label{nlin_diverge_ident}
Suppose that $u \in {_{0}}H^{1}(\Omega;\mathbb{R}^{n})$ and $\eta \in B_{\sp^s}(0,\delta) \subset \sp^s(\R^{n-1})$ for $s >(n-1)/2$ and $\delta >0$ the constant from Lemma \ref{nlin_eta_lem}.  Let $J$, $\A$, and $\n$ be defined in terms of $\eta$ as in \eqref{JK_def}, \eqref{A_def}, and \eqref{normal_def}.  Then
\begin{equation}\label{nlin_diverge_ident_0}
 u\cdot \n(x',b) - \int_0^b J(x',x_n) \diva u(x',x_n) dx_n = - \diverge'\left( \int_0^b J(x',x_n) \A^{\intercal}(x',x_n) u(x',x_n)dx_n \right)'.
\end{equation}
In particular, $u\cdot \n(\cdot,b) - \int_0^b J(\cdot,x_n) \diva u(\cdot,x_n) dx_n \in \dot{H}^{-1}(\R^{n-1})$
and there exists a constant $c = c(n,b)>0$ such that 
\begin{equation}\label{nlin_diverge_ident_01}
 \snorm{u\cdot \n(\cdot,b) - \int_0^b J(\cdot,x_n) \diva u(\cdot,x_n) dx_n}_{\dot{H}^{-1}} \le c \norm{J\A}_{L^\infty} \norm{u}_{L^2}. 
\end{equation}
\end{proposition}
\begin{proof}
Let $\varphi \in C^\infty_c(\R^{n-1})$ and let $\phi \in C^\infty_c(\bar{\Omega})$ be defined by $\phi(x) = \varphi(x')$.  The definition of $J$ and $\mathcal{A}$ imply that $J \mathcal{A} e_n = \n$ on $\Sigma_b$ and  $\sum_{j=1}^n \p_j(J\mathcal{A}_{ij}) =0$ in $\Omega$ for each $1 \le i \le n$, the latter of which implies that $J \diva u = \diverge(J \A^{\intercal} u)$.  Using these, we then compute 
\begin{multline}
\int_{\Omega} J \diva u \phi  = \int_{\Omega} \diverge(J \A^{\intercal} u) \phi = \int_{\Omega} - J \A^{\intercal} u \cdot \nab \phi + \int_{\Sigma_b} J \A^{\intercal} u \cdot e_n \phi \\
= \int_{\Omega} - (J \A^{\intercal} u)' \cdot \nab' \varphi + \int_{\Sigma_b} J \A^{\intercal} u \cdot e_n \varphi 
= \int_{\Omega} - (J \A^{\intercal} u)' \cdot \nab' \varphi + \int_{\Sigma_b} u\cdot \n \varphi.
\end{multline}
On the other hand, Fubini's theorem allows us to compute 
\begin{equation}
 \int_{\Omega} J \diva u \phi = \int_{\R^{n-1}} \varphi(x') \int_0^b J(x',x_n) \diva u(x',x_n) dx_n dx'
\end{equation}
and 
\begin{equation}
  \int_{\Omega}  (J \A^{\intercal} u)' \cdot \nab' \varphi = \int_{\R^{n-1}} \nab' \varphi(x') \cdot \int_0^b  J(x',x_n) \A^{\intercal}(x',x_n) u(x',x_n)dx_n dx'.
\end{equation}
Combining these, rearranging, and using the arbitrariness of $\varphi$ then proves \eqref{nlin_diverge_ident_0}.  Then \eqref{nlin_diverge_ident_01} follows directly from applying the Fourier transform to \eqref{nlin_diverge_ident_0} and using the bound 
\begin{equation}
 \int_{\R^{n-1}} \abs{  \int_0^b J(x',x_n) \A^{\intercal}(x',x_n) u(x',x_n)dx_n }^2 dx' \le b \norm{J \A}_{L^\infty}^2 \int_\Omega \abs{u}^2,
\end{equation}
which follows from the Cauchy-Schwarz inequality and Tonelli's theorem.
 
\end{proof}

Our final preliminary results show that the map we will use in the implicit function theorem is well-defined and $C^1$.  In stating this result we recall that the $\mathcal{A}-$based differential operators are defined in \eqref{A_op_def_1}--\eqref{A_op_def_3}.

\begin{theorem}\label{Xi_well_defd}
Let $s > n/2$,  $\sigma \ge 0$, and $\mathcal{X}^s$ be as defined in \eqref{Xs_def}.  For $\delta >0$ define the open set 
\begin{equation}
 U^s_\delta = \{ (u,p,\eta) \in \mathcal{X}^s \st \norm{\eta}_{\sp^{s+5/2}} < \delta\} \subset \mathcal{X}^s.
\end{equation}
There exists a constant $\delta = \delta(n,s,b)>0$ such that if for $\gamma \in \R$, $T \in H^{s+1/2}(\R^{n-1} ; \R^{n\times n}_{\operatorname*{sym}})$, and $(u,p,\eta) \in U^s_\delta$  we define $f: \Omega \to \R^n$, $g : \Omega \to \R$, $h: \Sigma_b \to \R$, and $k : \Sigma_b \to \R^n$ via $f = (u-\gamma e_1) \cdot \naba u - \diverge_{\A} S_{\A}(p,u) $, $g = J \diva{u}$, $h = u\cdot \n + \gamma \p_1 \eta$, and $k = (pI-  \sga u) \n - (\eta -\sigma \mathcal{H}(\eta) )\n  - S_b T \n$, where $J$,$\A$, $\n$, and $\mathcal{H}$ are defined in terms of $\eta$ as in \eqref{JK_def}, \eqref{A_def}, \eqref{normal_def}, and \eqref{MC_def}, respectively, and $S_b$ is as in Lemma \ref{sobolev_slice_extension_surface},  then $(f,g,h,k) \in \mathcal{Y}^s$, where $\mathcal{Y}^s$ is the Hilbert space defined in \eqref{Ys_def}.  Moreover, the map 
\begin{equation}
\R \times  H^{s+1/2}(\R^{n-1} ; \R^{n\times n}_{\operatorname*{sym}}) \times  U^s_\delta  \ni (\gamma, T,u,p,\eta) \mapsto (f,g,h,k) \in \mathcal{Y}^s
\end{equation}
is smooth.
\end{theorem}
\begin{proof}
Let $\delta >0$ denote the minimum of the constant from Lemma \ref{nlin_eta_lem}, $b/(2c)$ where $c$ is the constant from Theorem \ref{omega_power_series}, and the ball size from Theorem \ref{sigma_power_series} (with $r =s +3/2$ and $d = n-1$) divided by the embedding constant from \eqref{specialized_properties_00}.
 
To begin we note that thanks to the second item of Theorem \ref{omega_aniso_properties} and Theorem \ref{omega_power_series}, the maps  $\Gamma_1,\Gamma_2 :   B_{\sp^r}(0, \delta) \times H^r(\Omega)  \to H^r(\Omega)$ given by  $\Gamma_1(\eta,\psi) = \frac{\psi}{b+\eta}$ and $\Gamma_2(\eta,\psi) = \frac{\psi \eta}{b+\eta}$ are well-defined and smooth for all $r > n/2$.   From this, \eqref{A_op_def_1}--\eqref{A_op_def_3}, Theorem \ref{specialized_properties}, Theorem \ref{specialized_product_supercrit}, and standard trace theory we then deduce that the map
\begin{multline}
\R \times  {_{0}}H^{s+2}(\Omega;\mathbb{R}^{n}) \times B_{\sp^{s+5/2}}(0,\delta)  \ni (\gamma,u,\eta) \mapsto \\
 ( \sga u, \diva \sga u, (u-\gamma e_1) \cdot \naba u, J \diva u,  
  \sga u \vert_{\Sigma_b} \n , u\vert_{\Sigma_b} \cdot \n + \gamma \p_1 \eta  ) \\
\in H^{s+1}(\Omega;\R^{n\times n}_{\operatorname*{sym}}) \times H^s(\Omega;\R^n) \times H^s(\Omega;\R^n) \times H^s(\Omega) \times   H^{s+1/2}(\Sigma_b;\R^{n}) \times H^{s+3/2}(\Sigma_b)
\end{multline}
is well-defined and smooth.  Similarly, the smoothness of $\Gamma_1,\Gamma_2$, Theorem \ref{omega_aniso_properties}, the inclusion $p-\eta \in H^{s+1}(\Omega)$, standard trace theory, and the fact that $H^{s+1/2}(\R^{n-1})$ is an algebra imply that the map 
\begin{equation}
U^p_\delta \ni (u,p,\eta) \mapsto (\naba p, (p-\eta)\vert_{\Sigma_b} \n) \in H^s(\Omega;\R^n) \times H^{s+1/2}(\Sigma_b;\R^n)
\end{equation}
is well-defined and smooth.  Theorem \ref{specialized_properties} and Theorem \ref{sigma_power_series} imply that the map 
\begin{equation}
 B_{\sp^{s+5/2}}(0,\delta) \ni \eta \mapsto \sigma \mathcal{H}(\eta) = \sigma \diverge'\left(\frac{\nab' \eta}{\sqrt{1+\abs{\nab '\eta}^2}} \right) \in H^{s+1/2}(\R^{n-1})
\end{equation}
is well-defined and smooth as well.  Finally, Theorem \ref{specialized_properties}, Lemma \ref{sobolev_slice_extension_surface}, and the fact that $H^{s+1/2}(\R^{n-1})$ is an algebra imply that the map 
\begin{equation}
H^{s+1/2}(\R^{n-1} ; \R^{n\times n}_{\operatorname*{sym}}) \times B_{\sp^{s+5/2}}(0,\delta) \ni (T,\eta) \mapsto (S_bT) \n \in H^{s+1/2}(\Sigma_b ;\R^n)
\end{equation}
is also well-defined and smooth.

Arguing as above, we also have that the maps $F: {_{0}}H^{s+2}(\Omega;\mathbb{R}^{n}) \times B_{\sp^{s+5/2}}(0,\delta) \to H^{s+3/2}(\R^{n-1})$ and $G: {_{0}}H^{s+2}(\Omega;\mathbb{R}^{n}) \times B_{\sp^{s+5/2}}(0,\delta) \to H^{s+3/2}(\R^{n-1};\R^{n-1})$ given by 
\begin{equation}
F(u,p)(x') =  u\cdot \n(x',b) - \int_0^b J(x',x_n) \diva u(x',x_n) dx_n
\end{equation}
and
\begin{equation}
 G(u,p)(x') = \left( \int_0^b J(x',x_n) \A^{\intercal}(x',x_n) u(x',x_n)dx_n \right)'
\end{equation}
are well-defined and smooth.  Proposition \ref{nlin_diverge_ident} tells us that
\begin{equation}
 h - \int_0^b g(\cdot,x_n) dx_n = F(u,\eta) = -\diverge' G(u,\eta) \in \dot{H}^{-1}(\R^{n-1}).
\end{equation}
Moreover, for any $k \in \N$ we have $D^k F(u,\eta) = -\diverge' D^k G(u,\eta)$, from which we readily deduce that the map 
\begin{equation}
 {_{0}}H^{s+2}(\Omega;\mathbb{R}^{n}) \times B_{\sp^{s+5/2}}(0,\delta) \ni (u,\eta) \mapsto    F(u,\eta) = h - \int_0^b g(\cdot,x_n) dx_n \in H^{s+3/2}(\R^{n-1}) \cap \dot{H}^{-1}(\R^{n-1})
\end{equation}
is well-defined and smooth.

Synthesizing all of the above then shows that the map $(\gamma,T, u,p,\eta) \mapsto (f,g,h,k) \in \mathcal{Y}^s$ is well-defined and smooth.

\end{proof}

We also need a variant of this result.

\begin{proposition}\label{stress_Xi_C1}
Let $n/2 < s \in \N$ and $\delta_\ast >0$ be as in Proposition \ref{E_eta_properties}.  Then the map $H^{s+2}(\R^n ; \R^{n\times n}_{\operatorname*{sym}})   \times B_{\sp^{s+5/2}(\R^{n-1})}(0,\delta_\ast) \ni (\mathcal{T},\eta) \mapsto (\mathcal{T}\circ \mathfrak{F}_\eta \vert_{\Sigma_b} )\n \in  H^{s+1/2}(\Sigma_b;\R^n)$, where $\n$ is defined in terms of $\eta$ via \eqref{normal_def}, is well-defined and $C^1$. 
\end{proposition}
\begin{proof}
This follows by combining Theorems \ref{specialized_properties} and Corollary \ref{Lambda_Omega_diff}with   the fact that $H^{s+1/2}(\Sigma_b)$ is an algebra.
\end{proof}

\subsection{Solvability of \eqref{flattened_system}: proof of Theorem \ref{main_thm_flat}  }\label{sec_main_thms_flat}

We have now developed all of the tools needed to solve \eqref{flattened_system}.

\begin{proof}[Proof of Theorem \ref{main_thm_flat}]
We first consider the case $\sigma >0$ and $n \ge 2$.  Let $\delta>0$ be the smaller of $\delta(n,s,b)>0$ from Theorem \ref{Xi_well_defd} and $\delta_\ast >0$ from Proposition \ref{E_eta_properties}.  Define the open set
\begin{equation}
 U^s_\delta = \{ (u,p,\eta) \in \mathcal{X}^s \st \norm{\eta}_{\sp^{s+5/2}} < \delta\} \subset \mathcal{X}^s.
\end{equation}
Proposition \ref{Xs_embed} and the standard Sobolev embeddings show that any open subset of $U^s_\delta$ containing $(0,0,0)$ satisfies the assertions of the first item.

Define the Hilbert space 
\begin{equation}
\mathcal{E}^s = \R \times H^{s+2}(\R^n ; \R^{n\times n}_{\operatorname*{sym}})  \times  H^{s+1/2}(\R^{n-1}; \R^{n\times n}_{\operatorname*{sym}}) \times H^{s+1}(\R^n;\R^n) \times H^s(\R^{n-1};\R^n). 
\end{equation}
Corollary \ref{Lambda_Omega_diff}, Theorem \ref{Xi_well_defd}, Proposition \ref{stress_Xi_C1}, and Lemma \ref{sobolev_slice_extension} then tell us that the map $\Xi:  \mathcal{E}^s \times U^s_\delta   \to \mathcal{Y}^s$ given by
\begin{multline}
 \Xi(\gamma, \mathcal{T}, T, \mathfrak{f},f, u, p, \eta)  = -(\mathfrak{f}\circ \mathfrak{F}_\eta  + L_\Omega f, 0,0, (\mathcal{T}\circ \mathfrak{F}_\eta \vert_{\Sigma_b}+  S_b T) \n )  \\
 + ((u-\gamma e_1) \cdot \naba u  + \diva S_{\A}(p,u), J \diva{u}, u\cdot \n + \gamma \p_1 \eta, (pI-  \sga u) \n - (\eta -\sigma \mathcal{H}(\eta) )\n  ),
\end{multline}
where $\mathfrak{F}_\eta,$  $J$, $\A$, $\n$, and $\mathcal{H}$ are defined in terms of $\eta$ as in \eqref{flat_def}, \eqref{JK_def}, \eqref{A_def}, \eqref{normal_def}, and \eqref{MC_def}, respectively,  $L_\Omega$ is the linear map from Lemma \ref{sobolev_slice_extension}, and $S_b$ is the linear map from Lemma \ref{sobolev_slice_extension_surface}, is $C^1$.  Due to the product structure $\mathcal{E}^s \times U^s_\delta$, we may define the derivatives of $\Xi$ with respect to the first and second factors via 
\begin{equation}
 D_1 \Xi  : \mathcal{E}^s \times U^s_\delta \to  \L(\mathcal{E}^s; \mathcal{Y}^s) \text{ and } D_2 \Xi : \mathcal{E}^s \times U^s_\delta \to    \L(\mathcal{X}^s; \mathcal{Y}^s).
\end{equation}

Note that  Corollary \ref{Lambda_Omega_diff} show that for $\mathfrak{S}_b(\mathcal{T},\eta) = \mathcal{T} \circ \mathfrak{F}_\eta \vert_{\Sigma_b}$ and  $\Lambda_\Omega(\mathfrak{f},\eta) = \mathfrak{f} \circ \mathfrak{F}_\eta$ we have that $D_2 \mathfrak{S}_b(0,0) =0$ and  $D_2 \Lambda_\Omega (0,0) = 0$.  Thus, for any $\gamma \in \R$ we have that  
\begin{equation}
\Xi(\gamma,0,0,0,0,0,0,0) =  (0,0,0,0) 
\end{equation}
and 
\begin{equation}
 D_2 \Xi(\gamma,0,0,0,0,0,0,0)(u,p,\eta) = (\diverge S(p,u) - \gamma \p_1 u,  \diverge{u}, u_n + \gamma \p_1 \eta, S(p,u)e_n - (\eta - \sigma \Delta' \eta) e_n )
\end{equation}
for all $(u,p,\eta) \in \mathcal{X}^s$.   In other words, $D_2 \Xi(\gamma,0,0,0,0,0,0,0) = \Upsilon_{\gamma,\sigma}  \in \L(\mathcal{X}^s; \mathcal{Y}^s)$, where $\Upsilon_{\gamma,\sigma}$ is as defined in \eqref{Upsilon_def}.  Thus, for every $\gamma_\ast \neq 0$ Theorem \ref{iso_stokes_capillary} guarantees that $D_2 \Xi(\gamma_\ast,0,0,0,0,0,0,0)$ is an isomorphism.  The implicit function theorem (see, for instance, Theorem 2.5.7 in \cite{AMR_1988})  then provides us with open sets $\mathcal{U}(\gamma_\ast) \subset \mathcal{E}^s$ and $\mathcal{O}(\gamma_\ast) \subseteq U^s_\delta$  such that $(\gamma_\ast,0,0,0,0) \in \mathcal{U}(\gamma_\ast)$  and $(0,0,0) \in \mathcal{O}(\gamma_\ast)$, and a $C^1$ and Lipschitz map $\varpi_{\gamma_\ast} : \mathcal{U}(\gamma_\ast) \to \mathcal{O}(\gamma_\ast) \subseteq  U^s_\delta$ such that 
\begin{equation}
\Xi(\gamma,\mathcal{T},T,\mathfrak{f},f, \varpi_{\gamma_\ast}(\gamma,\mathcal{T},T,\mathfrak{f},f)) =(0,0,0,0) 
\end{equation}
for all  $(\gamma,\mathcal{T},T,\mathfrak{f},f) \in \mathcal{U}(\gamma_\ast)$.   Moreover, the implicit function theorem also implies that the triple $(u,p,\eta) = \varpi_{\gamma_\ast}(\gamma,\mathcal{T},T,\mathfrak{f},f) \in \mathcal{O}(\gamma_\ast)$ is the unique solution to $\Xi(\gamma,\mathcal{T},T,\mathfrak{f},f,u,p,\eta) = (0,0,0,0)$ in $\mathcal{O}(\gamma_\ast)$.

Define the open sets
\begin{equation}
 \mathcal{U}^s = \bigcup_{\gamma_\ast \in \R \backslash \{0\}}  \mathcal{U}(\gamma_\ast)   \subset \mathcal{E}^s
\text{ and }
\mathcal{O}^s = \bigcup_{\gamma_\ast \in \R \backslash \{0\}} \mathcal{O}(\gamma_\ast) \subseteq U^s_\delta. 
\end{equation}
By construction we have that $(\R \backslash \{0\}) \times \{0\} \times \{0\} \times \{0\} \times \{0\} \subset \mathcal{U}^s$, which is the second item.

Using the above, we may then define the map $\varpi : \mathcal{U}^s \to \mathcal{O}^s$ via $\varpi(\gamma,\mathcal{T},T,\mathfrak{f},f) = \varpi_{\gamma_\ast}(\gamma,\mathcal{T},T,\mathfrak{f},f)$ when $(\gamma,\mathcal{T},T,\mathfrak{f},f) \in \mathcal{U}(\gamma_\ast)$ for some $\gamma_\ast \in \R \backslash \{0\}$.  This is well-defined, $C^1$, and locally Lipschitz  by the above analysis.  The third and fourth items then follow by setting $(u,p,\eta) = \varpi(\gamma,\mathcal{T},T,\mathfrak{f},f)$ for $(\gamma,\mathcal{T},T,\mathfrak{f},f) \in \mathcal{U}^s$.

The result is now proved for $\sigma >0$ and $n \ge 2$.  The proof when $n=2$ and $\sigma =0$ is identical, except that we use Theorem \ref{iso_stokes_capillary_zero} and the isomorphism $\Upsilon_{\gamma,0}$ in place of Theorem \ref{iso_stokes_capillary}.  Moreover, in this case we know from \eqref{Xs_2equiv} that $\mathcal{X}^s = {_{0}}H^{s+2}(\Omega;\mathbb{R}^{2}) \times H^{s+1}(\Omega) \times H^{s+5/2}(\R)$.

\end{proof}

\subsection{Solvability of \eqref{traveling_euler}: proofs of Theorems \ref{main_thm_no_forcing} and  \ref{main_thm_euler} }\label{sec_main_thms_euler}
 
We have all the tools needed to prove Theorems \ref{main_thm_no_forcing} and \ref{main_thm_euler}.  We present these proofs now.

\begin{proof}[Proof of Theorem \ref{main_thm_no_forcing}]

Suppose that $\eta \in \sp^{s+5/2}(\R^{n-1})$, $v \in {_{0}}H^{s+2}(\Omega_{b+\eta};\mathbb{R}^{2})$, and $q \in \an^{s+1}(\Omega_{b+\eta})$ are nontrivial solutions to \eqref{traveling_euler} with $\mathcal{T}=0$ and $\mathfrak{f}=0$.  Further suppose that 
\begin{equation}\label{main_thm_no_forcing_1}
 \norm{v}_{{_{0}}H^{s+2}} + \norm{q}_{\an^{s+1}} + \norm{\eta}_{\sp^{s+5/2}} +  \norm{q-\eta}_{H^{s+1}} < R
\end{equation}
for some $0 < R < \delta$ to be chosen, where $\delta = \delta(n,s,b) = \delta(n,b) >0$ is as in Lemma \ref{nlin_eta_lem}.  In particular, this means that $\norm{\eta}_{C^0_b} < b/2$.

Define $u = v \circ \mathfrak{F} : \Omega \to \R^n$ and $p = q \circ \mathfrak{F} : \Omega \to \R$ for $\mathfrak{F}$ defined in terms of $\eta$ as in \eqref{flat_def}.  Then $(u,p,\eta)$ solves \eqref{flattened_system} with $f=0$ and $T=0$.  Since $s \in \N$ and $s > n/2$, Theorem \ref{G_composition} guarantees that $(u,p,\eta) \in \mathcal{X}^s$ and 
\begin{equation}\label{main_thm_no_forcing_2}
 \norm{u}_{{_{0}}H^{s+2}} + \norm{p}_{\an^{s+1}} + \norm{\eta}_{\sp^{s+5/2}} +  \norm{p-\eta}_{H^{s+1}} \le c(n,R) R,
\end{equation}
where $r \mapsto c(n,r)$ is non-decreasing.

Let $\mathcal{U}(\gamma)$ and $\mathcal{O}(\gamma)$ be the open sets constructed in the previous subsection in the proof of Theorem \ref{main_thm_flat}.  That proof shows that $(\dot{u},\dot{p},\dot{\eta}) = (0,0,0) = \varpi_{\gamma}(\gamma,0,0,0,0) \in \mathcal{O}(\gamma)$ is the unique solution to $\Xi(\gamma,0,0,0,0,\dot{u},\dot{p},\dot{\eta}) = (0,0,0,0)$ in $\mathcal{O}(\gamma)$.  Let $R_0 >0$ be such that $B_{\mathcal{X}^s}((0,0,0),R_0) \subseteq \mathcal{O}(\gamma)$.  From \eqref{main_thm_no_forcing_2} we know that if $R < r$ for $r$ small enough (in terms of $n$), then $c(n,R) R < R_0$, and hence $(u,p,\eta) = (0,0,0)$, which contradicts the fact that $(v,q,\eta)$ is nontrivial.  Thus \eqref{main_thm_no_forcing_1} cannot hold for $R \le r$, which completes the proof.

\end{proof}

\begin{proof}[Proof of Theorem \ref{main_thm_euler}]

Let $\mathcal{U}^s$ and $\mathcal{O}^s$ be the open sets from Theorem \ref{main_thm_flat}, and let $\varpi : \mathcal{U}^s \to \mathcal{O}^s$ be as in the proof of  Theorem \ref{main_thm_flat} above.  Then $(u,p,\eta) = \varpi(\gamma,\mathcal{T},T,\mathfrak{f},f)$ solves \eqref{main_thm_flat_0} for every $(\gamma,\mathcal{T},T,\mathfrak{f},f) \in \mathcal{U}^s$.  We also know that for this data we have $\norm{\eta}_{C^0_b} \le b/2$, and so Theorem \ref{G_composition} implies that the maps $\mathfrak{F}_\eta$ and $\mathfrak{F}_\eta^{-1}$ are $C^{3 + \lfloor s - n/2 \rfloor}$ diffeomorphisms.

Fix $(\gamma,\mathcal{T},T,\mathfrak{f},f) \in \mathcal{U}^s$ and set $(u,p,\eta) = \varpi(\gamma,\mathcal{T},T,\mathfrak{f},f)$,  $v = u \circ \mathfrak{F}_\eta^{-1}$, and $q = p \circ \mathfrak{F}_\eta^{-1}$.  Theorems \ref{omega_aniso_properties},  \ref{G_composition},  and the usual Sobolev embeddings then  imply that $v \in {_{0}}H^{s+2}(\Omega_{b+\eta};\mathbb{R}^{n}) \cap  C^{2 + \lfloor s-n/2 \rfloor}_b(\Omega_{b+\eta};\R^n)$ and $q \in \an^{s+1}(\Omega_{b+\eta}) \cap C^{1 + \lfloor s-n/2 \rfloor}_b(\Omega_{b+\eta})$.   Note that since $(\mathfrak{F}_\eta^{-1}(x))' = x'$, we have that
\begin{equation}
 \left(\mathfrak{f} \circ \mathfrak{F}_\eta  + L_\Omega f \right) \circ \mathfrak{F}_\eta^{-1}(x) = \mathfrak{f}(x) + f(x')  = \mathfrak{f}(x) + L_{\Omega+\eta} f(x) \text{ for all } x \in \Omega_{b+\eta},
\end{equation}
and, similarly, 
\begin{equation}
(\mathcal{T} \circ \mathfrak{F}_\eta + S_b T) \circ \mathfrak{F}_\eta^{-1}\vert_{\Sigma_{b+\eta}} = \mathcal{T}\vert_{\Sigma_{b+\eta}} + S_{b+\eta} T.
\end{equation}
Then since $(u,p,\eta)$ solve \eqref{main_thm_flat_0} we have that $(v,q,\eta)$ solve \eqref{main_thm_euler_0}, and this completes the proof of the first two items.  The third item follows from the fact that $\varpi$ is locally Lipschitz.

\end{proof}

\appendix

\section{Analysis tools}\label{sec_analysis_tools}

In this appendix we record various analytic tools used throughout the paper.

\subsection{A computation}\label{sec_computation}

Here we record the proof of the assertion made at the end of Section \ref{sec_eulerian_form}.

\begin{proposition}\label{trav_prop}
Suppose that  $\eta \in H^{5/2}(\R^{n-1})$ is such that $\eta$ is Lipschitz, bounded, and $\inf_{\R^{n-1}} \eta > -b$.  Further suppose that $v \in H^2(\Omega_{b+\eta};\R^n) \cap L^\infty(\Omega_{b+\eta};\R^{n}) $, $q \in H^1(\Omega_{b+\eta})$, $\mathfrak{f} \in L^2(\Omega_{b+\eta};\R^n)$, and $\mathcal{T} \in H^{1/2}(\Sigma_{b+\eta};\R^{n \times n}_{\operatorname*{sym}})$ solve \eqref{ns_euler}.  Then 
\begin{equation}\label{trav_prop_0}
\int_{\Omega_{b+\eta}}  \mathfrak{f} \cdot v - \int_{\Sigma_{b+\eta}} \mathcal{T} \nu \cdot v = \int_{\Omega_{b+\eta}} \hal \abs{\sg v}^2.
\end{equation}
\end{proposition}
\begin{proof}
First note that the first and fourth equations of \eqref{ns_euler} can be rewritten as  
\begin{equation}\label{trav_prop_1}
\mathfrak{f} = (v-\gamma e_1) \cdot \nab v  -  \Delta v + \nab q = \diverge( S(q,v)  + v \otimes (v - \gamma e_1) )  \text{ in } \Omega_{b+\eta}
\end{equation}
and 
\begin{equation}\label{trav_prop_2}
0 = v \cdot \n + \gamma \p_1 \eta = (v - \gamma e_1) \cdot \n = (v- \gamma e_1) \cdot \nu \sqrt{1+\abs{\nab'\eta}^2} \text{ on } \Sigma_{b+\eta}.
\end{equation}
We then take the dot product of  \eqref{trav_prop_1} with $v$, integrate by parts over $\Omega_{b+\eta}$ (which is possible since $\eta$ is Lipschitz, so $\Sigma_{b+\eta}$ enjoys a trace operator), and use the fact that $v =0$ on $\Sigma_0$ to deduce that 
\begin{multline}
 \int_{\Omega_{b+\eta}} \mathfrak{f} \cdot v = \int_{\Omega_{b+\eta}} v \cdot \diverge( S(q,v)  + v \otimes (v - \gamma e_1) ) \\ 
= \int_{\Omega_{b+\eta}} - \nab v : (S(q,v)   +   v \otimes (v - \gamma e_1) ) + \int_{\Sigma_{b+\eta}} (S(q,v)  + v \otimes (v - \gamma e_1))\nu \cdot v.
\end{multline}
Note that the second term in the last $\Omega_{b+\eta}$ integral and the second term in the $\Sigma_{b+\eta}$ integral are well-defined since $v$ is bounded.  We will compute each of the four terms on the right in turn.  

For the first term we use the fact that $\diverge v =0$ to compute
\begin{equation}
\int_{\Omega_{b+\eta}} -\nab v : S(q,v)  = \int_{\Omega_{b+\eta}} \nab v : \sg v - q \diverge{v} =    \int_{\Omega_{b+\eta}} \hal \abs{\sg v}^2.
\end{equation}
For the second we integrate by parts again and use the equations $\diverge v =0$ in $\Omega_{b+\eta}$, $v =0$ on $\Sigma_0$,  and \eqref{trav_prop_2} to compute
\begin{equation}
\int_{\Omega_{b+\eta}} -\nab v : v \otimes(v - \gamma e_1) = \int_{\Omega_{b+\eta}} - \nab \frac{\abs{v}^2}{2} \cdot (v - \gamma e_1) = \int_{\Omega_{b+\eta}} \frac{\abs{v}^2}{2} \diverge(v - \gamma e_1) - \int_{\Sigma_{b+\eta}} \frac{\abs{v}^2}{2} (v- \gamma e_1) \cdot \nu  = 0. 
\end{equation}
For the third term we use the third equation of \eqref{traveling_euler} to compute 
\begin{equation}
\int_{\Sigma_{b+\eta}} S(q,v)\nu \cdot v = \int_{\Sigma_{b+\eta}} (\eta - \sigma \mathcal{H}(\eta) )\nu  \cdot v + \mathcal{T} \nu \cdot v,
\end{equation}
but by \eqref{trav_prop_2} and an integration by parts
\begin{multline}
 \int_{\Sigma_{b+\eta}} (\eta - \sigma \mathcal{H}(\eta) )\nu  \cdot v   = \int_{\R^{n-1}} - \gamma \p_1 \eta (\eta - \sigma \mathcal{H}(\eta) ) = -\gamma \int_{\R^{n-1}} \eta \p_1 \eta  + \sigma \frac{\nab' \eta}{\sqrt{1+ \abs{\nab' \eta}^2}} \cdot \nab' \p_1 \eta    \\
 = - \gamma \int_{\R^{n-1}} \p_1 \left( \hal \abs{\eta}^2 + \sigma( \sqrt{1+ \abs{\nab' \eta}^2}-1  )  \right) = 0,
\end{multline}
so 
\begin{equation}
\int_{\Sigma_{b+\eta}} S(q,v)\nu \cdot v = \int_{\Sigma_{b+\eta}}  \mathcal{T}\nu \cdot v. 
\end{equation}
Finally, for the fourth term we again use \eqref{trav_prop_2} to compute 
\begin{equation}
\int_{\Sigma_{b+\eta}} v \otimes (v-\gamma e_1) \nu \cdot v = \int_{\Sigma_{b+\eta}} \abs{v}^2 (v- \gamma e_1) \cdot \nu =0. 
\end{equation}
Combining these computations and rearranging then yields \eqref{trav_prop_0}.

\end{proof}

\subsection{Fourier transform}

In the following lemma we will need to make use of the reflection operator defined as follows.  For $f: \R^d \to \C$ we define $Rf : \R^d \to \C$ via $Rf(x) = f(-x)$.

\begin{lemma}\label{tempered_real_lemma}
The following hold.
\begin{enumerate}
 \item Let $f \in L^2(\R^d;\C)$.  Then $f$ is real-valued, i.e. $f = \bar{f}$, if and only if $\bar{\hat{f}} = R \hat{f}$.

 \item The Fourier transform is a bijection from the real-valued Schwartz function $\{f \in \mathscr{S}(\R^d) \st f = \bar{f}\}$ to $\{f \in \mathscr{S}(\R^d) \st \bar{\hat{f}} = R \hat{f}\}$.
 
 \item   Recall that for a tempered distribution $T \in \mathscr{S}'(\R^d)$ we define the conjugate and reflected distributions $\bar{T}, R T \in \mathscr{S}'(\R^d)$ via 
\begin{equation}\label{tempered_real_lemma_0}
 \br{\bar{T}, \psi} = \br{T,\bar{\psi}} \text{ and } \br{RT,\psi} = \br{T,R\psi} \text{ for each } \psi \in \mathscr{S}(\R^d).
\end{equation}
Then $T \in \mathscr{S}'(\R^d)$ is real-valued, i.e. $T = \bar{T}$, if and only if $\overline{\hat{T}} = R \hat{T}$.

\end{enumerate}
\end{lemma}
\begin{proof}
If $f = \bar{f}$, then 
\begin{equation}
 \overline{\hat{f}(\xi)} = \int_{\R^d} \overline{f(x)} e^{2\pi i x \cdot \xi} dx = \int_{\R^d} f(x) e^{2 \pi i x\cdot \xi}dx = \hat{f}(-\xi).
\end{equation}
On the other hand, if $\bar{\hat{f}} = R \hat{f}$, then 
\begin{equation}
 \overline{f(x)} = \int_{\R^d} \overline{\hat{f}(\xi)} e^{-2\pi i x \cdot \xi} d\xi = \int_{\R^d} \hat{f}(-\xi) e^{-2 \pi i x\cdot \xi}dx = \int_{\R^d} \hat{f}(\xi) e^{2 \pi x \cdot \xi}d\xi = f(x).
\end{equation}
This proves the first item.  The second item follows from this and the fact that the Fourier transform is an isomorphism on the Schwartz class.

We now prove the third item.  For $\psi \in \mathscr{S}(\R^d)$ we have that $\hat{\bar{\psi}} = \overline{R \hat{\psi}}$ and $\widehat{R\psi} = R \hat{\psi}$.  Using these and \eqref{tempered_real_lemma_0}, for any $\psi \in \mathscr{S}(\R^d)$ we may compute 
\begin{equation}
\br{\overline{\hat{T}}, \psi} = \overline{ \br{\hat{T}, \bar{\psi}} }= \overline{ \br{T, \hat{\bar{\psi}}} }= \overline{ \br{T,\overline{R \hat{\psi}}  } }= \br{\bar{T}, R \hat{\psi}} 
\end{equation}
and
\begin{equation}
 \br{R\hat{T},\psi}  = \br{\hat{T},R\psi } = \br{T,\widehat{R\psi}} = \br{T,R \hat{\psi}}.
\end{equation}
The result then follows from these identities and the fact that the map $\mathscr{S}(\R^d) \ni \psi \mapsto R \hat{\psi} \in \mathscr{S}(\R^d)$ is a bijection.
\end{proof}

\subsection{Poincar\'{e} and Korn inequalities}

The following version of the Poincar\'{e} inequality will be useful.  

\begin{lemma}[Poincar\'{e} inequality]\label{poincare}
Suppose that $\zeta : \R^{n-1} \to (0,\infty)$ is bounded and lower semicontinuous.  Then 
\begin{equation}\label{poincare_inequality}
\int_{\Omega_{\zeta}}|f|^{2}  \leq\frac{1}{2}\Vert\zeta\Vert_{\infty}^{2}
\int_{\Omega_{\zeta}}|\nabla f|^{2} 
\end{equation}
for every $f \in H^1(\Omega_\zeta)$ such that $f =0$ on $\Sigma_0$.  Consequently, on $\{f \in H^1(\Omega_\zeta) \st f = 0 \text{ on }\Sigma_0\}$ the map $f \mapsto \Vert\nabla f\Vert_{L^{2}(\Omega_{\zeta})}$ defines a norm equivalent to the standard $H^1$ norm.
\end{lemma}
\begin{proof}
Theorem 13.19 in \cite{Leoni_2017} asserts this result for functions that also vanish on $\Sigma_{\zeta}$, but the proof works also for functions only vanishing on $\Sigma_0$.
\end{proof}

We record here a version of Korn's inequality for the space ${_{0}}H^{1}(\Omega;\mathbb{R}^{n})$.  A proof may be found, for instance, in Lemma 2.7 of \cite{Beale_1981}.

\begin{lemma}[Korn's inequality]\label{korn}
There exists a constant $c = c(n,b)>0$ such that 
\begin{equation}
 \norm{u}_{H^1} \le c \norm{\sg u}_{L^2} \text{ for all } u \in {_{0}}H^{1}(\Omega;\mathbb{R}^{n}).
\end{equation}
\end{lemma}

\subsection{Sobolev spaces}

We record here some basic results about standard Sobolev spaces.  Although these are well known, we include quick proofs for the benefit of the reader.  We begin with a lemma that relates Sobolev norms of functions in $\Omega$ to those of extension functions on all of $\R^n$.

\begin{lemma}\label{extension_char}
Let $s \ge 0$, $n \ge 2$, and $\zeta \in C^{0,1}_b(\R^{n-1})$ be such that $\inf \zeta > 0$.  Then the following hold.
\begin{enumerate}
 \item There exists a linear map $E$, mapping the measurable functions on $\Omega_\zeta$ to the measurable functions on $\R^n$, such that $Ef = f$ almost everywhere in $\Omega_\zeta$, and for every $0 \le t \le s$ the restriction of $E$ to $H^t(\Omega_\zeta)$ defines a bounded linear operator with values in $H^t(\R^n)$.  Moreover, there exists a constant $c = c(n,s,\zeta)>0$ such that $\norm{Ef}_{H^t(\R^n)} \le c \norm{f}_{H^t(\Omega_\zeta)}$ for all $0\le t \le s$ and $f \in H^t(\Omega_\zeta)$.
 
 \item A measurable function $f: \Omega_\zeta \to \R$ belongs to $H^s(\Omega_\zeta)$ if and only if there exists $F \in H^s(\R^n)$ such that $f = F$ almost everywhere in $\Omega_\zeta$.  Moreover, there exists a constant $c = c(n,s,\zeta)>0$ such that 
\begin{equation}
 \frac{1}{c} \norm{f}_{H^s(\Omega_\zeta)} \le \inf\{\norm{F}_{H^s(\R^n)} \st F =f \text{ a.e. in } \Omega_\zeta  \} \le c \norm{f}_{H^s(\Omega_\zeta)}
\end{equation}
for every measurable $f: \Omega_\zeta \to \R$.

\end{enumerate}
\end{lemma}
\begin{proof}

Let $s < m \in \N$.  The Stein extension theorem (see, for instance, Theorem VI.5 in \cite{Stein_1970}) provides a linear extension operator $E$ from the space of measurable functions on $\Omega_\zeta$ to the space of measurable functions on $\R^n$ such that $Ef = f$ almost everywhere on $\Omega_\zeta$ for each measurable $f: \Omega_\zeta \to \R$ and with the additional property that the restriction of $E$ to $H^k(\Omega_\zeta)$ is a bounded linear operator into $H^k(\R^n)$ for every $0 \le k \le m$.  Standard interpolation theory (see, for instance \cite{BL_1976,Leoni_2017,Triebel_1995}) then shows that $E$ is bounded from $H^s(\Omega_\zeta)$ to $H^s(\R^n)$ as well.  This proves the first item. 

Suppose now that $f : \Omega_\zeta \to \R$ is measurable, and consider $Ef: \R^n \to \R$.  If $f \in H^s(\R^n)$, then by the first item $Ef \in H^s(\R^n)$ and 
\begin{equation}\label{extension_char_1}
\inf\{\norm{F}_{H^s(\R^n)} \st F =f \text{ a.e. in } \Omega_\zeta  \} \le \norm{Ef}_{H^s(\R^n)} \le \norm{E}_{\L(H^s(\Omega);H^s(\R^n))} \norm{f}_{H^s(\Omega_\zeta)}.
\end{equation}
On the other hand, the intrinsic version of the $H^s(\Omega_\zeta)$ norm shows that 
\begin{equation}
 \norm{f}_{H^s(\Omega_\zeta)} \le c \norm{F}_{H^s(\R^n)} \text{ whenever } F \in H^s(\R^n) \text{ and } F = f \text{ a.e. in }\Omega_\zeta,
\end{equation}
so if there exists such an $F$ we deduce that $f \in H^s(\Omega_\zeta)$ with  
\begin{equation}\label{extension_char_2}
 \norm{f}_{H^s} \le c \inf\{\norm{F}_{H^s(\R^n)} \st F =f \text{ a.e. in } \Omega_\zeta  \}.
\end{equation}
To conclude we simply chain together the bounds.
\end{proof}

The second lemma provides an equivalent ``slicing norm'' on the space $H^s(\R^n)$.

\begin{lemma}\label{slicing_lemma}
Let $s \ge 0$ and $n \ge 2$.  Then there exists a constant $c = c(n,s)>0$ such that 
\begin{equation}
 \frac{1}{c} \norm{f}_{H^s}^2 \le \norm{f}_{L^2(\R;H^s(\R^{n-1}))}^2 + \norm{f}_{H^s(\R;L^2(\R^{n-1}))}^2   \le  c \norm{f}_{H^s}^2 
\end{equation}
for all $f \in \mathscr{S}'(\R^n)$ such that $\hat{f} \in L^1_{loc}(\R^n)$, where 
\begin{equation}
 \norm{f}_{L^2(\R;H^s(\R^{n-1}))}^2  = \int_{\R} \norm{f(\cdot,x_n)}_{H^s(\R^{n-1})}^2 dx_n \text{ and }  \norm{f}_{H^s(\R;L^2(\R^{n-1}))}^2  = \int_{\R} (1+\tau^2)^s \norm{\mathcal{F}_n f(\cdot,\tau)}_{L^2(\R^{n-1})}^2 d\tau,
\end{equation}
and in the latter equation $\mathcal{F}_n$ denotes the Fourier transform with respect to the $n^{th}$ variable.
\end{lemma}
\begin{proof}
Let $\hat{\cdot}$ denote the usual Fourier transform on $\R^n$ and $\mathcal{F}'$ denote the Fourier transform with respect to the first $n-1$ variables.  Write $\xi \in \R^n$ as $\xi = (\xi',\tau) \in \R^{n-1} \times \R$.  Then $\hat{f}(\xi) = \mathcal{F}' \mathcal{F}_n f(\xi',\tau) = \mathcal{F}_n \mathcal{F}'  f(\xi',\tau)$ for $f \in \mathscr{S}'(\R^n)$.

Then for any $f \in \mathscr{S}'(\R^n)$ such that $\hat{f} \in L^1_{loc}(\R^n)$ we have the equivalence
\begin{multline}
\norm{f}_{H^s}^2 = \int_{\R^n} (1+\abs{\xi}^2)^s \abs{\hat{f}(\xi)}^2 d\xi \\
\asymp \int_{\R} \int_{\R^{n-1}} (1+ \abs{\xi'}^2)^s  \abs{\mathcal{F}_n \mathcal{F}' f(\xi',\tau)}^2 d\xi' d\tau +  \int_{\R} \int_{\R^{n-1}} (1+ \abs{\tau}^2)^s  \abs{\mathcal{F}' \mathcal{F}_n  f(\xi',\tau)}^2 d\xi' d\tau.
\end{multline}
We then use the Parseval's theorem  on this to see that 
\begin{multline}
 \norm{f}_{H^s}^2
\asymp \int_{\R} \int_{\R^{n-1}} (1+ \abs{\xi'}^2)^s  \abs{\mathcal{F}'f (\xi',x_n)}^2 d\xi' dx_n +  \int_{\R} \int_{\R^{n-1}}  (1+ \abs{\tau}^2)^s  \abs{\mathcal{F}_n f(x',\tau)}^2 dx' d\tau \\
= \int_{\R} \norm{f(\cdot,x_n)}_{H^s(\R^{n-1})}^2   dx_n +  \int_{\R}    (1+ \abs{\tau}^2)^s  \norm{\mathcal{F}_n f(\cdot,\tau)}_{L^2(\R^{n-1})}^2  d\tau,
\end{multline}
which is the desired estimate.

\end{proof}

We also record a useful corollary.

\begin{corollary}\label{omega_slicing_bound}
Let $s \ge 0$ and $n \ge 2$.  Then there exists a constant $c = c(n,s,b) >0$ such that 
\begin{equation}
 \int_0^b \norm{f(\cdot,x_n)}_{H^s(\R^{n-1})}^2 dx_n \le c \norm{f}_{H^s(\Omega)}^2
\end{equation}
for all $f \in H^s(\Omega)$.
\end{corollary}
\begin{proof}
Let $Ef \in H^s(\R^n)$ be the Stein extension of $f$ defined in Lemma \ref{extension_char}.  From Lemma \ref{slicing_lemma} we may then bound 
\begin{equation}
  \int_0^b \norm{f(\cdot,x_n)}_{H^s(\R^{n-1})}^2 dx_n \le \int_{\R}  \norm{Ef(\cdot,x_n)}_{H^s(\R^{n-1})}^2 dx_n \le c \norm{Ef}_{H^s(\Omega)}^2 \le c \norm{f}_{H^s(\Omega)}^2.
\end{equation}
\end{proof}

Next we record a pair of product estimates.  The first is phrased for functions defined in sets of the form $\Omega_\zeta$.

\begin{lemma}\label{omega_zeta_subcrit}
Let $\zeta \in C^{0,1}_b(\R^{n-1})$ be such that $\inf \zeta >0$.  Suppose that $f \in H^s(\Omega_\zeta)$ for $s > n/2$.  Then for each $0 \le r \le s$ there exists a constant $c = c(n,\norm{\zeta}_{C^{0,1}_b},s,r)>0$ such that 
\begin{equation}\label{omega_zeta_subcrit_1}
 \norm{fg}_{H^r} \le c \norm{f}_{H^s} \norm{g}_{H^r} \text{ for all }g \in H^r(\Omega_\zeta).
\end{equation}
\end{lemma}
\begin{proof}
Define the linear map $T_f: L^1_{loc}(\Omega_\zeta) \to L^1_{loc}(\Omega_\zeta)$ via $T_f g = fg$.  Since $s > n/2$ we have that $H^s(\Omega_\zeta)$ is an algebra, and hence there is a constant $c = c(n,\norm{\zeta}_{C^{0,1}_b},s,r)>0$ such that 
\begin{equation}
 \norm{T_f g}_{H^s} \le c \norm{f}_{H^s} \norm{g}_{H^s} \text{ for all } g \in H^s(\Omega_\zeta).
\end{equation}
Similarly, since $H^s(\Omega_\zeta) \hookrightarrow C^0_b(\Omega_\zeta)$, we have that 
\begin{equation}
 \norm{T_f g}_{L^2} \le  \norm{f}_{C^0_b} \norm{g}_{L^2} \le   c \norm{f}_{H^s} \norm{g}_{L^2} \text{ for all } g \in L^2(\Omega_\zeta)
\end{equation}
for a constant $c = c(n,\norm{\zeta}_{C^{0,1}_b},s,r)>0$.  From these bounds we deduce that $T_f$ is a bounded linear operator on $L^2(\Omega_\zeta)$ and on $H^s(\Omega_\zeta)$.  By standard interpolation theory (see, for instance \cite{BL_1976,Leoni_2017,Triebel_1995}) we then have that $T_f$ is a bounded linear operator on $H^r(\Omega_\zeta)$ for all $0 < r < s$ and that the operator norm is bounded above by $c \norm{f}_{H^s}$ for a constant $c = c(n,\norm{\zeta}_{C^{0,1}_b},s,r)>0$.  The estimate \eqref{omega_zeta_subcrit_1} follows.

\end{proof}

The second is a full-space product estimate.

\begin{lemma}\label{prod_full_space}
Suppose that $n/2 < s \in \R$.  Then for $0 \le r \le s$ there exists a constant $c = c(r,s) >0$ such that 
\begin{equation}
 \norm{fg}_{H^r} \le c \norm{f}_{H^s} \norm{g}_{H^r} \text{ for all } f \in H^s(\R^n) \text{ and } g \in H^r(\R^n).
\end{equation}
\end{lemma}
\begin{proof}
The proof is a trivial modification of the proof of Lemma \ref{omega_zeta_subcrit} and is thus omitted.
\end{proof}

Finally, we record two results about the boundedness of simple lifting operators.  The first deals with sets of the form $\Omega_\zeta$.

\begin{lemma}\label{sobolev_slice_extension}
Let $\zeta \in C^{0,1}_b(\R^{n-1})$ be such that $\inf \zeta >0$.  For $0 \le s \in \R$ the map  $L_{\Omega_\zeta} : H^s(\R^{n-1};\R^n) \to H^s(\Omega_\zeta;\R^n)$ defined by $L_{\Omega_\zeta} f(x) = f(x')$ is bounded and linear.
\end{lemma}
\begin{proof}
The assertion is trivial for $s \in \N$, and the general case follows from these special cases and interpolation theory (see, for instance \cite{BL_1976,Leoni_2017,Triebel_1995}).   
\end{proof}

The second deals with the flat surface $\Sigma_b$.

\begin{lemma}\label{sobolev_slice_extension_surface}
Let $0 \le s \in \R$.  Define the map $S_b :  H^{s}(\R^{n-1} ; \R^{n\times n}_{\operatorname*{sym}}) \to H^{s}(\Sigma_b ; \R^{n\times n}_{\operatorname*{sym}})$ via $S_b T(x',b) = T(x')$.  Then $S_b$ is bounded and linear. 
\end{lemma}
\begin{proof}
This follows immediately from the fact that $\Sigma_b \ni (x',b) \mapsto x' \in \R^{n-1}$ is a smooth diffeomorphism.  
\end{proof}

\subsection{Difference quotients}

For $f : \R^d \to \R^m$, $1 \le j \le d$, and $h \in \R \backslash \{0\}$ we let the $j^{th}$ difference quotient of $f$ to be $\delta^j_h : \R^d \to X$ defined by 
\begin{equation}
 \delta_h^j f(x) = \frac{f(x+h e_j) - f(x)}{h}.
\end{equation}
Our next result provides a useful bound for this operator.

\begin{proposition}\label{diff_quote_fullspace}
Let $s \ge -1$.  Then for $1 \le j \le d$ and $h \in \R \backslash \{0\}$ we have that 
\begin{equation}
 \norm{\delta_h^j f}_{H^{s}} \le   \norm{\p_j f}_{H^{s}}  \le 2\pi \norm{f}_{H^{s+1}}.
\end{equation}
for all $f \in H^{s+1}(\R^d;\R^m)$.
\end{proposition}

\begin{proof}
Applying the Fourier transform, we find that 
\begin{equation}
 \widehat{\delta_h^j f}(\xi) = \frac{e^{2\pi i h e_j \cdot \xi} - 1}{h} \hat{f}(\xi)
\end{equation}
for all $\xi \in \R^d$ and $h \in \R \backslash \{0\}$.  From this and the simple inequality $\abs{e^{i\theta} -1} \le \abs{\theta}$ for $\theta \in \R$ we arrive at the estimate
\begin{equation}
 \abs{\widehat{\delta_h^j f}(\xi)} \le \frac{2 \pi \abs{h} \abs{\xi_j}}{\abs{h}} \abs{\hat{f}(\xi)} = \abs{2\pi i \xi_j \hat{f}(\xi)} = \abs{\widehat{\p_j f}(\xi)}.
\end{equation}
Hence
\begin{equation}
 \norm{\delta_h^j f}_{H^{s}}^2 = \int_{\R^d} (1+ \abs{\xi}^2)^s \abs{\widehat{\delta_h^j f}(\xi)}^2 \le   \int_{\R^d} (1+ \abs{\xi}^2)^s \abs{\widehat{\p_j f}(\xi)}^2 =    \norm{\p_j f}_{H^{s}}^2.
\end{equation}
To conclude we simply note that 
\begin{equation}
\int_{\R^d} (1+ \abs{\xi}^2)^s \abs{\widehat{\p_j f}(\xi)}^2 \le 4 \pi^2 \int_{\R^d} (1+ \abs{\xi}^2)^{s+1} \abs{\hat{f}(\xi)}^2 d\xi = 4 \pi^2 \norm{f}_{H^{s+1}}^2.
\end{equation}

\end{proof}

Next we examine difference quotients on $\Omega$.

\begin{corollary}\label{diff_quote_omega}
Let $\Omega = \R^{n-1} \times (0,b)$, $k \in \N$, and $f \in H^{k+1}(\Omega;\R^m)$.  If for $1 \le j \le n-1$ and $h \in \R \backslash \{0\}$  we define $\delta_h^j f : \Omega \to \R^m$ via
\begin{equation}
  \delta_h^j f(x) = \frac{f(x+h e_j) - f(x)}{h},
\end{equation}
then 
\begin{equation}
 \norm{\delta_h^j f}_{H^k} \le   \norm{\p_j f}_{H^k}.
\end{equation}
\end{corollary}
\begin{proof}
Let $\alpha \in \N^n$ be such that $\abs{\alpha} \le k$.  Then for almost every $x_n \in (0,b)$ we have that $\p^\alpha f(\cdot,x_n) \in H^1(\R^{n-1};\R^m)$, so we may apply Proposition \ref{diff_quote_fullspace} with $s=0$ to bound 
\begin{equation}
 \int_{\R^{n-1}} \abs{\p^\alpha \delta_h^j f(x',x_n)}^2 dx' =  \int_{\R^{n-1}} \abs{ \delta_h^j \p^\alpha f(x',x_n)}^2 dx'\le  \int_{\R^{n-1}} \abs{\p_j \p^\alpha f(x',x_n)}^2 dx'.
\end{equation}
The result then follows by integrating over $x_n \in (0,b)$, applying Tonelli's theorem, and summing over all such $\alpha$.
\end{proof}

\subsection{A smooth mapping}

Here we record an analog of Theorem \ref{omega_power_series} that is useful in dealing with the mean-curvature operator.

\begin{theorem}\label{sigma_power_series}
Let $r > d/2$.  Then there exists a constant $\delta = \delta(d,r)>0$ such that the map $\Gamma :  B_{H^{r}}(0, \delta)  \to H^r(\R^d;\R^d)$ given by 
\begin{equation}
 \Gamma(f) = \frac{f}{\sqrt{1+ \abs{f}^2 }} 
\end{equation}
is well-defined and smooth, where $B_{H^{r}}(0, \delta) \subset H^{r}(\R^d;\R^d)$ is the open ball of radius $\delta$.
\end{theorem}
\begin{proof}

First recall that since $r > d/2$ the standard theory of Sobolev spaces shows that $H^r(\R^d)$ is an algebra and that we have the continuous inclusion $H^r(\R^d) \hookrightarrow C^0_b(\R^d)$.  Consequently, we can choose a constant $c = c(d,r) >0$ such that $\norm{g}_{C^0_b} \le c \norm{g}_{H^r}$ and $\norm{g h}_{H^r} \le c \norm{g}_{H^r} \norm{h}_{H^r}$ for all $g,h \in H^r(\R^d)$.  In particular, for $f \in H^r(\R^d;\R^d)$ we have that
\begin{equation}
 \norm{\abs{f}^2 }_{H^r(\R^d)} \le \sum_{j=1}^d \norm{f_j^2}_{H^r(\R^d)}\le c \sum_{j=1}^d \norm{ f_j}_{H^r(\R^d)}^2  = c  \norm{f}_{H^{r}(\R^d;\R^d)}^2.
\end{equation}
Moreover, in any unital Banach algebra the power series 
\begin{equation}
 \sum_{k=0}^\infty \frac{(-1)^k (2k)!}{4^k (k!)^2} x^{k} = (1+x)^{-1/2}
\end{equation}
converges in the open unit ball and defines a smooth function there.  With these ingredients in hand we may then argue as in the proof of Theorem \ref{omega_power_series} (employing the unital Banach algebras $C^0_b(\R^d)$ and $\L(H^r(\R^d;\R^d))$) to conclude.

\end{proof}

\bibliographystyle{abbrv}
\bibliography{ns_traveling}

\end{document}